\documentclass[12pt,lot, lof]{puthesis}
\newcommand{\proquestmode}{}

\title{Efficient Representations of Signals in \\ Nonlinear~%
Signal~Processing \\ with Applications to Inverse Problems}

\submitted{September 2011}  % degree conferral date (January, April, June, September, or November)
\copyrightyear{2011}  % year in which the copyright is secured by publication of the dissertation.
\author{Eugene Brevdo}
\adviser{Peter J. Ramadge and Ingrid Daubechies}  %replace with the full name of your adviser
\department{Electrical Engineering}

\usepackage{amsmath}
\usepackage{amsthm}
\usepackage{amssymb}
\usepackage{caption} % required for subcaption
\usepackage{subcaption}
\usepackage{bm}
\usepackage{color}
\usepackage{array}
\usepackage{xr}
\usepackage{algorithmic}
\usepackage{algorithm}
\usepackage[stable]{footmisc} % Footnotes inside chapter/section headers

\usepackage{graphicx}

\usepackage{verbatim}

\usepackage{multirow}

\usepackage{longtable}

\usepackage{booktabs}

\ifdefined\printmode

\usepackage{url}

\else

\ifdefined\proquestmode
\usepackage{hyperref}
\hypersetup{bookmarksnumbered}

\makeatletter
\hypersetup{pdftitle=\@title,pdfauthor=\@author}
\makeatother

\else
\usepackage{hyperref}
\hypersetup{colorlinks,bookmarksnumbered}

\makeatletter
\hypersetup{pdftitle=\@title,pdfauthor=\@author}
\makeatother

\fi % proquest or online formatting
\fi % printed or online formatting

\usepackage{bm}
\usepackage{color}
\newcommand{\chattr}[1]{{%
\footnote{#1}
}}

\makeatletter \newcommand \listoftodos{\section*{Todo list} \@starttoc{tdo}}
\newcommand\l@todo[2]
  {\par\noindent \textit{#2}, \parbox{10cm}{#1}\par} \makeatother

\newcommand{\abs}[1]{\left\lvert#1\right\rvert}
\newcommand{\norm}[1]{\left\lVert#1\right\rVert}
\newcommand{\ceil}[1]{\left\lceil #1 \right\rceil}
\newcommand{\floor}[1]{\left\lfloor #1 \right\rfloor}
\newcommand{\ip}[2]{\left\langle {#1},{#2} \right\rangle}
\newcommand{\grad}[0]{\nabla}
\newcommand{\lap}[0]{\Delta}
\newcommand{\tr}[1]{\text{tr}\set{#1}}

\newcommand{\sinc}[0]{\text{sinc}}

\newcommand{\supp}[1]{\text{supp}\left(#1\right)}
\newcommand{\diag}[1]{\text{diag}\left(#1\right)}
\newcommand{\wh}[1]{\widehat{#1}}
\newcommand{\wt}[1]{\widetilde{#1}}
\newcommand{\ol}[1]{\overline{#1}}
\newcommand{\ul}[1]{\underline{#1}}

\newcommand{\pd}[2]{\frac{\partial #1}{\partial #2}}

\renewcommand*{\Im}[1]{\operatorname{\mathfrak{Im}}{\left( #1 \right)}}
\renewcommand*{\Re}[1]{\operatorname{\mathfrak{Re}}{\left( #1 \right)}}

\def \x     {{\times}}

\def \bbC     {{\mathbb C}}
\def \bbI     {{\mathbb I}}

\def \bbR     {{\mathbb R}}
\def \cA     {{\mathcal A}}
\def \cB     {{\mathcal B}}
\def \cC     {{\mathcal C}}
\def \cD     {{\mathcal D}}
\def \cE     {{\mathcal E}}
\def \cF     {{\mathcal F}}

\def \cH     {{\mathcal H}}
\def \cJ     {{\mathcal J}}
\def \cL     {{\mathcal L}}
\def \cM     {{\mathcal M}}
\def \cN     {{\mathcal N}}
\def \cP     {{\mathcal P}}
\def \cR     {{\mathcal R}}
\def \cS     {{\mathcal S}}

\def \cW     {{\mathcal W}}
\def \cX     {{\mathcal X}}
\def \cY     {{\mathcal Y}}
\def \sfN    {{\mathsf N}}
\def \sfS    {{\mathsf S}}

\def \bdM    {{\partial \cM}}

\def \cAp     {{\mathcal A'}}

\def \bn     {\bm{n}}
\def \bv     {\bm{v}}

\def \bp     {{\bar{p}}}
\def \bd     {{\bar{d}}}
\def \bcD     {{\bar{\cD}}}

\def \td     {{\tilde{d}}}

\def \tf     {{\tilde{f}}}
\def \tg     {{\tilde{g}}}
\def \tr     {{\tilde{r}}}
\def \tp     {{\tilde{p}}}

\def \tG     {{\wt{G}}}
\def \tT     {{\wt{T}}}
\def \tW     {{\wt{W}}}
\def \tnu     {{\wt{\nu}}}

\def \tcP    {{\wt{\cP}}}

\def \hI     {\hat{I}}

\def \Ae     {{A_\epsilon}}
\def \De     {{D_\epsilon}}
\def \Le     {{L_\epsilon}}
\def \Ne     {{N_\epsilon}}
\def \Pe     {{P_\epsilon}}
\def \tLe    {{\tilde{L}_\epsilon}}
\def \hDe    {{\hat{D}_\epsilon}}
\def \hPe    {{\hat{P}_\epsilon}}

\def \bs     {\backslash}

\newcommand{\set}[1]{\left\{#1\right\}}
\newcommand{\prn}[1]{\left(#1\right)}

\def \downto {{\downarrow}}

\DeclareMathOperator*{\Cut}{Cut}
\DeclareMathOperator*{\argmax}{argmax}
\DeclareMathOperator*{\argmin}{argmin}
\DeclareMathOperator*{\sign}{sign}

\DeclareMathOperator*{\med}{med}
\DeclareMathOperator*{\maximize}{maximize}

\newcommand{\beq}{\begin{equation}}
\newcommand{\eeq}{\end{equation}}

\newcommand{\pmat}[1]{\begin{pmatrix}#1\end{pmatrix}}

\ifx \nodefthm \undefined
  \usepackage{amsthm}
  \theoremstyle{plain}% default
  \newtheorem{thm}{Theorem}[section]
  \newtheorem{lem}[thm]{Lemma}
  \newtheorem{prop}[thm]{Proposition}
  \newtheorem{cor}[thm]{Corollary}
  \theoremstyle{definition}
  \newtheorem{defn}{Definition}[section]

  \newtheorem{exmp}{Example}[section]
  \theoremstyle{remark}
  \newtheorem*{rem}{Remark}
  
  \newtheorem{case}{Case}
\else
  \ifx \newtheorem \undefined
  \else
    \newtheorem{thm}[theorem]{Theorem}
    \newtheorem{lem}[lemma]{Lemma}
    \newtheorem{prop}[theorem]{Proposition}
    \newtheorem{cor}[theorem]{Corollary}
    \newtheorem{defn}[definition]{Definition}
    \newtheorem{exmp}{Example}[theorem]

  \fi
\fi

\ifodd 0
\renewcommand{\maketitlepage}{}
\renewcommand*{\makecopyrightpage}{}
\renewcommand*{\makeabstract}{}

\else

\abstract{
The focus of this thesis is the construction and analysis of
efficient representations in nonlinear signal processing, and the
applications of these structures to inverse problems in a variety of
fields.  The work is composed of three major sections, each associated
with a different form of data:

\begin{itemize}
\item Regression and Distance Estimation on Graphs and Riemannian Manifolds.
\item Instantaneous Time-Frequency Analysis via Synchrosqueezing.
\item Multiscale Dictionaries of Slepian Functions on the Sphere.
\end{itemize}

}

\acknowledgements{
The work in this thesis is only possible due to the continual support
of many wonderful people, whom I feel privileged to thank here.

My advisors, Peter~Ramadge and Ingrid~Daubechies, have
had a particularly important role in my direction over the last six
years.  I thank Peter for teaching me to be direct,
careful, and professional in my work and my career; these lessons
will stay with me always.  I thank Ingrid for her warmness, her
excitement (about everything!), and her friendship over these years.
I thank them both for their flexibility in allowing me to explore my
interests throughout my PhD, and for the many discussions
that have directed the development, the tone, and the substance of my
work.

I owe a great deal to my collaborators, with whom it has been a great
pleasure to work, and who I am lucky to call friends:
Shannon~Hughes, Rodolfo~R\'{i}os-Zertuche, Amit~Singer, Hau-Tieng~Wu,
Gaurav~Thakur, Neven~S.~Fu\v{c}kar, and Frederik~J.~Simons.  Without
their energy and expertise, this thesis would not have been possible.

Many people have changed my life over the last six years, and I
especially want to thank them here.  Carol, for her support and
guidance.  My Magie open-door roommates, more like brothers: Abhishek,
Aman, Deep, and Sushobhan.  Adam, Reza, and Dano: always far
and always near.  Arvind, and Theodor, roommates extraordinaire!  My
Princeton Anonymous group: Stephanie, Yury, Lorne, Kevin, Daniel,
Jeff, Matt, Ana, CJ, DJ, Seth, Manos, Kostas, Silvana, Jessbee, Katie,
Sibren, Tanushree, Sharon, and Sahar.  My fellow Ramadge
groupies: Mert~Rory, Bryan, Taylor, James, Alex, David, and Hao.

Most importantly, I thank my parents and my sister Marina for their calm
and unwavering confidence in my ability to always achieve my goals.

And Taya, whose support, encouragement, and love throughout has been
nothing short of miraculous.  Thank you.

}

\fi  % disable frontmatter

\begin{document}

\makefrontmatter

% If you've disabled frontmatter, you can insert the toc manually
%\tableofcontents\clearpage

% \include lets us split up the document (and each include starts a new page):
\chapter{Introduction\label{ch:intro}}

This chapter introduces the main topics covered in this thesis, and
the major tools and ideas that link these topics.  The discussion is
philosophical and purely motivational, and should not be considered
mathematically rigorous or exhaustive.  Chapters \ref{ch:npdr},
\ref{ch:rl}, \ref{ch:ss}, and \ref{ch:sltr} contain their own relevant
introductory discussions, descriptions of prior work, and literature
reviews.

An inverse problem is a general question of the following form:

\begin{center}
\framebox[\linewidth]{\parbox{.92\linewidth}{%
You are given observed measurements $Y=F(X)$, where $F : D \to R$,
is a mapping from domain $D$ to range $R$.  Provide an estimate of
a physical object or underlying set of parameters~$X$.
}}
\end{center}

\;

Inverse problems encompass a variety of types of questions in
many fields, from geophysics and medical imaging to
computer graphics.  In a high level sense, all of the following
questions are inverse problems:
\begin{itemize}
\item Given the output of a CT scanner (sinogram), reconstruct the 3D
  volumetric image of the patient's chest.
\item Given noisy and limited satellite observations of the magnetic
  field over Australia, estimate the magnetic field everywhere on the
  continent.
\item Given the lightcurve of a periodic dwarf star, extract the
  significant slowly time-varying modes of oscillation.  Determine
  which modes are associated with the rotation of the star, which are
  caused by transiting giant planets, and which are caused by one or
  more earth-sized planets.
\item Given a spline curve corrupted by additive Gaussian noise,
  identify its order, knot positions, and the parameters at each knot.
\item Given a social network graph with a variety of both known and
  unknown parameters at each node, and given a sample of
  known zip codes on a subset of the graph,
  estimate the zip codes of the rest of the nodes.
%\item Given a triangulation of a mesh, and a set of nodes, build a
%  Voronoi diagram for the given set of nodes.
\end{itemize}

Of the questions above, only the first two are traditionally
considered to be inverse problems because they have a clear
physical system $F(\cdot)$ mapping inputs to outputs, and the domain
and range are also clear from the physics of the problem.
Nevertheless, all of these problems can
be formulated as general inverse problems: a domain and range are
given, as are observations, and the simplest, most accurate, or
most statistically likely parameters to match the observations are
required.
This basic premise encourages the ``borrowing'' and mixing
of ideas from many fields of mathematics and engineering: graph
theory, statistics, functional analysis, differential equations, and
geometry.

This thesis focuses on finding representations of the data $X$, or
of the underlying transform $F(\cdot)$, that lead to either simpler or
more robust ways to solve particular inverse problems, or that provide
a new perspective for existing solutions.

We especially focus on finding representations within domains $D$
which are important in fields such as medical imaging and geoscience,
or are becoming important in the signal processing community due to
the recent explosion of ``high dimensional'' data sets.  These domains are:
\begin{itemize}
\item General compact Riemannian manifolds.
\item Large graphs (especially nearest neighbor graphs from point samples of
  Riemannian manifolds embedded in $\bbR^p$).
\item The Sphere $S^2$.
\end{itemize}
A short technical description of Riemannian manifolds, and a set of
references, is given in App.~\ref{app:diffgeom}.  A description of
(Nearest Neighbor) graphs can be found in \S\ref{sec:prelim}.

The central tool that brings together all of the problems,
constructions, and insights in our work is the Laplacian.
In a way, the thrust of this thesis is the importance and flexibility
of the Laplacian as a data analysis and problem solving tool.  Roughly
speaking, the Laplacian is defined as follows:

\begin{center}
\framebox[\linewidth]{\parbox{.9\linewidth}{
Given a domain $D$, with measure $\mu$ and knowing for any
  point $x \in D$ its neighborhood $N(x)$, the Laplacian $L$ is a
  linear operator that, for any well defined
  function $f(x)$, subtracts some weighted average of $f(\cdot)$ in
  $N(x)$ from its value at~$x$: 
  $$[Lf](x) = f(x) - \int_{x' \in N(x)} w(x,x') f(x') d\mu(x'),$$
  where the weights are determined by the intrinsic relationship
  (e.g., distance) between $x$ and $x'$ on $D$.
}}
\end{center}

\;

The Laplacian is a powerful tool because it synthesizes local
information at each point $x$ in domain $D$ into global information
about the domain; and this in turn can be used to analyze functions
and other operators on this domain.  This statement is formalized in
the classical language of Fourier analysis and the recent work of
Mikhail Belkin, Ronald Coifman, Peter Jones, and~Amit~Singer~%
\cite{Belkin2008th,Coifman2006,Jones2008,Singer2011}:
the eigenfunctions of the Laplacian are a powerful tool for
both the analysis of points in $D$ and functions defined on $D$.

Throughout this thesis, we use and study a variety of Laplacians, each
for a slightly different purpose:
\begin{itemize}
\item The weighted graph Laplacian: its ``averaging'' component in
  Chapter~\ref{ch:npdr} and its regularized inverse in Chapters~%
  \ref{ch:rl}~and~\ref{ch:rlapp}.
\item The Laplace-Beltrami on a compact Riemannian manifold in
  Chapters~\ref{ch:rl}~and~\ref{ch:rlapp}.
\item The Fourier transform (projections on the eigenfunctions of the
  Laplacian~on~$\bbR$) in Chapters~\ref{ch:ss}~and~%
  \ref{ch:ssapp}.
\item Spherical harmonics (eigenfunctions of the Laplace-Beltrami
  operator~on~$S^2$) and their use in analyzing bandlimited functions
  on the sphere in Chapter \ref{ch:sltr}.
\end{itemize}
The definition and properties of the weighted graph Laplacian are
given in \S\ref{sec:ssllimit}.  A definition of the Laplacian for a
compact Riemannian manifold is given in App.~\ref{app:diffgeom}, with
specific definitions for the real line and the sphere in
App.~\ref{app:fourier}.  This appendix also contains a
comprehensive discussion of Fourier analysis on Riemannian manifolds,
and on the Sphere in particular.

Though the definition of the Laplacian differs depending on the
underlying domain $D$, all Laplacians are intricately linked: the
Laplacian on a domain $D$ converges to that on $D'$ as
the former domain converges to the latter, or when one is a special
case of the other.  We study these relationships when they are relevant
(e.g.,~in~\S\ref{sec:ssllimit}).

The term Nonlinear in the title of this thesis refers to the ways in
which our constructions use properties of the Laplacian, or of its
eigenfunctions, to represent~data:
\begin{itemize}
\item In Chapter \ref{ch:npdr}, the regularized inverse of the graph
  Laplacian is used to denoise graphs (this is not directly clear from
  the context but see \S\ref{sec:geonpdr} for a discussion).
\item In Chapters \ref{ch:rl} and \ref{ch:rlapp}, the regularized
  inverse of both the graph Laplacian, and the Laplace Beltrami
  operator, are analyzed; and a special ``trick'' of taking the
  nonlinear logarithm is used to both perform regression and to
  estimate geodesic~distances.
\item In Chapters \ref{ch:ss} and \ref{ch:ssapp}, the Fourier
  transform of the Wavelet representation of harmonic signals is used
  to motivate a time localized estimate of amplitude and frequency.
  This is followed up by a nonlinear reassignment
  of~the~Wavelet~representation.
\item In Chapter \ref{ch:sltr}, we construct a multi-scale dictionary
  composed of bandlimited Slepian functions on the sphere (which are
  in turn constructed via Fourier analysis).  This dictionary is then
  combined with nonlinear ($\ell_1$) methods for signal~estimation.
\end{itemize}

We hope that the underlying theme of this thesis is now clear: as the
movie \emph{Manhattan} is Woody Allen's ode to the city of New
York, the work herein is a testament to the flexibility and power
of the Laplacian.

Without further ado, we now describe in detail the focus of the
individual chapters of this thesis.

\section{Denoising and Inference on Graphs}
Many unsupervised and semi-supervised learning problems contain
relationships between sample points that can be modeled with a
graph.  As a result, the weighted adjacency matrix of a graph, and
the associated graph Laplacian, are often used to solve such
problems.  More specifically, the spectrum of the graph Laplacian,
and its regularized inverse, can both be used to determine
relationships between observed data points (such as
``neighborliness'' or ``connectedness''), and to perform regression
when a partial labeling of the data points is available.  Chapters
\ref{ch:npdr}, \ref{ch:rl}, and \ref{ch:rlapp} study the inverse
regularized graph Laplacian.

Chapter \ref{ch:npdr} focuses on the unsupervised problem of detecting
``bad'' edges, or bridges, in a graph constructed with erroneous
edges.  Such graphs arise, e.g., as nearest neighbor graphs from
high dimensional points sampled under noise.  A novel bridge detection
rule is constructed, based on Markov random walks with restarts,
that robustly identifies bridges.  The detection rule uses the
regularized inverse of the graph's Laplacian matrix, and its
structure can be analyzed from a geometric point of view under certain
assumptions of the underlying sample points.  We compare this
detection rule to past work and show its improved performance
as a preprocessing step in the estimation of geodesic distances on the
underlying graph, a global estimation problem.  We also show its superior
performance as a preprocessing tool when solving the random projection
computational tomography inverse problem.

Chapter \ref{ch:rl} studies a closely related problem, that of
performing regression on points sampled from a high dimensional space,
only some of which are labeled.  We focus on the common case when the
regression is performed via the nearest neighbor graph of the points,
with ridge and Laplacian regularization.
This common solution approach reduces to a
matrix-vector product, where the matrix is the regularized inverse of
the graph Laplacian, and the vector contains the partially known label
information.

In this chapter, we focus on the geometric aspects of
the problem. First, we prove that in the noiseless, low
regularization case, when the points are sampled from a smooth,
compact, Riemannian manifold, the matrix-vector product converges to a
sampling of the solution to a elliptic PDE. We use the theory of
viscosity solutions to show that in the low regularization case, the
solution of this PDE encodes geodesic distances between labeled points
and unlabeled ones.  This geometric PDE framework provides key
insights into the original semisupervised regression problem, and into
the regularized inverse of the graph Laplacian matrix in general.

Chapter \ref{ch:rlapp} follows on the theoretical analysis in Chapter
\ref{ch:rl} by displaying a wide variety of applications for the
Regularized Laplacian PDE framework.  The contributions of this
chapter include:
\begin{itemize}
\item A new consistent geodesics distance estimator on Riemannian
  manifolds, whose complexity depends only on the number of sample
  points $n$, rather than the ambient dimension $p$ or the manifold
  dimension $d$.
\item A new multi-class classifier, applicable to high-dimensional
  semi-supervised learning problems.
\item New explanations for negative results in the machine learning
  literature associated with the graph Laplacian.
\item A new dimensionality reduction algorithm called Viscous ISOMAP.
\item A new and satisfying interpretation for the bridge detection
  algorithm constructed in Chapter \ref{ch:npdr}.
\end{itemize}

\section{Instantaneous Time-Frequency Analysis}

Time frequency analysis in the form of the Fourier transform, short
time Fourier transform (STFT), and their discrete variants (e.g. the power
spectrum), have long been standard tools in science, engineering, and
mathematical analysis.  Recent advances in time-frequency
reassignment, wherein energies in the magnitude plot of the STFT are
shifted in the time-frequency plane, have found application in
``sharpening'' images of, e.g., STFT magnitude plots, and have been
used to perform ridge detection and other types of time- and
frequency-localized feature detection in time series.

Chapters \ref{ch:ss} and \ref{ch:ssapp} focus on a novel
time-frequency reassignment transform, Synchrosqueezing, which can be
constructed to work ``on top of'' many invertible transforms (e.g.,
the Wavelet transform or the STFT).  As Synchrosqueezing is itself an
invertible transform, it can be used to filter and denoise
signals.  Most importantly, it can be used to extract
individual components from superpositions of ``quasi-harmonic''
functions: functions that take the form $A(t) e^{i \phi(t)}$, where
$A(t)$ and $d\phi(t)/dt$ are slowly varying.  

Chapter \ref{ch:ss} focuses on two aspects of Synchrosqueezing.
First, we develop a fast new numerical implementation of Wavelet-based 
Synchrosqueezing.  This implementation, and other useful utilities,
have been included in the Synchrosqueezing MATLAB toolbox.  Second, 
we present a stability theorem, showing that Synchrosqueezing
can extract components from superpositions of quasi-harmonic signals
when the original observations have been corrupted by 
bounded perturbations (of the form often encountered
during the pre-processing of signals).

Chapter \ref{ch:ssapp} builds upon the work of the previous
chapter to develop novel applications of Synchrosqueezing.  We present
a wide variety of problem domains in which the Synchrosqueezing
transform is a powerful tool for denoising, feature extraction, and
more general scientific analysis.  We especially focus on estimation
in inverse problems in medical signal processing and the geosciences.
Contributions include: 
\begin{itemize}
\item The extraction of respiratory signals from ECG
  (Electrocardiogram) signals.
%\item The estimation of physiologically relevant T-peak and T-end
%  times from ECG signals of both healthy patients and those with
%  Atrial Fibrillation (AF). These estimates are shown to be more
%  precise than those of the ``standard'' state-of-the-art estimator.
\item Precise new analyses of paleoclimate simulations (solar
  insolation models), individual paleoclimate proxies, and proxy
  stacks in the last 2.5 Myr.  These results are compared to
  Wavelet- and STFT-based analyses, which are less precise and harder to
  interpret.
%\item Modern climate analyses, finding teleconnections in sea surface
%  temperatures in the last 250 yr.
\end{itemize}

\section{Multiscale Dictionaries of Slepian Functions on the Sphere}

Just as audio signals and images are constrained by the physical
processes that generate them, and by the sensors that observe them, so
too are many geophysical and cosmological signals, which reside on the
sphere $S^2$.  On the real line and in the plane, both physical
constraints and sampling constraints lead to assumptions of a
bandlimit: that a signal contains zero energy outside some supporting
region in the frequency domain.  Similarly, bandlimited signals on the
sphere are zero outside the low-frequency spherical harmonic
components.

In Chapter \ref{ch:sltr}, following up on Claude Shannon's initial
investigations into sampling, we describe Slepian, Landau, and
Pollak's spatial concentration problem on subsets of the real line,
and the resulting Slepian functions.  The construction of Slepian
functions have led to many important modern algorithms for the
inversion of bandlimited signals from their samples within an
interval, and especially Thompson's multitaper spectral estimator.  We
then study Simons and Dahlen's extension of these results to subsets
of the Sphere, where now the definition of frequency and bandlimit has
been appropriately modified.

Building upon these results, we develop an algorithm for the
construction of dictionary elements that are bandlimited, multiscale,
and localized.  Our algorithm is based on a subdivision scheme that
constructs a binary tree from subdivisions of the region of interest
(ROI).  We show, via numerous examples, that this dictionary has many
nice properties: it closely overlaps with the most concentrated
Slepian functions on the ROI, and most element pairs have low
coherence.

The focus of this construction is to solve ill-posed inverse problems
in geophysics and cosmology.  Though the new dictionary is no longer
composed of purely orthogonal elements like the Slepian basis, it can
be combined with modern inversion techniques that promote sparsity in
the solution, to provide significantly lower residual error after
reconstruction (as compared to classically optimal Slepian inversion
techniques).

We provide additional numerical results showing the solution path that
these techniques take when combined with the multiscale dictionaries,
and their efficacy on a standard model of the Earth's magnetic field,
POMME-4.  Finally, we show via randomized trials that the combination
of the multiscale construction and $\ell_1$-based estimation provides
significant improvement, over the current state of the art, in the
inversion of bandlimited white and pink random processes within
subsets of the sphere.

\chapter{Graph Bridge Detection via Random Walks%
\chattr{This chapter is based on work in collaboration with
Peter~J.~Ramadge, Department of Electrical Engineering,
Princeton~University.  A preliminary version appears in
\cite{Brevdo2010}.}}
\label{ch:npdr}

\section{Introduction}
Many new problems in machine learning and signal processing require
the robust estimation of geodesic distances between nodes of a nearest
neighbors (NN) graph.  For example, when the nodes represent points
sampled from a manifold, estimating feature space distances between
these points can be an important step in unsupervised
\cite{Tenenbaum2000} and semi-supervised \cite{Belkin2006} learning.
This problem often reduces to that of having accurate estimates of
each point's neighbors, as described below.

In the simplest approach to estimating geodesic distances,
the NN graph's edges are estimated from either the $k$ nearest
neighbors around each point, or from all of the neighbors
within an ambient (Euclidean) $\delta$-ball around each point.  Each
graph edge is then assigned a weight: the ambient distance between its
nodes.  A graph shortest path (SP) algorithm, e.g. Dijkstra's
\cite[\S24.3]{Cormen2001}, is then used to estimate geodesic distances
between pairs of points.

When the manifold is sampled with noise, or contains outliers,
bridges (short circuits between distant parts of the manifold) can
appear in the NN graph and this has a catastrophic effect on geodesics
estimation \cite{Balasubramanian2002}.

In this chapter, we develop a new approach for calculating point
``neighborliness'' from the NN graph.  This approach allows the robust
removal of bridges from NN graphs of manifolds sampled with noise.
This metric, which we call ``neighbor probability,'' is based on a
Markov Random walk with independent restarts.  The bridge decision
rule based on this metric is called the neighbor probability decision
rule (NPDR), and reduces to removing edges from the NN graph whose
neighbor probability is below a threshold.  We study some of the NPDR's
geometric properties when the number of samples grow large.  We also
compare the efficacy of the NPDR to other decision rules and show its
superior performance on removing bridges in simulated data, and in
the novel inverse problem of computational tomography with random
(and unknown) projection angles.

\section{Preliminaries}
\label{sec:prelim}

Let $\cX=\set{x_i}_{i=1}^n$ be nonuniformly sampled points from
manifold $\cM \subset \bbR^r$. We observe
$\cY=\set{y_i = x_i + \nu_i}_{i=1}^n$, where $\nu_i$ is noise.
A nearest neighbor (NN) graph $G = (\cY,\cE,d)$ is constructed 
from $k$-NN or $\delta$-ball neighborhoods of $\cY$
with the scale ($k$ or $\delta$)  chosen via cross-validation or
prior knowledge.  The map $d \colon \cE \to \bbR$ assigns cost $d_e =
\norm{x_k-x_l}_2$ to edge $e=(k,l) \in \cE$.  Let $\cD = \set{d_e : e \in
\cE}$.  The set $\cE$ gives initial estimates of neighbors
on the manifold. Let $\cF_k$ denote the neighbors of $x_k$~in~$G$.

In \cite{Tenenbaum2000} the geodesic distance between $(i,j) \in \cY^2$
is estimated by $\hat{g}_{ij} = \sum_{e \in \cP_{ij}} d_e$
where $\cP_{ij}$ is a minimum cost path from $i$ to $j$ in $G$
(this can be calculated via Dijkstra's algorithm).
When there is no noise, this estimate converges to the true geodesic
distance on $\cM$ as $n \to \infty$ and neighborhood size $\delta \to 0$.
However, in the presence of noise bridges form in the NN graph
and this results in significant estimation error. Forming
the shortest path in $G$ is too greedy in the presence of~bridges.

If bridges could be detected, their anomalous effect could be
removed without disconnecting the graph by substituting
a surrogate weight:
$\td_e = d_e + M, e \in \cB$ where $\cB$ is the set of detected
bridges and $M=n (\max_{e \in \cE} d_e)$, larger than the diameter of
$G$, is a penalty.
Let $\tG = (\cY,\cE,\td)$ and $\tcP_{ij}$ be
a minimum cost path between $i$ and $j$ in $\tG$.  The adjusted estimate
of geodesic distance is $\tg_{ij} = \sum_{e \in \tcP_{ij}} d_e$.

With this in mind, we first review some bridge detection methods, and
discuss recent theoretical work in random walks on graphs.

\section{Prior Work}

The greedy nature of the SP solution encourages the traversal
of bridges, thereby significantly underestimating geodesic distances.
Previous work has considered denoising the nearest neighbors
graph via rejection of edges based on local distance statistics
\cite{Chen2006,Singer2009}, or via local
tangent space estimation \cite{Li2008, CHANG2006}.  However, unlike
the method we propose (NPDR), these methods use local rather
that global statistics. We have found that using only local statistics
can be unreliable. For example, with state of the art robust
estimators of the local tangent space (as in \cite{Subbarao2006a}), local
rejection of neighborhood edges is not reliable with moderate noise or
outliers. Furthermore, edge removal (pruning) based on local edge
length statistics is based on questionable assumptions. For example, a
thin chain of outliers can form a bridge without unusually long edge
lengths.

As an example, we first describe the simplest class of bridge decision
rules (DRs): ones that classify bridges by a threshold on edge
length.  We call this the length decision rule (LDR); it is similar to the DR
of \cite{Chen2006}.  It is calculated with the following steps:
\begin{enumerate}
\item Normalize edge lengths for local sampling density by setting
$$\bd_{kl} = \frac{d_{kl}}{\sqrt{d_{k\cdot}d_{l\cdot}}},$$
where $d_{k\cdot} =\sum_{m \in \cF_k} d_{km}$ sums outgoing edge lengths from
$x_k$. 
\item Let $\bcD = \set{\bd_e : e \in \cE}$.
Select a ``good edge percentage'' $0 < q < 1$ (e.g. $99\%$) and
calculate the detected bridge set $\cB$ via:
$$\cB = \set{e \in \cE \colon  \bd_e \geq Q(\bcD,q)},$$
where $Q(\cD,q)$ is the $q$-th quantile of the set $\cD$. 
\end{enumerate}

The second decision rule, Jaccard similarity DR (JDR),
classifies bridges as edges between points with dissimilar
neighborhoods \cite{Singer2009}.  As opposed to the LDR,
the JDR uses information from immediate neighbors of two points to
detect bridges:
\begin{enumerate}
\item The Jaccard similarity between the neighborhoods of $x_l$ and
  $x_m$ is
$$j_{lm} = \frac{\abs{\cF_l \cap \cF_m}}{\abs{\cF_l \cup \cF_m}}.$$
\item Let $\cJ = \set{j_e : e \in \cE}$ be the set of Jaccard
  similarities.  Select a $q \in (0,1)$; the estimated bridge set is
$$\cB = \set{e \in \cE \colon j_e < Q(\cJ,1-q)}.$$
\end{enumerate}

We now describe a more global neighborliness metric.  The main
motivation is that bridges are short cuts for shortest paths. This
suggests detecting bridges by counting the traversals of each edge by
estimated geodesics.  This is the concept of edge centrality
(\emph{edge betweenness}) in networks \cite{Newman2004}.
In a network, the centrality of edge $e$ is
$$BE(e) = \sum_{(i,j) \in \cY^2} \bbI(e \in \cP_{ij}).$$
Edge centrality can be calculated in $O(n^2 \log n)$ time and
$O(n + \abs{\cE})$ space using algorithms of Newman or Brandes
\cite{Newman2004,Brandes2001}.
However, caution is required in using edge centrality for our purpose.
Consider a bridge $(a,b) \in \cE$ having high centrality. Suppose
there exists a point $y_c$ with $(a,c), (c,b)\in \cE$ such that
$d_{ac}+d_{cb} < d_{ab} + \delta'$,
$\delta'$ small. These edges are never preferred over $(a,b)$ in a
geodesic path, hence have low centrality.
However, once $(a,b)$ is placed into $\cB$ and given increased weight,
$(a,c),(c,b)$ reveals itself as a secondary bridge in $\tG$.
So detection by centrality must be done in
rounds, each adding edges to $\cB$. This allows
secondary bridges to be detected in subsequent rounds.
We now describe the Edge Centrality Decision Rule (ECDR):

Select quantile $q \in (0,1)$ and iterate the following steps $K$
times on $\tG$:
\begin{enumerate}
\item Calculate $BE(e)$ for each edge $e$.
\item Place $(1-q) n/K$ of the most central edges into $\cB$ and
update $\tG$.
\end{enumerate}

The result is a bridge set $\cB$ containing approximately $\lfloor
(1-q) n \rfloor$ edges, matching the $q$-th quantile sets of the previous
DRs.  The iteration count parameter $K$ trades off between
computational complexity (higher $K$ implies more iterations of edge
centrality estimation) and robustness (it also more likely to detect
bridges).  To our knowledge, the use of centrality as a bridge
detector is new.  While an improvement over LDR, the deterministic
greedy underpinnings of ECDR are a limitation: it initially fails to
see secondary bridges, and may also misclassify true high centrality
edges as bridges, e.g. the narrow path in the dumbbell manifold
\cite{Coifman2005}.

The Diffusion Maps approach to estimating feature space distances
\cite{Coifman2005}, has experimentally exhibited robustness to noise,
outliers, and finite sampling.  Diffusion distances, based on random
walks on the NN graph, are closely related to ``natural'' distances
(\emph{commute times}) on a manifold \cite{Ham2004}.  Furthermore,
Diffusion Maps coordinates (based on these distances) converge to
eigenfunctions of the Laplace Beltrami operator on the underlying
manifold.

The neighbor probability metric we construct is a global measure of edge
reliability based on diffusion distances.  The NPDR, based on this
metric, is then used to inform geodesic estimates.

%% The remainder of the chapter is organized as follows. \S\ref{sec:prelim}
%% introduces notation and describes the basic framework of surrogate
%% distances in geodesics estimation. We also describe methods for bridge
%% detection that will be used as benchmarks.
%% \S\ref{sec:npdr} covers the main contribution of this chapter: a novel
%% method for bridge detection based on random walks. We show
%% its improved performance experimentally in \S\ref{sec:experiments}.

\section{Neighbor Probability Decision Rule}
\label{sec:npdr}

We now propose a DR based on a Markov random walk that assigns
a probability that two points are neighbors.
This steps back from immediately looking for a shortest path
and instead lets a random walk ``explore'' the NN graph.
To this end, let $P$ be a row-stochastic matrix
with $P_{ij}=0$ if $i\neq j$ and $(i,j)\notin \cE$. Let $p \in
(0,1)$ be a parameter and $\bp = 1-p$.
Consider a random walk $s(m)$, $m = 0,1,\ldots$, on $G$ starting at $s(0)=i$
and governed by $P$. For each $t\geq 0$, with probability $p$ we stop the walk
and declare $s(m)$ a neighbor of $i$.
Let $N$ be the matrix of probabilities that the walk starting at node $i$,
stops at node $j$. The $i$-th row of $N$ is the neighbor distribution
of node $i$. This distribution can be calculated:
%% Since stopping at each time $t$ is mutually exclusive:
%% $$
%% N =\sum_{t\geq 0} P^t p(1-p)^t = p (I - \bp P)^{-1}
%% $$
\begin{align}
\label{eq:pNN1} N_{ij} &= \sum_{m=0}^\infty \Pr(\text{stop at $m$}, y^{(m)}=y_j|y^{(0)}=y_i) \\
\label{eq:pNN2} &= \sum_{m=0}^\infty \Pr(\text{stop at $m$}) \Pr(y^{(m)}=y_j|y^{(0)}=y_i) \\
\label{eq:pNN3} &= \sum_{m=0}^\infty \Pr(\text{stop at $m$}) \left( P^m \right)_{ij} \\
\label{eq:pNN4} &= \sum_{m=0}^\infty p (1-p)^m \left( P^m \right)_{ij}.
\end{align}
In \eqref{eq:pNN1} we decompose the stopping event into the disjoint
events of stopping at time $t$.  In \eqref{eq:pNN2} we separate the
independent events of stopping from being at the current
position $y^{(m)}$.  In \eqref{eq:pNN3} we use the well known Markov
property $P_{ij}^{(m)} = \left( P^m \right)_{ij}$ (with $P^{(0)} =I$).
Finally, in \eqref{eq:pNN4} we recall that stopping at time $m$
means we choose not to stop, independently, for each $m'=0,\ldots,m-1$,
and finally stop at time $m$, so the probability of this event is
$(1-p)^m p$.  Thus, the neighbor probabilities are:
$$
N = \sum_{t\geq 0} p(1-p)^m P^m = p (I - \bp P)^{-1}.
$$
A smaller stopping probability $p$ induces greater weighting of long-time
diffusion effects, which are more dependent on the topology of $\cM$.
%(as seen in \S\ref{sec:experiments}, where a small $p$ leads
%to a more robust DR)
$N$ is closely related to
% the inner product in
the \emph{normalized commute time} of \cite{Zhou2004}.
Its computation requires $O(n^3)$ time and $O(n^2)$ space ($P$ is sparse but $N$ may
not be).
%Let $w_p(z) = p (1 - \bp z)^{-1}$.  
%We now prove some interesting facts about $N$.

The neighborhood probability matrix $N$ is heavily dependent on the
choice of underlying random walk matrix $P$.  First we relate several
key properties of $N$ to those of $P$, then we introduce a geometrically
intuitive choice of $P$ when the sample points are sampled from a
manifold.

\begin{lem}\label{lem:propN}
$N$ is row-stochastic, shares
the left and right eigenvectors of $P$, and has
spectrum $\sigma(N)=\{p(1-\bp \lambda)^{-1}$, $\lambda\in\sigma(P)\}$.
\end{lem}
\begin{proof}
Clearly, $N^{-1}=(I-\bp P)/p$ is row-stochastic, shares the left and right eigenvectors of $P$ and has
spectrum $\{(1-\bp\lambda)/p, \lambda\in\sigma(P)\}$. 
\end{proof}

Following \cite{Coifman2005}, we select $P$ to be the popular
Diffusion Maps kernel on $G$:
\begin{enumerate}
\item Choose $\epsilon > 0$ (e.g. $\epsilon = \med_{e\in\cE}d_e/2$)
and let $(\hPe)_{lk} =
\exp \left(-d_{kl}^2 / \epsilon \right)$ if $(l,k) \in \cE$,
$1$ if $l=k$, and $0$ otherwise.
\item Let $\hDe$ be diagonal
with $(\hDe)_{kk} = \sum_l (\hPe)_{kl}$, and normalize for sampling density
by setting $\Ae = \hDe^{-1} \hPe \hDe^{-1}$.
% $\Ae$ is symmetric.
\item Let $\De$ be diagonal with $(\De)_{kk} = \sum_l (\Ae)_{kl}$
and set $\Pe = \De^{-1}\Ae$.
\end{enumerate}
$\Pe$ is row-stochastic and, as is well known, has bounded spectrum:

\begin{lem}\label{lem:eigP}
The spectrum of $\Pe$ is contained in $[0,1]$.
\end{lem}
\begin{proof}
$\hPe$, $\Ae$ and $S=\De^{-1/2}\Ae\De^{-1/2}$, are symmetric PSD.
$\Pe = \De^{-1} \Ae$ is row-stochastic and similar to $S$. 
\end{proof}

\noindent We define the Neighbor Probability Decision Rule (NPDR) as follows.
\begin{enumerate}
\item Let $\Ne=p(I-\bp \Pe)^{-1}$ and find the restriction $\cN =
  \set{(\Ne)_e : e \in \cE}$ of $\Ne$ to $\cE$.
\item Choose a $q \in (0,1)$ and let $\cB = \set{e \in
\cE : (\Ne)_e < Q(\cN,1-q)}$, i.e., edges connecting nodes with a
low probability of being neighbors, are bridges.
\end{enumerate}
Calculating $\Ne$ can be prohibitive for very large datasets.
Fortunately, we can effectively order the elements of $\cN$
using a low rank approximation to $\Ne$.

By Lemmas \ref{lem:propN}, \ref{lem:eigP}, we calculate
$N_{lm} \approx \sum_{j=1}^J
p(1-\bp\lambda_{j})^{-1}(v^R_j)_l(v^L_j)_m$ for $(l,m) \in \cE$,
where $\lambda_j$,$v^R_j$ and $v^L_j$ are the $J$ largest eigenvalues
of $\Pe$ and the associated right and left eigenvectors.
In practice, an effective ordering of $\cN$ is obtained with $J \ll n$
and since $\Pe$ is sparse, its largest eigenvalues and
eigenvectors can be computed efficiently using a standard
iterative algorithm.

Matrix $\Ne$ has, thanks to our choice of $\Pe$, a more geometric
interpretation which we now discuss.

\subsection{Geometric Interpretation of $N$}
\label{sec:Ngeo}

In this section, we show the close relationship between the matrices
$\Ne$, $\Pe$, and the weighted graph Laplacian matrix (to be defined
soon).  One important property of $\Pe$ is
its relationship to the Laplace Beltrami operator on $\cM$. Specifically,
when $n \to\infty$, and as $\epsilon \to 0$ at the appropriate rate
\footnote{We discuss this convergence in greater detail in Chapter~\ref{ch:rl}.}, 
$(I-\Pe)/\epsilon \to c \Delta_{LB}$ both pointwise and in spectrum
\cite{Coifman2005}. Here $\Le = (I-\Pe)/\epsilon$ is the
weighted graph Laplacian and $\Delta_{LB}$ is the Laplace Beltrami
operator on $\cM$ and $c$ is a constant.  Thus, for large $n$ and
small $\epsilon$, $\Pe$ is a neighborhood averaging operator. This
property also holds for $\Ne$:

\begin{thm} \label{thm:LB}
As $n \to \infty$ and $\epsilon \to 0$, $(I-\Ne)/\epsilon \to c' \Delta_{LB}$.
\end{thm}
\begin{proof}
For $I-S$ invertible, $(I-S)^{-1} = I+S(I-S)^{-1}$.  From
Lem. \ref{lem:eigP}, $I-\bp \Pe$ is invertible.  Thus
\begin{align}
\label{eq:LBp1}
p(I-\bp\Pe)^{-1} &= \prn{I-\frac{\bp}{p}(\Pe-I)}^{-1}  \\
\label{eq:LBp2}
    &= I + \frac{\bp}{p}(\Pe-I)\prn{I-\frac{\bp}{p}(\Pe-I)}^{-1}.
\end{align}

Therefore
$$
 \frac{I-\Ne}{\epsilon} = \frac{\bp}{p} \frac{I-\Pe}{\epsilon}
 \left(I-\frac{\bp}{p}(\Pe-I) \right)^{-1}
  = \bp \frac{I-\Pe}{\epsilon} (I - \bp \Pe)^{-1}.
$$
The first factor on the RHS converges to $\bp \Delta_{LB}$ and the second
to $p^{-1}I$ since $\Pe \to I$.
%Hence $c \Delta_{LB} \frac{\bp}{p} = c  \Delta_{LB}$ as $\epsilon \to 0$.
\end{proof}
\noindent By Theorem \ref{thm:LB}, as $n \to \infty$ and $\epsilon \to 0$,
$\Ne$ acts like $\Pe$.  For finite sample sizes, however, experiments
indicate that $\Ne$ is more informative of neighborhood
relationships.  We provide here a simple justification based on the
original random walk construction.

Were we to replace $\Ne$ with $\Pe$ in the implementation of the NPDR,
edge $(i,j)$ would be marked as a bridge essentially according to
its normalized weight (that is, proportional to $P_{ij}$ and
normalized for sampling density).  As $P_{ij} =
\exp(-d_{ij}^2/\epsilon)$, this edge would essentially be marked as
a bridge if the pairwise distance between its associated points is
above a threshold.  NPDR would in this case yield a performance very
similar to that of LDR.  In contrast, NPDR via $\Ne$ uses multi-step
probabilities with an exponential decay weighting to determine whether
an edge is a bridge.  Thus, the more ways there are to
get from $y_i$ to $y_j$ over a wide variety of possible step counts,
the less likely that $(i,j)$ is a bridge.

We now provide some synthetic examples comparing NPDR to the other
decision rules for finite sample sizes.

\section{Denoising the Swiss Roll}
\label{sec:experiments}
We first test our method on the synthetic Swiss roll,
parametrized by $(a,b)$ in $U = [\pi,4\pi] \x [0,21]$
via the embedding $(x^1,x^2,x^3) =
f(a,b) = (a \cos a, b, a \sin a)$.
True geodesic distances $\set{g_{1j}}_{j=1}^n$
are computed via:
$$g_{ij} = \int_0^1 \norm{Df(\bv(t))
  \frac{\partial{\bv(t)}}{\partial t}}_2 dt$$
where $\bv(t) =(1-t)[a_i\ \ b_i]^T + t [a_j\ \ b_j]^T$, and $Df(\bv)$ is the
differential of $f$ evaluated at $(\bv^1,\bv^2)$.
We sampled $n=500$ points uniformly in the ambient space $\bbR^3$
with $x_1$ fixed ($(a_1,b_1) = (\pi,0)$).
For $t=1,\dots,T$ ($T=100$)
we generated $n$ random noise values,
$\set{u_{ti}}_{i=1}^n$, uniformly on $[-1,1]$.
Then $y_{ti} = x_i + \mu u_{ti} \bn_i$, $t=1,\dots,T$, where $\bn_i$ is the normal to
$\cM$ at $x_i$. Each experiment was repeated for $\mu \in \set{0,.05,\cdots,1.95,2}$.

\begin{figure}[h!]
  \centering
  \begin{subfigure}[b]{.45\linewidth}
    \includegraphics[width=\linewidth]{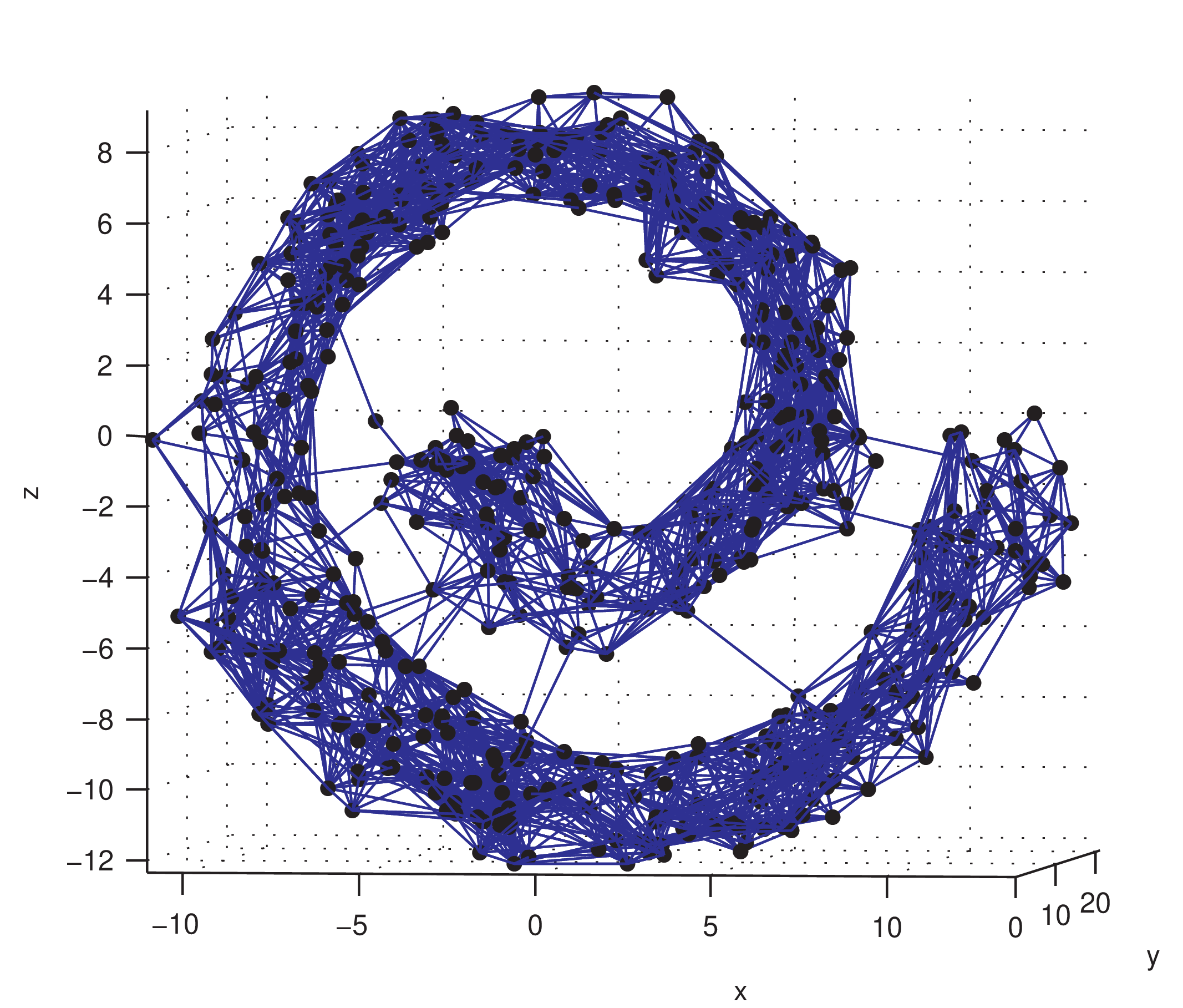}
    \caption{NN graph of Swiss roll ($\mu$=1.6) \label{sfig:swissroll_edges}}
  \end{subfigure}
  \begin{subfigure}[b]{.45\linewidth}
    \centering
    \begin{tabular}[b]{|c|c|} \hline
      $\mu$   &\#B   \\ \hline
      1.23    &0       \\
      1.28    &2       \\
      1.44    &10      \\
      1.54    &20      \\
      1.64    &32      \\
      1.74    &46      \\
      1.85    &66      \\
      1.90    &76      \\ \hline
    \end{tabular}
    \caption{Median bridge count vs. $\mu$. \label{sfig:bridges}}
  \end{subfigure}
  \caption{Noisy Swiss Roll}
\end{figure}

\begin{figure}[h!]
  \begin{subfigure}[b]{.45\linewidth}
    \caption{$\mu=.1$ \label{sfig:simple_noiseless}}
    \includegraphics[width=\linewidth]{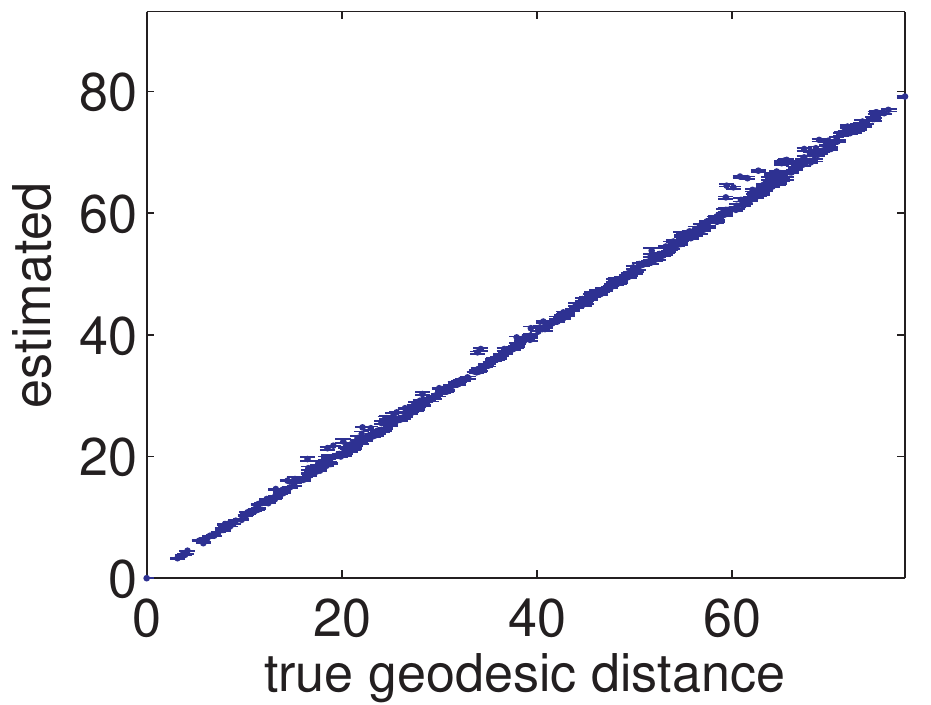}
  \end{subfigure}
  \begin{subfigure}[b]{.45\linewidth}
    \caption{$\mu=1.54$ \label{sfig:simple_noisy}}
    \includegraphics[width=\linewidth]{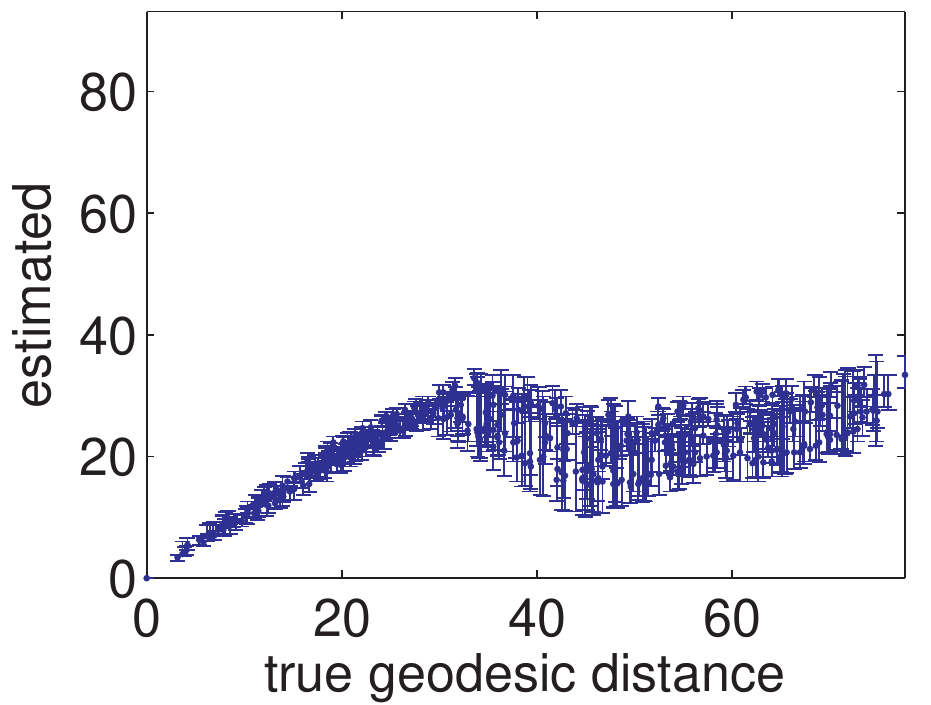}
  \end{subfigure}
  \caption{SP Denoising: Swiss Roll Geodesic estimates vs. ground truth (from $x_1$)}
\end{figure}

\begin{figure}[h!]
  \begin{subfigure}[b]{.45\linewidth}
    \caption{$\mu=.1, q=.92$ \label{sfig:ECDR_noiseless}}
    \includegraphics[width=\linewidth]{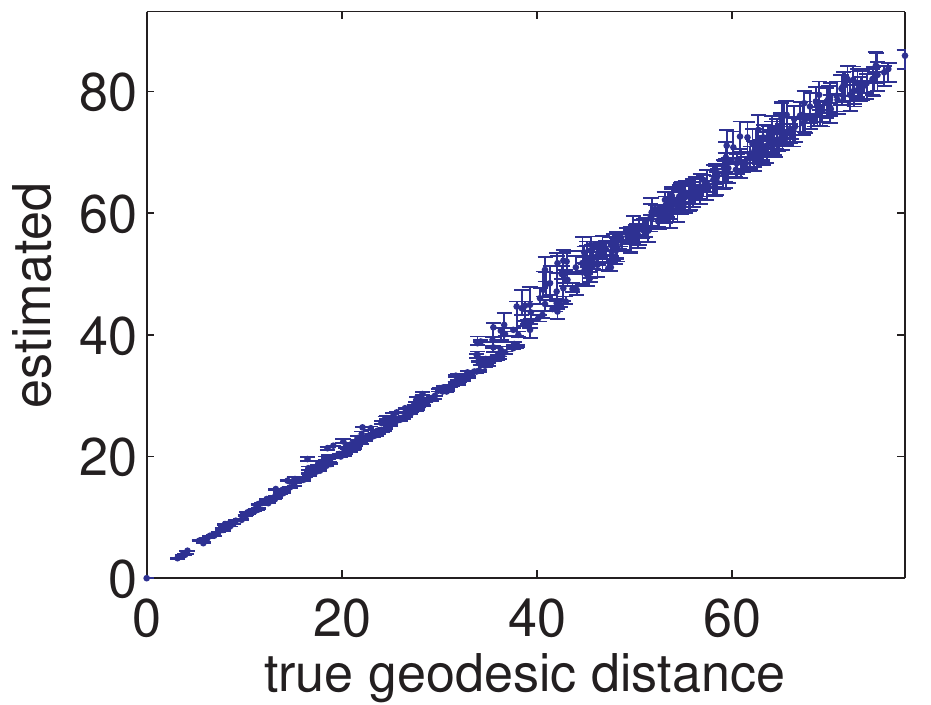}
  \end{subfigure}
  \begin{subfigure}[b]{.45\linewidth}
    \caption{$\mu=1.54, q=.92$ \label{sfig:ECDR_noisy_q1}}
    \includegraphics[width=\linewidth]{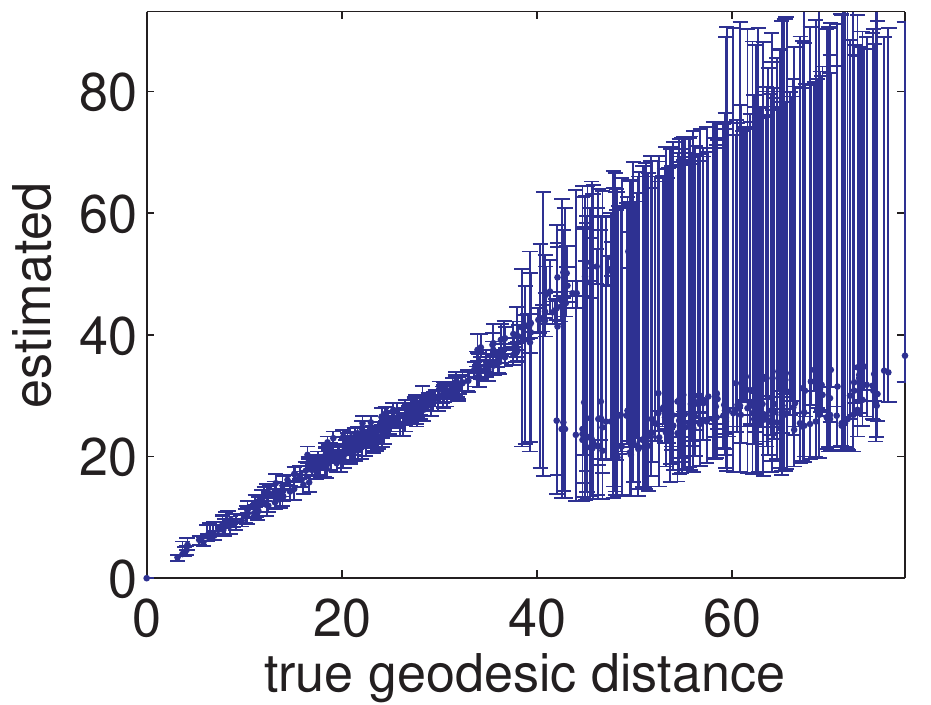}
  \end{subfigure}
  \caption{ECDR Denoising: Swiss Roll Geodesic estimates vs. ground truth (from $x_1$)}
\end{figure}

\begin{figure}[h!]
  \begin{subfigure}[b]{.45\linewidth}
    \caption{$\mu=.1, q=.92$ \label{sfig:NPDR_noiseless}}
    \includegraphics[width=\linewidth]{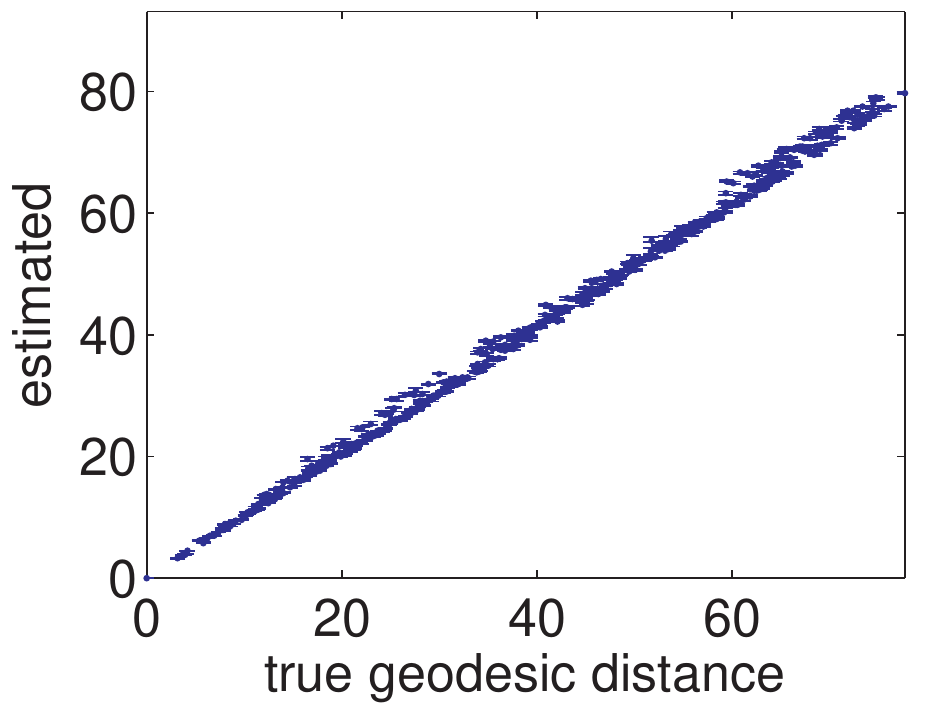}
  \end{subfigure}
  \begin{subfigure}[b]{.45\linewidth}
    \caption{$\mu=1.54, q=.92$ \label{sfig:NPDR_noisy_q1}}
    \includegraphics[width=\linewidth]{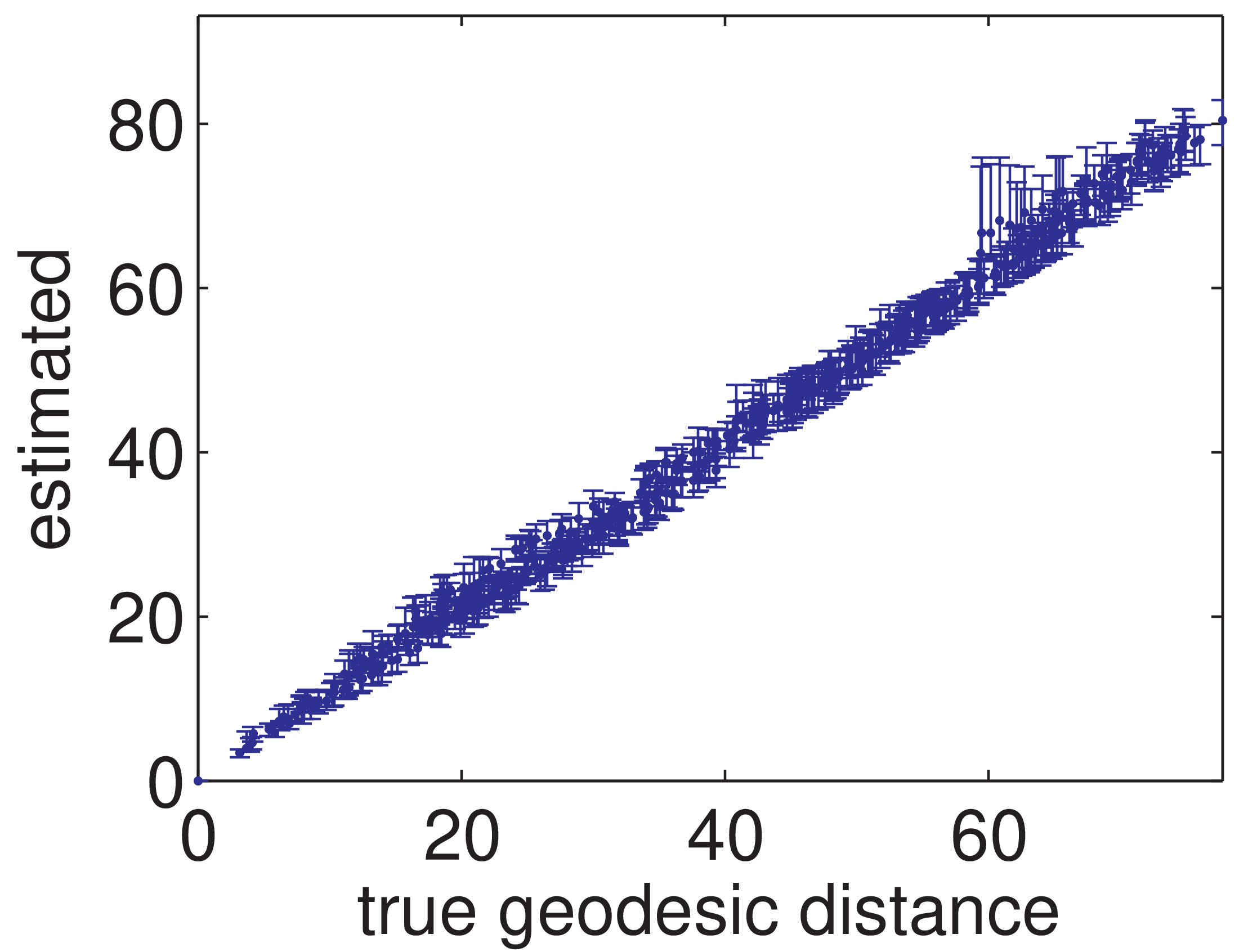}
%   \subfigure[$\mu=1.54, q=.99$]
%   {
%     \label{sfig:ECDR_noisy_q2}
%     \includegraphics[width=.42\linewidth]{ch-npdr/geo_ecdr_ci_001_noisy_q2}
%   }
%   \subfigure[$\mu=1.54, q=.99$]
%   {
%     \label{sfig:NPDR_noisy_q2}
%     \includegraphics[width=.42\linewidth]{ch-npdr/geo_npdr_ci_001_noisy_q2}
%   }
  \end{subfigure}
  \caption{NPDR Denoising: Swiss Roll Geodesic estimates vs. ground truth (from $x_1$)
%(e),(g) ECDR estimates. % (i)
%(f),(h) NPDR estimates. % (j)
}
\end{figure}

The initial NN graph $G$ was constructed using $\delta$-balls
($\delta=4$). The median bridge counts over $T$ realizations are shown in
Fig.~\ref{sfig:bridges}.  Bridges first appear at $\mu \approx 1.2$.
%The minimum distance between two branches of the swiss roll is $2\pi$.
Fig.~\ref{sfig:swissroll_edges} shows one realization of $\cY$ and the NN graph $G$ (note bridges).
We compare the simple SP with the LDR, ECDR ($K=15$), and NPDR ($p=0.01$)
 -based estimators by plotting the estimates of geodesic distance versus ground truth (sorted
by distance from $x_1$).  We plot the median estimate, 33\%, and 66\%
quantiles over the $T$ runs, for $\mu=.1$ and $\mu = 1.54$.  The LDR
and JDR based estimators' performance is comparable to SP for $q>.9$
(plots not included).

With no noise: SP provides excellent estimates; NPDR estimates are
%(Fig.~\ref{sfig:simple_noiseless})
%However, its performance degrades quickly when bridges appears in the NN graph
%(Fig.~\ref{sfig:simple_noisy}).
accurate even after removing $8\%$ of the graph edges
(Fig.~\ref{sfig:NPDR_noiseless}); however, ECDR removes important
edges (Fig.~\ref{sfig:ECDR_noiseless}).
At $\mu = 1.54$ with approximately $25$ bridges: SP has failed (Fig.~\ref{sfig:simple_noisy});
ECDR is removing bridges but also important edges, resulting in an upward
estimation bias (Fig.~\ref{sfig:ECDR_noisy_q1}); in contrast,
NPDR is successfully discounting bridges without
any significant upward bias even at $q=.92$ (Fig.~\ref{sfig:NPDR_noisy_q1}).
This supports our claim that bridges occur between edges with low
neighbor probability in the NPDR random walk.
Lower values of $q$ remove more edges, including bridges, but removing
non-bridges always increases SP estimates, and can lead to an upward bias.
%A better DR will remove more bridges at lower $q$.
The choice of $q$ should be based on prior knowledge of the
noise or cross-validation.

\begin{table}[h!]
\caption{Comparison of mean error $E$, varying $\mu$.}
\label{fig:cmp}
\centering
\begin{tabular}{|c||c|c|c|c|c|c|c|c|}
\hline
\multirow{2}{*}{$\mu$} &
SP &
LDR, &
\multicolumn{3}{|c|}{ECDR, q=} &
\multicolumn{3}{|c|}{NPDR, q=} \\
{}        & {}    & q=.92 & .92 & .95 & .99 & .92 & .95 & .99 \\ \hline
0.10  & \textbf{0.8}   & 1.5   & 4.4   & 3.8   & 1.9   & 2.4   & 1.9   & 1.2 \\ \hline
1.44 & 12.0  & 10.9  & 10.7  & 8.1   & 2.1   & 3.6   & 2.5   & \textbf{1.6} \\
1.54 & 13.6  & 12.8  & 11.7  & 7.6   & 2.5   & 3.8   & 2.6   & \textbf{2.3} \\
1.64 & 14.4  & 13.9  & 11.6  & 6.8   & 5.1   & 4.0   & \textbf{3.1}   & 6.5 \\
1.74 & 14.9  & 14.4  & 11.4  & 6.5   & 8.9   & \textbf{5.0}   & 5.6   & 11.1 \\
1.85 & 15.3  & 14.9  & 12.1  & 8.6   & 12.0  & \textbf{8.1}   & 9.4   & 13.4 \\
\hline
\end{tabular}
\end{table}

Table~\ref{fig:cmp} compares the performance of DRs at moderate noise
levels.  For this experiment, we chose $5$ points,
$\set{x_r}_{r=1}^5$, well distributed over $\cM$.  Over $T=100$
noise realizations, we calculated the mean of the value
$$
E = \frac{1}{5n} \sum_{r=1}^5 \sum_{i=1}^n \abs{g_{ri}-\tg_{ri}},
$$
the average absolute error of the geodesic estimate from all of the
points to these 5.  As seen in this figure, given an appropriate
choice of $q$,  NPDR outperforms the other DRs at moderate noise
levels.  As expected, $q$ must grow with the noise level as more
bridges are found in the initial graph.

\section{Denoising a Random Projection Graph}
We now consider the random projection tomography problem of
\cite{Singer2009}.  Random projections of $I$ are taken at angles
$\theta \in [0, 2\pi)$.  More specifically, these projections are
$f(\theta) = R_{\theta}(I)$, where $R_\theta$ is the Radon transform
at angle $\theta$.

In \cite{Singer2009}, we observe $n=1024$ random projections:
$$y_i = f(\theta_i) + \nu_i \qquad \text{where} \qquad \nu_i \sim N(0,\sigma^2).$$
for which the ambient dimension of the projection is $r = 512$ and
$\sigma^2 = \sigma^2_f /10^{SNR_{db}/10}$, where the signal power is
$\sigma_f^2 \approx 0.0044$.  The image used in all experiments is the
Shepp-Logan phantom (Fig.~\ref{fig:tomo_true}), and the projection
angles are unknown (Fig.~\ref{fig:tomo_proj}).  After some initial
preprocessing, a NN graph ($k=50$) is constructed from the noisy
projections, and JDR is used to detect bridges in this graph.  After
detected bridges are removed (pruned), nodes with less than two
remaining edges (that is, isolated nodes) are removed from the
graph. An eigenvalue problem on the new graph's adjacency matrix is
then solved to find an angular ordering of the remaining projections
(nodes). Finally, $\hI$ is reconstructed via an inverse Radon
transform of these resorted projections.

\begin{figure}[h!]
\centering
\begin{subfigure}[b]{.45\linewidth}
  \includegraphics[width=\linewidth]{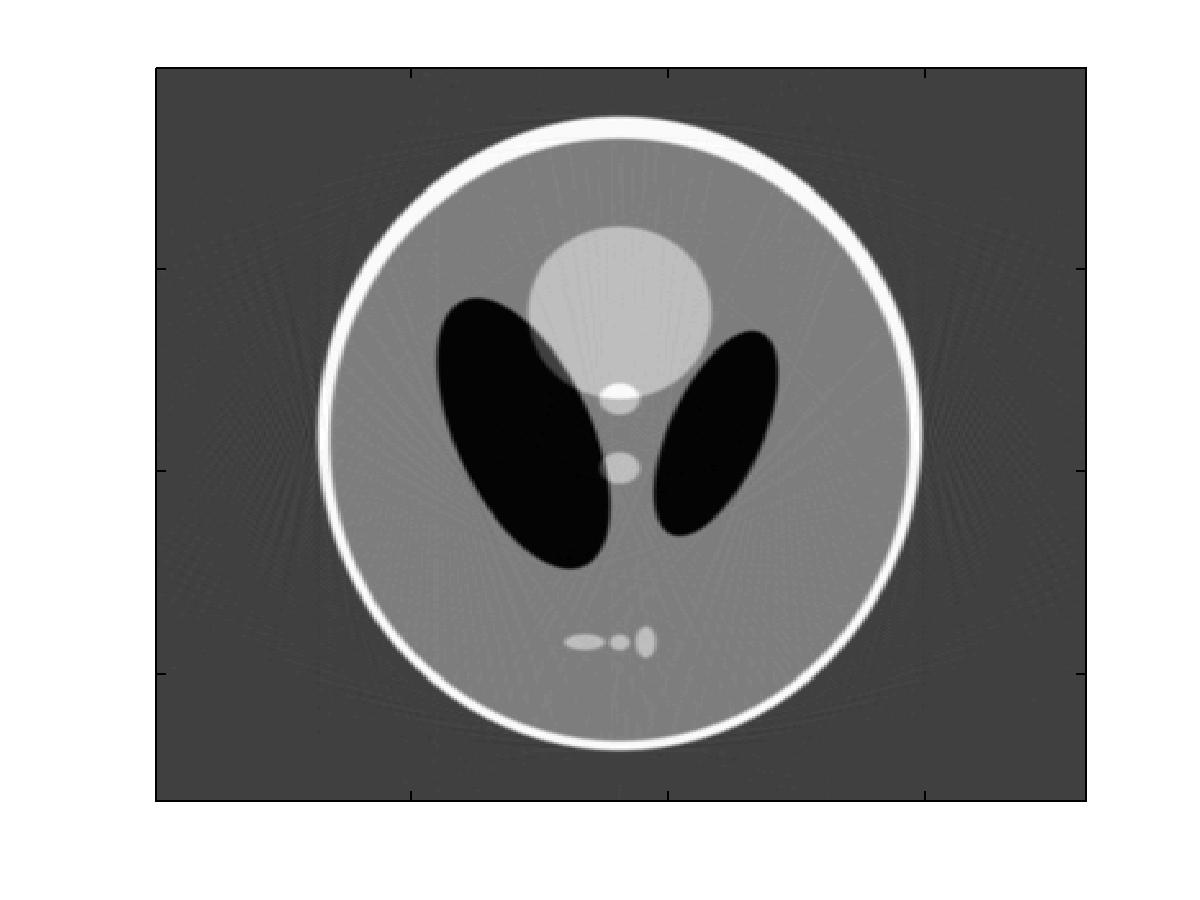}
  \caption{Original Phantom, $I$ \label{fig:tomo_true}}
\end{subfigure}
\begin{subfigure}[b]{.45\linewidth}
  \includegraphics[width=\linewidth]{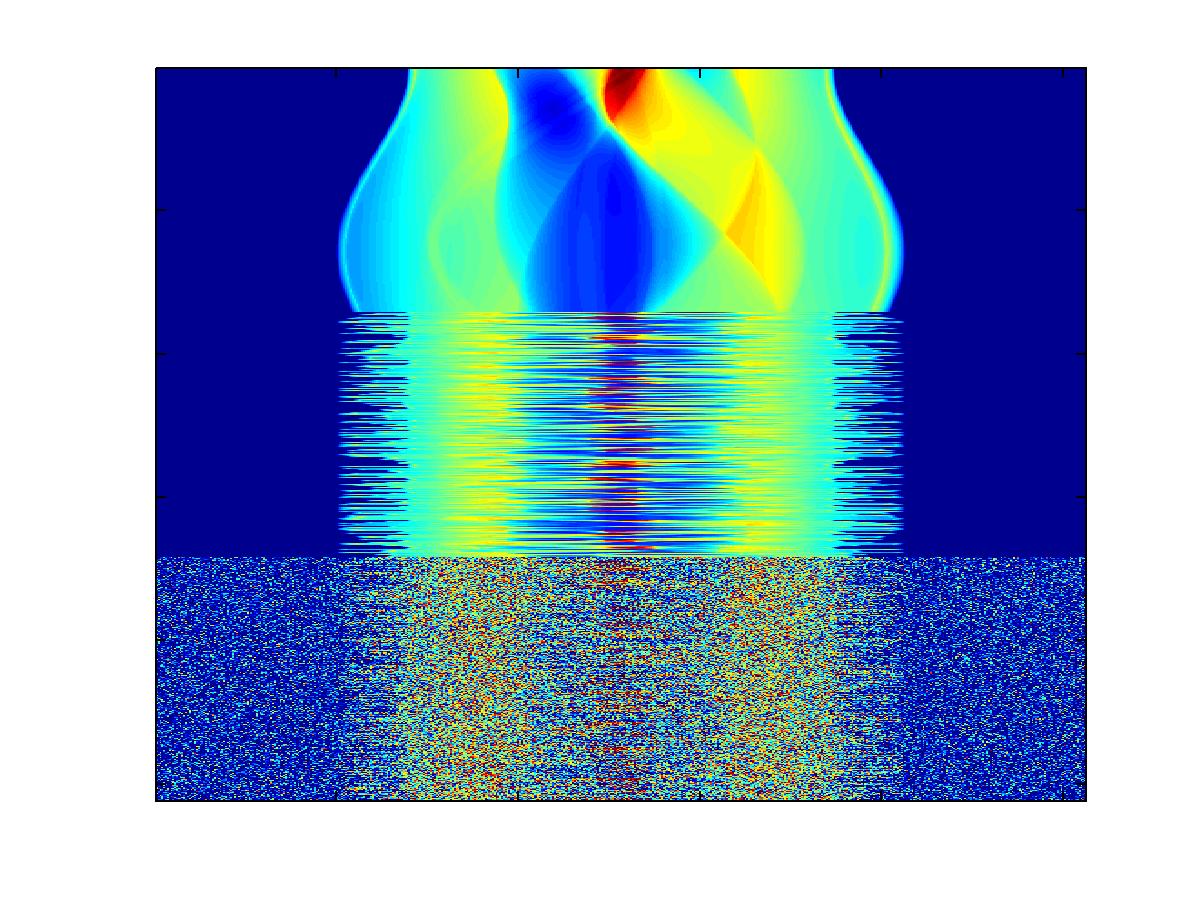}
  \caption{Projections of $I$ \label{fig:tomo_proj}}
\end{subfigure}
\caption[Shepp-Logan Phantom and Radon Projections.]%
  {Shepp-Logan Phantom and Radon Projections.  In (b), the
  y-axis represents projection angle.  The top third shows regularly
  ordered projections $f(\theta)$.  The middle third shows projections when
  $\theta$ has been randomized.  The lower third shows projections
  after angles have been randomized and noise has been added (SNR is
  -2db)}
\end{figure}

\begin{figure}[h!]
\centering
\begin{subfigure}[b]{.45\linewidth}
  \includegraphics[width=\linewidth]{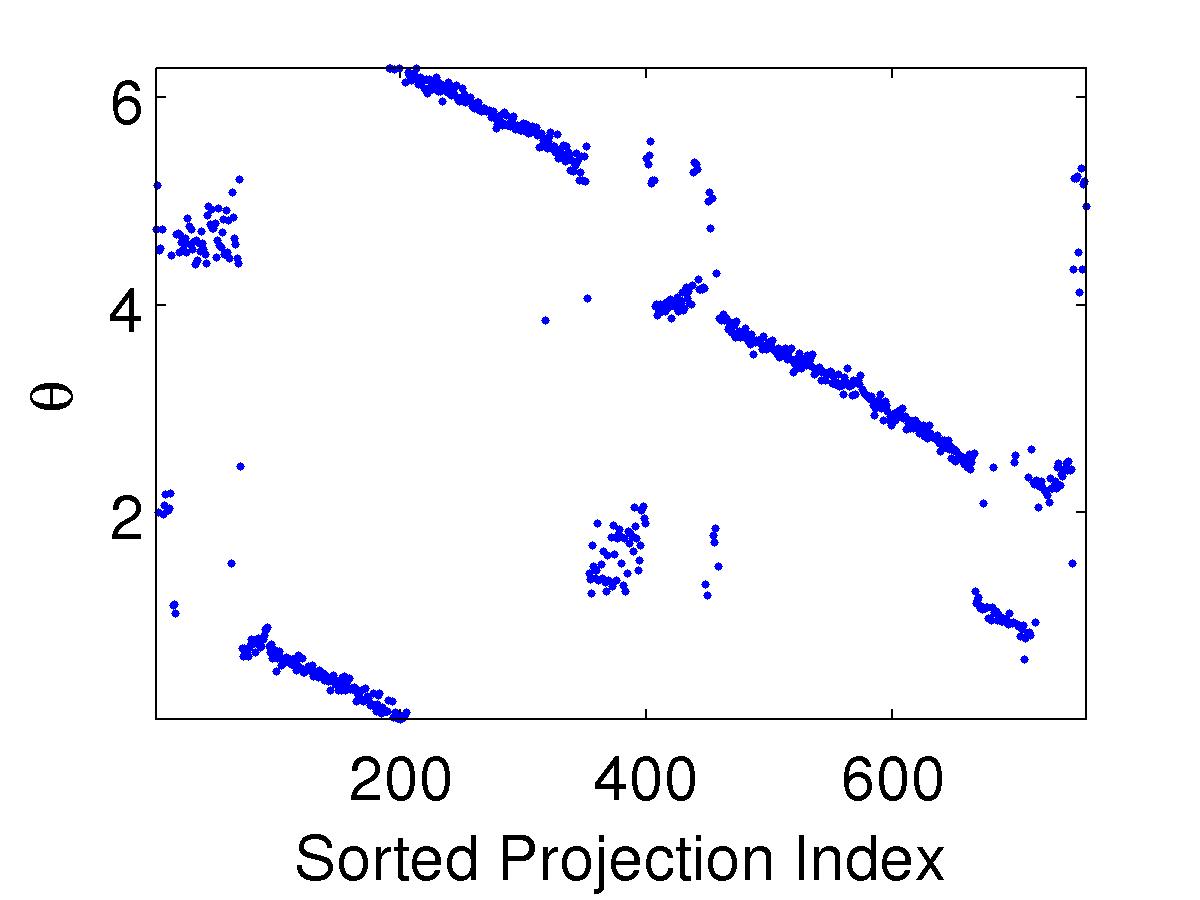}
  \caption{Estimated $\theta$, JDR \label{fig:tomo_sorted_jdr}}
\end{subfigure}
\begin{subfigure}[b]{.45\linewidth}
  \includegraphics[width=\linewidth]{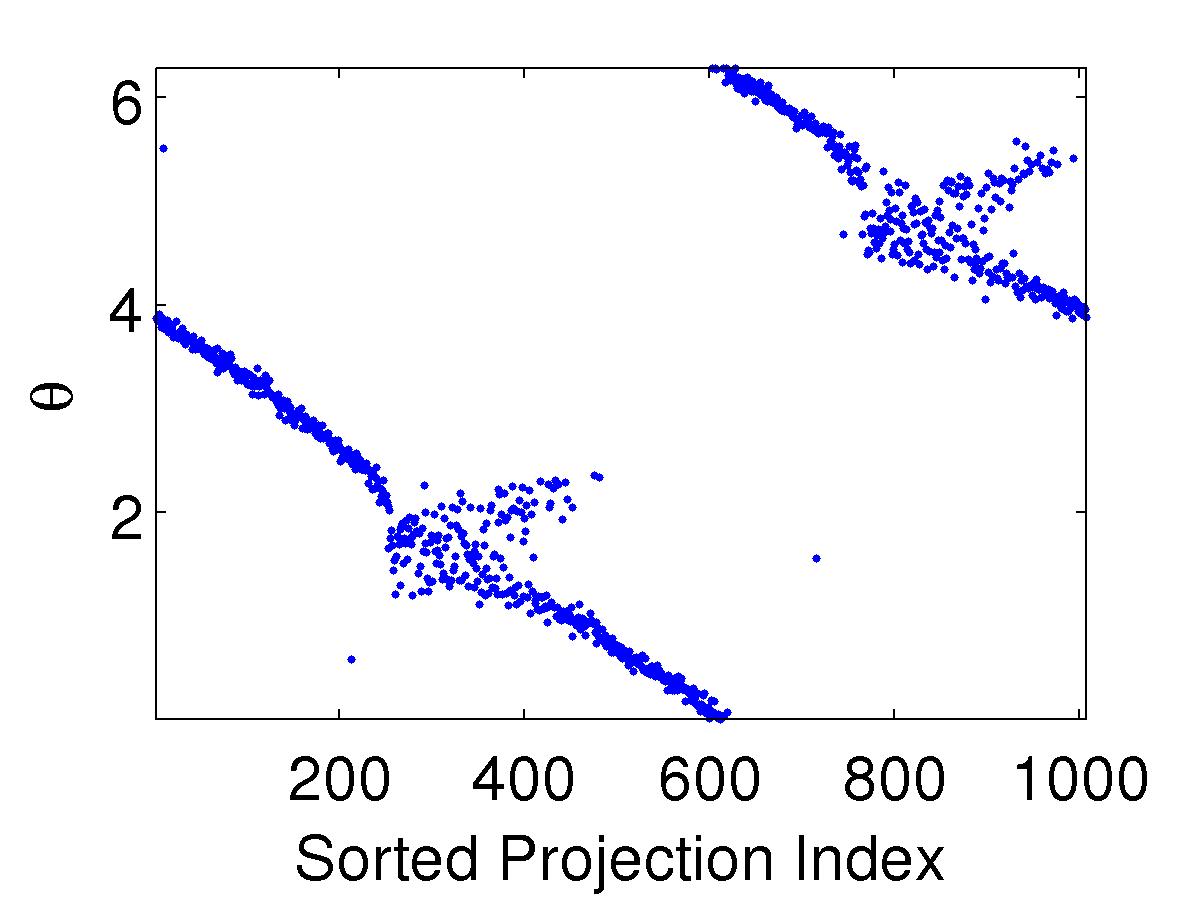}
  \caption{Estimated $\theta$, NPDR   \label{fig:tomo_sorted_npdr}}
\end{subfigure}
\begin{subfigure}[b]{.45\linewidth}
  \includegraphics[width=\linewidth]{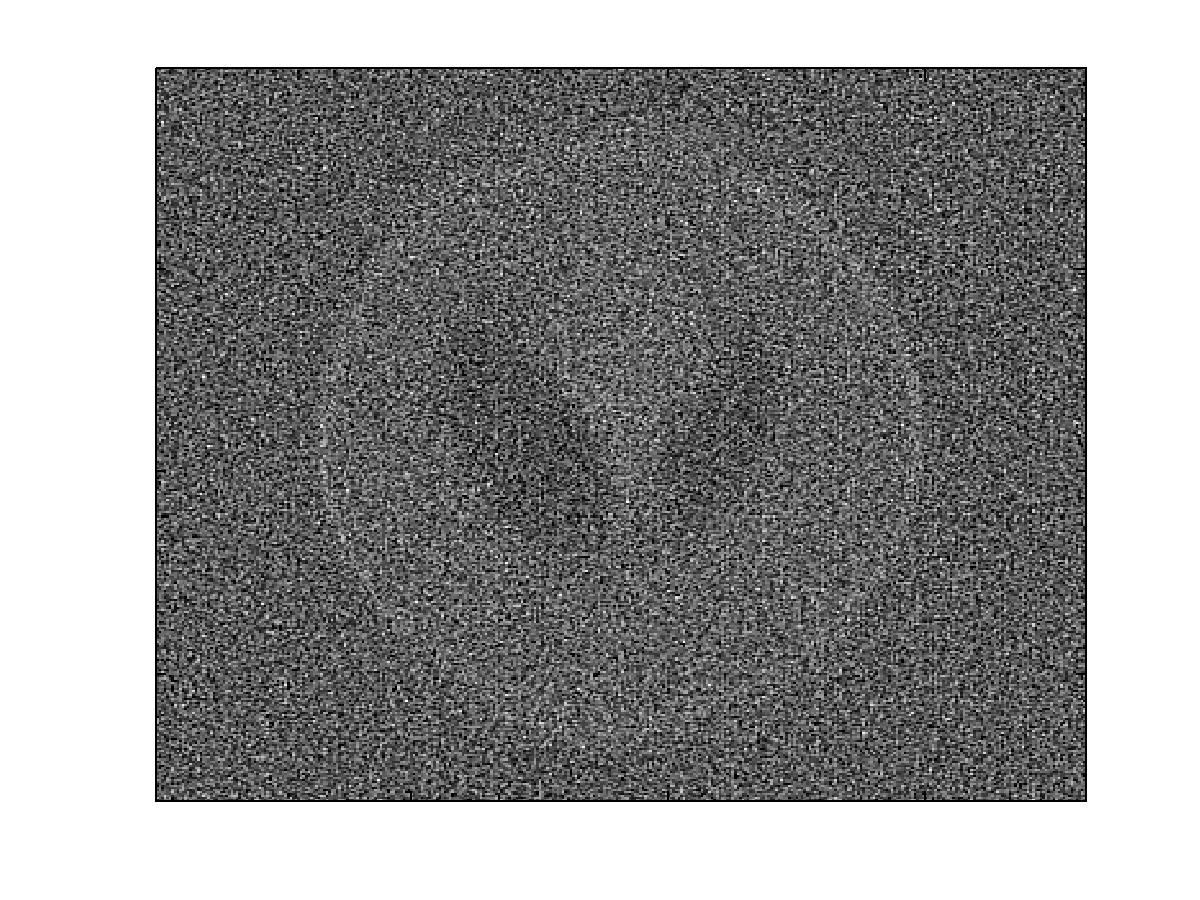}
  \caption{Reconstruction $\hI$, JDR \label{fig:tomo_recon_jdr}}
\end{subfigure}
\begin{subfigure}[b]{.45\linewidth}
  \includegraphics[width=\linewidth]{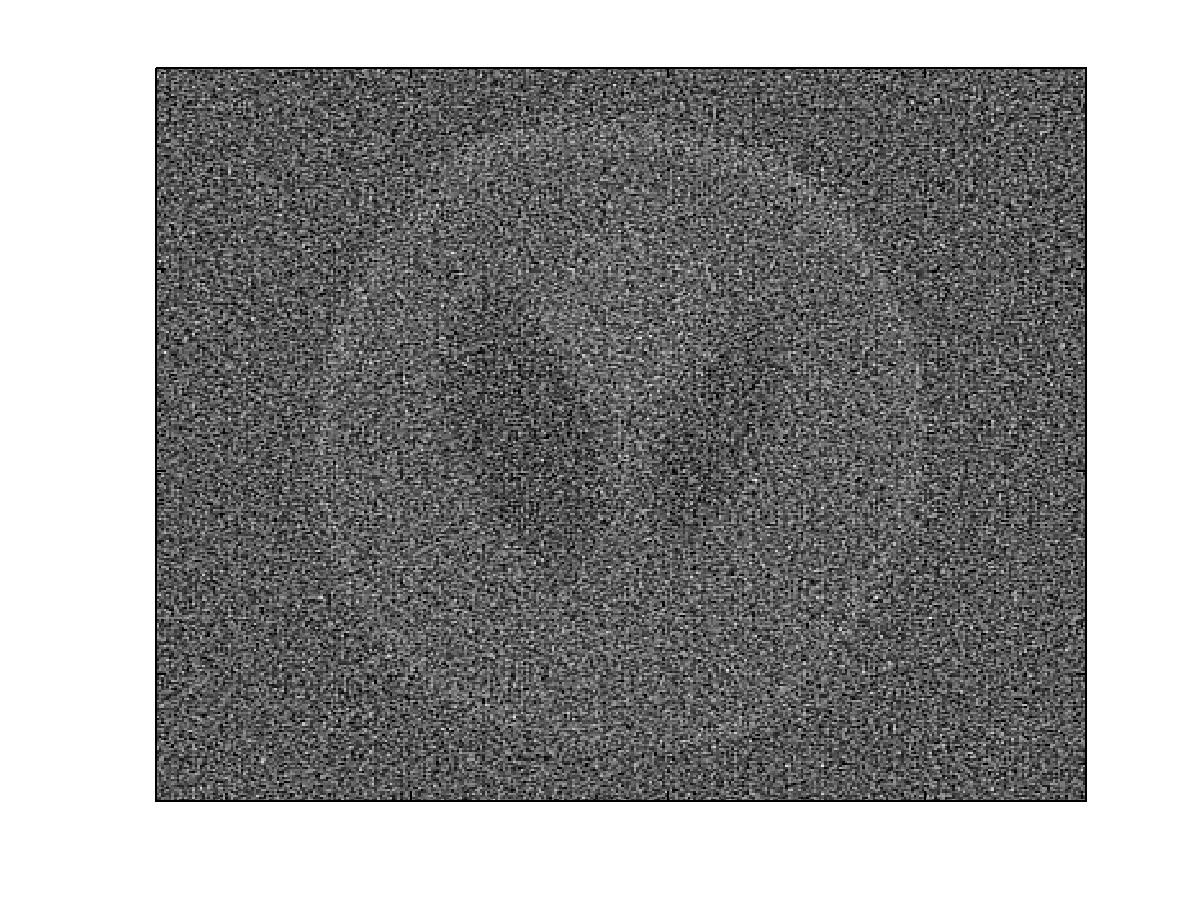}
  \caption{Reconstruction $\hI$, NPDR \label{fig:tomo_recon_npdr}}
\end{subfigure}
\caption{Tomography Reconstructions from Random Projections}
\end{figure}

We compared JDR to NPDR pruning at a SNR of $-2\text{db}$.
Exhaustive search finds the optimal $q$ for
JDR at $q=.78$.  For NPDR we used $\hat{P}_\infty$ (all edges in $G$
have weight $1$), $p=.01$, and $q=.8$. After pruning, JDR
disconnected 277 nodes compared to 21 for NPDR.
The estimated sorted angles are shown in
Figs. \ref{fig:tomo_sorted_jdr},\ref{fig:tomo_sorted_npdr}, and the
rotated reconstructions in
Figs. \ref{fig:tomo_recon_jdr},\ref{fig:tomo_recon_npdr}.
Under the similarity metric $\rho = \frac{I^T\hI}{\norm{I}\norm{\hI}}$,
with alignment of $\hI$ with $I$, the increase in NPDR similarity (0.15)
over JDR similarity ($0.12$) is 25\%.  Note the clearer boundaries
in the NPDR phantom, thanks to 256 additional (unpruned) projections
(best viewed on screen).
At moderate noise levels, NPDR removes fewer NN graph nodes
and yields a more accurate reconstruction.
As more projections are left after pruning, the final accuracy is
higher.

\section{Conclusion and Connections}
We studied the problem of estimating geodesics in a point cloud.  A
slight revision for removing edges from a neighborhood graph allows us
to avoid disconnecting weakly connected groups.  Building on this
framework, we studied several global measures for detecting
topological bridges in the NN graph.  In particular, we developed and
analyzed the NPDR bridge detection rule, which is based on a special
type of Markov random walk.  Using a special random walk matrix
derived from the geometry of the sample points, we constructed the
NPDR to detect bridges by thresholding entries of the neighborhood
matrix $\Ne$.  The entries of column $i$ in this matrix converge to
those of a special averaging operator in the neighborhood of point
$x_i$ in $\cM$: the averaging intrinsically performed by the Laplace
Beltrami operator around $x_i$.

Our experiments indicate that NPDR robustly detects bridges in the NN
graph without misclassifying edges important for geodesic estimation
or tomographic angle estimation.  Furthermore, it does so over a wider
noise range than competing methods, e.g. LDR and ECDR.  It can be
calculated efficiently via a sparse eigenvalue decomposition.
Preliminary evidence from synthetic experiments indicates that, as
\S\ref{sec:npdr} suggests, for NPDR one should choose $p$ as small as
possible while retaining numerical conditioning of $I-\bp\Pe$.

For very large $n$ and small $\epsilon$, the matrices $\Ne$ and $\Pe$
are equivalent, but in practical cases $\Ne$ yields significantly
better performance.  More testing of NPDR on non-synthetic
datasets is needed.  Possible applications include determining
bridges in social and webpage (hyperlink) network graphs, and in the
common line graphs estimated in the blind 3D tomography ``Cryo-EM''
problem~\cite{Singer2010}.

Furthermore, the matrix $\Ne$ is closely related to the regularized
inverse of the graph Laplacian.  The term $\Pe-I$ in \eqref{eq:LBp1}
is proportional to the weighted graph Laplacian $\Le$, and from this
equation it is clear that $\Ne$ is proportional to the inverse of the
Tikhonov regularized weighted graph Laplacian (with regularization
parameter $p / \bp$).  The efficacy of the NPDR, and the initial
theoretical results developed in \S\ref{sec:Ngeo} lead us to study the
regularized inverse of the graph Laplacian in more detail;
this is the focus of Chapter~\ref{ch:rl}.

\chapter{The Inverse Regularized Laplacian: Theory\label{ch:rl}%
\chattr{This chapter, and the next, are based on work in collaboration with
Peter~J.~Ramadge, Department of Electrical Engineering, Princeton
University, as submitted in \cite{Brevdo2011a}.}}

%\title{Semisupervised Learning: The Regularized Manifold Laplacian, Geodesics, and the Small Viscosity Limit}

%% \begin{abstract}Motivated by recent problems in semi-supervised learning,
%% we show a connection between the regularized Laplacian PDE 
%% on a compact Riemannian manifold and the Eikonal equation, which
%% generates geodesics on this manifold.
%% %
%% This connection leads to intuitive geometric interpretations of
%% of learning algorithms whose solutions include a regularized inverse
%% of the graph Laplacian.  It also enables us to build a robust
%% geodesic distance estimator, a competitive new
%% multiclass classifier, and a regularized version of ISOMAP.
%% \end{abstract}

\section{Introduction}

Semi-supervised learning (SSL) encompasses a class of machine learning
problems in which both labeled data points and unlabeled data points
are available during the training (fitting) stage~\cite{Chapelle2006}.
In contrast to supervised learning, the goal of SSL algorithms is to
improve future prediction accuracy by including information from the
unlabeled data points during training.  SSL extends
a wide class of standard problems, such as classification and
regression.

A number of recent nonlinear SSL algorithms use aggregates of nearest
neighbor (NN) information to improve inference performance.  These
aggregates generally take the form of some transform, or
decomposition, of the weighted adjacency, or weight, matrix of the NN
graph.  Formal definitions of NN graphs, and associated weight
matrices, are given in~\S\ref{sec:prelim}; we will also review them in
the SSL learning context in~\S\ref{sec:sslintro}.

Motivated by several of these algorithms, we show a connection between
a certain nonlinear transform of the NN graph weight matrix, the
regularized inverse of the graph Laplacian matrix, and the
solution to the regularized Laplacian partial differential equation
(PDE), when the underlying data points are sampled from a compact
Riemannian manifold.
We then show a connection between this PDE and the Eikonal equation,
which generates geodesics on the manifold.

These connections lead to intuitive geometric interpretations of
learning algorithms whose solutions include a regularized inverse
of the graph Laplacian.
As we will show in chapter~\ref{ch:rlapp}, it
also enables us to build a robust geodesic distance estimator, a
competitive new multiclass classifier, and a regularized version of
ISOMAP.

This chapter is organized as follows:
\S\ref{sec:sslintro}-\S\ref{sec:ssllimitthm} motivate our study by
showing that in a certain limiting case, a standard SSL problem can be
modeled as a regularized Laplacian (RL) PDE problem.
\S\ref{sec:rlpde}-\S\ref{sec:lap} derive the relationship between
the regularized Laplacian and geodesics and discusses convergence
issues as an important regularization term (viscosity) goes to zero.
%
%\S\ref{sec:sslr} and~\S\ref{sec:irr} interpret the original SSL
%problem, and a recent result of \cite{Nadler2009},
%within this geometric framework.
%
%\S\ref{sec:geo} constructs the geodesic distance estimator and
%shows the efficacy of the resulting multiclass classifier.
%
%\S\ref{sec:beyond} discusses applications beyond
%basic classification, including to ISOMAP and to a graph denoising
%algorithm.
%
%We conclude in~\S\ref{sec:concl}.

\section{Prior Work}

The graph Laplacian is an important tool for regularizing
the solutions of unsupervised and semi-supervised learning
problems, such as classification and regression, in high-dimensional
data analysis \cite{Smola2003, Belkin2004, Belkin2006, Zhu2003, Luxburg2007}.
Similarly, estimating geodesic distances from sample points on a
manifold has important applications in manifold learning and SSL
(see, e.g., chapter~\ref{ch:npdr} and \cite{Tenenbaum2000,
Chapelle2005, Bengio2004}). Though heavily used in the learning and
geometry communities, these methods still raise many questions.
For example, with dimension $d>1$, graph Laplacian regularized
SSL does not work as expected in the large sample limit \cite{Nadler2009}.
It is also desirable to have geometric intuition about the behavior 
of the solutions of models like those proposed in \cite{Belkin2006, Zhu2003}
in the limit when the data set is large.

To this end, we elucidate a connection between three important 
components of analysis of points sampled from manifolds:
\begin{enumerate}
\item The inverse of the regularized weighted graph Laplacian matrix.
\item A special type of elliptic Partial Differential Equation (PDE) on the
manifold.
\item Geodesic distances on the manifold.
\end{enumerate}
This connection provides a novel geometric interpretation for
machine learning algorithms whose solutions require the regularized
inverse of a graph Laplacian matrix.
It also leads to a consistent geodesic distance estimator with two
desired properties: the complexity of the estimator depends only on the
number of sample points and the estimator naturally incorporates a
smoothing penalty. 

\section{Manifold Laplacian, Viscosity, and Geodesics}
\label{sec:summary}
We motivate our study by first looking at a standard
semisupervised learning problem (\S\ref{sec:sslintro}).
We show that as the amount of data increases and regularization is
relaxed, this problem reduces to a PDE
(\S\ref{sec:sslassume}-\S\ref{sec:ssllimit}).  We then analyze this PDE in 
the low regularization setting to uncover new geometric insights into its
solution (\S\ref{sec:rlpde}-\S\ref{sec:transport}).  In
\S\ref{sec:sslr}, these insights will allow us to analyze the original
SSL problem from a geometric perspective.

\subsection{A Standard SSL Problem}
\label{sec:sslintro}
We present a classic SSL problem in which points are sampled, some
with labels, and a regularized least squares regression is performed to
estimate the labels on the remaining points.  The regression contains
two regularization penalties: a ridge regularization penalty and a graph Laplacian
``smoothing'' penalty.

Data points $\cX = \set{x_i}_{i=1}^n$ are sampled from a space $\cM
\subset \bbR^p$.  The first $l$ of these $n$ points have associated
labels: $w_i \in \cL \subset \bbR, i=1,\ldots,l$.  For binary
classification, one could take $\cL = \set{\pm 1}$.  
The goal is to find a vector $\tf \in \bbR^n$ that approximates the $n$
samples at the points in $\cX$ of an unknown smooth function $f$ on $\cM$ and
minimizes the regression penalty $\sum_{i=1}^l (f(x_i)-w_i)^2$.
The solution is regularized by two penalties: the ridge
penalty $\|\tf\|_2$ and the graph Laplacian penalty
$\wt{J}(\tf) = \tf^T L \tf$; details of the construction
of $L$ are given in~\S\ref{sec:ssllimit}.
%.  The first is the
%ridge regression penalty $\|\tf\|_2$, 
The second penalty approximately penalizes the gradient of $f$.  It is
a discretization of the functional:
$
{J(f) = \int_{\cM} \norm{\grad f(x)}^2 d\cP(x) = \int_{\cM} f(x)\Delta
 f(x) d\cP(x),}
$
where $\cP(x)$ is the sampling density.

To find $\tf$ we solve the convex minimization problem:
\beq
\label{eq:rlsp}
\min_{\tf \in \bbR^n} \sum_{i=1}^l (\tf_i-w_i)^2
 + \gamma_A \|\tf\|^2_2 + \gamma_I \wt{J}(\tf),
\eeq
where the nonnegative regularization parameters $\gamma_A$
and $\gamma_I$ depend on $l$ and $n$.

We can rewrite \eqref{eq:rlsp} in its matrix form.  Let
$E_\cA = \diag{1\cdots 1\ \ 0\cdots 0}$ have $l$ ones followed by
$n-l$ zeros on its diagonal, and let $E_\cAp = I -
E_\cA$.  Further, let $\wt{w} =
\left[w_1\ \cdots\ w_l\ \ 0\ \cdots\ 0\ \right]^T$ be a vector of the 
labels $w_i$ for the first $l$ sample points, and zeros for the
unlabeled points.  Then Eq.~\eqref{eq:rlsp} can be written as:
\beq
\label{eq:rlspdisc}
\min_{\tf \in \bbR^n} \|\wt{w}-E_\cA \tf\|^2_2 + \gamma_A \tf^T \tf +
\gamma_I \tf^T L \tf
\eeq
This is a quadratic program for $\tf$.  Setting the gradient, with
respect to $\tf$, to zero yields the linear system
%\cite[Eq.~35]{Belkin2006}
\beq
\label{eq:lrlssysorig}
\left(E_\cA + \gamma_A I + \gamma_I L \right) \tf = \wt{w}.
\eeq

The optimization problem \eqref{eq:rlsp} and the linear
system \eqref{eq:lrlssysorig} are related to two previously
studied problems.  The first is graph regression with Tikhonov
regularization \cite{Belkin2004}.  Problem
\eqref{eq:lrlssysorig} is closely related to the one solved in
Algorithm 1 of that paper, where we replace their general penalty
$\gamma S$ term with the more specific form $\gamma_A I + \gamma_I L$.
The ridge penalty $\gamma_A I$ encourages stability in the solution,
replacing their zero-sum constraint on $f$.
The second related problem is Laplacian
Regularized Least Squares (LapRLS) of \cite{Belkin2006}.
Specifically, \eqref{eq:lrlssysorig} is identical to Eq.~35 of
\cite{Belkin2006} with one of the regularization terms removed by
setting the kernel matrix $K$ equal to the identity.  In that
framework, the matrix $K$ is defined as $K_{ij} \propto
e^{-\norm{x_i-x_j}^2/(2 \sigma^2)}$, $i,j = 1,\ldots,n$ for some
$\sigma > 0$.
%The problem we pose is thus the limiting case where
%$\sigma \ll 1$.
A choice of $\sigma > 0$
regularizes for finite sample size and sampling noise
\cite[Thm.~2, Remark 2]{Belkin2006}.  Eq.~\eqref{eq:lrlssysorig} is
thus closely related to the limiting solution to the LapRLS problem
in the noiseless case, where the sampling size $n$ grows and $\sigma$
shrinks quickly with $n$.

We study the following problem: suppose the function $w$
is sampled without noise on specific subsets of
$\cM$.  The estimate $\tf$ represents an \emph{extension} of $w$ to
the rest of $\cM$.  What does this extension look like, and how
does it depend on the geometry of $\cM$?  The first step is to
understand the implications of the noiseless case on
\eqref{eq:lrlssysorig}; we study this next.

\subsection{SSL Problem -- Assumptions}
\label{sec:sslassume}

We now list our main assumptions for the SSL problem in
\eqref{eq:lrlssysorig}:
\begin{enumerate}
 \item \label{it:as1} $\cM$ is a $d$-dimensional ($d<p$) manifold
   $(\cM,g)$,  which is compact, and Riemannian.
 \item \label{it:as2} The labeled points $\set{x_i}_{i=1}^l$ are
   sampled from a regular and closed nonempty subset $\cA \subset \cM$.
 \item \label{it:as3} The labels $\set{w_i}_{i=1}^l$ are sampled from a
   smooth (e.g. $C^2$ or Lipschitz) function $w : \cA \to \cL \subset
   \bbR$.
 \item \label{it:as4} The density $\cP$ is nonzero everywhere on
   $\cM$, including on $\cA$.
\end{enumerate}
These assumptions ensure that the points $\cX$ are sampled without
noise from a bounded, and smooth space, that the labels are sampled
without noise, and that the label data is also sampled from a bounded
and smooth function.  We will use these assumptions to show the
convergence of $\tf$ to a smooth function $f$ on $\cM$.

In assumption \ref{it:as2} we use the term regular in the PDE
sense \cite[Irregular Boundary Point]{Hazewinkel1995}; we discuss
this further in~\S\ref{sec:irr}.  We call $\cA$ the
\emph{anchor set} (or anchor); note that it is not necessarily
connected.  In addition, let $\cAp = \cM \bs
\cA$ denote the complement set.  In assumption \ref{it:as3}, for
$w$ to be smooth it suffices that it is smooth on all connected
components of $\cA$.  Thus we can allow $\cL$ to take on discrete
values, as long as the classes they represent are separated from each
other on $\cM$.  We call the function $w$ the \emph{anchor condition}
or \emph{anchor function}. Note finally that assumption \ref{it:as4}
implies that the labeled data size $l$ grows with $n$.

As we have assumed that there is no noise
on the labels (assumptions \ref{it:as2} and \ref{it:as3}), we will not
apply a regularization penalty to the labeled data.  On the labeled
points, therefore, \eqref{eq:lrlssysorig} reduces to $E_\cA \tf = \wt{w}$.
Hence, the regularized problem becomes an interpolation problem.
The ridge penalty, now restricted to the unlabeled data, changes from
$\tf^T \tf$ to $\tf^T E_\cAp \tf$.  The Laplacian penalty function
becomes ${J(f) = \int_{\cAp} f(x) \lap f(x) d\cP(x)}$,  and the
discretization of this penalty similarly changes from $\tf^T L \tf$ to
$(E_\cAp \tf)^T E_\cAp L \tf$.
The original linear system \eqref{eq:lrlssysorig} thus becomes
\begin{align}
\label{eq:lrlssys} &M_n \tf = \wt{w}, \\
\nonumber \text{where } 
  &M_n = E_\cA + \gamma_A(n) E_\cAp + \gamma_I(n) E_\cAp L.
\end{align}
We note that the dimensions of $M_n$, $\tf$, and $\wt{w}$
in \eqref{eq:lrlssys} grow with $n$.

A more detailed explanation of the component terms in the
original problem \eqref{eq:rlspdisc} and the linear systems
\eqref{eq:lrlssysorig} and \eqref{eq:lrlssys} are available in
\S\ref{app:lrldetail}, where we show that the solution
to the simplified problem \eqref{eq:lrlssys} depends only on the ratio
$\gamma_I(n) / \gamma_A(n)$.  We will return to this in~\S\ref{sec:ssllimit}.

\subsection{SSL Problem -- the Large Sample Limit -- Preliminary Results}
\label{sec:ssllimit}

We study the linear system \eqref{eq:lrlssys} as $n \to \infty$,
and prove that given an appropriate choice of graph Laplacian $L$ and
growth rate of the regularization parameters $\gamma_A(n)$ and $\gamma_I(n)$,
the solution $\tf$ converges to the solution $f$ of a
particular PDE on $\cM$. The convergence occurs in the $\ell_\infty$
sense with probability $1$.

Our proof of convergence has two parts.  First, we must
show that $M_n \tf$ models a forward PDE with increasing accuracy as
$n$ grows large; this is called \emph{consistency}.  Second, we must
show that the norm of $M_n^{-1}$ does not grow too quickly as $n$
grows large; this is called \emph{stability}.  These two results will
combine to provide the desired proof.  As all of our estimates rely on
the choice of graph Laplacian matrix on the sample points $\cX$, we
first detail the specific construction that we use throughout:

\begin{enumerate}
\item Construct a weighted nearest neighbor (NN) graph $G = (\cX,\cE,\td)$
from $k$-NN or $\delta$-ball neighborhoods of $\cX$,
where edge $(l,k) \in \cE$ is assigned the distance
$\td_{lk} = \norm{x_k-x_l}_{\bbR^p}$.
\item Choose $\epsilon > 0$ and let $(\hPe)_{lk} = \exp(-\td_{kl}^2 /
  \epsilon)$ if $(l,k) \in \cE$, $1$ if $l=k$, and $0$ otherwise.
\item Let $\hDe$ be diagonal with $(\hDe)_{kk} = \sum_l (\hPe)_{kl}$,
  and normalize for sampling density by setting $\Ae = \hDe^{-1} \hPe
  \hDe^{-1}$.
\item Let $\De$ be diagonal with $(\De)_{kk} = \sum_l (\Ae)_{kl}$ and define the
row-stochastic matrix  $\Pe = \De^{-1}\Ae$.
\item The asymmetric normalized graph Laplacian is $\Le = (I-\Pe)/\epsilon$.
\end{enumerate}
%Note that the symmetric normalized graph Laplacian, $\tLe \propto I -
%\De^{-1/2} \Ae \De^{-1/2}$, is similar to $\Le$ in the linear algebra
%sense \todo{REMOVE?}.
We will use $\Le$ from now on in place of the generic Laplacian matrix
$L$.

Regardless of the sampling density, as $n \to
\infty$ and $\epsilon(n) \to 0$ at the appropriate rate, $\Le$
converges (with probability 1) to the Laplace-Beltrami operator on
$\cM$: ${\Le \to -c \lap}$, for some $c > 0$, uniformly pointwise
\cite{Hein2005,Singer2006,Coifman2006,Ting2010} and in spectrum
\cite{Belkin2008th}.  The concept of correcting for sampling density
was first suggested in \cite{Lafon2004}. 

This convergence forms the basis of our consistency argument.
We first introduce a result of \cite{Singer2006}, which shows that
with probability 1, as $n \to \infty$, the system \eqref{eq:lrlssys} with
$L=\Le_{(n)}$ consistently models the Laplace-Beltrami operator in the
$\ell_\infty$ sense.

Let $\pi_\cX : L^2(\cM) \to \bbR^n$ map any square integrable function
on the manifold to the vector of its samples on the discrete set $\cX
\subset \cM$.

\begin{thm}[Convergence of $\Le$: \cite{Singer2006}, Eq.~1.7]
\label{thm:convLe}
Suppose we are given a compact Riemannian manifold $\cM$ and smooth
function $f : \cM \to \bbR$.  Suppose the points $\cX$ are sampled iid
from everywhere on $\cM$.  Then, for $n$ large and $\epsilon$ small,
with probability $1$
$$
(\Le \pi_\cX(f))_i = -c \lap f(x_i)
  + O\left(\frac{1}{n^{1/2} \epsilon^{1/2 + d/2}} + \epsilon \right),
$$
where $\lap$ is the negatively defined Laplace-Beltrami operator on
$\cM$, and $c>0$ is a constant.  Choosing $\epsilon = C
n^{-1/(3+d/2)}$ (where $C$ depends on the geometry of $\cM$) leads to
the optimal bound of $O(n^{-1/(3+d/2)})$.  Following
\cite{Hein2005}, the convergence is uniform.
\end{thm}
As convergence is uniform in Thm.~\ref{thm:convLe}, we may write the bound
in terms that we will use throughout:
\beq
\label{eq:bdconvLe}
\norm{\Le \pi_\cX (f) - c' \pi_\cX (\lap f)}_\infty =
  O(n^{-1/(3+d/2)}),
\eeq
where $c'=-c$ and $\norm{\cdot}_\infty$ is the vector infinity norm.

We now show that $M_n$ (as defined in \eqref{eq:lrlssys}) is consistent:
\begin{cor}[Consistency of $M_n$]
\label{cor:consistency}
Assume that $\bdM = \emptyset$ or that $\bdM \subset \cA$.  Let $F_n$ be
the following operator for functions $f \in C^2(\cAp) \cap C^0(\cA)$:
\beq
\label{eq:defFn}
F_n f(x) = \begin{cases}
 \gamma_A(n) f(x) - c \gamma_I(n) \lap f(x) & x \in \cAp \\
 f(x) & x \in \cA
\end{cases}
\eeq

Then, under the same conditions as in Thm.~\ref{thm:convLe}, with probability $1$
for $n$ large and $\epsilon$ chosen as in Thm.~\ref{thm:convLe},
$\norm{M_n \pi_\cX(f) - \pi_\cX(F_n f)}_\infty = O(n^{-1/(3+d/2)})$.
\end{cor}

\begin{proof}
This follows directly from Thm.~\ref{thm:convLe} and the fact that
$(E_\cS f)_i = f(x_i) \delta_{x_i \in \cS}$ for any set
$\cS \subset \cM$.
\end{proof} 

Our notion of stability is described in terms of certain limiting
inequalities.  We use the notation $a_n <_n b_n$ to mean that
there exists some $n_0$ such that for all $n' > n_0$, $a_{n'} <
b_{n'}$.

\begin{prop}[Stability of $M_n$]
\label{prop:fstable}
Suppose that $\gamma_I(n)/\gamma_A(n) <_n
\epsilon(n)/\kappa$, $\kappa>2$.  Then
$\norm{M_n^{-1}}_\infty = O(\gamma_A^{-1}(n))$.
\end{prop}
\noindent The proof of Prop. \ref{prop:fstable} is mainly technical and is given
in~\S\ref{sec:deferproof}.

If we modify $M_n$ as 
\beq
\label{eq:lrlssysp}
M'_n = E_\cA + E_\cAp(I + h^2 \Le),
\eeq
we can also state the following corollary, which will be useful in
chapter~\ref{ch:rlapp}.

\begin{cor}[Stability of $M'_n$]
\label{cor:fpstable}
Suppose $h^2(n) <_n \epsilon(n)/\kappa$, $\kappa>2$.  Then
$\norm{{M'}_n^{-1}}_\infty = O(1)$.
\end{cor}

\subsection{SSL Problem -- the Large Sample Limit -- Convergence Theorem}
\label{sec:ssllimitthm}

We are now ready to state and prove our main theorem about the convergence of
$\tf$. We assume that $\cM$ has empty boundary, or that $\bdM \subset
\cA$.  For the case of nonempty manifold boundary $\bdM \not \subset
\cA$, there are additional constraints at this boundary in the
resulting PDE.  We discuss this case in~\S\ref{app:derivbd}.
Note, by assumption \ref{it:as2} of~\S\ref{it:as2}, the anchor set
$\cA$ and its boundary $\partial \cA$ are not empty.

\begin{thm}[Convergence of $\tf$ under $M_n$]
\label{thm:fconv}
Consider the solution of $M_n \tf_n = \wt{w}$
(Eq.~\eqref{eq:lrlssys} with $L=\Le$), with $\cM$, $\cA$, and $w$ as 
described in~\S\ref{sec:sslassume}, and with either $\bdM =
\emptyset$ or $\bdM \subset \cA$.
Further, assume that
\begin{enumerate}
\item \label{it:fc1} $\epsilon(n)$ shrinks as given in
  Thm.~\ref{thm:convLe}.
\item \label{it:fc2} $\lim_{n \to \infty} \epsilon(n)/\gamma_A(n) = 0$.
\item \label{it:fc3} As in Prop. \ref{prop:fstable}, $\gamma_I(n) <_n
  \gamma_A(n) \epsilon(n) / \kappa$ for some $\kappa > 2$.
\end{enumerate}

\noindent Then for $n$ large, with probability 1,
$$
\norm{\tf_n - \pi_\cX f_n}_\infty <_n \epsilon(n)/\gamma_A(n)
$$
where the function $f_n \in C^2(\cM)$ is the unique, smooth, solution
to the following PDE for the given $n$:
\beq
\label{eq:lrlssim2}
f_n(x) - c\frac{\gamma_I(n)}{\gamma_A(n)} \lap f_n(x) = 0, \quad x \in \cAp
\qquad \text{and} \qquad
f_n(x) = w(x), \quad x \in \cA.
\eeq
\end{thm}

\begin{proof}[Proof of Thm.~\ref{thm:fconv}, Convergence]
Let $F_n$ be given by \eqref{eq:defFn} and
let $f_n$ be the solution to \eqref{eq:lrlssim2}.  The existence and
uniqueness of $f_n$, under assumptions \ref{it:as1}-\ref{it:as3} of
\S\ref{sec:sslassume}, is well known (we show it in
\S\ref{sec:rlpde}).

We bound $\norm{\tf_n-\pi_\cX f_n}_\infty$ as follows:
\begin{align*}
\norm{\tf_n - \pi_\cX f_n}_\infty &= \norm{M_n^{-1}\wt{w} - \pi_\cX f_n}_\infty \\
 &= \norm{M_n^{-1} (\pi_\cX (F_n f_n) - M_n \pi_\cX f_n)}_\infty \\
 &\leq \norm{M_n^{-1}}_\infty \norm{\pi_\cX(F_n f_n) - M_n \pi_\cX f_n}_\infty.
\end{align*}

From stability (Prop. \ref{prop:fstable}), $\norm{M_n^{-1}}_\infty =
O(\gamma_A^{-1}(n))$, and by consistency (Cor. \ref{cor:consistency}),
with probability 1 for large $n$,
$\norm{\pi_\cX(F_n f_n) - M_n  \pi_\cX f_n}_\infty = O(\epsilon(n))$.

The theorem follows upon applying assumption \ref{it:fc2}.
\end{proof}

If we modify $M_n$ in Thm.~\ref{thm:fconv} as in \eqref{eq:lrlssysp},
we obtain Cor. \ref{cor:fpconv} below.

\begin{cor}[Convergence of $\tf'$ under $M'_n$]
\label{cor:fpconv}
Consider the solution of $M'_n \tf_n' = \wt{w}$, with $\cM$, $\cA$,
and $w$ as described in~\S\ref{sec:sslassume}, and with $\partial
\cM = \emptyset$ or $\bdM \subset \cA$.
Further, assume $\epsilon(n)$ shrinks as given in
Thm.~\ref{thm:convLe}, and $h(n)$ as in
Cor. \ref{cor:fpstable}.  Then for $n$ large, with probability 1,
$$
\norm{\tf_n' - \pi_\cX f_n}_\infty = O(\epsilon(n))
$$
where $f_n \in C^2(\cM)$ again solves \eqref{eq:lrlssim2}:
$$
f_n(x) - h^2(n) \lap f_n(x) = 0, \quad x \in \cAp
\qquad \text{and} \qquad
f_n(x) = w(x), \quad x \in \cA.
$$
\end{cor}

\begin{proof}[Proof of Cor. \ref{cor:fpconv}]
Similar to that of Thm.~\ref{thm:fconv}, with stability given by
Cor. \ref{cor:fpstable}.
\end{proof}

We will call \eqref{eq:lrlssim2} the Regularized Laplacian PDE (RL
PDE) with the regularization parameter $h(n) = \sqrt{c \gamma_I(n) /
\gamma_A(n)}$.  By assumptions \ref{it:fc1} and \ref{it:fc2} of
Thm.~\ref{thm:fconv} , $h(n) \to 0$ as $n \to \infty$.  This, and the
analysis in~\S\ref{sec:sslassume}, motivate us to study the RL PDE
when $h$ is small to gain some insight into its solution, and hence
into the behavior of the original SSL problem.

\subsection{The Regularized Laplacian PDE}
\label{sec:rlpde}
We now study the RL PDE in greater detail.
We will assume a basic knowledge of differential geometry on compact
Riemannian manifolds throughout.  For basic definitions and notation,
see App.~\ref{app:diffgeom}.

We rewrite the RL PDE \eqref{eq:lrlssim2}, now
denoting the explicit dependence on a parameter $h$, and making the
problem independent of sampling:
\beq
\label{eq:rl}
  - h^2 \lap f_h(x) + f_h(x) = 0 \quad x \in \cAp
  \qquad\text{and}\qquad
  f_h(x) = w(x) \quad x \in \cA,
\eeq
where $\cM$, $\cA$, and $w$ are defined as in~\S\ref{sec:sslassume}.
The idea is that $f_h$ is specified smoothly on $\cA$ and by solving
\eqref{eq:rl} we seek a smooth extension of $f_h$ to all of $\cM$.

The RL PDE, \eqref{eq:rl}, has been well studied
\cite[Thm.~6.22]{Gilbarg1983}, \cite[App.~A]{Jost2001}.
It is uniformly elliptic, and for $h>0$ it admits a unique, bounded,
solution ${f_h \in C^2(\cAp) \cap C^0(\cM)}$.
The boundedness of $f_h$ follows from the strong maximum principle
\cite[Thm.~3.5]{Gilbarg1983}.  One consequence is that $f_h$
will not extrapolate beyond an interval determined by the anchor
values.

\begin{prop}[\cite{Jost2001},~\S A.2] \label{prop:bound}
The RL PDE \eqref{eq:rl} has a unique, smooth solution that is bounded within the range
$(w_-,w_+)$ for $x \in \cAp$, where
$$
%\label{eq:bounded}
w_- = \inf_{y \in \cA} \min(w(y),0)
\qquad\text{and}\qquad
w_+ = \sup_{y \in \cA} \max(w(y),0).
$$
\end{prop}

Our goal is to understand the solution of \eqref{eq:rl} as the
regularization term vanishes, i.e., $h \downarrow 0$.  To do so,
we introduce the Viscous Eikonal Equation.

\subsection{The Viscous Eikonal Equation}
\label{sec:ve}
The RL PDE is closely related to what we will call the Viscous Eikonal
(VE) equation.  This is the following ``smoothed'' Hamilton-Jacobi
equation of Eikonal type:
\beq
\label{eq:ve}
  - h \lap S_h(x) + \norm{\grad S_h(x)}^2 = 1 \quad x \in \cAp
  \qquad\text{and}\qquad
  S_h(x) = u_h(x) \quad x \in \cA.
\eeq
The term containing the Laplacian is called the \emph{viscosity term},
and $h$ is called the \emph{viscosity parameter}.

The two PDEs, \eqref{eq:rl} and \eqref{eq:ve}, are connected via
the following proposition:
\begin{prop}
\label{prop:rl2ve}
Consider \eqref{eq:rl} with $h>0$ and $w(x) > 0$, and \eqref{eq:ve} with
$u_h(x) = -h \log w(x)$.  When solution $f_h$ exists for
\eqref{eq:rl}, then $S_h = -h \log f_h$ is the unique, smooth,
bounded solution to \eqref{eq:ve}.
\end{prop}

\begin{proof}
Let $f_h$ be the unique solution of \eqref{eq:rl}.  From Prop. \ref{prop:bound},
$f_h > 0$ for $x \in \cAp$.  Apply the inverse
of the smooth monotonic bijection $\tau_h(t) = e^{-t/h}$, $\tau_h :
[0,\infty) \to (0,1]$ to $f_h$.  Let $R_h = -h \log f_h$, hence $f_h =
e^{-R_h/h}$.

We will need the standard product rule for the divergence
``$\grad \cdot$''.  When $f$ is a differentiable function and $\bm{f}$
is a differentiable vector field,
\beq
\label{eq:prodrule}
\grad \cdot (f \bm{f}) = (\grad f) \cdot \bm{f} + (\grad \cdot \bm{f}) f.
\eeq

As $f_h$ is harmonic, on $\cAp$:
\begin{align}
\label{eq:rl2ve1}
0 &= -h^2 \grad \cdot \grad e^{-R_h/h} + e^{-R_h/h} \\
\label{eq:rl2ve2}
&= h \grad \cdot (e^{-R_h/h} (\grad R_h)) + e^{-R_h/h} \\ %\text{ (chain rule)}\\
\label{eq:rl2ve3}
&= h (\grad e^{-R_h/h} \cdot \grad R_h
 + e^{-R_h/h} \grad \cdot \grad R_h) + e^{-R_h/h} \\ %\text{ (product rule)} \\
\label{eq:rl2ve4}
&= e^{-R_h/h} (- \norm{\grad R_h}^2 + h \lap R_h + 1).
\end{align}
Here, from \eqref{eq:rl2ve1} to \eqref{eq:rl2ve2} we use the
chain rule, and from \eqref{eq:rl2ve2} to \eqref{eq:rl2ve3} we
use the product rule \eqref{eq:prodrule}.
After dropping the positive multiplier in \eqref{eq:rl2ve4},
we see that that $R_h$ satisfies the first part of \eqref{eq:ve}.
Further, $R_h \in C^2(\cAp) \cap C^0(\cM)$ because $\tau_h$ is a
smooth bijection.  Similarly, $R_h$ is bounded because $f_h$ is:
${R_h \in [-h \log w_+, -h \log w_-)}$.

Finally, on the boundary $\cA$ we have $e^{-R_h/h} = w$, equivalently $R_h =
-h \log w$.  Hence, $R_h$ solves \eqref{eq:ve}.

\end{proof}

%We use the connection formed in Prop. \ref{prop:rl2ve} to study
%\eqref{eq:rl} for vanishing $h$.

To summarize: for $h>0$, \eqref{eq:ve} has a unique solution $S_h \in
C^2(\cAp)$ (Prop. \ref{prop:rl2ve}; see also
\cite{Gilbarg1983,Jost2001}).

We are interested in the solutions of the VE Eq. for the case of $h=0$
as well as for solutions obtained for $h>0$ small and for more general
$u_h$. When $h=0$ and $u_h=0$, it is well known that on a compact
Riemannian manifold $\cM$, \eqref{eq:ve} models propagation from $\cA$
through $\cAp$ along shortest paths.  Results are known
for a number of important cases, and we will discuss them after
describing the following assumption.

\theoremstyle{plain}
\newtheorem{assm}{Assumption}[section]
\begin{assm}
\label{conj:vedist}
For $h=0$ and $u_0$ sufficiently regular, \eqref{eq:ve} has the
unique \emph{viscosity} solution:
\beq
\label{eq:vedist}
S_0(x) = \inf_{y \in \cA} \left( d(x,y) + u_0(y) \right),
\eeq
where $d(x,y)$ is the geodesic distance between $x$ and $y$ through $\cAp$.
Furthermore, as $h \to 0$, $S_h$ converges to $S_0$ in $L^p(\cAp)$, $1
\leq p <  \infty$, and in $(L^\infty(\cAp))^*$ (i.e. essentially
pointwise) when $u_h$ converges to $u_0$ in the same sense.  The rate
of convergence is 
%given as 
$\norm{S_0-S_h}_\infty^* = O(h)$.
\end{assm}
\noindent From now on we will denote  $S_0$ simply by $S$.

\begin{proof}[Discussion of Assum. \ref{conj:vedist} on compact $\cM$]
To our knowledge, a complete proof of \eqref{eq:vedist} for compact
Riemannian $\cM$ is not known; the theory of unique viscosity
solutions (nondifferentiable in some areas), on manifolds is
an open area of research \cite{Crandall1992, Azagra2008}.  However,
below we cite known partial results.

Eq.~\eqref{eq:vedist} was shown to hold for $u_0 = 0$ on compact
$\cM$ in \cite[Thm.~3.1]{Mantegazza2002}, and for $u_0$ sufficiently regular on
bounded, smooth, and connected subsets of $\bbR^d$ in \cite[Thms. 2.1, 6.1,
6.2]{Lions1982}, and e.g., when $u_0$ is Lipschitz 
\cite[Eq.~4.23]{Kruzkov1975}.  Convergence and the convergence rate of
$S_h$ to $S_0$ were also shown on such Euclidean subsets in
\cite[Eq.~69]{Lions1982}.  Conditions of convergence to a viscosity
solution are not altered under the exponential map \cite[Cor.
2.3]{Azagra2008}, thus convergence in local coordinates around
$\cA$ (which follows from \cite[Thm.~6.5]{Lions1982} and
Prop. \ref{prop:rl2ve}) implies convergence on open subsets of
$\cAp$. However, global convergence of $S_h$ to $S_0$ on $\cM$ is
still an open problem. 

Not surprisingly, despite the lack of formal proof, and in light of
the above evidence, our numerical experiments on a variety of
nontrivial compact Riemannian manifolds (e.g. compact subsets of
hyperbolic paraboloids) give additional evidence that this convergence
\emph{is} achieved.
\end{proof}

\subsection{What happens when $h$ converges to $0$: Transport Terms}
\label{sec:transport}
To study the relationship between $S_h$ and
$S$,  we look for a higher order expansion of $f_h$ using a tool
called Transport Equations \cite{Schuss2005}.

Assume $f_h$ can be expanded into the following form:
\beq
\label{eq:trans}
 f_h(x) = e^{-R(x)/h} \sum_{k \geq 0} h^{\alpha k} Z_k(x)
\eeq
with $\alpha>0$.  The terms $Z_k$, $k = 0,1,\ldots$ are called the
transport terms.  Substitution of this form into \eqref{eq:rl} will
give us the conditions required on $R$ and $Z_k$.

\begin{thm}
\label{thm:transport}
If \eqref{eq:trans}, (with $\alpha = 1$) is a solution to the RL PDE
\eqref{eq:rl} for all $h>0$, then:
\beq
\forall x \in \cA \text{ and } \forall k \geq 1 \qquad
 R(x) = 0,
 \qquad Z_0(x) = w(x),
 \qquad Z_k(x) = 0,
\eeq
and \eqref{eq:rl} reduces to a series of PDEs:
\begin{align}
0 &= -\norm{\grad R}^2 + 1 \nonumber \\
\label{eq:trl}
0 &= Z_0 \lap R + 2 \grad R \cdot \grad Z_0 \\
0 &= Z_{k} \lap R + 2 \grad R \cdot \grad Z_{k} - \lap Z_{k-1},
 \quad k > 0 \nonumber
\end{align}

In particular, letting $d_\cA(x) = \inf_{y \in \cA} d(x,y)$ denote
the shortest geodesic distance from $x$ to $\cA$, we have that
$R(x) = d_\cA(x)$ everywhere.

\end{thm}

\begin{proof}
The anchor conditions follow from the fact that for all $h>0$, $w=Z_0
e^{-R/h}+ h Z_1 e^{-R/h} + \ldots$ (thus forcing $R=0$ and
therefore $Z_0 = w$, and $Z_k = 0, \forall k > 0$).

Plugging \eqref{eq:trans} into \eqref{eq:rl}, and applying the product
and chain rules, we get
$$
0 = \sum_{k \geq 0} e^{-R/h} h^k \left(-h^2 \lap Z_k + 2 h \grad R \cdot \grad Z_k
 + h Z_k \lap R - Z_k \norm{\grad R}^2 + Z_k \right).
$$
Eqs. \eqref{eq:trl} follow after collecting like powers of $h$ and
simplifying.
\end{proof}

Thm.~\ref{thm:transport} shows first that $R$ is determined by the
Eikonal equation with zero boundary conditions.  Second, it shows that
$Z_0$ is the dominant term affected by the boundary values $w$ as $h
\downto 0$.  For $k>0$, the transport terms $Z_k$ are affected by $w$
via $Z_{k-1}$, but these are not the dominant terms for small $h$.
The existence, uniqueness, and smoothness of $Z_0$ on $\cAp$ and
within the cut locus of $\cA$, is proved in~\S\ref{sec:deferproof}
(Thm.~\ref{thm:Zsmooth}).

%Coming back to Assum. \ref{conj:vedist}, Thm.~\ref{thm:transport}
%also provides evidence that the order of convergence in
%Prop. \ref{prop:rl2ve} is $O(h)$:
Note that the choice of $\alpha=1$ is not arbitrary.
For $\alpha<1$ in \eqref{eq:trans}, \eqref{eq:rl} does not admit a
consistent set of solvable transport equations.
For $\alpha = 2$, the resulting transport equations reduce to those of
Eqs. \eqref{eq:trl} (the nonzero odd $k$ terms are forced to zero and
the even $k$ terms are related to each other via Eqs. \eqref{eq:trl}).

%% Note also that $Z_0$ is closely related to a diffusion
%% equation modeling anelastic (a special type of incompressible) flow.
%% Specifically, the second of Eqs. \eqref{eq:trl} is equivalent to
%% $0 = \grad \cdot (Z_0^2 \grad S)$, where $S$ is the (known) diffusion
%% and $Z_0^2$ is the diffusivity.

\subsection{Manifold Laplacian and Vanishing Viscosity}
\label{sec:lap}

We now combine Assum. \ref{conj:vedist} and Thm.~\ref{thm:transport} in
a way that summarizes the solution of the RL PDE \eqref{eq:rl} for
small $h$, taking into account possible arbitrary nonnegative boundary
conditions.

\begin{thm}
\label{thm:rl2e}
Let $w(x)\geq 0, \forall x \in \cA$ and let $\hat{\cA} = \supp{w}$.
Further, define
$
  f^*_h(x) = w(x') e^{-d_{\cA}(x)/h}
$
where $x' = \arg\inf_{y \in \cA} d(x,y)$.
Then for a situation where Assum. \ref{conj:vedist} holds,
and for small $h > 0$, the solution of \eqref{eq:rl} with sufficiently
regular anchor $\cA$, satisfies:
\begin{align}
\abs{f_h(x) - f^*_h(x)} &= O(h e^{-d_\cA(x)/h}); \text{and}
\label{eq:rlappr}\\
\lim_{h \to 0} -h \log f_h(x) &= d_{\hat{\cA}}(x). \label{eq:rllim}
\end{align}
\end{thm}

\begin{proof} First, apply Thm.~\ref{thm:transport} to decompose $f_h$
in terms of $R$ and the transport terms $Z_k$, $k \geq 0$.  Next,
by Thm.~3.1 of \cite{Mantegazza2002}, as discussed in
Assum. \ref{conj:vedist}, we obtain $R(x) = d_{\cA}(x)$.
We can therefore write $f_h(x) = Z_0(x) e^{-d_\cA(x)/h} + O(h
e^{-d_\cA(x)/h})$.  Further, $Z_0(x)$ is unique and smooth 
within an intersection of $\cA'$ and a cut locus of $\cA$,
and satisfies the boundary conditions ($Z_0(x) = w(x)$ for $x \in
\cA$) .  This can be shown using the method of characteristics
(Thm.~\ref{thm:Zsmooth}).  This verifies \eqref{eq:rlappr}.

Showing that \eqref{eq:rllim} holds requires more work due to possible
zero boundary conditions on $\cA$.  To prove \eqref{eq:rllim}, we
find a sequence of PDEs, parametrized by viscosity $h$ and ``height''
$c>0$; we denote these solutions $\hat{f}_{h,c}(x)$.  These solutions
match $f_h$ as $h \to 0$.    We then show that for large $c$, they
also match $f_h$ for nonzero $h$.

Let $\hat{\cA}(c,h) = \set{x \in \cA : w(x) > e^{-c/h}}$ and
$\cA_0(c,h) = \set{x \in \cA : w(x) \leq e^{-c/h}}$.  We define
$\hat{f}_{h,c}$ as the solution to \eqref{eq:rl} with the modified
boundary conditions
\beq
w_{h,c}(x) = \begin{cases}
 w(x) & x \in \hat{\cA}(c,h) \\
 e^{-c/h} & x \in \cA_0(c,h)
\end{cases}. \label{eq:whc}
\eeq
This is a modification of the original problem with a lower bound
saturation point of $e^{-c/h}$.  Clearly, as $h \to 0$, $w_{h,c}(x)
\to w(x)$ on the boundary.

As in Prop. \ref{prop:rl2ve}, for fixed $c$ we can write $w_{h}(x) =
e^{-u_{h}(x)/h}$ for any $x \in \cA$ and $h>0$.  Then 
$$
u_{h}(x) = \begin{cases} c & x \in \cA_0(c,h) \\
  -h \log w(x) & x \in \hat{\cA}(c,h)
\end{cases}.
$$
Let $\cA_0 = \set{x \in \cA : w(x) = 0}$ and $\hat{\cA} = \set{x
\in \cA : w(x) > 0}$, and define
$$
u_{0}(x) = \begin{cases} c & x \in \cA_0 \\
  0 & x \in \hat{\cA}
\end{cases}.
$$
Then
$$
u_h(x)-u_0(x) = \begin{cases}
0 & w(x) = 0 \\
c & 0 < w(x) \leq e^{-c/h} \\
-h \log w(x) & w(x) > e^{-c/h}
\end{cases}.
$$
As $w(x)$ is regular and $\cA$ is compact, $u_h \to u_0$ pointwise on
$\cA$, and the convergence is also uniform.
Clearly, then, $u_h \to u_0$ almost everywhere on $\cA$.  Furthermore,
for any $h>0$, $u_h \leq c$ everywhere on $\cA$.
Therefore $u_h \to u_0$ in $L^p$ for all $p \geq 1$
\cite[Prop. 6.4]{Kubrusly2007}.
The rate of convergence, $O(h)$, is determined
by the set of points $\set{x : w(x) > e^{-c/h}}$.
Thus, by Prop. \ref{prop:rl2ve} and Assum. \ref{conj:vedist},
\beq
\hat{S}_{0,c}(x) = \lim_{h \to 0} -h \log \hat{f}_{h,c}(x) = \min
\left(d_{\hat{\cA}}(x), c + d_{\cA_0}(x)\right).
\label{eq:Scx}
\eeq

To match the boundary conditions of $\hat{f}_{h,c}$ to those of
$f_h$ for a fixed $h>0$, we must choose $c$ large in \eqref{eq:whc}.
Subsequently, when $c$ is large in \eqref{eq:Scx}, e.g., when $c \geq
\text{diam}(\cM)$, we have $\hat{S}_{0,c}(x) = d_{\hat{\cA}}(x)$.
This verifies \eqref{eq:rllim}.
\end{proof}

When $w(x) = 1$ on $\cA$, and for small $h$, the exponent
of $f_h$ directly encodes $d_\cA$.  The following simple
example illustrates Thm.~\ref{thm:rl2e}. Additional examples on the
Torus $T=S^1 \x S^1$ and on a complex triangulated mesh are included
in~\S\ref{sec:geoexmp} of chapter~\ref{ch:rlapp}.

\begin{exmp}[The Annulus in $\bbR^2$]
\label{ex:annul}
Let $\cM = \set{r_0 \leq r \leq 1}$, where $r=\norm{x}$ is the
distance to the origin.  Let $\cA = \set{r=r_0} \cup \set{r=1}$ be the inner
and outer circles.  Letting $w=1$ ($u_h=0$), we get $S(r) = d_{\cA}(r)
= \min(1-r,r-r_0)$.
For symmetry reasons, we can assume a radially symmetric solution to
the RL Eq.  For a given dimension $d$, the radial Laplacian is
$\lap f(r) = f''(r) + (d-1)r^{-1} f'(r)$.
So \eqref{eq:rl}
becomes:
$-h^2 \left(f_h''(r) + r^{-1} f_h'(r) \right) + f_h(r) = 0$
for $r \in (r_0,1)$, and $f_h(r_0)=f_h(1)=1$.
The solution, as calculated in Maple \cite{Maple10}, is
$$
f_h(r) = \frac{I_0(r/h)K_0(1/h)-I_0(r/h)K_0(r_0/h) -
  K_0(r/h)I_0(1/h)+K_0(r/h)I_0(r_0/h)}{K_0(1/h)I_0(r_0/h) -
  K_0(r_0/h)I_0(1/h)},
$$
where $I_j$ and $K_j$ are the $j$'th order modified Bessel functions
of the first kind and second kind, respectively.
A series expansion of $f_h(r)$ around $h=0$ (partially calculated with
Maple) gives
\beq
\label{eq:fhr}
f_h(r) = \sqrt{r_0/r} e^{-(r-r_0)/h} + \sqrt{1/r} e^{-(1-r)/h} + O(h)
\eeq
As the limiting behavior of $f_h$, as $h$ grows small, depends on the
exponents of the two terms in \eqref{eq:fhr}, one can check that the limit
depends on whether $r$ is nearer to $r_0$ or $1$.  Depending on this,
one of the terms drops out in the limit.
From here, it is easy to check that $\lim_{h \to 0} -h \log
f_h(r) = S(r)$, confirming \eqref{eq:rllim}.

We simulated this problem with $r_0=0.25$ by sampling
$n=1500$ points from the ball $B(0,1.25)$, rescaling points having
$r \in [0,0.25)$ to $r=0.25$, and rescaling points having $r \in
  (1,1.25]$ to $r=1$.  $S_h$ is approximated up to a constant using
the numerical discretization, via \eqref{eq:lrlssysp}, of
\eqref{eq:lrlssim2}.  For the graph Laplacian we used a $k=20$
NN graph and $\epsilon=0.001$.

\begin{figure}[h!]
\centering
\begin{minipage}{\linewidth}
  \centering
  \parbox{.3\linewidth}{
    \centering
    \includegraphics[width=1.2\linewidth]{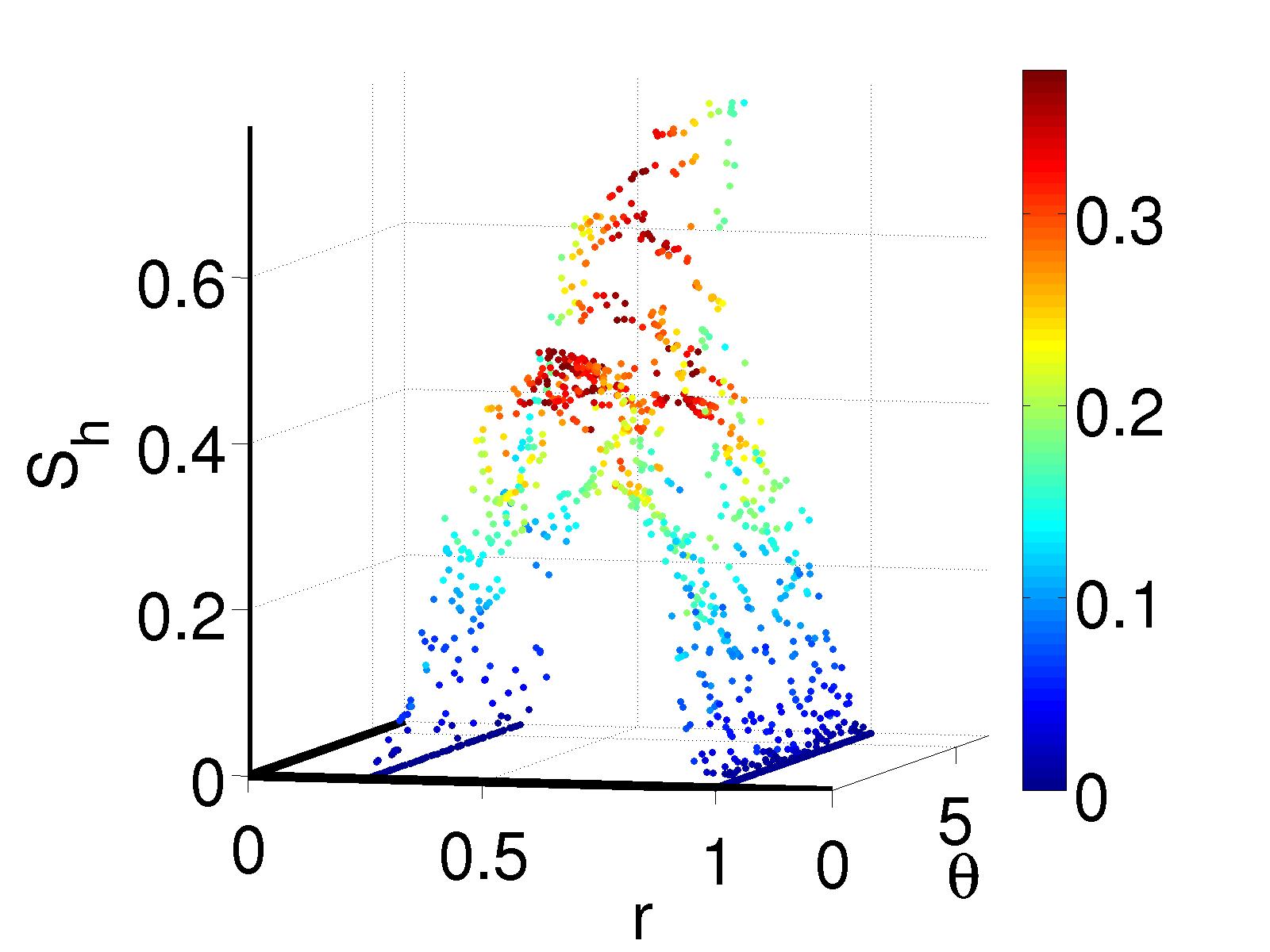}
  }
  \quad
  \parbox{.3\linewidth}{
    \centering
    \includegraphics[width=1.2\linewidth]{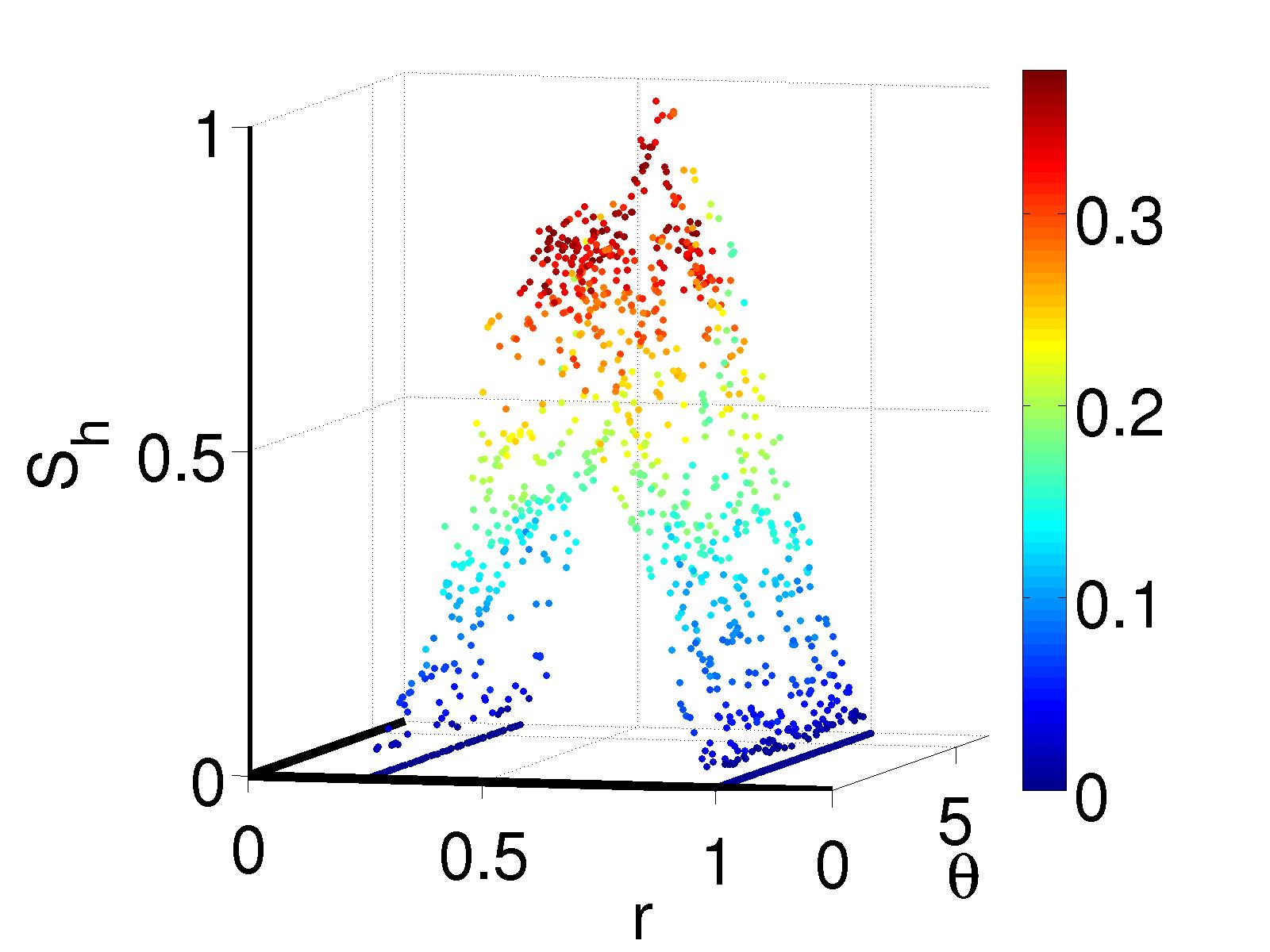}
  }
  \quad
  \parbox{.3\linewidth}{
    \centering
    \includegraphics[width=1.2\linewidth]{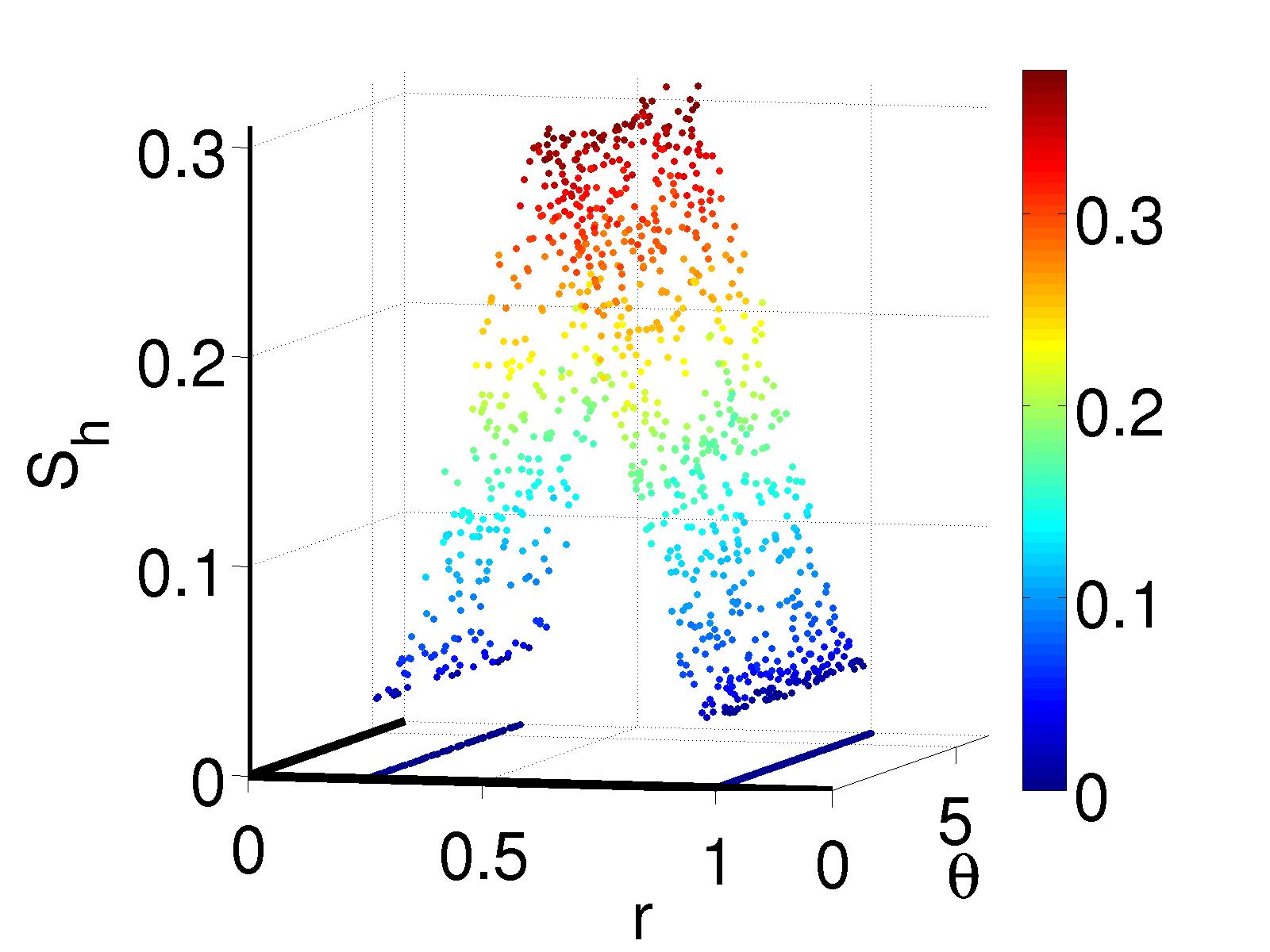}
  }
  \caption[Geodesics Estimates $S_h(r,\theta)$ on Annulus.]%
    {Geodesics Estimates $S_h(r,\theta)$ on Annulus.  Left to
    Right: $h=1,\ 0.1,\ 0.001$.}
  \label{fig:geoannul}
\end{minipage}
\end{figure}

Fig.~\ref{fig:geoannul} shows (in the $z$ axis) the estimate $S_h(x)$
as $h$ grows small.  The colors of the points reflect the true
distance to $\cA$: $S(r) = S_0(r) = \min(1-r,r-r_0)$.  Note the
convergence as $h \downto 0$, and also the clear offset of $S_h$ which
is especially apparent in the right panel at $r=0.25$ and $r=1$.

%% \hfill $\square$

%% Example \ref{ex:annul} raises several important questions about the
%% convergence of $f_h$ to $f^*_h$.  There is no leading term of the form
%% $\sqrt{r_0/r}$ \eqref{eq:rlappr}.  This factor is
%% important for calculation: it leads to the offset of
%% $S_h$ in the right pane of Fig.~\ref{fig:geoannul}.  We therefore
%% ask, for what values of $h$ is the approximation valid?
%% and what is the error in the approximation?
%% We give the simplest (and weakest) rule of thumb:

%% \begin{cor}
%% \label{cor:rough}
%% The approximation $f_h(x) \approx f^*_h(x)$ holds when
%% $S(x) \gg h$.
%% \end{prop}
%% \begin{proof}
%% Note that $\abs{f_h - f^*_h} = \abs{e^{-S_h/h}-e^{-S/h}}$
%% and $\abs{S-S_h} \leq C h$ for $h$ small. When
%% \begin{align*}
%% S_h > S
%%  &: \abs{f_h - f^*_h}
%%     = e^{-S/h} \left(1 - e^{-(S_h-S)/h} \right)
%%     \leq e^{-S/h} \left( 1-e^{-C} \right), \\
%% S > S_h
%%  &: \abs{f_h - f^*_h}
%% %    = e^{-S_h/h} \left(1 - e^{-(S-S_h)/h} \right)
%%     = e^{-S/h} e^{(S-S_h)/h} \left(1 - e^{-(S-S_h)/h} \right)
%%     \leq e^{-S/h} \left(e^C-1 \right).
%% \end{align*}
%% The result follows after combining these two inequalities.
%% \end{proof}

%% For geodesics estimation, $f_h \approx f_h^*$ when $d_\cA(x) \gg h$ (because $w=1$
%% and therefore $S = d_\cA(x)$).

%Example \ref{ex:annul} can now be revisited, and
%we will focus on finding the term $Z_0$:

%% \begin{exmp}[The Annulus in $\bbR^2$: $Z_0$]
%% \label{ex:annultrans}

From the second of Eqs. \eqref{eq:trl} and the fact that $\grad S(r) =
1, \lap S(r) = 1/r$, we have $Z_0/r + 2 Z'_0 = 0$ for $r_0 < r \leq 1$,
and $Z_0 = 1$ for $r \in \set{r_0,1}$. To solve this near $r=r_0$, we use
the boundary condition $Z_0(r_0)=1$ and get $Z_0(r) = \sqrt{r_0/r}$.
Likewise, near $r=1$ we use the boundary condition $Z_0(1)=1$ and get
$Z_0(r) = \sqrt{1/r}$.
Near $r=r_0$, the solution becomes $f_h(r) =
e^{-(r-r_0)/h}(\sqrt{r_0/r} + O(h))$, and near $r=1$, it becomes
$f_h(r) = e^{-(1-r)/h}(\sqrt{1/r} + O(h))$, which match the earlier
series expansion of the full solution.
Furthermore, upon an additional Taylor expansion near $r=r_0$, we have
$S_h(r) = r-r_0 - h \log(r_0/r)/2 + O(h \log h)$.
Note the extra term in the $S_h$ estimate,
which has a large effect when $r-r_0$ is small (as seen
in the right pane of Fig.~\ref{fig:geoannul}).  A similar expansion
can be made around the outer circle, at $r=1$.
\end{exmp}

\subsection{The SSL Problem of~\S\ref{sec:sslintro}, Revisited}
\label{sec:sslr}
Armed with our study of the RL PDE, we can now return
to the original SSL problem of~\S\ref{sec:sslintro}.

Suppose the anchor is composed of two simply connected domains $\cA_0$
and $\cA_1$, where $w$ takes on the constant values $c_0$ and $c_1$,
respectively, within each domain.
When $c_1 > c_0 \geq 0$, we can directly apply the result of
Thm.~\ref{thm:rl2e} to \eqref{eq:lrlssim2}.  The solution, for
$\gamma_I \ll \gamma_A$, is given by \eqref{eq:rlappr} with
$h=\sqrt{c  \gamma_I / \gamma_A}$:
$$
f(x) \approx w(x') \sup_{y \in \cA} {e^{- d(x,y) \sqrt{\gamma_A / c
      \gamma_I} }}
\qquad \text{where} \qquad x' = \arg\inf_{y \in \cA} d(x,y)
$$
The solution depends on both the geometry of $\cM$ (via the
geodesic distance to $\cA_0$ or $\cA_1$) and on
the values chosen to represent the class labels.  For example,
suppose $\cL = \set{0,1}$.  As $n$ grows large and $h$ grows small, we
apply \eqref{eq:rllim} to see that the classifier is biased towards
the class in $\cA_0$:
\beq
\label{eq:sslasymm}
f(x) \approx \sup_{y \in \cA_1} e^{-d(x,y) \sqrt{\gamma_A / c
    \gamma_I}}.
\eeq
Choosing the symmetric labels $\cL = \set{c_0,-c_0}$ is more natural.
In this case, we decompose \eqref{eq:rl} into two problems:
\begin{align*}
-h^2 \lap f_{h,0} + f_{h,0} &= 0 \quad x \in \cAp
\qquad \text{and} \qquad  f_{h,0} = c_0 \quad x \in \cA_0;
\quad  f_{h,0} = 0 \qquad x \in \cA_1 \\
-h^2 \lap f_{h,1} + f_{h,1} &= 0 \quad x \in \cAp
\qquad \text{and} \qquad  f_{h,1} = 0\ \quad x \in \cA_0;
\quad  f_{h,1} = -c_0 \quad x \in \cA_1,
\end{align*}
and note that by linearity of the problem and the separation of the
anchor conditions, the solution to \eqref{eq:rl} is
given by $f_h = f_{h,0} + f_{h,1}$.
Therefore, by taking $h=\sqrt{c \gamma_I / \gamma_A}$, we separate
\eqref{eq:lrlssim2} into two problems with
nonnegative anchor conditions (one in $f_{h,0}$ and one in $-f_{h,1}$).
Applying the result of Thm.~\ref{thm:rl2e} to each of these
individually, and combining the solutions, yields
\beq
\label{eq:sslsymm}
f(x) \approx c_0 \sup_{y \in \cA_0} e^{-d(x,y)/h}
             - c_0 \sup_{y \in \cA_1} e^{-d(x,y)/h}
     \propto e^{-d_{\cA_0}(x)/h} - e^{-d_{\cA_1}(x)/h}.
\eeq
This solution is zero when $d_{\cA_0}(x) = d_{\cA_1}(x)$, positive
when $d_{\cA_0}(x) < d_{\cA_1}(x)$, and negative otherwise.  That is,
in the noiseless, low regularization regime with symmetric anchor values,
algorithms like LapRLS classification assign the point $x$ to the
class that is closest in geodesic distance.
We illustrate this with a simple example of classification on the
sphere $S^2$.

\begin{exmp}
\label{ex:sslsphere}
We sample $n=1000$ points from the sphere $S^2$ at random, and define
the two anchors $\cA_0 = B_g((1,0,0),\pi/16)$ and $\cA_1 =
B_g((-1,0,0), \pi/16)$.  Here $B_g(x,\theta)$ is a cap of angle
$\theta$ around point $x$.  The associated anchor labels are $w_0 =
+1$ and $w_1 = -1$.
We discretize the Laplacian $\Le$ using $k=50$,
and $\epsilon = 0.001$ and solve \eqref{eq:lrlssys}.
Fig.~\ref{fig:sslsphere} compares the numerical
solutions at small $h$ to our estimates from \eqref{eq:sslsymm}.
The two solutions are comparable up to a positive multiplicative
factor (due to the fact that $\Le$ converges to $\lap$ times a
constant).  \hfill $\square$

\begin{figure}[h!]
\centering
\begin{subfigure}[b]{.49\linewidth}
  \includegraphics[width=\linewidth]{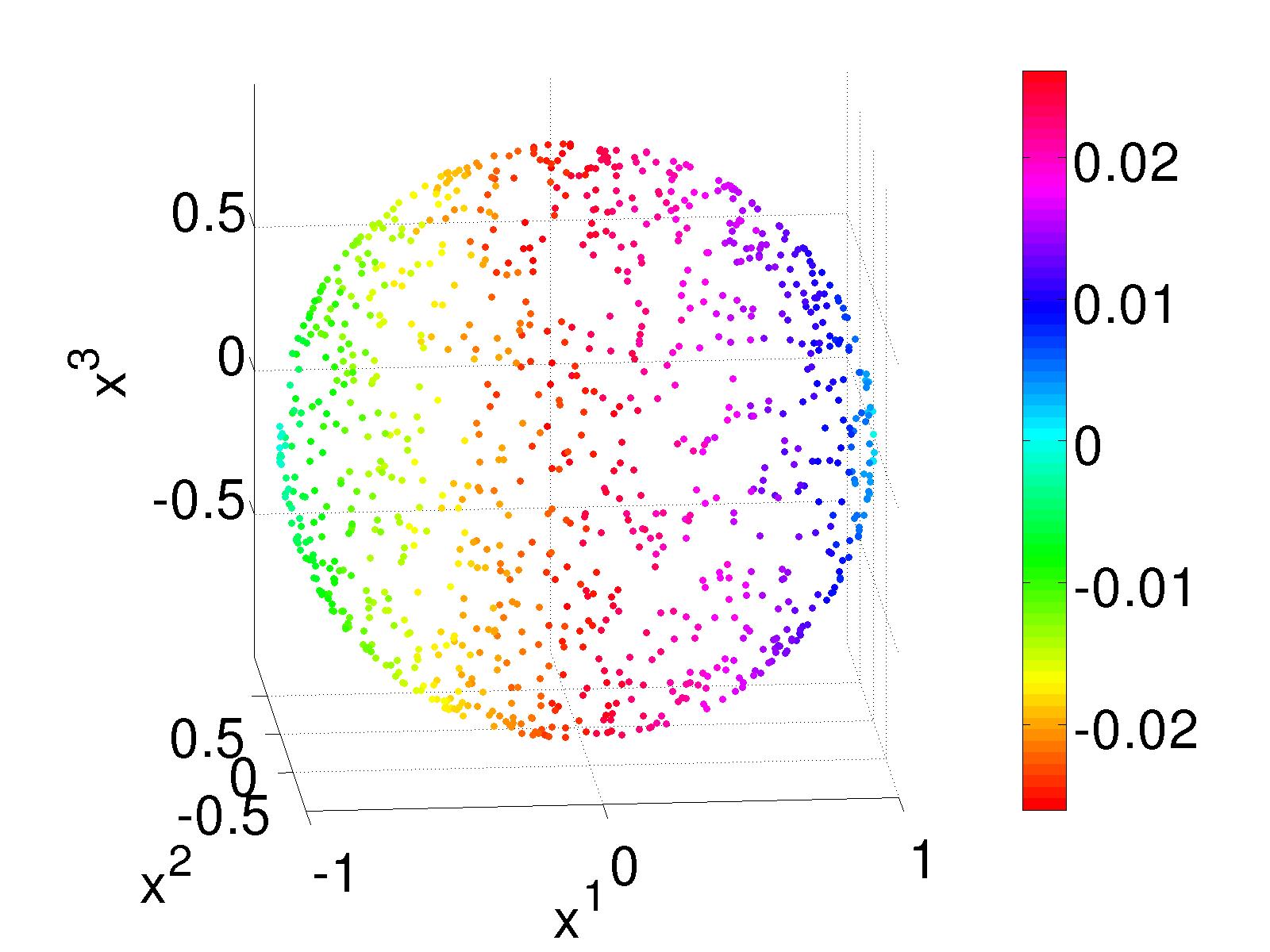}
  \caption{$f_h$ on $\cM$, numerical, log scale}
  \label{fig:sslsa}
\end{subfigure}
\begin{subfigure}[b]{.49\linewidth}
  \includegraphics[width=\linewidth]{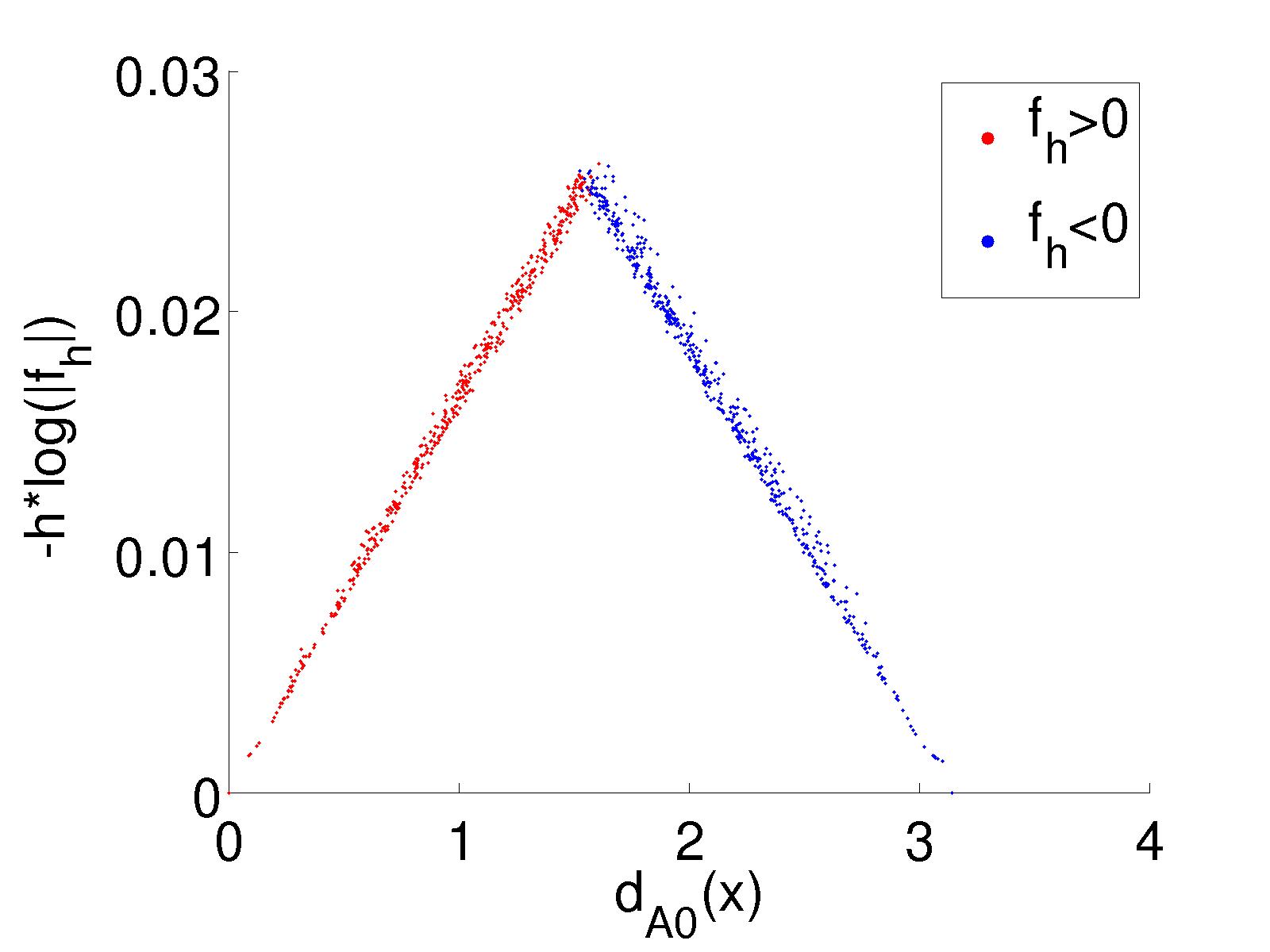}
  \caption{$S_h$ (numerical)} % (via \eqref{eq:lrlssys}).}
  \label{fig:sslsc}
\end{subfigure}
\begin{subfigure}[b]{.49\linewidth}
  \includegraphics[width=\linewidth]{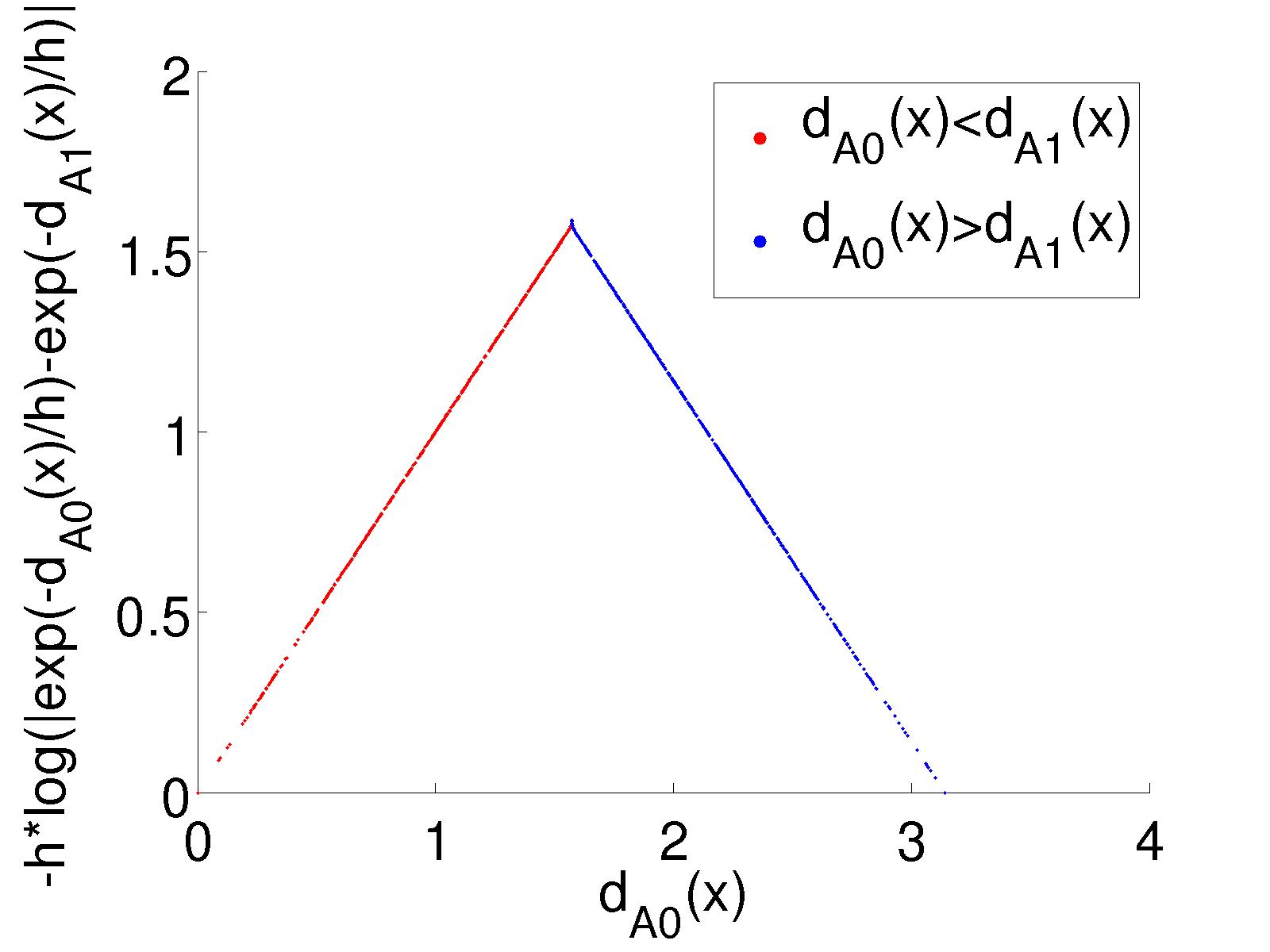}
  \caption{$S_h$ (prediction)} % (via \eqref{eq:sslsymm})}
  \label{fig:sslsb}
\end{subfigure}
\caption[Comparing solution of \eqref{eq:lrlssys} to model prediction \eqref{eq:sslsymm}]% 
  {Comparing solution of \eqref{eq:lrlssys} to model prediction \eqref{eq:sslsymm}.
   Colors
    in (\subref{fig:sslsa}), are given by $-h \sign(f_h) \log \abs{f_h}$;
    in (\subref{fig:sslsc}), encode $\sign(f_h)$;
    in (\subref{fig:sslsb}), encode closeness of $x$ to $\cA_0$ vs.
    $\cA_1$.
}
\label{fig:sslsphere}
\end{figure}
\end{exmp}

%%%%%%\appendix

%%%%%%%%%%%%%%%%%%%%%%%%% \input{appendices.tex}
\section{Technical Details}

\subsection{Deferred Proofs}
\label{sec:deferproof}

\subsubsection{Stability of $M_n$}
To prove the stability of $M_n$, we first need to present
some notation.  The matrices $\Pe, \Le,$ and $\De$ can be written in
terms of submatrices to simplify the exposition.
Separating these matrices into submatrices associated with the
$l$ labeled points and the $n-l$ unlabeled points, we write:
\begin{align*}
\Pe &= \pmat{P_{ll} & P_{lu} \\ P_{ul} & P_{uu} }, \quad
\Le = \pmat{L_{ll} & L_{lu} \\ L_{ul} & L_{uu}}
   = \frac{1}{\epsilon} \pmat{I - P_{ll} & - P_{lu} \\ - P_{ul} & I - P_{uu} }, \quad
\\
\text{and }
\De &= \pmat{D_{ll} & 0 \\ 0 & D_{uu} }.
\end{align*}
Note that the two identities in the definition of $\Le$ are of size $l$
and $n-l$, respectively.

We will also also need a lemma bounding the spectrum of the matrix
$I-P_{uu}$.

\begin{lem}
\label{lem:bdspec}
The matrix $I-P_{uu}$ is bounded in spectrum between $0$ and $1$.
\end{lem}
\begin{proof}
The hermitian matrix $\tilde{P}_\epsilon = \De^{-1/2} \Ae \De^{-1/2}$ has
eigenvalues bounded between $0$ and $1$.  The eigenvalues of its lower
right principal submatrix, $\tilde{P}_{uu}$, are therefore also bounded
between $0$ and $1$ \cite[Thm.~4.3.15]{Horn1985}.  Finally, $P_{uu}$
is similar to $\tilde{P}_{uu}$ via the transformation $P_{uu} =
D_{uu}^{-1/2} \tilde{P}_{uu} D_{uu}^{1/2}$.
\end{proof}

We are now ready to prove the stability of $M_n$.

\begin{proof}[Proof of Prop. \ref{prop:fstable}, Stability]
We first expand $M_n$ in block matrix form: %bound $\norm{M_n^{-1}}_\infty$:
$$
M_n = \pmat{I & 0 \\ \gamma_I L_{ul} & G_n}
\quad \text{where} \quad G_n = \gamma_A(n) I + \gamma_I(n) L_{uu}.
$$
From the block matrix inverse formula, the inverse of $M_n$ is:
$$
M_n^{-1} = \pmat{I & 0 \\ -G_n^{-1} \gamma_I L_{ul} &  G_n^{-1}},
$$
and the norm may be bounded as:
\beq
\label{eq:maxFn}
\norm{M_n^{-1}}_\infty \leq \max \left(1, \norm{G_n^{-1}\gamma_I
  L_{ul}}_\infty + \norm{G_n^{-1}}_\infty \right)
\eeq
where we use the inequalities $\norm{(A^T \ B^T)^T}_\infty =
\max(\norm{A}_\infty,\norm{B}_\infty)$ and $\norm{(C\ \ D)}_\infty \leq
\norm{C}_\infty + \norm{D}_\infty$.

%% \begin{align}
%% \nonumber
%% \norm{M_n^{-1}}_\infty &= \max\left(1,\max_{i=1}^{n-l}
%% \left( \sum_{j=1}^l |(-G_n^{-1}\gamma_I L_{ul})_{ij}|
%%   + \sum_{j'=1}^{n-l} |(G_n^{-1})_{ij'}| \right)  \right) \\
%% \nonumber
%% &\leq \max \left(1,\max_{i=1}^{n-l} \sum_{j=1}^l |(-G_n^{-1}\gamma_I L_{ul})_{ij}|
%%   + \max_{i=1}^{n-l} \sum_{j'=1}^{n-l} |(G_n^{-1})_{ij'}| \right) \\
%% \label{eq:maxFn}
%% &= \max \left(1, \norm{G_n^{-1}\gamma_I L_{ul}}_\infty
%%   + \norm{G_n^{-1}}_\infty \right).
%% \end{align}

We first expand $G_n^{-1} = (\gamma_A(n) I + \gamma_I(n) L_{uu})^{-1} =
\gamma_A^{-1}(n) \left(I + \frac{\gamma_I}{\gamma_A \epsilon(n) }(I - P_{uu})
\right)$.  
%$\tilde{P}_{uu}$ 
By Lem. \ref{lem:bdspec}, $I-P_{uu}$ is bounded in spectrum between $0$ and
$1$, and by the first assumption in the proposition, when $\kappa > 2$
we have $\gamma_I <_n \gamma_A \epsilon(n) / \kappa <_n \gamma_A(n) \epsilon(n)$.
Thus there exists some $n_0$ so that for all $n'>n_0$ we can write
$$
G_n^{-1}
 = \gamma_A^{-1}(n) \left(I + \frac{\gamma_I(n)}{\gamma_A(n) \epsilon(n) }(I - P_{uu}) \right)^{-1}
 = \gamma_A^{-1} \sum_{k = 0}^\infty \left( \frac{\gamma_I(n)}{\gamma_A(n)
   \epsilon(n)} \right)^k (P_{uu}-I)^k.
$$
Now we use this expansion to bound $\norm{G_n^{-1}}_\infty$.  Let $a_n
= \gamma_I(n) (\gamma_A(n) \epsilon(n))^{-1}$.  As the norm is
subadditive,
\beq
\label{eq:Gnibnd}
\norm{G_n^{-1}}_\infty  \leq \gamma_A^{-1}(n) \sum_{k \geq 0} a^k_n \norm{P_{uu}-I}_\infty^k.
\eeq
Furthermore, we can bound
$\norm{P_{uu}-I}_\infty^k$ as follows.  Since the entries of $P_{uu}$
are nonnegative, $\norm{P_{uu}-I}_\infty \leq
\norm{(P_{uu}+I)1}_\infty$ where $1$ is a vector of all ones.  Further,
since $P_{uu}$ is a submatrix of a stochastic matrix,
$\norm{(P_{uu}+I)1}_\infty \leq 2$.  Thus
since $a_n = \gamma_I(n) (\gamma_A(n) \epsilon(n))^{-1} <_n \kappa^{-1} < 1/2$,
for $n$ large enough \eqref{eq:Gnibnd} is bounded by the
geometric sum:
$$
\norm{G_n^{-1}}_\infty \leq \gamma_A^{-1}(n) \sum_{k \geq 0} (2a_n)^k = \frac{1}{\gamma_A(n)(1-2 a_n)}
= \frac{\epsilon(n)}{\epsilon(n) \gamma_A(n) - 2 \gamma_I(n)}
$$
and this last term is bounded based on our initial assumption:
$\epsilon(n)(\epsilon(n) \gamma_A(n) - 2 \gamma_I(n))^{-1} <_n \gamma_A^{-1}(n) (1-2 \kappa^{-1})^{-1}$.  Thus, for large enough $n$,
$\norm{G_n^{-1}}_\infty \leq \gamma_A^{-1}(n) (1-2 \kappa^{-1})^{-1}$.

Now we bound $\norm{G_n^{-1} \gamma_I L_{ul}}_\infty$.  Note
that $\norm{G_n^{-1} \gamma_I(n) L_{ul}}_\infty \leq \gamma_I(n)
\norm{G_n^{-1}}_\infty \norm{L_{ul}}_\infty$.  As $L_{ul} = -P_{ul}$ and
$\Pe$ is stochastic, $\norm{L_{ul}}_\infty \leq 1$.  Putting together
these two steps, we have $\norm{G_n^{-1} \gamma_I(n) L_{ul}}_\infty \leq
\gamma_I(n) \norm{G_n^{-1}}_\infty$.

Combining these two bounds, \eqref{eq:maxFn} finally becomes
$$
\norm{M_n^{-1}}_\infty \leq \max{\left( 1,
\gamma_A(n)^{-1}(1-2 \kappa^{-1})^{-1} (1 + \gamma_I(n))\right)}.
$$
For small $\gamma_A(n) > 0$, the second term is the maximum and the result
follows.
\end{proof}

\subsubsection{Characterization of $Z_0$}
We first need some preliminary definitions and results.

We define the \emph{cut locus} of the set $\cA$ as closure of the set
of points in $\cAp$ where $d_\cA^2(x)$ is not differentiable (i.e.,
where there is more than one minimal geodesic between $x$ and $\cA$):
$$
\Cut(\cA) = \overline{\set{x \in \cAp \ |\ 
    d_\cA^2 \text{ is not differentiable at } x}}.
$$

The cut locus and $d_\cA$ have several important properties, which
we now list:
\begin{enumerate}
\item \label{it:cutloc1} The Hausdorff dimension of $\Cut(\cA)$ is at
most $d-1$ \cite[Cor. 4.12]{Mantegazza2002}.
\item \label{it:cutloc2} $\Cut(\cA) \cup \cA$ is closed in $\cM$.
\item \label{it:cutloc3} The open set $\cAp \bs \Cut(\cA)$ can be
continuously retracted to $\partial \cA$.
\item \label{it:cutloc4} If $\cA \in C^r$ then $d_\cA$ is $C^r$ in
 $\cAp \bs \Cut(\cA)$.
\end{enumerate}
\noindent Items \ref{it:cutloc2}-\ref{it:cutloc4} are proved in
\cite[Prop 4.6]{Mantegazza2002}.

%\hl{TODO: M\&M see also refs. 16,17,19,28?}

Property \ref{it:cutloc1} shows that $d_\cA$ is smooth almost
everywhere on $\cAp$.  Properties \ref{it:cutloc2}-\ref{it:cutloc3}
show that $\cAp \bs \Cut(\cA)$ is composed of a finite number of
disjoint connected components, each touching $\cA$.  Finally, property
\ref{it:cutloc4} shows that $d_\cA$ is as smooth as the boundary
$\partial \cA$.

Let $\cM$ be a $d$-dimensional Riemannian manifold and let $\cA
\subset \cM$ be a Riemannian submanifold such that $\partial \cA$ is
regular (in the PDE sense).  As in Thm.~\ref{thm:transport}, define the
differential equation in $Z_0$ as
\begin{align}
\label{eq:diffZ0}
Z_0(x) \lap d_\cA(x) + 2 \grad d_\cA(x) \cdot \grad Z_0(x) = 0 &\quad x \in \cAp \\
Z_0(x) = w(x) &\quad x \in \cA \nonumber
\end{align}

We first show that $Z_0$ of Thm.~\eqref{thm:transport} has a unique,
smooth, local solution in a chart at $\cA$.  To do this we will use
the method of characteristics \cite[Chap. 3]{Evans1998}.  We will need
an established result for the local solutions of PDEs on open subsets
of $\bbR^d$.

Let $V$ be an open subset in $\bbR^d$ and let $\Gamma \subset \partial
V$.  Let $u : V \to \bbR$ and $Du$ be its derivative on $\bbR^d$.
Finally, suppose $x \in V$ and let $w : \Gamma \to \bbR$.  We study
the first-order~PDE 
\begin{align*}
  F(Du, u, x) = 0 &\qquad x \in V, \\
  u = w &\qquad x \in \Gamma
\end{align*}
\noindent Note that we can write $F = F(p, z, x) : \bbR^d \x \bbR \x
\overline{V} \to \bbR$.  The main test for existence, uniqueness, and
smoothness is the test for noncharacteristic boundary conditions.

\begin{defn}[\cite{Evans1998}, Noncharacteristic boundary condition]
\label{defn:nonchar}
Let $p^0 \in \bbR^d$, $z^0 \in \bbR$ and and $x^0 \in \Gamma$.  We say
the triple $(p^0,z^0, x^0)$ is noncharacteristic if
$$
D_p F(p^0, z^0, x^0) \cdot \nu(x^0)  \neq 0,
$$
where $\nu(x^0)$ is the outward unit normal to $\partial V$ at $x^0$.
We also say that the noncharacteristic boundary condition holds at
$(p^0,z^0, x^0)$.
\end{defn}

\noindent This test is sufficient for local existence:

\begin{prop}[\cite{Evans1998},~\S 3.3, Thm.~2 (Local Existence)]
\label{prop:locexist}
Assume that $F(p,z,x)$ is smooth and that the noncharacteristic
boundary condition holds on $F$ for some triple $(p^0, z^0, x^0)$.
Then there exists a neighborhood $V'$ of $x^0$ in $\bbR^d$ and a
unique, $C^2$ function $u$ that solves the PDE
\begin{align*}
F(Du(x), u(x), x) = 0 &\qquad x \in V', \\
u(x) = w(x) &\qquad x \in \Gamma \cap V'.
\end{align*}
\end{prop}

\noindent We are now ready to prove the existence, uniqueness, and
smoothness of $Z_0$.

\begin{lem}Let $\cA'_0$ be one of the connected components of $\cAp
\bs \Cut(\cA)$.  Then on any chart $(U,\phi)$ that satisfies $U
\subset \cA'_0 \cup \cA$ and for which $U \cap \cA$ is sufficiently
regular, the differential equation \eqref{eq:diffZ0} has a unique, and
smooth solution.
\label{lem:Zsmoothchart}
\end{lem}

\begin{proof}
Under the diffeomorphism $\phi$, \eqref{eq:diffZ0} is modified.
Choose a point $u_0 \in U \cap \cA$ and apply $\phi$.  The boundary
$\cA$ becomes a boundary $\Gamma$ in $\bbR^d$.  Let $V$
represent the rest of the mapped space.  Eq.~\eqref{eq:diffZ0} then
becomes
\begin{align*}
Z_0(v) l(v) + 2 \sum_{i,j}^{d} g^{ij}(v) \partial_i d_{\cA}(v) \partial_j Z_0(v) = 0 &\quad v \in V \\
Z_0(x) = w(x) &\quad v \in \Gamma \nonumber,
\end{align*}
where we use the abusive notation $f(v) =
f(\phi^{-1}(v))$ for a function $f : \cM \to \bbR$, and where $l(v) =
\lap d_\cA(v)$ is the (smooth) Laplacian of $d_\cA(x)$ mapped into
local coordinates.  Using the notation of Lem. \ref{defn:nonchar}, we
can write the equation above as $F(DZ_0(v), Z_0(v), v) = 0$ where
$$
F(p,z,v) = z l(v) + 2 \sum_{i,j=1}^{d} g^{ij}(v) \partial_i d_{\cA}(v) p_j,
$$
and therefore $D_p F(p,z,v)$ becomes
$$
(D_pF)^k(p,z,v) = 2 \sum_{i=1}^{d} g^{ik}(v) \partial_i d_{\cA}(v).
$$
\noindent for $k=1,\ldots,d$.  At the point $\phi(u^0) = v^0 \in \Gamma$, the
outward unit normal is $\nu(v^0) = \grad d_A(v^0)$, which in local
coordinates is given by the vector $\nu^j(v^0) = \sum_{k=1}^{d} g^{jk}(v^0)
\partial_k d_{\cA}(v^0)$ for $j=1,\ldots,d$.

The uniqueness, existence, and smoothness of $Z_0$ near $\cA$ in this
chart follows by Prop. \ref{prop:locexist} after checking the
noncharacteristic boundary condition for $F$ at $v^0$:
\begin{align*}
D_p F(p^0, z^0, v^0) \cdot \nu(v^0)
  &= 2 \sum_{kj} g_{kj}(v^0)
    \sum_{i=1}^{d} g^{ik}(v^0) \partial_i d_{\cA}(v^0)
    \sum_{k=1}^{d} g^{jk}(v^0) \partial_k d_{\cA}(v^0) \\
  &= 2 \sum_{kj} g^{kj}(v^0) \partial_k d_{\cA}(v^0) \partial_j  d_{\cA}(v^0) \\
  &= 2 \norm{\grad d_{\cA}(v^0)}^2 = 2 \neq 0,
\end{align*}
where the last equality follows by definition of the distance function
in terms of the Eikonal equation.
\end{proof}

% Move to an earlier section?
\begin{thm}
Let $\cA'_0$ be one of the connected components of $\cAp
\bs \Cut(\cA)$.  The differential equation \eqref{eq:diffZ0} has a
unique, and smooth solution on $\cA'_0$.
\label{thm:Zsmooth}
\end{thm}

\begin{proof}[Proof of Thm.~\ref{thm:Zsmooth}]
A local solution exists in an open ball around each point $u_0$ in the
region $\cA'_0 \cap \cA$,  due to Lem. \ref{lem:Zsmoothchart}.  The
size of each ball is bounded from below, so by compactness we can find
a finite number of subsets $U = \cup_i U_{0,i}$ that cover $\cA'_0
\cap \cA$, for which \eqref{eq:diffZ0} has a smooth unique solution,
and which overlap.  As the charts overlap and the associated mappings
are diffeomorphic, a consistent, smooth, unique solution therefore
exists near $\cA$.

To extend this solution away from the boundary, we choose a small
distance $d_0$ such that for all $x$ with $d_\cA(x) \leq d_0$, that
$x$ is also in the initially solved region $U$.
This set, which we call $\cD_0$, is a contour of $d_\cA$ within
$\cA'_0$.  From the previous argument, $Z_0$ has been solved up to
this contour, and we now look at an updated version of
\eqref{eq:diffZ0} by setting the new Dirichlet anchor conditions at
$\cD_0$ from the solved-for $Z_0$, and setting the interior of the
updated problem domain to the remainder of $\cA'_0$.

Let $U_0 \ \set{x \in U : d(x) \leq d_0}$ and let $\cA'_1 = \cA'_0 \bs
U_0$.  The method of characteristics also applies on $\cA_1'$ near
$\cD_0$.  We apply Lem. \ref{lem:Zsmoothchart} with the updated
Dirichlet boundary conditions.  As $\cD_0$ defines a contour of $d_\cA$, its outward
normal direction is $\grad d_\cA$.  Similarly, $D_p F$ (of
Lem. \eqref{lem:Zsmoothchart}) has not changed.  A solution
therefore exists locally around each point $u_1 \in \cA'_1 \cap
\cD_0$.  The process above can be repeated to ``fill in'' the solution
within all of $\cA'_0$.
\end{proof}

\subsection{Details of the Regression Problem of~\S\ref{sec:sslintro}}
\label{app:lrldetail}
%Defining $M_n = E_\cA + \gamma_A E_\cAp + \gamma_I E_\cAp L$,
In this deferred section, we decompose the problem \eqref{eq:rlspdisc} into two
parts: elements associated with the first $l$ labeled points (these
are given subscript $l$) and elements associated with the remaining
unlabeled points (given subscript $u$).  This decomposition provides a
more direct look into the how the assumptions in~\S\ref{sec:sslassume}
simplify the original problem, and how the resulting optimization
problem depends only on the ratio of the two parameters $\gamma_I$ and
$\gamma_A$.

We first rewrite \eqref{eq:rlspdisc}, expanding all the parts:
$$
\min_{\tf_l, \tf_u} \left\{
  \norm{\pmat{w_l \\ 0} - \pmat{\tf_l \\ 0}}^2_2
  + \gamma_A \pmat{\tf_l \\ \tf_u}^T \pmat{\tf_l \\ \tf_u}
  + \gamma_I \pmat{\tf_l \\ \tf_u}^T \pmat{L_{ll} & L_{lu} \\ L_{ul} & L_{uu}} \pmat{\tf_l \\\tf_u}
  \right\}.
$$
In this system, the optimization problems on $\tf_u$ and $\tf_l$ are
coupled by the matrix $L$.

The assumptions in~\S\ref{sec:sslassume} decouple
\eqref{eq:rlspdisc}.  This comes from the equality constraint $E_\cA
\tf = \tilde{w}$, equivalently $\tf_l = \tilde{w}_l$.  The problem is
further simplified by the restriction of the
integral domain from $\cM$ to $\cAp$ in the modified penalty
$J(f)$.  We can write $J(f) = \int_\cAp f(x) \lap f(x) dx =
\int_\cM f(x) 1_{\set{x \in \cAp}} \lap f(x) dx$, and the
discretization of this term is $\wt{J}(\tf) = (E_\cAp \tf)^T L \tf$.
After these reductions, and the reduction of the ridge term to $\tf^T
E_\cAp \tf$, the problem \eqref{eq:rlspdisc} becomes:
\begin{align*}
 \min_{\tf_u} \left\{
  \gamma_A \tf_u^T \tf_u  
  + \gamma_I \pmat{\tilde{w}_l \\ \tf_u}^T \pmat{0 & 0 \\ L_{ul} & L_{uu}}
  \pmat{\tilde{w}_l \\ \tf_u}
  \right\} \\
= \min_{\tf_u} \left\{
  \gamma_A \tf_u^T \tf_u
  + \gamma_I (\tf_u^T L_{ul} \tilde{w}_l + \tf_u^T L_{uu} \tf_u )
  \right\}.
\end{align*}
The solution to this problem, combined with the constraint $\tf_l =
\tilde{w}_l$, leads to \eqref{eq:lrlssys}.

The $\gamma_A$ term above normalizes the Euclidean norm of $\tf_u$, thus
earning it the mnemonic ``ambient regularizer''.  The first $\gamma_I$
term is an inner product between $\tf_u$ and $L_{ul} \tilde{w}_l$.  As
$L_{ul}$ is an averaging operator with negative coefficients, the
component $(L_{ul} \tilde{w}_l)_j$ contains the negative average of
the labels for points in $\cA$ near $x_j \in \cAp$.  If $x_j$ is far
from $\cA$, this component is near zero.  Minimizing $\tf_u^T L_{ul}
\tilde{w}_l$ therefore encourages points near $\cA$ to take on the
labels of their labeled neighbors.  For points away from $\cA$ it has
no direct effect.  Minimizing the second $\gamma_I$ term encourages a
diffusion of values between points in $\cAp$, thus diffusing these
near-boundary labels to the rest of the space.  This process earns the
$\gamma_I$ term the mnemonic ``intrinsic regularizer'', because it
encourages diffusion of the labels across $\cM$.

Dividing the problem by $\gamma_A$, we see that the solution
depends only on the ratio $\gamma_I(n) / \gamma_A(n)$.
When $\gamma_I \gg \gamma_A$ the solution of \eqref{eq:rlsp}
is biased towards a constant \cite{Kim2009} on $\cAp$, equivalent to
solving the Laplace equation $\lap f = 0$ on $\cAp$ with the anchor
conditions $f = w$ on $\cA$.  This case of heavy regularization is
useful when $n$ is small, but offers little insight about how the
solution depends on the geometry of $\cM$.  We
are interested in the situation of light regularization: $\gamma_I(n)
= o(\gamma_A(n))$.  We also independently see this assumption as a
requirement for convergence of $\tf$ in~\S\ref{sec:ssllimit}.

\subsection{The RL PDE with Nonempty Boundary ($\bdM \neq \emptyset$)\label{app:derivbd}}

When the boundary of $\cM$ is not empty,
Thm.~\ref{thm:fconv} and Cor. \ref{cor:fpconv} no longer apply in
their current form.  In this section, we provide a road map for how
these results must be modified.  We also argue why in the case of
small $h>0$, the limiting results (expressions for $S_h$ and $f^*_h$
in Assum. \ref{conj:vedist}, Thm.~\ref{thm:transport}, and
Thm.~\ref{thm:rl2e}) are not affected by these modifications.

Let $\cM_{\epsilon} = \set{x \in \cM : d_{\bdM}(x) >
\epsilon^{\gamma}}$ where $\gamma \in (0,1/2)$.  For points
in the intersection of $\cA$ and $\cM \bs \cM_{\epsilon}$ (for example, when
$\cM \bs \cM_{\epsilon} \subset \cA$, and the anchor ``covers'' the boundary of $\cM$), we
need only consider the standard anchor conditions.  For other cases,
we proceed thus:

It has been shown \cite[Prop. 11]{Coifman2006} that as $n \to \infty$:
\begin{enumerate}
\item $\Le \pi_{\cX}(f)_i = -c \lap f(x_i) + O(\epsilon)$ for $x_i \in
\cM_{\epsilon}$ (this matches Thm.~\ref{thm:convLe}).  
\item For $x_i \in \cM \bs \cM_{\epsilon}$, $(\Le \pi_{\cX}(f))_i \approx
\pd{f}{\nu}(x'_i)$, where $x'_i$ is the nearest point in $\bdM$ to
$x_i$ and $\nu$ is the outward normal at $x'_i$.  That is, near the
boundary $\Le$ takes the outward normal derivative.
\item This region $\cM \bs \cM_{\epsilon}$ is small, and shrinks with
decreasing $\epsilon$: ${\mu(\cM \bs \cM_{\epsilon}) = O(\epsilon^{1/2})}$.
\end{enumerate}

One therefore expects that Thm.~\ref{thm:fconv} and Cor. \ref{cor:fpconv}
still hold, albeit with the norms restricted to points in
$\cM_{\epsilon}$.  More specifically, the set $\cA'$ must necessarily
become $\cAp(n) = \cM_{\epsilon(n)} \bs \cA$.
Furthermore, as $n \to \infty, \epsilon \to 0$, this set grows to
encompass more of $\cAp$.

As a result, the domains of the RL PDE \eqref{eq:rl} change.
It is hard to write down the boundary condition at
$\bdM$, precisely because there is no analytical description for how
$\Le$ acts on functions in $\cM \bs \cM_{\epsilon}$.  However, from
item 2 above, we can model it as an unknown Neumann condition.

Fortunately, for vanishing viscosity (small $h$), the effect of this
second boundary condition disappears: the Eikonal equation depends only on
the (Dirichlet) conditions at $\cA$.  More specifically, regardless of
other Neumann boundary conditions away from $\cA$,
Assum. \ref{conj:vedist} still holds and, 
as a result, so do Thm.~\ref{thm:transport} and Thm.~\ref{thm:rl2e}.
This follows because the Eikonal equation is a first order
differential equation, and so some of the boundary conditions may be
dropped in the small $h$ approximation.  A more rigorous discussion
requires a perturbation analysis (see, e.g., \cite{Nayfeh1973}).
We instead provide an example, mimicking Ex. \ref{ex:annul},
except now we let the anchor domain be the inner circle only.

\begin{exmp}[The Annulus in $\bbR^2$ with reduced anchor]
\label{ex:annulred}
Let $\cM = \set{r_0 \leq r \leq 1}$, where $r=\norm{x}$ is the
distance to the origin.  Let $\cA = \set{r=r_0}$ be the inner
circle.  Letting $w=1$ ($u_h=0$), we get $S(r) = d_{\cA}(r) = r-r_0$.
We again assume a radially symmetric solution to
the RL Eq. and \eqref{eq:rl} becomes:
$-h^2 \left(f_h''(r) + r^{-1} f_h'(r) \right) + f_h(r) = 0$
for $r \in (r_0,1]$, $f_h(r_0)=1$.
Furthermore, since the boundary condition at $r=1$ is unknown, we
set it to be an arbitrary Neumann condition: $f'_h(1) = b$.
The solution is
$$
f_h(r) = \frac{bh[I_0(r/h)K_0(r_0/h)-K_0(r/h)I_0(r_0/h)]
    + I_0(r/h)K_1(1/h) + K_0(r/h) I_1(1/h)}
  {K_0(r_0/h)I_1(1/h)+K_1(1/h)I_0(r_0/h)},
$$
A series expansion of $f_h(r)$ around $h=0$ gives
$f_h(r) = \sqrt{r_0/r} e^{-(r-r_0)/h} + O(h)$,
and therefore $\lim_{h \to 0} -h \log f_h(r) = S(r)$
(again confirming \eqref{eq:rllim}).

We simulated this problem with $r_0=0.25$ by sampling
$n=1000$ points from the ball $B(0,1)$, and rescaling points
with $r \leq r_0$ to $r=r_0$.  $S_h$ is approximated up to a constant
using the numerical discretization, via \eqref{eq:lrlssysp}, of
\eqref{eq:lrlssim2}.  For the graph Laplacian we used a $k=20$ NN
graph and $\epsilon=0.001$.

\begin{figure}[h!]
\centering
\begin{minipage}{\linewidth}
  \centering
  \parbox{.3\linewidth}{
    \centering
    \includegraphics[width=1.2\linewidth]{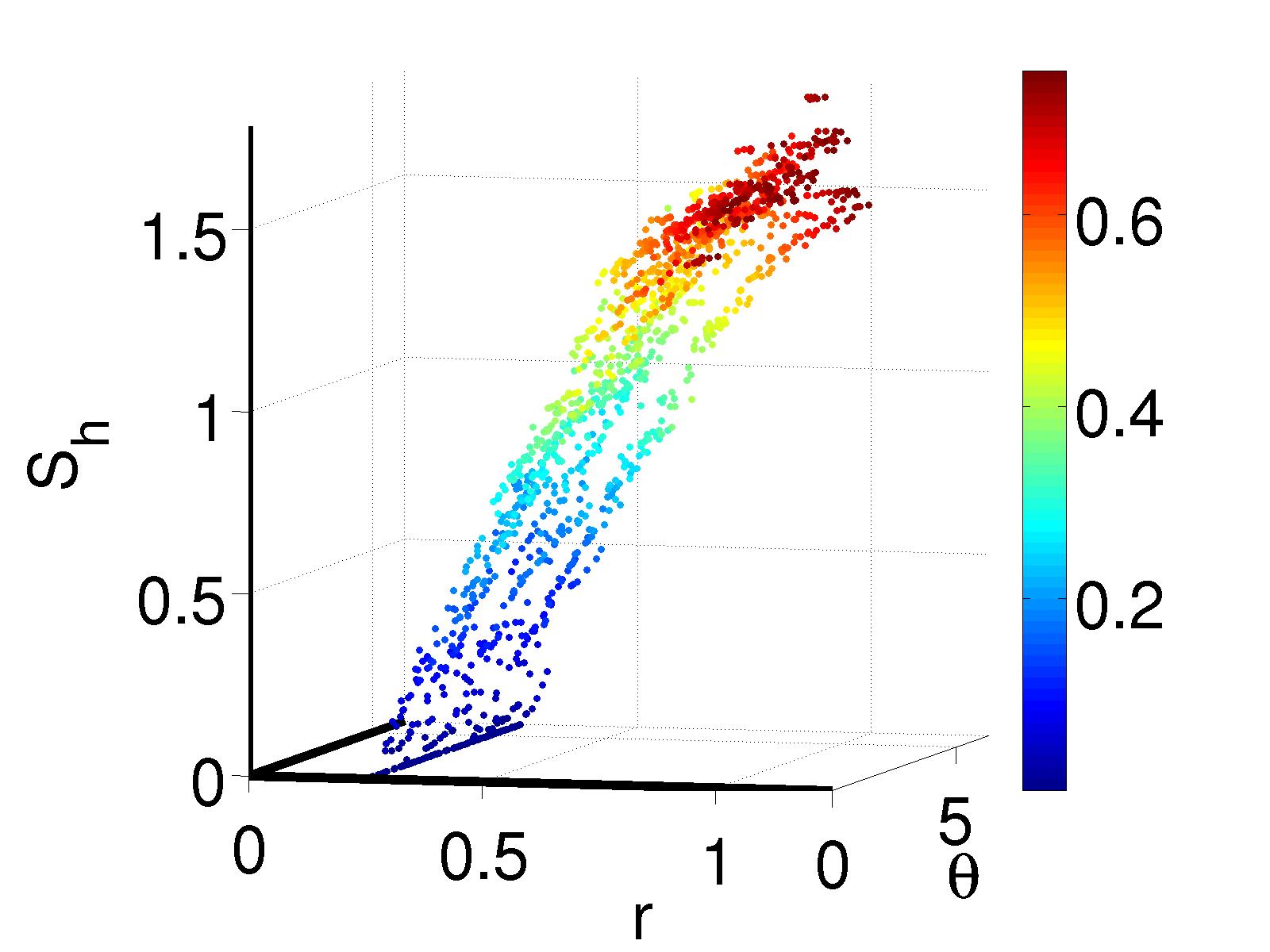}
  }
  \quad
  \parbox{.3\linewidth}{
    \centering
    \includegraphics[width=1.2\linewidth]{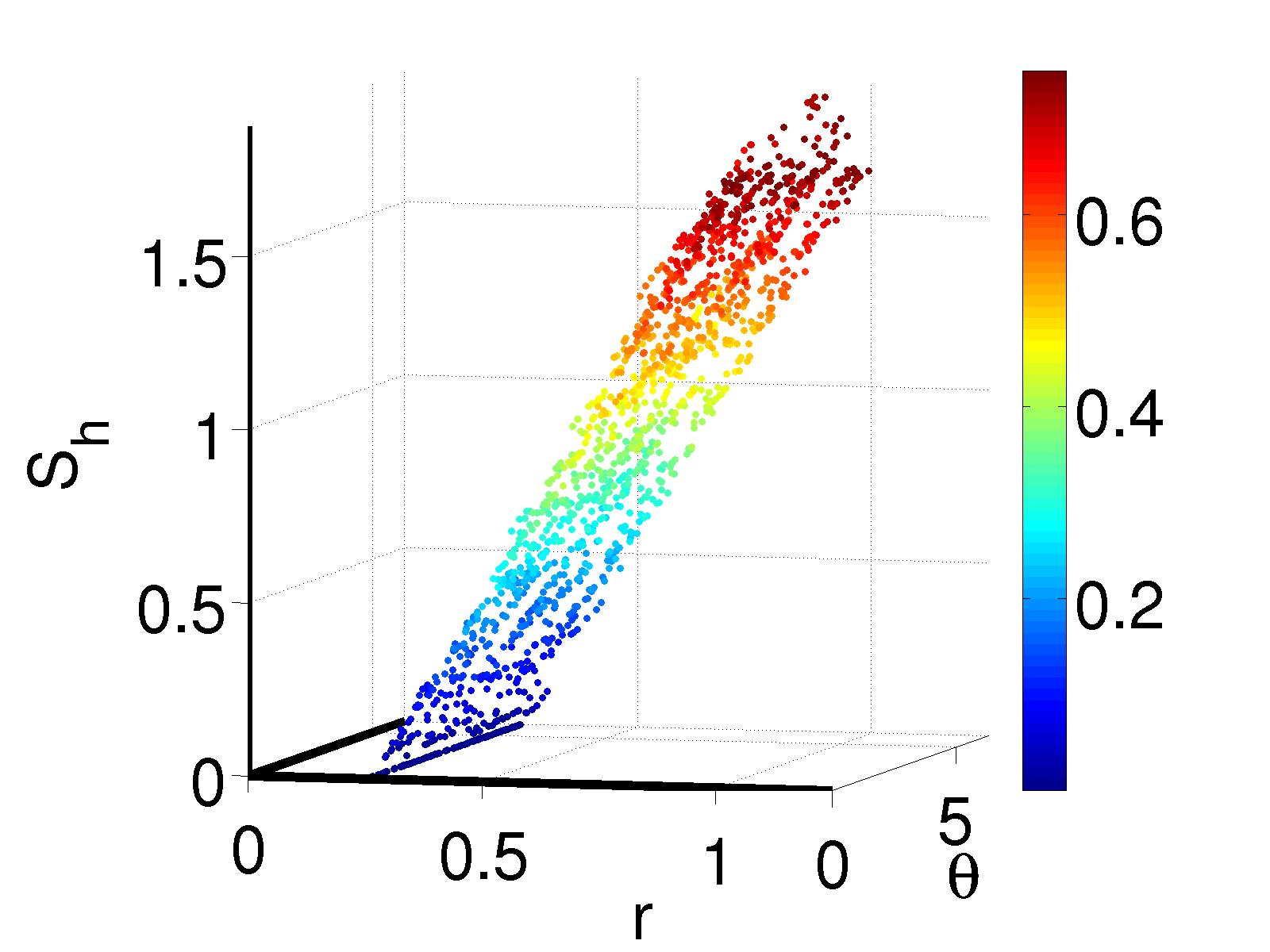}
  }
  \quad
  \parbox{.3\linewidth}{
    \centering
    \includegraphics[width=1.2\linewidth]{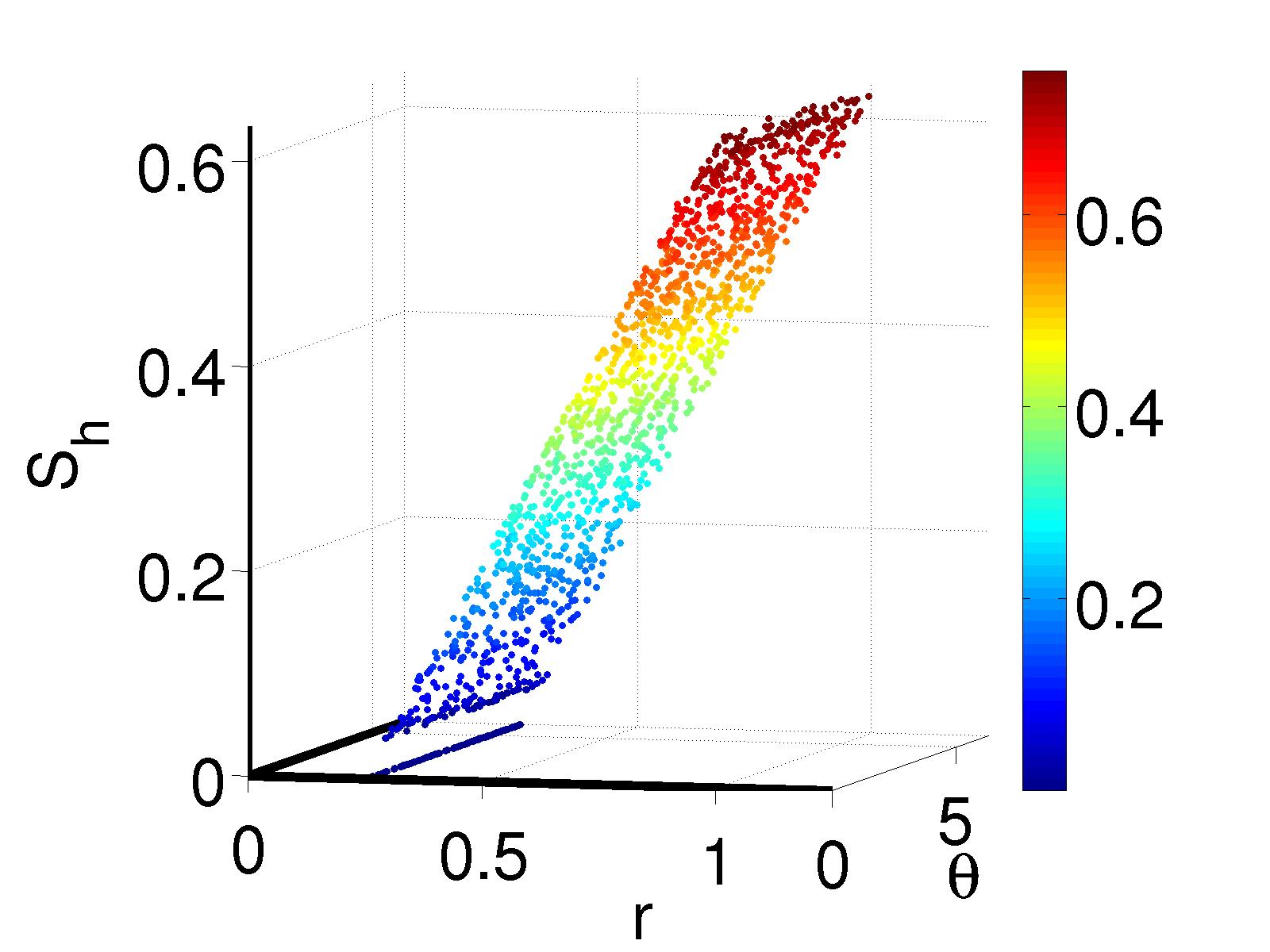}
  }
  \caption[Geodesics Estimates $S_h(r,\theta)$ on Modified Annulus.]%
   {Geodesics Estimates $S_h(r,\theta)$ on Modified Annulus.  Left to
    Right: $h=1,\ 0.1,\ 0.001$}
  \label{fig:geoannulred}
\end{minipage}
\end{figure}

Fig.~\ref{fig:geoannulred} shows (in the $z$ axis) the estimate $S_h(x)$
as $h$ grows small.  The colors of the points reflect the true
distance to $\cA$: $S(r) = S_0(r) = r-r_0$.  Note the convergence as $h
\downto 0$, and also the clear offset of $S_h$ which is especially
apparent in the right pane near $r=0.25$.

From the second of Eqs. \eqref{eq:trl} and the fact that $\grad S(r) =
1, \lap S(r) = 1/r$, we have $Z_0/r + 2 Z'_0 = 0$ for $r_0 < r \leq 1$
and $Z_0 = 1$ for $r=r_0$. Solving this we get $Z_0(r) = \sqrt{r_0/r}$.
The solution becomes $f_h(r) = e^{-(r-r_0)/h}(\sqrt{r_0/r} +
O(h))$, which matches the earlier series expansion of the full solution.
Furthermore, upon an additional Taylor expansion we have
$S_h(r) = r-r_0 - h \log(r_0/r)/2 + O(h \log h)$.
As before, the extra term in the $S_h$ estimate
has a large effect when $r-r_0$ is small (as seen
in the right pane of Fig.~\ref{fig:geoannulred}).
\end{exmp}

\section{Conclusion and Future Work\label{sec:concl}}

We have proved that the solution to the SSL problem \eqref{eq:lrlssys}
converges to the sampling of a smooth solution of a Regularized
Laplacian PDE, in certain limiting cases.  Furthermore, we have
applied the established theory of Viscosity PDE solutions to analyze
this Regularized Laplacian PDE.   Our analysis leads to a geometric
framework for understanding the regularized graph Laplacian in the
noiseless, low regularization regime (where $h \to 0$). This framework
provides intuitive explanations for, and validation of, machine
learning algorithms that use the inverse of a regularized Laplacian
matrix.

We have taken the first steps in extending the theoretical analysis in
this chapter to manifolds with boundary (\S\ref{app:derivbd})
%and to irregular anchor domains (\S\ref{app:irranch}).
While the results within this section can be confirmed numerically, in
some cases additional work must be done to confirm them in full
generality. Furthermore, Assum. \ref{conj:vedist} awaits confirmation
within the viscosity theory community.

There are a host of applications derived from the work in this
chapter, and we turn our focus to them in chapter~\ref{ch:rlapp}.

%  It also makes it easy to motivate and construct a new
%regularized geodesics estimator, and algorithms such as the NSM
%Classifier and Viscous ISOMAP.

\chapter{The Inverse Regularized Laplacian: Applications\label{ch:rlapp}}

\section{Introduction}

Thanks to the theoretical development in chapter \ref{ch:rl},
we now have a framework within which we can construct new tools for
learning (e.g. a regularized geodesic distance estimator and a new
multiclass classifier).  These tools can also shed light on other
results in the literature (e.g. a result of \cite{Nadler2009}).
Throughout this chapter we will use the notation developed in chapter
\ref{ch:rl}.

\section{Regularized Nearest Sub-Manifold (NSM) Classifier}
\label{sec:geo}
We now construct a new robust geodesic distance estimator and employ
it for classification.  We then demonstrate the classifier's efficacy on
several standard data sets.
To construct the estimator, first choose some anchor
set $\cA \subset \cM$, and suppose the points $\set{x_i}_{i=1}^l$ are
sampled from $\cA$.  To calculate the distance $d_\cA(x_i)$ for $i=l+1
\ldots n$, construct the normalized graph Laplacian
$\Le$.  Choosing $\wt{h} > 0$ appropriately,
solve the linear system \eqref{eq:lrlssysp}:
\beq
\label{eq:egeo}
(E_\cA + E_\cAp(I + \wt{h}^2 \Le)) a = \wt{w}
\eeq
where $\wt{w} = \left[1\ \cdots\ 1\ \ 0\ \cdots\ 0\ \right]^T$ is a
vector of all zeros for sample points in $\cAp$ and all ones for
sample points in $\cA$.  For $n$ large, $\epsilon$ small, and
$\wt{h}$ small, this linear system approximates \eqref{eq:rl}
with $h = \wt{h} \sqrt{c}$.  Applying Thm. \ref{thm:rl2e}, we see
that $\wt{S}_i = -\wt{h} \log a_i \approx c^{-1/2} d_\cA(x_i)$.

While the estimator $\wt{S}$ is approximate and only valid up to a
constant, it is also simple to implement and consistent (due to 
Cor. \ref{cor:fpconv}).

We know of two other consistent geodesics estimators that work on
point samples from $\cM$.  One performs fast
marching by constructing complex local upwind schemes
that require the iterative solution of sequences of
high dimensional quadratic systems \cite{Sethian2000}.  Another
performs fast marching in $\bbR^p$ on offsets of $\cX$ and is also
approximate \cite{Memoli2005}.  The first scheme is complex
to implement; the second is exponential in the ambient dimension $p$.
Our estimator, on the other hand, can be implemented in Matlab in
under 10 lines, given one of many fast approximate NN
estimators. Furthermore, it requires the solution of a linear system
of size essentially $n$, so its complexity depends only on the number
of samples $n$, not on the ambient dimension $p$.
Finally, our scheme allows for a natural regularization by tweaking
the viscosity parameter $\wt{h}$.
\S\ref{sec:geoexmp} contains numerical comparisons between our
estimator and, e.g. Dijkstra's Shortest Path and Sethian's Fast
Marching estimators.

The lack of dynamic range in the estimator $\wt{S}$, following
\eqref{eq:egeo}, leads to important numerical considerations.  According to
Thm. \ref{thm:rl2e}, for a given sampling $\cX$ one would choose
$\wt{h} \ll \min_{e \in \cE} \td_e$ to have an accurate estimate of
geodesics for all point samples. In this case, however, many
points far from $\cA$ may have their associated estimate $a_i$ drop
below the machine epsilon.  In this case an iterative multiscale
approach will work: estimates are first calculated for points nearest
to $\cA$ for which no estimate yet exists (but $a_i$ is above machine
epsilon), then $\wt{h}$ is multiplied by some factor $\gamma > 1$,
and the process is repeated.

We now use the above estimator to form the Nearest Sub-Manifold (NSM)
classifier.  The classifier is based on two simplifications.  First,
for noisy samples, one would want to select $\wt{h}$ based on the
noise level or via cross-validation; it therefore becomes
a regularization term.
Second, as seen in \S\ref{sec:sslr}, for classification the exact estimate of geodesic
distance is less important than relative distances; hence
there is no need to estimate scaling constants.

As before, suppose we are given $n$ samples from a manifold $\cM$.  Of
these, each of the first $l$ belong to one of $M$ classes; that
is, $x_i \in \cC_m$, $m \in \set{1,\ldots,M}$.  We assume that all
points within class $m$ belong to a smooth closed subset of
$\cM$, which we call anchor $\cA_m$, $m=1,\ldots,M$.  For each anchor,
we define the anchor data vector $\wt{w}^m$ via $\wt{w}^m_i =
\delta(x_i \in \cC_m)$, $i=1,\ldots,n$.
To classify, first choose $\wt{h}>0$ and solve
\eqref{eq:egeo} for each of the $M$ different anchor sets $\cA_m$ (and
associated $\wt{w}^m$), to get solutions $\set{a^m}_{m=1}^M$.  Then
for each unlabeled point $x_i$, $a^m_i$ encodes its distance to anchor
$\cA_m$.  The decision rule is $C(x_i) = \argmax_m a^m_i$.

For $n$ and $l$ large, $\epsilon$, $\wt{h} > 0$ small, and
no noise, $C(x_i)$ will accurately estimate the class
which is closest in geodesic distance to $x_i$.  In the noisy, finite
sample case with irregular boundaries, $C$ provides a regularized
estimate of the same.

\section{NSM Classifier: Performance}

We compare the classification performance of the NSM classifier to
several state-of-the-art classifiers using the test set from
\cite{Chapelle2006} (testing protocol and datasets:
\url{http://www.kyb.tuebingen.mpg.de/ssl-book/}).
For the NSM classifier, we performed a parameter search as described
in \cite[\S 21.2.5]{Chapelle2006}, and additionally cross-validated
over the viscosity parameter $\wt{h} \in \set{0.1, 1, 10}$ scaled by
the median distance between pairs of points in $\cE$.

We compare our results with publicly available implementations of:
\begin{itemize}
\item LapRLS from M. Belkin's website,
with obvious modifications for one-vs-all multiclassification
and with the exception that, as opposed to \cite{Chapelle2006}, we
used $p=1$ in the kernel $\tLe^p$ instead of $p=2$.  Here, we also
performed a parameter search as in \cite[\S 21.2.5]{Chapelle2006}.
\item LDS from O. Chapelle's website with parameters optimized as in
\cite[\S 21.2.11]{Chapelle2006}.
\item Kernel TSVM using primal gradient descent (available in the LDS
  package) with parameters optimized as in \cite[\S
    21.2.1]{Chapelle2006}.
\end{itemize}
For testing, we also included the LIBRAS (LIB) dataset with 12 splits
of $l=30, 100$ labeled points and the ionosphere (Ion) dataset with 12
splits of $l=10, 100$ labeled points.
All datasets have $M=2$ (the task is binary classification) except
COIL, which has $M=6$.

Table \ref{tab:ssl} shows percent classification error vs. percentage
labeled points, over 12 randomized splits of the testing and training
data set.  Parameter optimization (cross-validation) was always
performed on the training splits only; classification error is
reported over the testing data.
Note that that the NSM classifier is competitive with the
others, especially on those datasets where we expect a manifold
structure (e.g. the image sets USPS and COIL).

\begin{table}[t]
\caption{Percent classification error over 12 splits. Clear~winners~in~bold.}
\footnotesize
\label{tab:ssl}
\centering
\begin{tabular}{|>{\bfseries}c|c|c|c|c|c|c|c|c|c|c|c|c|c|c|c|}
\hline
\ & 
\multicolumn{2}{|c|}{{\bf USPS}} & 
\multicolumn{2}{|c|}{{\bf BCI}} &
\multicolumn{2}{|c|}{{\bf g241n}} &
\multicolumn{2}{|c|}{{\bf COIL}} &
\multicolumn{2}{|c|}{{\bf LIB}} &
\multicolumn{2}{|c|}{{\bf Ion}} \\
\hline
$100 (l/n)$
      & 0.66 & 6.66
      & 2.5 & 25
      & 0.66 & 6.66
      & 0.66 & 6.66
      & 8.3 & 28
      & 2.8 & 28 \\
\hline
$l$ & 10 & 100
    & 10 & 100
    & 10 & 100
    & 10 & 100
    & 30 & 100
    & 10 & 100 \\
\hline
NSM
& {\bf 10.0} & {\bf 5.9} % 2-USPS    10,100
& 49.0 & 44.8 % 4-BCI     10,100
& 46.0 & 39.0 % 7-g241n   10,100
& 58.8 & {\bf 11.6} % 6-COIL    10,100
& 51.0 & {\bf 30.3} % 20-LIB     30,100
& 30.4 & 14.2 \\ % 21-Ion  10,100

LapRLS
& 15.0 & 10.6 % USPS    10,100
& 48.7 & 45.4 % BCI     10,100
& 43.0 & 33.9 % 7-g241n 10,100
& 63.1 & 20.6 % COIL    10,100
& 48.9 & {\bf 30.4} % LIB     30,100
& 31.9 & 13.0 \\ % Ion  10,100

LDS
& 22.5 & 11.5 % 2-USPS    10,100
& 48.7 & 46.0 % 4-BCI     10,100
& 49.0 & 41.0 % 7-g241n   10,100
& 56.7 & 16.2 % 6-COIL    10,100
& 54.9 & 36.6 % 20-LIB     30,100
& {\bf 17.3} & 13.5 \\ % 21-Ion  10,100

TSVM
& 17.4 & 12.0 % USPS    10,100
& 48.8 & 46.1 % BCI     10,100
& 47.3 & {\bf 23.7} % 7-g241n 10,100
& 69.5 & 39.9 % COIL    10,100
& 66.0 & 36.6 % LIB     30,100
& 48.1 & 20.7 \\ % Ion  10,100
\hline
\end{tabular}
\end{table}

\section{Irregular Boundaries and the counterexample of Nadler et al.}
\label{sec:irr}
We relate the Annulus example (Ex. \ref{ex:annul}) to a
negative result of \cite[Thm. 2]{Nadler2009}, which
essentially states that no solution exists for \eqref{eq:rl} for $\cM$
with $d \geq 2$ and the anchor set a countable number of points.
%
%In the counter example of \cite{Nadler2009}, $\cM$ is the ball
%$B(0,1)$ in $\bbR^2$ and $\cA = \set{(0,0),(1,0)}$ with anchor
%values $f_h(0,0)=0$, $f_h(1,0) =1$.
%
This yields a special case of a result known in PDE theory: no
solution exists to \eqref{eq:rl} when $\cA$ is irregular; and isolated
points on subsets of $\bbR^d$, $d \geq 2$, are irregular
\cite[Irregular Boundary Point]{Hazewinkel1995}.

This is very clearly seen in Ex. \ref{ex:annul}, where attempting
to let $r_0 \to 0$ (thus forcing a single point anchor) forces
the first term of the solution $f_h$ in \eqref{eq:fhr} to zero for any
$r>r_0$, regardless of the anchor condition at $r=r_0$ and of $h$.
The major culprit here is the $(d-1)/r$ term that appears in the
radial Laplacian and is unbounded at the origin.
Note, however, that viscosity solutions to
\eqref{eq:ve} do exist even for singular anchors
\cite{Mantegazza2002}.
%
%In the noiseless case $f_h$ can
%be approximated for small $h$ via the SP algorithm and
%Thm. \ref{thm:rl2e}.

In many practical cases (i.e., if we had chosen single point anchors
in \S\ref{sec:sslr}, Ex. \ref{ex:sslsphere}, etc), the sampling
size is finite and we keep $\epsilon \geq \epsilon_0$ for some
$\epsilon_0>0$.  In these cases, the issues raised here do not affect
the numerical analysis because even single points act like
balls of radius $\epsilon$ in $\bbR^p$.

\section{Beyond Classification: Graph Denoising, Manifold Learning \label{sec:beyond}}

The ideas presented in the previous sections can also be applied to
other areas of machine learning.  As illustrations, we show that the
graph denoising scheme of \cite{Brevdo2010} is a special case of our
geodesics estimator.  Further, we show how to construct a regularized
variant of ISOMAP and provide some numerical examples of geodesics
estimation.

\section{The Graph Denoising Algorithm of Chapter \ref{ch:npdr}}
\label{sec:geonpdr}
In chapter \ref{ch:npdr}, we studied decision rules for denoising
(removing) edges from NN graphs that have been corrupted by sampling
noise.  We examine the Neighborhood Probability Decision Rule (NPDR)
of \S\ref{sec:npdr}, and show that in the low noise,
low regularization regime, it removes graph edges between
geodesically distant points.

The NPDR is constructed in three stages:  (a) the NN graph $G$ is
constructed from the sample points $\cX$.  $\cE$ contains an initial
estimate of neighbors in $G$, but may contain incorrect edges
due to sampling noise;  (b) a special Markov random walk
is constructed on $G$, resulting in the transition probability matrix
$\Ne \propto (I-\bp \Pe)^{-1}$ for $\bp \in (0,1)$;
(c) the edges $(l,k) = e \in \cE$ with the smallest associated
entries $(\Ne)_{lk}$ are removed from $G$.
In chapter \ref{ch:npdr}, we provide a probabilistic
interpretation for the coefficients of $\Ne$.

We show that $\Ne$ encodes geodesic distances by 
reducing  $\Ne^{-1}$ to look like \eqref{eq:rl}:
\begin{align}
I - \bp \Pe &= (1-\bp)I + \bp (I - \Pe) %\nonumber \\
% &\propto I + \frac{\bp}{1-\bp} (I-\Pe)
 \propto ({1}/{\epsilon}) I + {\bp}(1-\bp)^{-1} \Le.
\label{eq:redNe}
\end{align}
%The $i$th column of $\Ne$ is $\Ne e_i$. 
For $\epsilon$ small,
after applying the RHS of \eqref{eq:redNe}, $\Ne e_i$
approximately solves \eqref{eq:rl} with $\cA = \set{x_i}$,
$w(x_i) = c'$ (where $c'$ is a function of ${\bp,\epsilon}$), and
$h^2 = c \epsilon \bp / (1-\bp)$.  
Then by Thm. \ref{thm:rl2e}
$$(\Ne)_{lk} \approx c' e^{-d(x_l,x_k) \sqrt{(1-\bp)/(c \epsilon
\bp)}}, \mbox{some } c'>0.
$$
%
%\begin{prop} When $\cM$ is sampled without noise and for
%$\epsilon \ll \bp^{-1}(1-\bp)$, $\Ne$ encodes geodesic distances:
%$(\Ne)_{lk} \approx c' e^{-d(x_l,x_k) \sqrt{(1-\bp)/(c \epsilon
%\bp)}}$, some $c'>0$.
%\end{prop}
%\begin{proof}
%We reduce the form of $\Ne^{-1}$ to one that looks like \eqref{eq:rl}:
%\begin{align}
%I - \bp \Pe &= (1-\bp)I + \bp (I - \Pe) %\nonumber \\
%% &\propto I + \frac{\bp}{1-\bp} (I-\Pe)
 %\propto \frac{1}{\epsilon} I + \frac{\bp}{1-\bp} \Le.
%\label{eq:redNe}
%\end{align}
%The $i$th column of $\Ne$ is $\Ne e_i$. 
%For $\epsilon$ small,
%after applying the RHS of \eqref{eq:redNe}, $\Ne e_i$
%approximately solves \eqref{eq:rl} with $\cA = \set{x_i}$,
%$w(x_i) = c'_{\bp,\epsilon}$, and $h^2 = c \epsilon \bp / (1-\bp)$.  The
%result then follows by Thm. \ref{thm:rl2e}.
%\end{proof}
Thus in the noiseless case and with $\bp \approx 0$, the NPDR
algorithm will remove edges in the graph between
points that are geodesically far from each other. 
As edges of this type are the most detrimental to learning
\cite{Balasubramanian2002}, the NPDR is a powerful denoising rule.
As shown in chapter \ref{ch:npdr},
for noisy samples one would choose $\bp \approx 1$ to regularize for
noisy edges.  In this case, one can think
of $\Ne$ as a highly regularized encoding of pairwise geodesic distances.

\section{Viscous ISOMAP}
\label{sec:isomap}
As a second example of how the ideas from \S\ref{sec:sslr} can be
used, we construct a regularized variant of ISOMAP
\cite{Tenenbaum2000}, which we call Viscous ISOMAP.

ISOMAP is a dimensionality reduction algorithm that constructs an
embedding for $n$ points sampled from a high-dimensional space by
performing Multidimensional Scaling (MDS) on the estimated geodesic
distance matrix of the NN graph of these points.

The first step of ISOMAP is to estimate all pairwise geodesic
distances.  Traditionally this is done via Dijkstra's Shortest Path
algorithm.  We replace this step with our regularized geodesics
estimator.  A direct implementation requires $n$ calculations of
\eqref{eq:egeo}.  However, a faster estimator can be constructed,
based on our analysis of the NPDR algorithm in \S\ref{sec:geonpdr}.
Specifically, to calculate pairwise distances, first
calculate $M = (I + \wt{h}^2 \Le)^{-1}$. Then the symmetrized
geodesics estimates are $H = -\wt{h}(\log M + \log M^T)$, where the
logarithm is taken elementwise.  Finally, perform MDS on the matrix
$H$ to calculate the ISOMAP embedding.

For small $\wt{h}$, the Viscous ISOMAP embedding matches that
of standard ISOMAP.  For large $\wt{h}$, the additional
regularization can remove the effects of erroneous edges caused by
noise and outliers.

We provide a rather simple numerical example.  It confirms
that for small viscosity $\wt{h}$, Viscous ISOMAP embeddings match
standard ISOMAP embeddings, and that for larger viscosities the
embeddings are less sensitive to outliers in the original sampling set
$\cX$ and in $\cE$.

Fig.~\ref{fig:isomap} compares Viscous ISOMAP to regular ISOMAP
on a noisy Swiss Roll with topological shortcuts.  We used the same
$n=1000$ samples and $\delta=4$ for NN estimation for both algorithms,
and $\epsilon=1$ for Viscous ISOMAP.  Note how for small $\wt{h}$,
the Viscous ISOMAP embedding matches the standard one. Also note how
increasing the viscosity term $\wt{h}$ leads to the an accurate
embedding in the principal direction, ``unrolling'' the Swiss Roll.

\begin{figure}[ht]
 \centering
 \includegraphics[width=.45\linewidth]{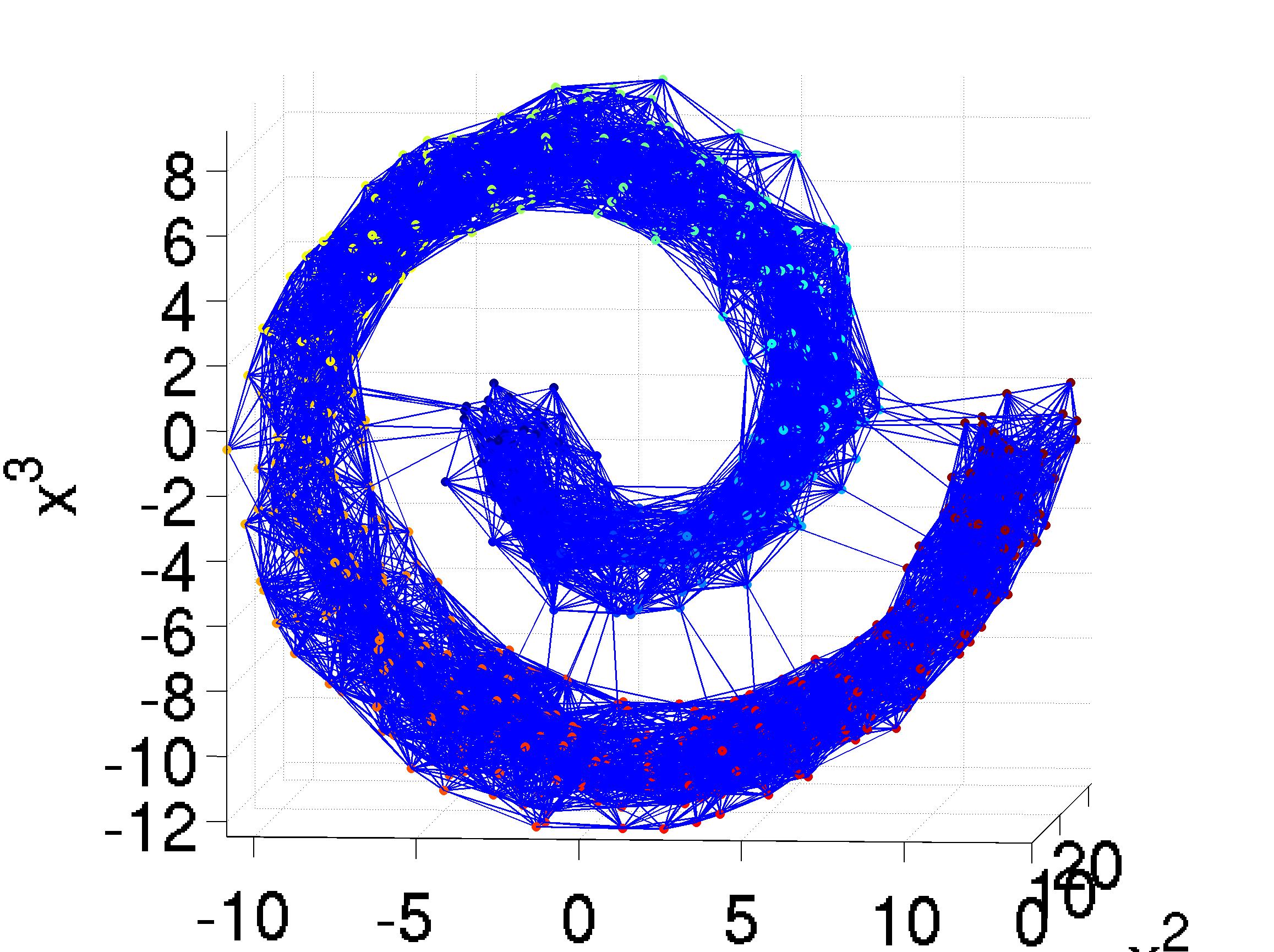}
 \includegraphics[width=.45\linewidth]{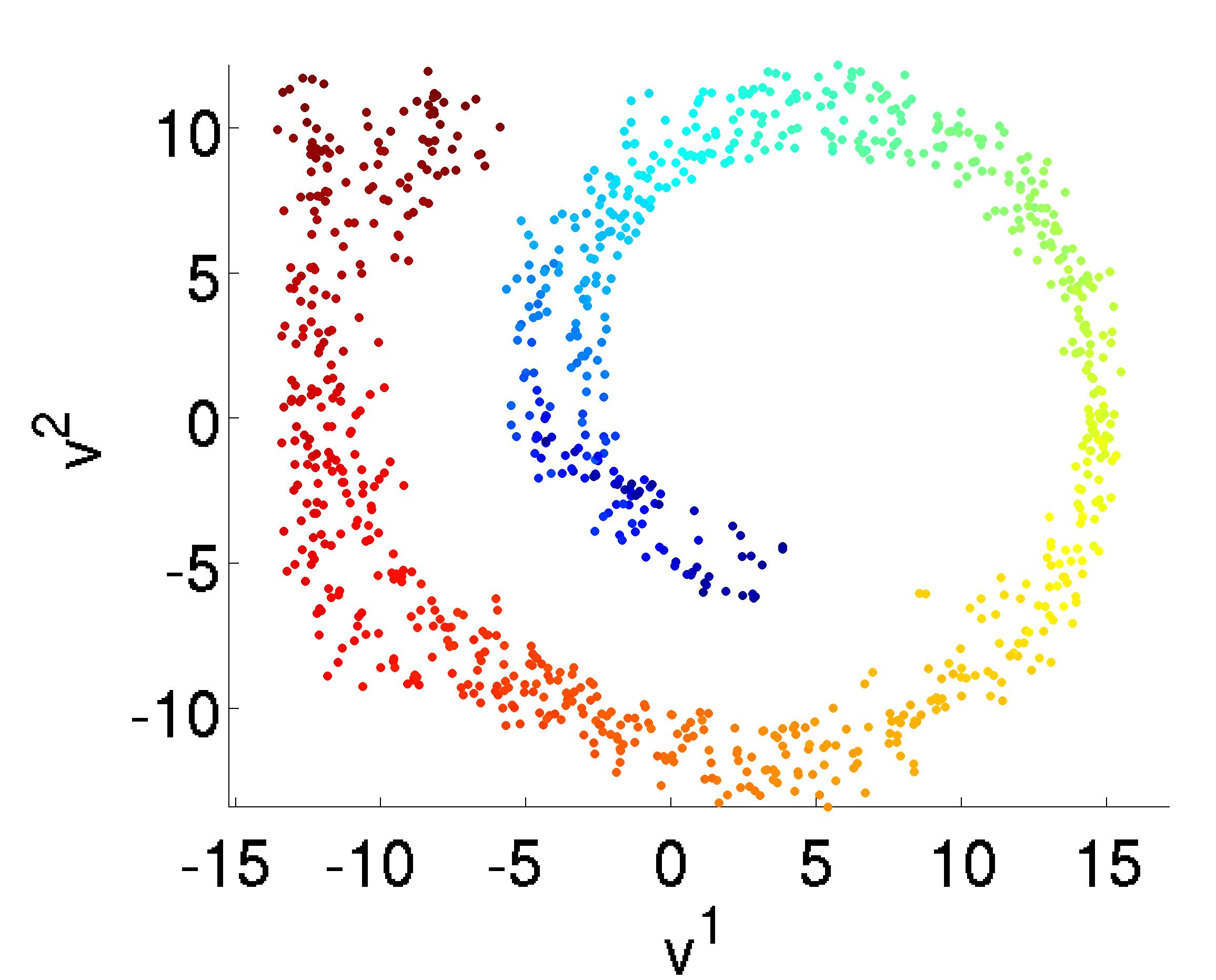}
 \includegraphics[width=.45\linewidth]{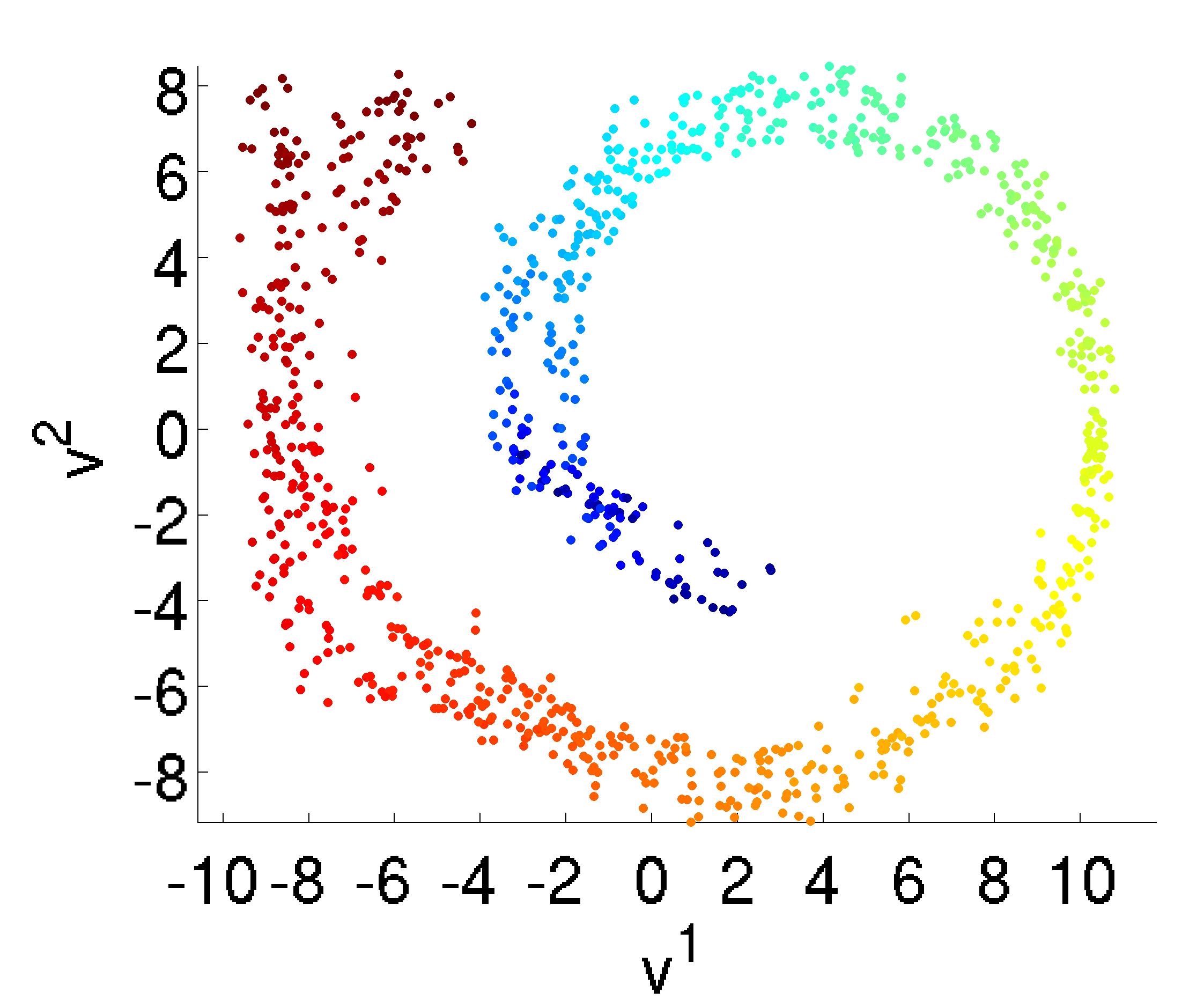}
 \includegraphics[width=.45\linewidth]{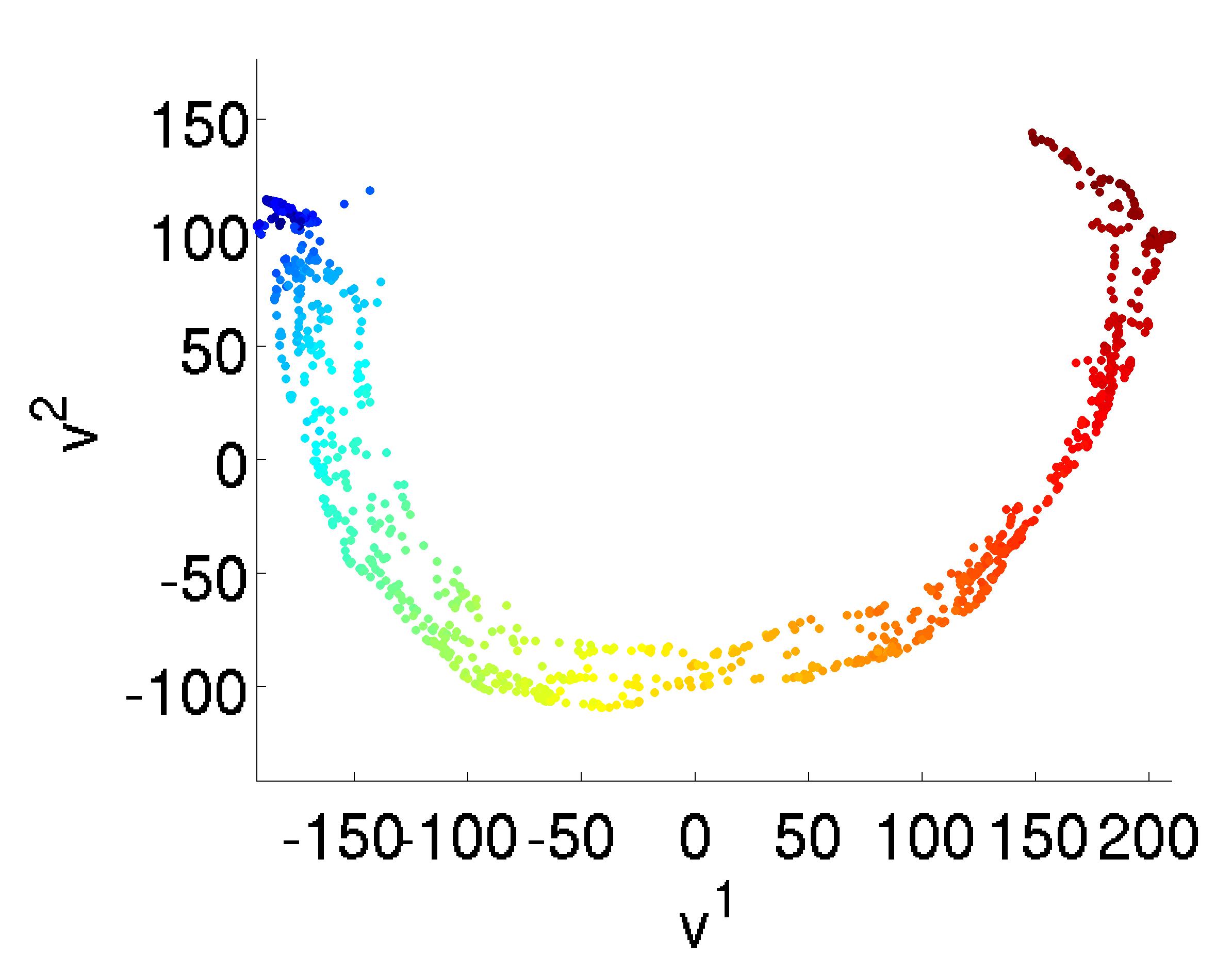}
 \caption[ISOMAP vs. Viscous ISOMAP.]%
  {ISOMAP vs. Viscous ISOMAP.  Top Left: Noisy Swiss Roll,
   Top Right: ISOMAP Embedding, first 2 components.
   Bottom Left: Viscous ISOMAP Embedding, $\wt{h}=.1$.  Bottom Right:
   Viscous ISOMAP Embedding, $\wt{h}=20$.}
 \label{fig:isomap}
\end{figure}

\section{Numerical Examples of Geodesics Estimation}
\label{sec:geoexmp}

We provide two examples of Geodesic Estimation: on the Torus, and on a
triangulated mesh.  For the Torus, we used the normalized Graph Laplacian
of \S\ref{sec:sslintro} and ground truth geodesic distances were given
by Dijkstra's Shortest Path algorithm.  On the mesh, we used the mesh
surface Laplacian of \cite{Belkin2008}, and for ground truth geodesic
distances the mesh Fast Marching algorithm of \cite{Kimmel1998} (As
implemented in Toolbox Fast Marching at
\url{http://www.ceremade.dauphine.fr/~peyre/}).  In both cases, $S_h$
was calculated via the geodesics estimator of \S\ref{sec:geo}; thus,
as always, $S_h$ estimates geodesic distances up to a constant.

\begin{exmp}[The Torus $T = S^1 \x S^1$ in $\bbR^3$]
\label{ex:torus}
The torus $T$ is defined by the points $(x^1,x^2,x^3) = ((2+\cos
v^1)\cos v^2, (2+\cos v^1)\sin v^2, \sin  v^2)$ for $v^1 \in [0,2\pi)$
and $v^2 \in [0,2\pi)$.  We used $n=1000$ randomly sampled points,
with $k=100$ neighbors for the initial NN graph, $\epsilon = .01$ and
$\wt{h} = .001$.  Setting $\cA = \set{x_i}$, $i=1,3$ where $x_1$
corresponds to $(v^1,v^2) = (0,0)$, and $x_3$ is a randomly chosen point,
we can calculate geodesic distances of all points in $\cX$ to these
anchors.  The results are shown in Fig.~\ref{fig:geotorus}.

\begin{figure}[h!]
\centering
\includegraphics[width=.49\linewidth]{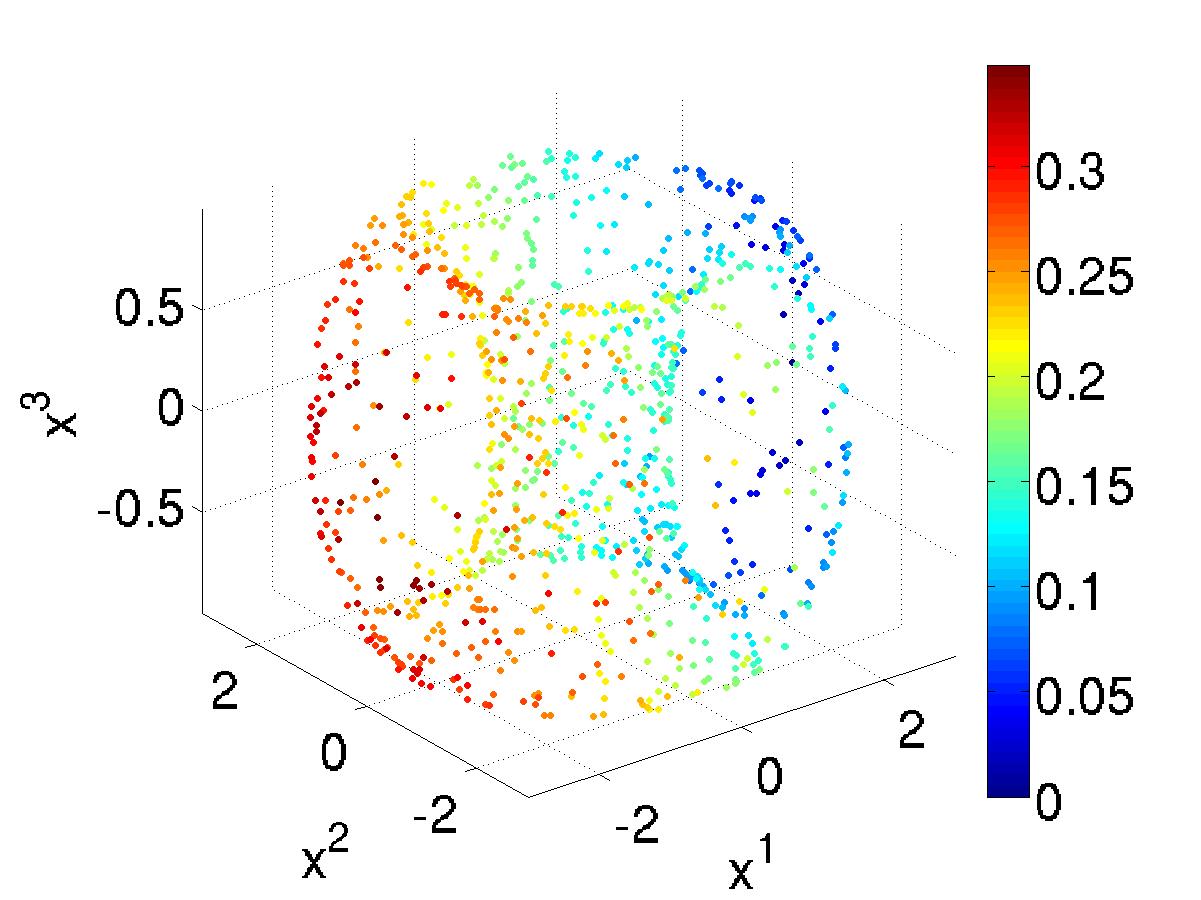}
\includegraphics[width=.49\linewidth]{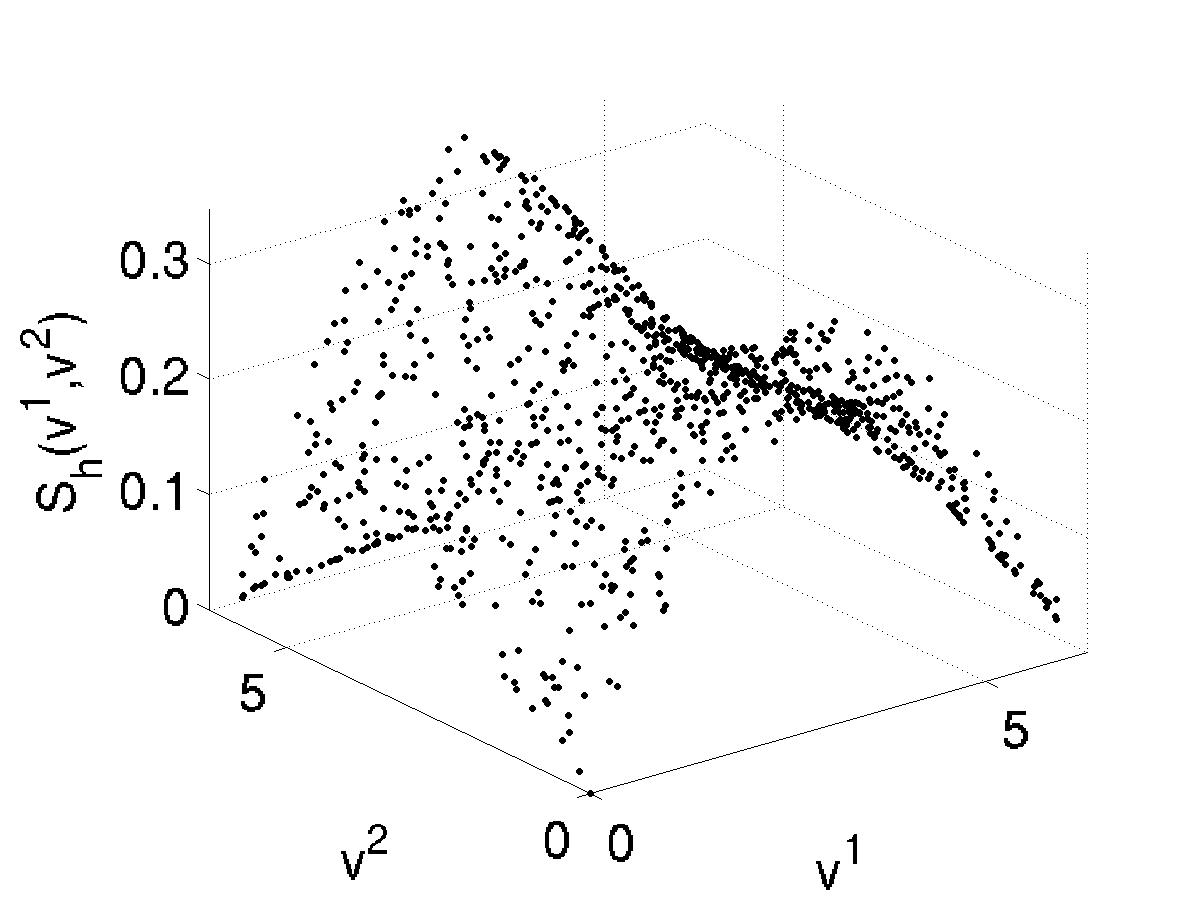} \\
\includegraphics[width=.6\linewidth]{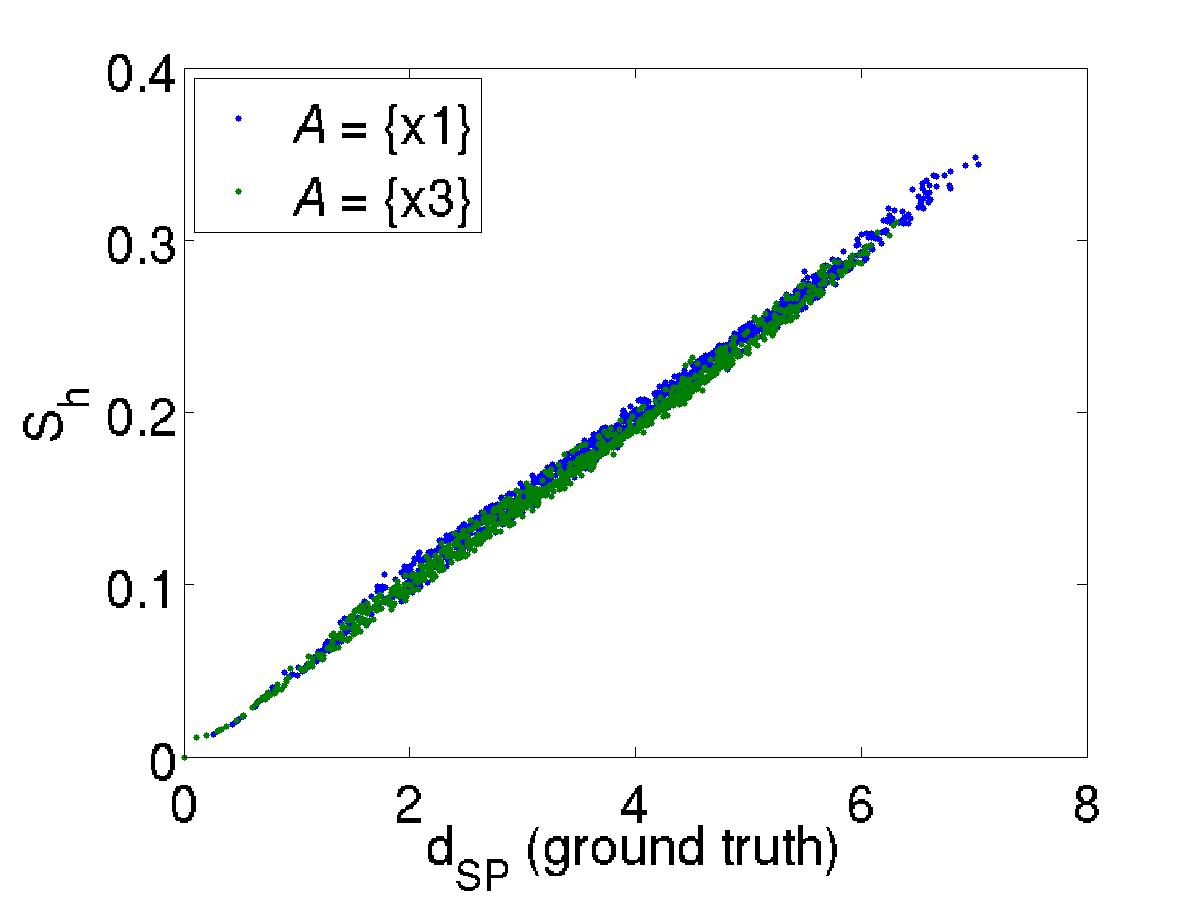}
\caption[Geodesics Estimation on $T$.]%
 {Geodesics Estimation on $T$.
  Top Left: $S_h(x)$ in $\bbR^3$; $\cA=\set{x_1}$.
  Top Right: $S_h(v^1,v^2)$ in the parameter space; $\cA = \set{x_1}$.
  Bottom: $S_h$ vs. Shortest Path estimates for $\cA = \set{x_1}$, $\set{x_3}$.}
\label{fig:geotorus}
\end{figure}
\end{exmp}

\begin{exmp}[The Dancing Children Mesh]
\label{ex:mesh}
The Dancing Children mesh is a complex (high genus) mesh from the
Aim$@$Shape Repository (\url{http://shapes.aimatshape.net/}).
The mesh $G$ is composed of $n \approx 36000$ vertices; and we used
$\epsilon=.1$ (times the mean edge distance), and $\wt{h}=.01$ for
the estimation procedure.
The anchor point $x_1$ was chosen randomly.
Our results are shown in Fig.~\ref{fig:geomesh}.  Note that minor
discrepancies between the Fast Marching estimate $d_{FM}$ and our
estimate $S_h$ occur near areas with complex topology and areas of high
curvature (e.g. near a hole in the mesh).

\begin{figure}[h!]
  \centering
  \includegraphics[width=.49\linewidth]{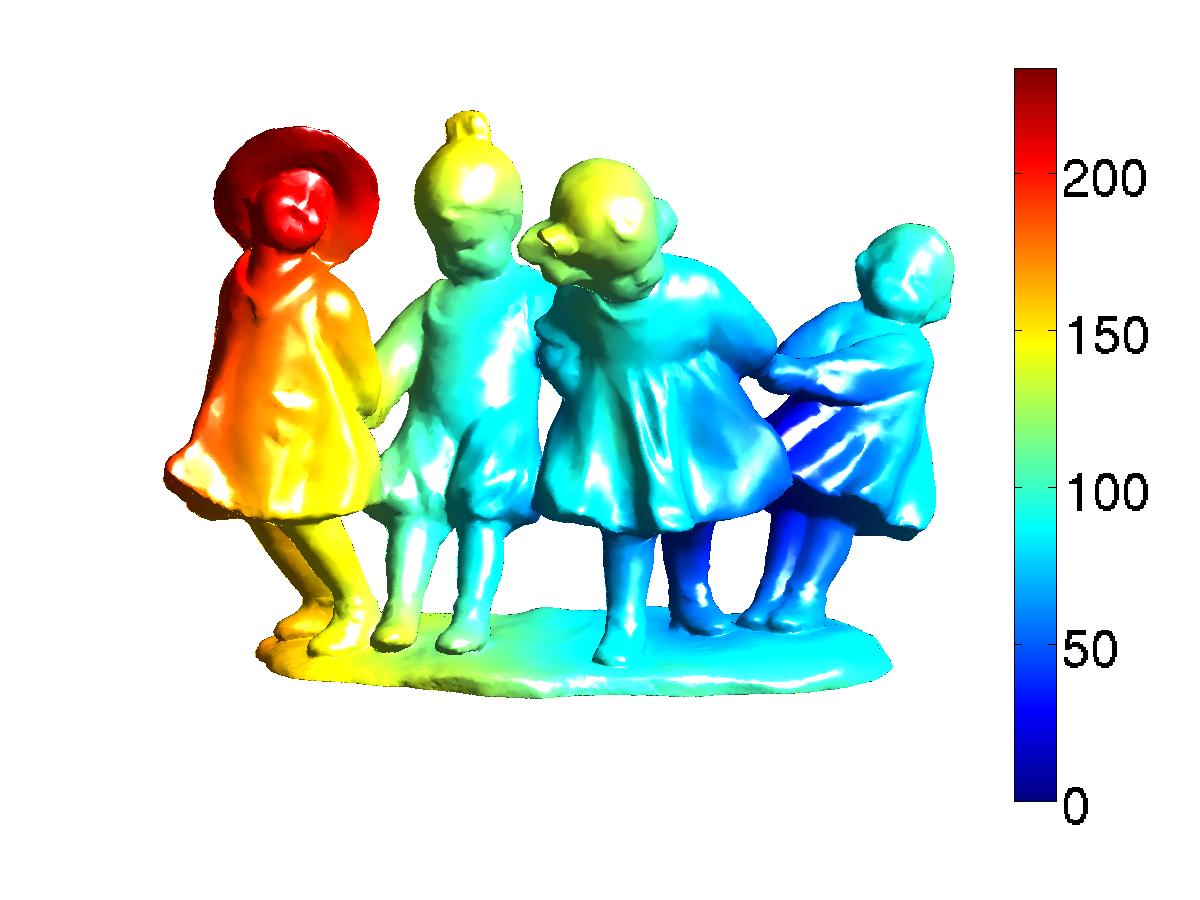}
  \includegraphics[width=.49\linewidth]{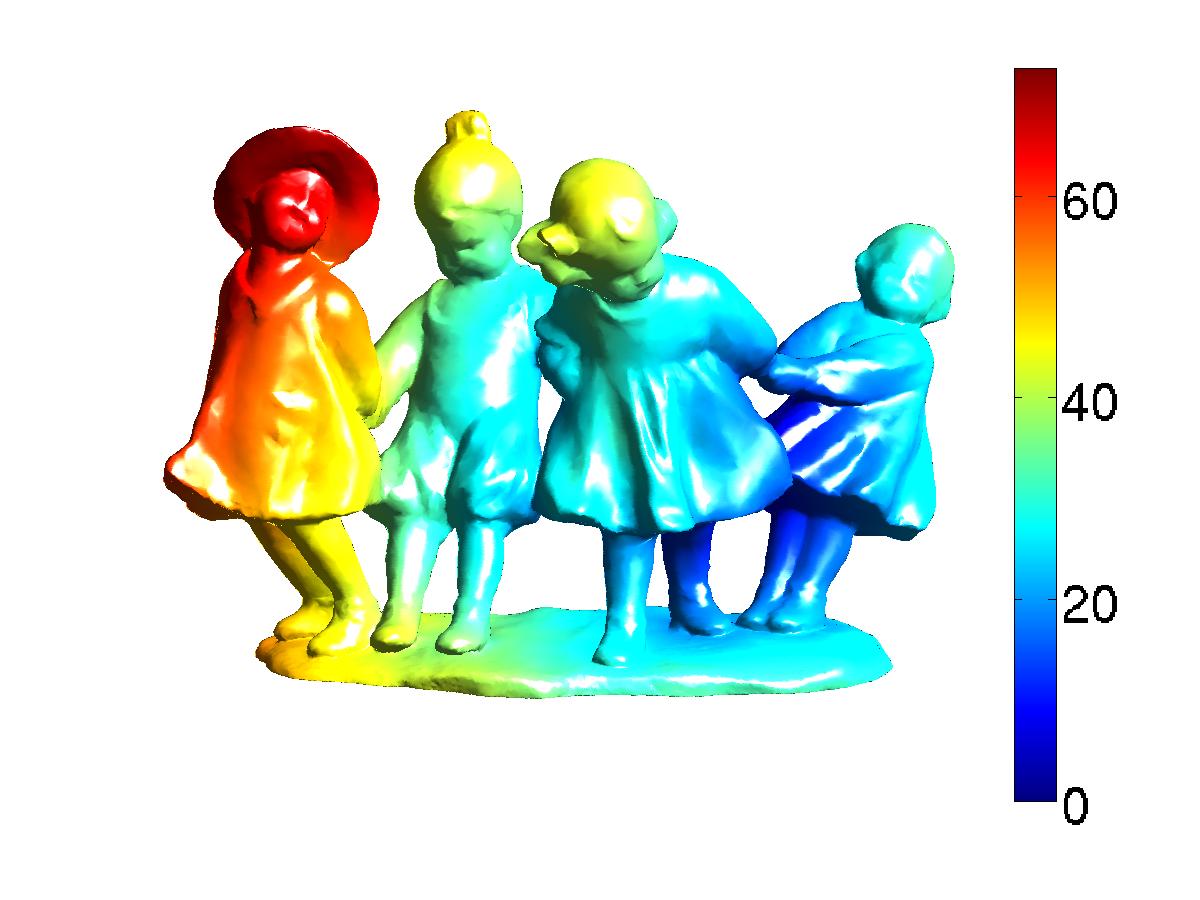} \\
  \includegraphics[width=.6\linewidth]{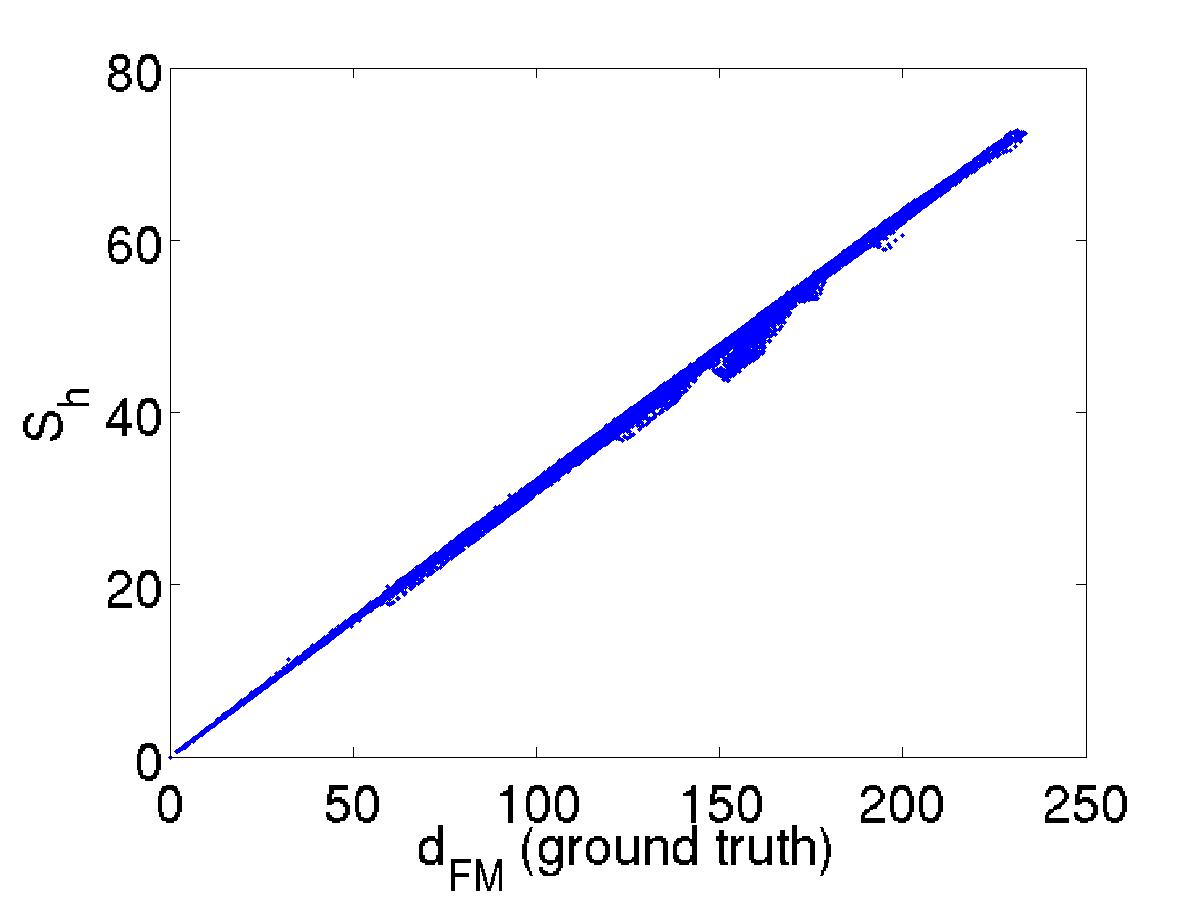}
  \caption[Geodesics Estimation on the Dancing Children mesh.]%
   {Geodesics Estimation on the Dancing Children mesh.
    Top Left: Fast Marching Estimate $d_{FM}(x)$ on the mesh; $\cA=\set{x_1}$.
    Top Right: $S_h(x)$ on the mesh; $\cA = \set{x_1}$.
    Bottom: $S_h$ vs. Fast Marching estimates for $\cA = \set{x_1}$.}
  \label{fig:geomesh}
\end{figure}
\end{exmp}

%\section{Conclusion and Future Work}

\chapter{Synchrosqueezing\label{ch:ss}%
\chattr{This chapter is based on work in collaboration with
Hau-Tieng~Wu, Department of Mathematics, and Gaurav~Thakur, Program in
Applied and Computational Mathematics, Princeton University, as
submitted in \cite{Brevdo2011b}.}}

\section{Introduction}

In this chapter, we analyze the Synchrosqueezing transform, a
consistent and invertible time-frequency analysis tool that can
identify and extract oscillating components (of time-varying frequency
and amplitude) from regularly sampled time series.  We first describe
a fast algorithm implementing the transform.  Second, we show
Synchrosqueezing is robust to bounded perturbations
of the signal.  This stability property extends the applicability of
Synchrosqueezing to the analysis of nonuniformly sampled and noisy
time series, which are ubiquitous in engineering and the natural
sciences.  Numerical simulations show that Synchrosqueezing provides a
natural way to analyze and filter a variety of signals.  In Chapter
\ref{ch:ssapp}, we use Synchrosqueezing to analyze a variety of
data, including ECG signals and climate proxies.

The purpose of this chapter is twofold.
We first describe the Synchrosqueezing transform in detail and
highlight the subtleties of a new fast numerical implementation.
Second, we show both numerically and theoretically that
Synchrosqueezing is stable under bounded signal perturbations.  It is
therefore robust to noise and to errors incurred by preprocessing
using approximations, such as interpolation.
%
%This extensive example shows how Synchrosqueezing improves on
%the state of the art in scientific signal analysis.

The chapter is organized as follows.  We first describing
Synchrosqueezing, and in \S\ref{sec:ssimp} we provide a fast new
implementation%
\footnote{The Synchrosqueezing Toolbox for MATLAB, and the codes
used to generate all of the figures in this chapter, are available
at \url{http://math.princeton.edu/~ebrevdo/synsq/}.}.
In \S\ref{sec:SStheory} we provide theoretical evidence that
Synchrosqueezing analysis and reconstruction are stable to bounded
perturbations.  In \S\ref{sec:wmisc}, we numerically compare
Synchrosqueezing to other common transforms, and provide
examples of its stability properties.  Conclusions and ideas for
future theoretical work are in \S\ref{sec:ssfuture}.

Comprehensive numerical examples and applications are deferred to
Chapter~\ref{ch:ssapp}.

\section{Prior Work}

Synchrosqueezing is a tool designed to extract and compare
oscillatory components of signals that arise in complex systems. It
provides a powerful method for analyzing signals with time-varying behavior
and can give insight into the structure of their constituent
components. Such signals $f(t)$ have the general form
\begin{equation}\label{eq:sig}
  f(t) \,=\, \sum_{k=1}^K f_k(t) + e(t),
\end{equation}
where each component $f_k(t) = A_k(t)\cos(\phi_k(t))$ is an
oscillating function, possibly with smoothly time-varying amplitude and
frequency, and $e(t)$ represents noise or observation
error.  The goal is to extract the amplitude factor $A_k(t)$ and the
Instantaneous Frequency (IF) $\phi'_k(t)$ for each $k$.

Signals of the form \eqref{eq:sig} arise naturally in engineering and
scientific applications, where it is often important to understand
their spectral properties. Many time-frequency (TF) transforms exist
for analyzing such signals, such as the Short Time Fourier Transform
(STFT), Wavelet Transform, and Wigner-Ville distribution
\cite{Flandrin1999}, but these methods can fail to capture key
short-range characteristics of the signals.  As we will see,
Synchrosqueezing deals well with such complex data.

Synchrosqueezing is a TF transform that is ostensibly
similar to the family of time-frequency reassignment (TFR) algorithms,
methods used in the estimation of IFs in signals of the form given in
\eqref{eq:sig}. TFR analysis originates from a study of the STFT,
which smears the energy of the superimposed IFs around their center
frequencies in the spectrogram. TFR analysis ``reassigns'' these
energies to sharpen the spectrogram \cite{Flandrin2002, Fulop2006}.
However, there are some significant differences between
Synchrosqueezing and most standard TFR techniques.

Synchrosqueezing was originally introduced in the context of audio
signal analysis \cite{Daubechies1996}.  In \cite{Daubechies2010}, it
was further analyzed theoretically as an alternative way to understand
the {\em Empirical Mode Decomposition} (EMD) algorithm
\cite{Huang1998}.  EMD has proved to be a
useful tool for analyzing and decomposing natural
signals. Like EMD, Synchrosqueezing can extract and
clearly delineate components with time varying spectrum.  Furthermore,
like EMD, and unlike most TFR techniques, it allows individual
reconstruction of these components.
%In contrast to classical EMD \cite{Wu2005}, and 

\section{\label{sec:analysis}Synchrosqueezing: Analysis}
Synchrosqueezing is performed in three steps.  First, the Continuous
Wavelet Transform (CWT) $W_f(a,b)$ of $f(t)$ is calculated
\cite{Daubechies1992}.  Second, an initial estimate of the
FM-demodulated frequency, $\omega_f(a,b)$, is calculated on the
support of $W_f$.  Finally, this estimate is used to squeeze $W_f$ via
reassignment; we thus get the Synchrosqueezing representation
$T_f(\omega, b)$. Synchrosqueezing is invertible: we can calculate
$f$ from $T_f$.  Our ability to extract individual components
stems from filtering $f$ by keeping energies from
specific regions of the support of $T_f$ during reconstruction.

Note that Synchrosqueezing, as originally
proposed \cite{Daubechies1996}, estimates the FM-demodulated frequency
from the wavelet representation $W_f(a,b)$ before performing reassignment.
However, it can be adapted to work on ``on top of'' many invertible
transforms (e.g. the STFT \cite{Thakur2010}).  We focus on the
original wavelet version as described in \cite{Daubechies2010}.

We now detail each step of Synchrosqueezing, using the harmonic signal
$h(t) = A \cos(\Omega t)$ for motivation.  As a visual aid,
Fig.~\ref{fig:simple} shows each step on the signal
$h(t)$ with $A=1$ and $\Omega = 4 \pi$.  Note that
Figs.~\ref{fig:simple}(b,d) show that Synchrosqueezing is more
``precise'' than the CWT.
%
%See Ex. 1 of \S\ref{sec:wmisc} for a more comprehensive comparison
%with the CWT and the STFT.

\begin{figure}[ht]
  \centering
  \includegraphics[width=\columnwidth]{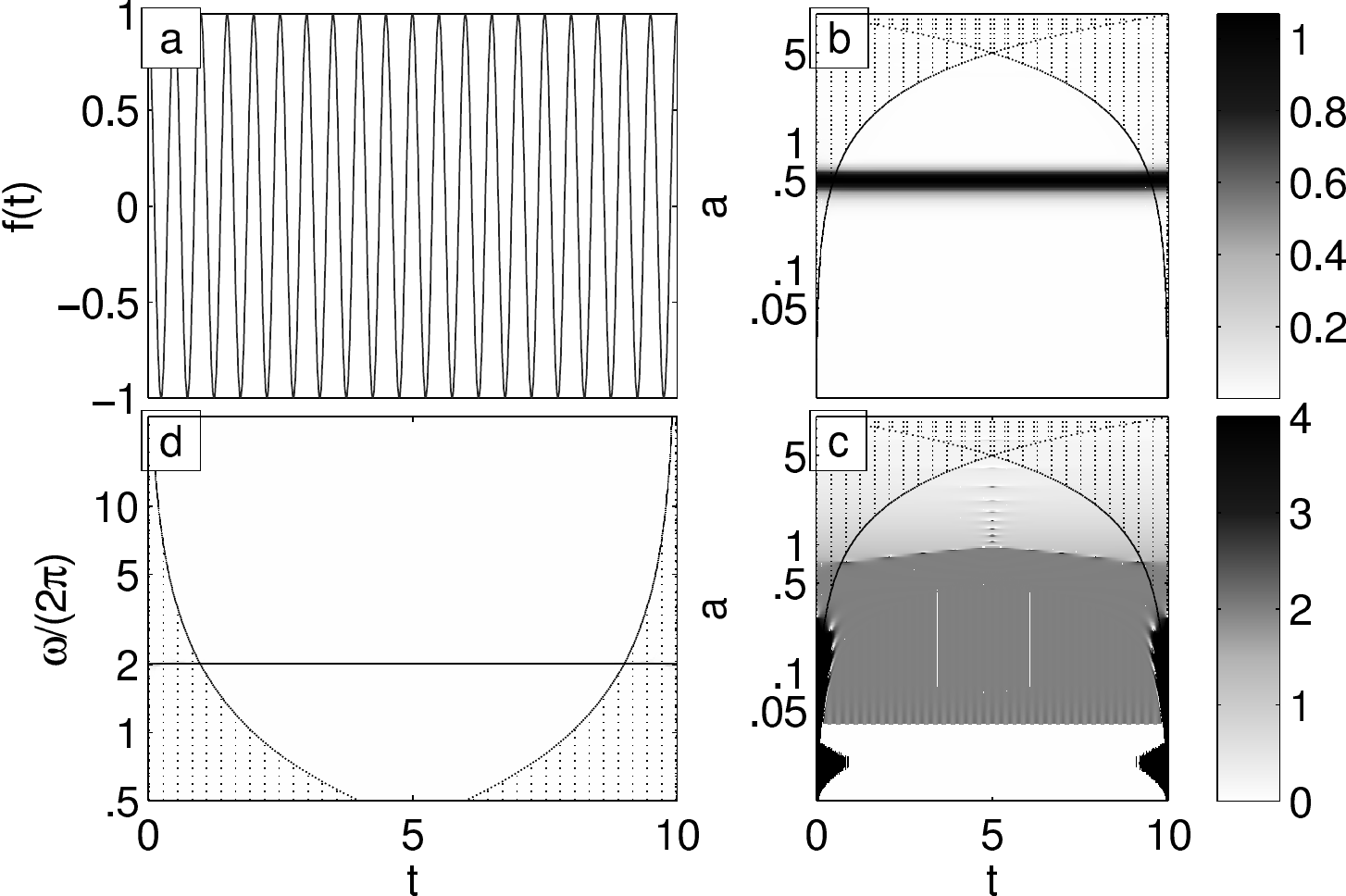}
  \caption[Synchrosqueezing example for $h(t) = \cos(4 \pi t)$.]%
   {\label{fig:simple}
    Synchrosqueezing example for $h(t) = \cos(4 \pi t)$.
    Clockwise:
    a) $h(t)$, sampled, $n=1024$ points.
    b) CWT of $h$, $|W_h|$.
    c) FM-demodulated frequency from $W_h$, ${\omega}_{h}$.
    d) Synchrosqueezing of $h$: $|T_h|$; line at
    $w_{\hat{l}} \approx 2$.
  }
\end{figure}

\subsection{CWT of ${f(t)}$}
For a given mother wavelet $\psi$, the CWT of
$f$ is given by $W_f(a,b) = \int_{-\infty}^\infty f(t) a^{-1/2}
\overline{\psi \left(\frac{t-b}{a} \right)} dt,$ where $a$ is the
scale and $b$ is the time offset.
We assume that $\psi$ has fast decay, and that its Fourier
transform $\wh{\psi}(\xi) = (2\pi)^{-1/2} \int_{-\infty}^\infty
\psi(t) e^{-i \xi t} dt$ is approximately zero in the negative frequencies
\footnote{More details about the Fourier transform, and analysis
  on intervals, are available in App.~\ref{app:fourier}}:
$\wh{\psi}(\xi) \approx 0$ for $\xi < 0$, and is concentrated around
some positive frequency $\xi = \omega_0$ \cite{Daubechies2010}. Many
wavelets have these properties (several examples and
compared in \S\ref{sec:wcmp}).  For $h(t)$, the harmonic signal above,
upon applying our assumptions we get $W_{h}(a,b) = \frac{1}{2 \sqrt{2
\pi}} A a^{1/2} \overline{\wh{\psi}(a \Omega)} e^{i b \Omega}$.

\subsection{Calculate the FM-demodulated frequency ${\omega(a,b)}$}
The wavelet representation of the harmonic signal $h(t)$
(with frequency $\Omega$) will have its energy spread out in the time-scale
plane around the line $a = \omega_0/\Omega$, and this frequency will
be encoded in the phase \cite{Daubechies1996, Daubechies2010}.
In those regions where $\abs{W_h} > 0$ we would like to remove the
effect of the Wavelet on this frequency.   We perform
a type of FM demodulation by taking derivatives: $ (W_{h}(a,b))^{-1}
\partial_b W_{h}(a,b) = i \Omega $.  This simple model
leads to an estimate of the frequency in the time-scale plane:
\beq
\omega_{f}(a,b)=
\label{eq:omegax0}
\begin{cases} \frac{-i\partial_{b}W_{f}(a,b)}{W_{f}(a,b)} & |W_{f}(a,b)|>0\\
\infty & |W_{f}(a,b)|=0
\end{cases}.
\eeq
\subsection{Squeezing in the time-frequency plane: ${T_f(\omega,b)}$}
The final step of Synchrosqueezing is reassigning energy
in the time-scale plane to the TF plane according to the frequency map
$(a,b) \to (\omega(a,b), b)$.  Reassignment follows from the inversion
property of the CWT: when $f(t)$ is real,
\beq
\label{eq:Winv}
f(b) = 2 \cR_\psi^{-1} \Re{\int_0^\infty W_f(a,b) a^{-3/2} da},
\eeq
where $\cR_\psi = \sqrt{2 \pi} \int_0^\infty \xi^{-1}
\overline{\wh{\psi}(\xi)} d\xi$ is a normalizing constant.

We first break up the integrand in \eqref{eq:Winv} according to the
FM-demodulated frequency estimate $\omega_f$.  Define frequency divisions
$\set{w_l}_{l=0}^\infty$ s.t. $w_0 > 0$ and $w_{l+1} > w_l$ for all
$l$.  Further, let the frequency bin $\cW_l$ be the set of
points $w' \in \bbC$ closer to $w_l$ than to any other $w_l'$.
We define the Discrete-Frequency Wavelet Synchrosqueezing transform of
$f$
as:
\beq
\label{eq:Tfdef}
T_f(w_l, b) = \int_{\set{a : \omega_f(a,b) \in \cW_l}} W_f(a,b) a^{-3/2} da.
\eeq
In other words, $T_f(w_l,b)$ is the ``volume'' of the
frequency preimage set
$\cW_l^{-1}(b) = \set{a : \omega_f(a,b) \in \cW_l}$ under
the signed measure $\mu_{f,b}(a) = W_f(a,b) a^{-3/2} da$.

This definition has several favorable properties.  First, it allows
us to reconstruct $f$ from $T_f$:
\beq
\label{eq:Tfrecon}
f(b) = 2 \cR_\psi^{-1} \Re{\sum_l T_f(w_l, b)}.
\eeq
Second, for the
harmonic signal $h(t)$, with $\omega_{h}(a,b) = \Omega$, there
will be a single $\hat{l}$ such that $w_{\hat{l}}$ is closest to
$\omega_{h}(a,b)$.  From \eqref{eq:Winv}, we have $h(b) =
2\Re{\cR_\psi^{-1} T_{h}(w_{\hat{l}},b)}$.  Further, the
magnitude of $T_h$ is proportional to that of $h(t)$:
$\abs{T_{h}(w_{\hat{l}},b)} = \frac{\abs{A}}{2 \pi} \abs{\cR_\psi}$.

More generally, for a wide class of signals with slowly varying
$A_k(t)$ and well separated ${\phi_k'}(t)$, given a
sufficiently fine division of the frequency bins $\set{w_l}$, each of
the $K$ components can be well concentrated into its own ``curve'' in the
TF plane (see Thm. \ref{SSThm} below).
This allows us to analyze such signals: by looking at $\abs{T_f(w,b)}$
to identify and extract the curves, and to reconstruct
their associated components.

\section{\label{sec:ssimp}A Fast Implementation}

In practice, we observe the vector $\tf \in \bbR^n$, $n = 2^{L+1}$,
where $L$ is a nonnegative integer.  Its elements, $\tf_m,
m=0,\ldots,n-1$, correspond to a uniform discretization of $f(t)$ taken at
the time points $t_m = t_0 + m \Delta t$.  To prevent boundary
effects, we pad $\tf$ on both sides (using, e.g., reflecting boundary
conditions).

We now describe a fast numerical implementation of Synchrosqueezing.
The speed of our algorithm lies in two key steps.  First, we
calculate the Discrete Wavelet Transform (DWT) of the vector $\tf$ using
the Fast Fourier Transform (FFT).  Second, we discretize the squeezing
operator $T$ in a way that lends itself to a fast numerical
implementation.

\subsection{DWT of sampled signal $\tf$}
The DWT samples the CWT $W_f$ at the locations $(a_j,t_m)$, where $a_j
= 2^{j/n_v} \Delta t$, $j=1,\ldots,L n_v$, and the number of voices
$n_v$ is a user-defined ``voice number'' parameter
\cite{Goupillaud1984} (we have found that $n_v = 32$ works well).  The
DWT of $\tf$ can be calculated in $O(n_v n \log_2^2 n)$ operations
using the FFT.  We outline the steps below.

First note that $W_f(a,b) = \left[ a^{-1/2} \overline{\psi(-t/a)} *
  f(t) \right](b)$, where $*$ denotes convolution over $t$.
In the frequency domain, this relationship becomes:
$\wh{W}_f(a,\xi) = a^{1/2} \wh{f}(\xi) \wh{\psi}(a \xi)$.
We use this to calculate the DWT, $\tW_\tf(a_j,t_m)$.  Let $\cF_n$
($\cF_n^{-1}$) be the standard (inverse) circular Discrete Fourier
Transform.  Then
\beq
\label{eq:Wxdisc}
\tW_\tf(a_j, \cdot) = \cF_n^{-1} \left( (\cF_n \tf) \odot \wh{\psi}_j \right).
\eeq
Here $\odot$ denotes elementwise
multiplication and $\wh{\psi}_j$ is an $n$-length vector with
%%$(\wh{\psi}_j)_m = (-1)^m a_j^{1/2} \wh{\psi}(a_j \xi_m)$; $\xi_m$
$(\wh{\psi}_j)_m = a_j^{1/2} \wh{\psi}(a_j \xi_m)$; $\xi_m$
are samples in the unit frequency interval: $\xi_m = 2 \pi m / n$,
$m=0,\ldots, n-1$.

%\footnote{Here we use the standard MATLAB sampling.}:
%$$ \xi_k = \begin{cases} 2 \pi k / n & k=0,\ldots,n/2 \\ 2 \pi
%  (n-k)/n & k = n/2+1, \ldots, n-1
%\end{cases}
%$$

%Note the additional $(-1)^k$ term in the discretization of
%$\wh{\psi}$, which shifts the $\IFFT{\wh{\psi}_j}$ so that the $t=0$
%term is at the center, and the SHIFT operator in \eqref{eq:Wxdisc},
%which undoes this shift.  These operations do not make any
%difference here, but will be important in the IF extraction step,
%which we discuss next.

\subsection{A Stable Estimate of $\omega_f$: $\wt{\omega}_\tf$}

We first require a slight modification of the FM-demodulated frequency
estimate 
\eqref{eq:omegax0},
\beq
\label{eq:omegax}
\omega_f(a,b) = \Im{(W_f(a,b))^{-1} \partial_b W_f(a,b)}.
\eeq
This definition is equivalent to \eqref{eq:omegax0} when
Synchrosqueezing is performed via \eqref{eq:Tfdef}, and simplifies the
algorithm.

In practice, signals have noise and other artifacts due to,
e.g., sampling errors, and the phase of $W_f$ is unstable when
$\abs{W_f} \approx 0$.  As such the user should choose some
$\gamma > 0$ (we often use $\gamma \approx 10^{-8}$) as a hard
threshold on $\abs{W_f}$.  We define the numerical support of
$\tW_\tf$, on which $\omega_f$ can be estimated:
\begin{center}
$ \wt{\cS}^\gamma_\tf(m) = \set{j : \abs{\tW_\tf(a_j,t_m)} >
  \gamma}$,
for $m = 0,\ldots,n-1$.
\end{center}

The estimate of $\omega_f$, $\wt{\omega}_\tf$, can be calculated
by taking differences of $\tW_\tf$ with respect to $m$ before
applying \eqref{eq:omegax}, but we provide a more direct way.
Let 
Using the property $\wh{\partial_b W_f}(a,\xi) = i \xi
\wh{W_f}(a,\xi)$, we estimate the FM-demodulated frequency, for
$j \in \wt{\cS}^\gamma_\tf(m)$, as
\begin{center}
$ \wt{\omega}_\tf(a_j,t_m) =
  \Im{\left(\tW_\tf(a_j,t_m)\right)^{-1} \partial_b \tW_\tf(a_j,t_m) },
$
\end{center}
with the time derivative of $W_f$ estimated via (e.g., \cite{Tadmor1986}): % Gottlieb1984,
\begin{center}
$ \partial_b \tW_\tf(a_j,\cdot) = 
   \cF^{-1}_n \left( (\cF_n \tf) \odot \wh{\partial \psi}_j \right),
$
\end{center}
where
$(\wh{\partial \psi}_j)_m = a_j^{1/2} i \xi_m \wh{\psi}(a_j \xi_m)/\Delta
  t$ for $m=0,\ldots,n-1$.

Finally, we normalize $\wt{\omega}$ by $2 \pi$ so that the dominant
frequency estimate is $\alpha$ when $f(t) = \cos(2 \pi \alpha t)$.

\subsection{Fast estimation of ${T_f}$ from ${\tW_\tf}$ and ${\wt{\omega}_\tf}$}
The representation $\tW_\tf$ is given with respect to $n_a = L n_v$
log-scale samples of the scale $a$, and this leads to several
important considerations when estimating $T_f$ via \eqref{eq:Winv}
and \eqref{eq:Tfdef}.  First, due to lower resolutions in coarser
scales, we expect to get lower resolutions in the lower frequencies.
We thus divide the frequency domain into $n_a$ components on a
log scale.  Second, sums with respect to $a$ on a
log scale, $a(z) = 2^{z/n_v}$ with $da(z) = a \frac{\log
  2}{n_v} dz$, lead to the modified integrand
${W_{f}(a, b) a^{-1/2} \frac{\log 2}{n_v} dz}$ in \eqref{eq:Tfdef}.

To choose the frequency divisions, note that the discretization
period $\Delta t$ limits the maximum frequency $\ol{w}$ that can be
estimated.  The Nyquist theorem suggests that this frequency is
$\ol{w} = w_{n_a-1} = \frac{1}{2 \Delta t}$.  Further, if we assume
periodicity, the maximum period of an input signal is $n \Delta t$;
thus the minimum frequency is $\ul{w} = w_0 = \frac{1}{n \Delta
  t}$.  Combining these limits with the log scaling of
the $w$'s we get the divisions: $w_l = 2^{l \Delta w} \ul{w}$, $l =
0, \ldots, n_a-1$, where $\Delta w = \frac{1}{n_a-1} \log_2 (n/2)$.
Note, the voice number $n_v$ has a big
effect on the frequency resolution.

We can now calculate the Synchrosqueezed estimate $\tT_\tf$.  Our
fast implementation of \eqref{eq:Tfdef} finds the associated $\cW_l$
for each $(a_j,t_m)$ and adds it to the
correct sum, instead of performing a search over all
scales for each $l$.  This is possible because
$\wt{\omega}_\tf(a_j, t_m)$ only ever lands in one frequency bin.
We provide pseudocode for this $O(n_a)$ implementation in
Alg. \ref{alg:Tf}.

\begin{algorithm}[ht]
\caption{Fast calculation of $\tT_\tf$ for fixed $m$}
\label{alg:Tf}
  \begin{algorithmic}
    \small
    \FOR[Initialize $\tT$ for this $m$]{$l = 0$ to $n_a-1$}
    \STATE $\tT_{\tf}(w_l, t_m) \leftarrow 0$
    \ENDFOR
    \FORALL[Calculate \eqref{eq:Tfdef}]{$j \in
      \wt{\cS}^\gamma_{\tf}(m)$}
    \STATE \COMMENT{Find frequency bin via $w_l = 2^{l \Delta w}
      \ul{w}$, and $\wt{\omega}_{\tf} \in \cW_l$}
    \STATE $l \leftarrow
    \textrm{ROUND} \left[ \frac{1}{\Delta w} \log_2 \left(
      \frac{\wt{\omega}_{\tf}(a_j,b_m)}{\ul{w}} \right) \right]$
    \IF{$l \in [0, n_a-1]$} 
    \STATE \COMMENT{Add normalized term to appropriate integral;
      $\Delta z = 1$}
    \STATE $\tT_{\tf}(w_l,t_m) \leftarrow \tT_{\tf}(w_l,t_m) + \frac{\log 2}{n_v}
    \tW_{\tf}(a_j,t_m) a^{-1/2}_j $
    \ENDIF
    \ENDFOR
  \end{algorithmic}
\end{algorithm}

\subsection{IF Curve Extraction and Filtered Reconstruction}

A variety of signals, especially sums of quasi-harmonic signals with
well-separated IFs, will have a frequency image
$\abs{T_f(w,b)}$ composed of several curves in the $(w,b)$ plane.
The image of the $k$th curve corresponds to both the IF
${\phi_k'}(b)$, and the entire component $A_{k}(b) \cos(\phi_{k}(b))$.

To extract a discretized curve $c^*$ we maximize a functional of the
energy of the curve that penalizes variation%
\footnote{The implementation of this step in the Synchrosqueezing
  Toolbox is a heuristic (greedy) approach that maximizes the
  objective at each time index, assuming the objective has been
  maximized for all previous time indices.}:
\beq
\label{eq:Cextract}
 \max_{c \in \set{w_l}^n} \sum_{m=0}^{n-1}
  E_\tf(w_{c_m}, t_m) - 
  \lambda \sum_{m=1}^{n-1} \Delta w |c_m - c_{m-1}|^2,
\eeq
where $E_\tf(w_l, t_m) = \log (|\tT_\tf(w_l,t_m)|^2)$ is the
normalized energy of $\tT$.  The user-defined parameter $\lambda > 0$
determines the ``smoothness'' of the resulting curve estimate 
(we use $\lambda = 10^5$).  Its associated component
$\tf^*$ can be reconstructed via \eqref{eq:Tfrecon}, by restricting the
sum over $l$, at each $t_m$, to the neighborhood $\cN_m =
[{c^*_m-n_w},{c^*_m+n_w}]$ (we use the window size $n_w =
n_v/2$).  The next curve is extracted by setting $\tT_\tf(\cN_m,t_m)=0$
for all $m$ and repeating the process above.

\section{Consistency and Stability of Synchrosqueezing}
\label{sec:SStheory}

We first review the main theorem on wavelet-based
Synchrosqueezing, as developed in \cite{Daubechies2010} (Thm. \ref{SSThm}).
Then we show that the components extracted via Synchrosqueezing
are stable to bounded perturbations such as noise and
discretization error.

We specify a class of functions on which these results hold.  In
practice, Synchrosqueezing works on a wider function class.

\begin{defn}[Sums of Intrinsic Mode Type (IMT) Functions]
The space $\mathcal{A}_{\epsilon,d}$ of superpositions
of IMT functions, with smoothness $\epsilon$ and separation $d$,
consists of functions having the form
$f(t)=\sum_{k=1}^{K}f_{k}(t)$ with $f_{k}(t) = A_{k}(t)
e^{i\phi_{k}(t)}$.  For $t \in \bbR$ the IF components
$\phi'_{k}$ are ordered and relatively well separated (high
frequency components are spaced further apart than low frequency
ones):
\begin{align*}
\forall t \quad \phi_{k}'(t)&>\phi_{k-1}'(t), \quad \text{and} \\
\inf_{t}\phi'_{k}(t) - \sup_{t}\phi'_{k-1}(t) 
  &\geq d(\inf_{t}\phi'_{k}(t)+\sup_{t}\phi'_{k-1}(t)).
\end{align*}
Functions in the class $\mathcal{A}_{\epsilon,d}$ are
essentially composed of components with time-varying
amplitudes. Furthermore, the amplitudes vary slowly, and the
individual IFs are sufficiently smooth.  For each $k$,
\begin{align*}
& A_k \in L^{\infty}\cap C^{1}, \quad \phi_k \in C^{2}, 
  \quad \phi'_k,\phi''_k \in L^{\infty}, \quad \phi'_k(t)>0, \\
& \| A'_k \| _{L^{\infty}} \leq \epsilon \| \phi'_k \|_{L^{\infty}}, 
\quad \text{and} \quad \| \phi''_k \|_{L^{\infty}} \leq \epsilon \| \phi'_k \|_{L^{\infty}}.
\end{align*}
\end{defn}

For the theoretical analysis, we also define the Continuous
Wavelet Synchrosqueezing transform, a smooth version of $T_f$.
% (as defined in \eqref{eq:Tfdef}).

\begin{defn}[Continuous Wavelet Synchrosqueezing]
Let $h\in C_{0}^{\infty}$ be a smooth function such that
$\norm{h}_{L^1}=1$.  The Continuous Wavelet
Synchrosqueezing transform of function $f$, with accuracy $\delta$ and
thresholds $\epsilon$ and $M$, is defined by
\beq
S_{f,\epsilon}^{\delta,M}(b,\eta) = 
\int_{\Gamma_{f,\epsilon}^{\delta, M}}
   \frac{W_f(a,b) }{a^{3/2}} \delta^{-1} h\left(\frac{\abs{\eta-\omega_{f}(a,b)}}{\delta}\right) da
\label{SS}
\eeq
where $\Gamma_{f,\epsilon}^{M} = \set{(a,b) : a \in
[M^{-1},M],|W_{f}(a,b)|>\epsilon}$.  We also denote
$S_{f,\epsilon}^{\delta} = S_{f,\epsilon}^{\delta,\infty}$ and
$\Gamma_{f,\epsilon}^{\infty} = \Gamma_{f,\epsilon}$, where the
condition $a \in [M^{-1},M]$ is replaced by $a>0$.
\end{defn}

The continuous ($S_f^{\delta}$) and discrete frequency
($T_f$) Synchrosqueezing transforms are equivalent for small
$\delta$ and large $n_v$, respectively.  The frequency term $\eta$ in
\eqref{SS} is equivalent to $w_l$ in \eqref{eq:Tfdef}, and the
integrand term $\frac{1}{\delta}h(\frac{\cdot}{\delta})$ in \eqref{SS}
takes the place of constraining the frequencies to $\cW_l$ in
\eqref{eq:Tfdef}. Signal reconstruction and filtering analogues via
the continuous Synchrosqueezing transform thus reduce to
integrating $S_{f,\epsilon}^\delta$ over $\eta>0$, similar
to summing over $l$ in \eqref{eq:Tfrecon}.

The following consistency theorem was proved in \cite{Daubechies2010}:

%% For Thm. \ref{SSThm} and Thm. \ref{SSStableThm}, assume that $f \in
%% \cA_{\epsilon,d}$, that the Synchrosqueezing wavelet $\psi$ is smooth,
%% and its Fourier transform $\wh{\psi}(\xi)$ is concentrated in
%% $[1-\Delta,1+\Delta]$ for $\Delta>0$.  In theory and in practice,
%% $\wh{\psi}$ should be well concentrated ($\Delta$ is small).
\begin{thm}[Synchrosqueezing Consistency]
\label{SSThm}
Suppose $f \in \cA_{\epsilon,d}$. Pick a wavelet $\psi\in C^{1}$
such that its Fourier transform $\wh{\psi}(\xi)$
%$\widehat{\psi}(\xi)=(2\pi)^{-1/2}\int_{-\infty}^{\infty}\psi(x)e^{-i\xi x}dx$
is supported in $[1-\Delta,1+\Delta]$ for some $\Delta<\frac{d}{1+d}$.
Then for sufficiently small $\epsilon$, Synchrosqueezing can
identify and extract the components $\set{f_k}$ from $f$:

\textbf{1}. The Synchrosqueezing plot $|S^\delta_f|$ is concentrated
around the IF curves $\{\phi_{k}'\}$.
For each $k$, define the ``scale band''
$Z_k = \{(a,b):|a\phi_{k}'(b)-1|<\Delta\}$.
For sufficiently small $\epsilon$, the FM-demodulated frequency
estimate $\omega_f$ is 
accurate inside $Z_k$ where $W_f$ is sufficiently
large ($|W_{f}(a,b)|>{\epsilon^{1/3}}$):

\begin{center}
$\abs{\omega_f(a,b) - \phi_k'(b)} \leq {\epsilon^{1/3}}$.
\end{center}
\noindent Outside the scale bands $\set{Z_k}$, $W_f$ is small:
\begin{center}
$|W_{f}(a,b)|\leq{\epsilon^{1/3}}$.
\end{center}

\textbf{2}.  Each component $f_k$ may be reconstructed
by integrating $S^\delta_f$ over a neighborhood around $\phi_{k}'$.
Choose the Wavelet threshold $\epsilon^{1/3}$ and
let $N_k(b) = \{\eta : |\eta-\phi'_{k}(b)|\leq{\epsilon^{1/3}}\}$.
For sufficiently small $\epsilon$, there is a constant $C_{1}$
such that for all $b\in\mathbb{R}$,
$$
\left|\lim \limits_{\delta\rightarrow0}\left(\mathcal{R}_{\psi}^{-1}
\int_{N_k(b)}
S_{f,\epsilon^{1/3}}^{\delta}(b,\eta)d\eta\right)-f_k(b)\right|\leq
C_{1} \epsilon^{1/3}.
$$
\end{thm}

%Then for small $\epsilon$,
%$\cR_{\psi}^{-1} \sum\limits_{l \in N_k(b)} T_{f}(w_l,b)
%\approx f_k(b) + O({\epsilon^{1/3}}).$

Note that, as expected, Thm. \ref{SSThm} implies that components $f_k$
with low amplitude may be difficult to identify and extract (as their
Wavelet magnitudes may fall below $\epsilon^{1/3}$).

Thm. \ref{SSThm} also applies to discrete Synchrosqueezing,
with the following modifications: letting $\delta \to 0$ is equivalent
to letting $n_v \to \infty$.  For reconstruction via
\eqref{eq:Tfrecon}, the integral over $\eta$ should be replaced by a sum
over $l$ in the discrete neighborhood $N_k(b) = \set{l : \abs{w_l -
\phi'_k(b)} \leq \epsilon^{1/3}}$.  Finally, the threshold
$\epsilon^{1/3}$ in Thm. \ref{SSThm} part 2 can be applied numerically
by letting $\gamma > \epsilon^{1/3}$ when calculating the discrete support
$\cS_{\tf}^\gamma$.

We prove the following theorem in \cite{Brevdo2011b}:

\begin{thm}[Synchrosqueezing stability to small perturbations]
\label{SSStableThm}
The statements in Thm. \ref{SSThm} essentially still hold
if $f$ is corrupted by a small error $e$, especially for
mid-range IFs.

Let $f\in\mathcal{A}_{\epsilon,d}$ and suppose we have a corresponding
$\epsilon$, $h$, $\psi$, $\Delta$, and $Z_k$ as given in Thm. \ref{SSThm}.
Furthermore, assume that $g=f+e$, where $e$ is a bounded perturbation
such that $\norm{e}_{L^\infty} \leq C_\psi \epsilon$, where
$C_\psi^{-1}=\max(\|\psi\|_{L^{1}},\|\psi'\|_{L^{1}})$.
For each $k$ define the ``maximal frequency range'' $M_k \geq 1$ such that
$\phi'_k(t) \in [M^{-1},M]$ for all $t$.  A mid-range IF is
defined as having $M_k$ near $1$.

\textbf{1}. The Synchrosqueezing plot $|S_g^\delta|$ is
concentrated around the IF curves $\set{\phi_k'}$.
For sufficiently small $\epsilon$, the FM-demodulated frequency
estimate $\omega_g$ is accurate inside $Z_k$ where $W_g$ is sufficiently large
($|W_{g}(a,b)|>M_{k}^{1/2}\epsilon+{\epsilon^{1/3}}$):
\begin{center}
$|\omega_{g}(a,b)-\phi_{k}'(b)| \leq C_2 \epsilon^{1/3}$,
\end{center}
where $C_2 = O(M_k)$.
\noindent Outside the scale bands $\set{Z_{k}}$, $W_g$ is small:
\begin{center}
$|W_{g}(a,b)|\leq M_{k}^{1/2}\epsilon+{\epsilon^{1/3}}$.
\end{center}

\textbf{2}. Each component $f_k$ may be reconstructed with 
accuracy proportional to the noise magnitude and its maximal frequency
range by integrating $S_g^\delta$ over a neighborhood around $\phi_{k}'$.
Choose the wavelet threshold $M_k^{1/2}{\epsilon^{1/3}} + \epsilon$
and let $N'_k(b) = \set{\eta : |\eta-\phi'_k(b)| \leq C_2
 \epsilon^{1/3}}$, where (as before) $C_2 = O(M_k)$.
For sufficiently small $\epsilon$,
$$
%\mathcal{R}_{\psi}^{-1} \sum\limits_{l \in N'_k(b)} T_g(w_l,b)
% \approx f_k(b) + O(M_k {\epsilon^{1/3}}).
\left|\lim_{\delta\rightarrow0}\left(\mathcal{R}_{\psi}^{-1}
\int \limits_{N'_k(b)}
S_{g,M_{k}^{1/2}\epsilon+\epsilon^{1/3}}^{\delta,M_{k}}(b,\eta)d\eta\right)
 - f_k(b)\right|\leq C_{3}\epsilon^{1/3},
$$
where $C_3 = O(M_k)$.
\end{thm}

Thm. \ref{SSStableThm} has two important implications.  First,
components with mid-range IF tend to have the best estimates and lowest
reconstruction error under bounded noise.  Second,
to best identify signal component $f_k$ with IF $\phi'_k \in
[M^{-1},M]$, from a noisy signal, the
threshold $\gamma$ should be chosen proportional to $M^{1/2}
\epsilon$, where $\epsilon$ is an estimate of the noise magnitude.

\subsection{Stability under Spline Interpolation}
In many applications, samples of a signal $f \in \cA$ are
only given at irregular sample points $\set{t'_m}$, and these
are spline interpolated to a function $f_s$.
Thm. \ref{SSStableThm} bounds the error incurred due to this
preprocessing:

\begin{cor}
\label{cor:splinestable}
Let $\displaystyle D = \max_m |t'_{m+1}-t'_m|$ and let
$e = f_s-f$.  Then the error in the estimate of the $k$th IF of $T_{f_s}$
is $O(M_k D^{4/3})$, and the error in extracting $f_k$ is 
$O(M_k D^{4/3})$.
\end{cor}

\begin{proof}
This follows from Thm. \ref{SSStableThm} and the following standard
estimate on cubic spline approximations \cite[p. 97]{Stewart1998}:
\begin{center}
$
\left\Vert e \right\Vert_{L^{\infty}}
  \leq\frac{5}{384} D^4 \| f^{(4)} \|_{L^{\infty}}.
$
\end{center}
\end{proof}
Thus, we can Synchrosqueeze $f_s$ instead of $f$ and, as long as the
minimum sampling rate $D^{-1}$ is high enough, the results will match.
Furthermore, in practice errors are localized in time to areas of low
sampling rate, low component amplitude, and/or high component
frequency (see, e.g., \S\ref{sec:wmisc}).

\section{\label{sec:wmisc}Examples of Synchrosqueezing Properties}
We now provide numerical examples of several important properties of
Synchrosqueezing.  First, we compare Synchrosqueezing with two
common analysis transforms.
% Second, we show that Synchrosqueezing does
% not ``insert'' spurious spectral information when used with spline
% interpolation on subsampled data.

\subsection{Comparison of Synchrosqueezing to the CWT and STFT}
We compare Synchrosqueezing to the Wavelet
transform and the Short Time Fourier Transform (STFT)
\cite{Oppenheim1999}.  We show its superior precision, in both time and
frequency, at identifying components of sums of
quasi-harmonic signals.

\begin{figure}[ht]
  \centering
  \includegraphics[width=.9\columnwidth]{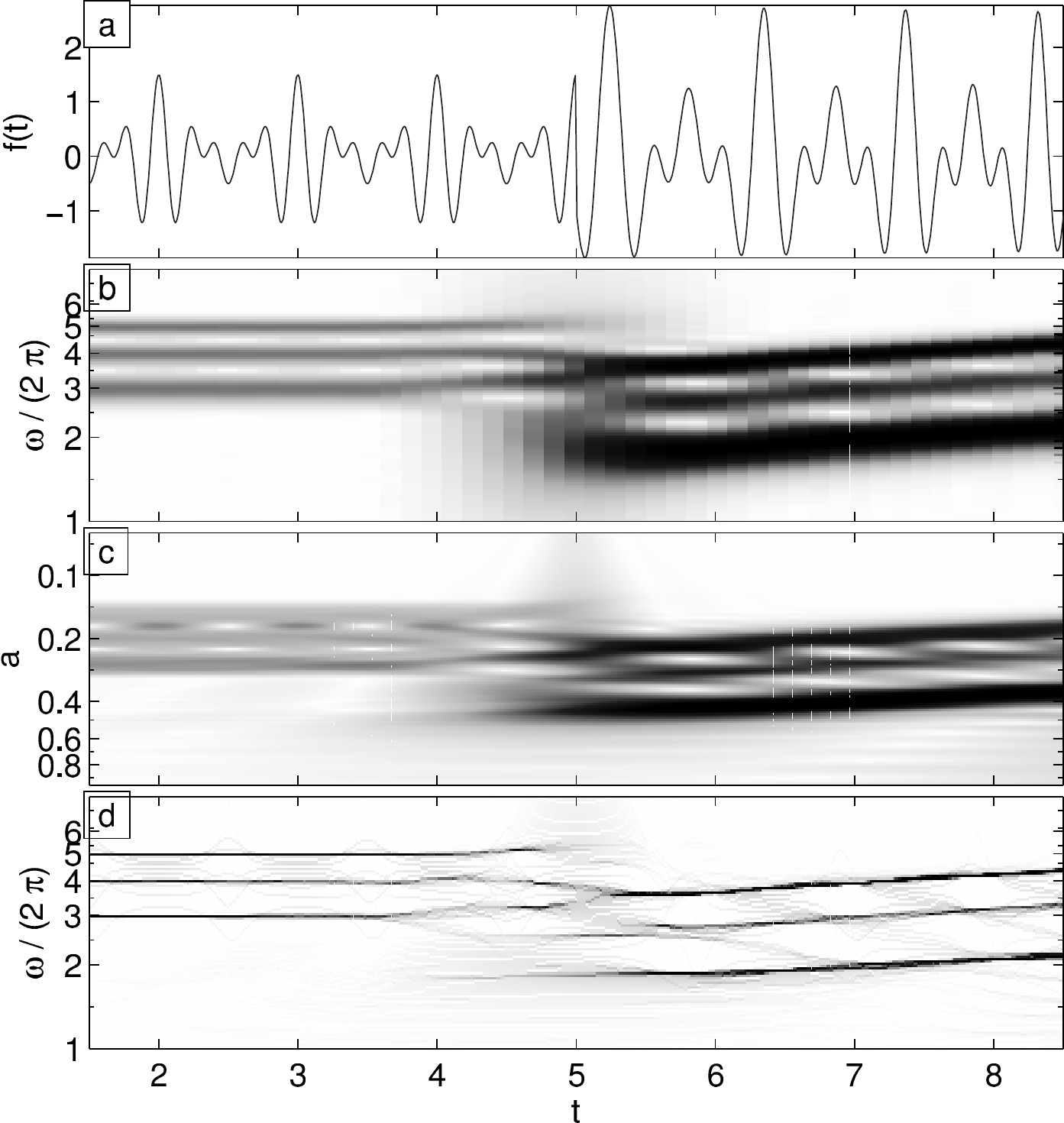}
  \caption[Comparison of Synchrosqueezing with Wavelet and the STFT.]%
   {\label{fig:cmpstftwave} 
    Comparison of Synchrosqueezing with Wavelet and the STFT.  (a)
    Synthetic signal $s(t)$.  (b) Short Time Fourier Transformed signal.
    (c) Wavelet transformed: $W_s(a,t)$.  (d) Synchrosqueezed: $T_s(\omega,t)$.  
  }
\end{figure}

In Fig.~\ref{fig:cmpstftwave} we focus on a signal $s(t)$ defined on
$t \in [0,10]$, that contains an abrupt transition at $t=5$, and
time-varying AM and FM modulation.  It is discretized to $n=1024$
points and is composed of the following components:
\begin{align*}
t < 5 :
 s_1(t) &= .5 \cos(2 \pi (3 t)), s_2(t) = .5 \cos(2 \pi (4t)), \\
 s_3(t) &= .5 \cos(2 \pi (5t)) \\
t \geq 5 :
 s_1(t) &= \cos(2 \pi (.5 t^{1.5})), \\
 s_2(t) &= \exp(-t/20) \cos(2 \pi (.75 t^{1.5})), \\
 s_3(t) &= \cos(2 \pi t^{1.5}).
\end{align*}
We used the shifted bump wavelet (see \S\ref{sec:wcmp}) and $n_v = 32$
for both the Wavelet and  Synchrosqueezing transforms, and a Hamming
window with length 300 and overlap of length 285 for the STFT.  These
STFT parameters focused on optimal precision in frequency, but not in
time \cite{Oppenheim1999}.
For $t<5$, the harmonic components of $s(t)$ are
clearly identified in the Synchrosqueezing plot $T_s$
(Fig.~\ref{fig:cmpstftwave}(d)) and the STFT plot
(Fig.~\ref{fig:cmpstftwave}(b)), though the
frequency estimate is more precise in $T_s$.  The higher frequency
components are better estimated up to the singularity at $t=5$ in
$T_s$, but in the STFT there is mixing at the singularity.  For $t \geq
5$, the frequency components are more clearly visible in $T_s$ due
to the smearing of lower frequencies in the STFT.
The temporal resolution in the STFT is also significantly lower than
for Synchrosqueezing due to the selected parameters.  A
shorter window in the STFT will provide higher temporal
resolution, but lower frequency resolution and more smearing between
the three components.

\subsection{\label{sec:ssnonunif}Nonuniform Sampling and Splines}

We now demonstrate how Synchrosqueezing and extraction work for a
more complicated signal that contains multiple time-varying
amplitude and frequency components, and has been irregularly
subsampled.  Let
\begin{align}
\label{eq:fsum}
f(t) &= \cos(4 \pi t)\\
     &+ (1+0.2\cos(2.5t))\cos(2\pi(5t+2t^{1.2})) \nonumber \\
      &+ e^{-0.2t}\cos(2\pi(3t+0.2\cos(t))), \nonumber
\end{align}
and let the sampling times be perturbations of uniformly spaced times
having the form $t'_{m}=\Delta t_1 m + \Delta t_2 u_{m}$, where
$\Delta t_2 < \Delta t_1$ and $\{u_{m}\}$ is sampled from
the uniform distribution on $[0,1]$. Here we fix $\Delta
t_1=11/180$ and $\Delta t_2=11/600$.  This leads to $\approx 160$
samples on the interval $t \in [0,10]$.  To correct for nonuniform
sampling, we fit a spline through $(t'_m,f(t'_m))$ to get the
function $f_s(t)$ and discretize on the finer
grid $t_m = m \Delta t$, with $\Delta t=10/1024$ and
$m=0,\ldots,1023$.  The resulting vector, $\tf_s$,
is a discretization of the original signal plus a spline
error term.  Fig.~\ref{fig:nonunif}(a) shows
$\tf_s$ for $t \in [2,8]$.

\begin{figure}[hb]
  \centering
  \includegraphics[width=.9\columnwidth]{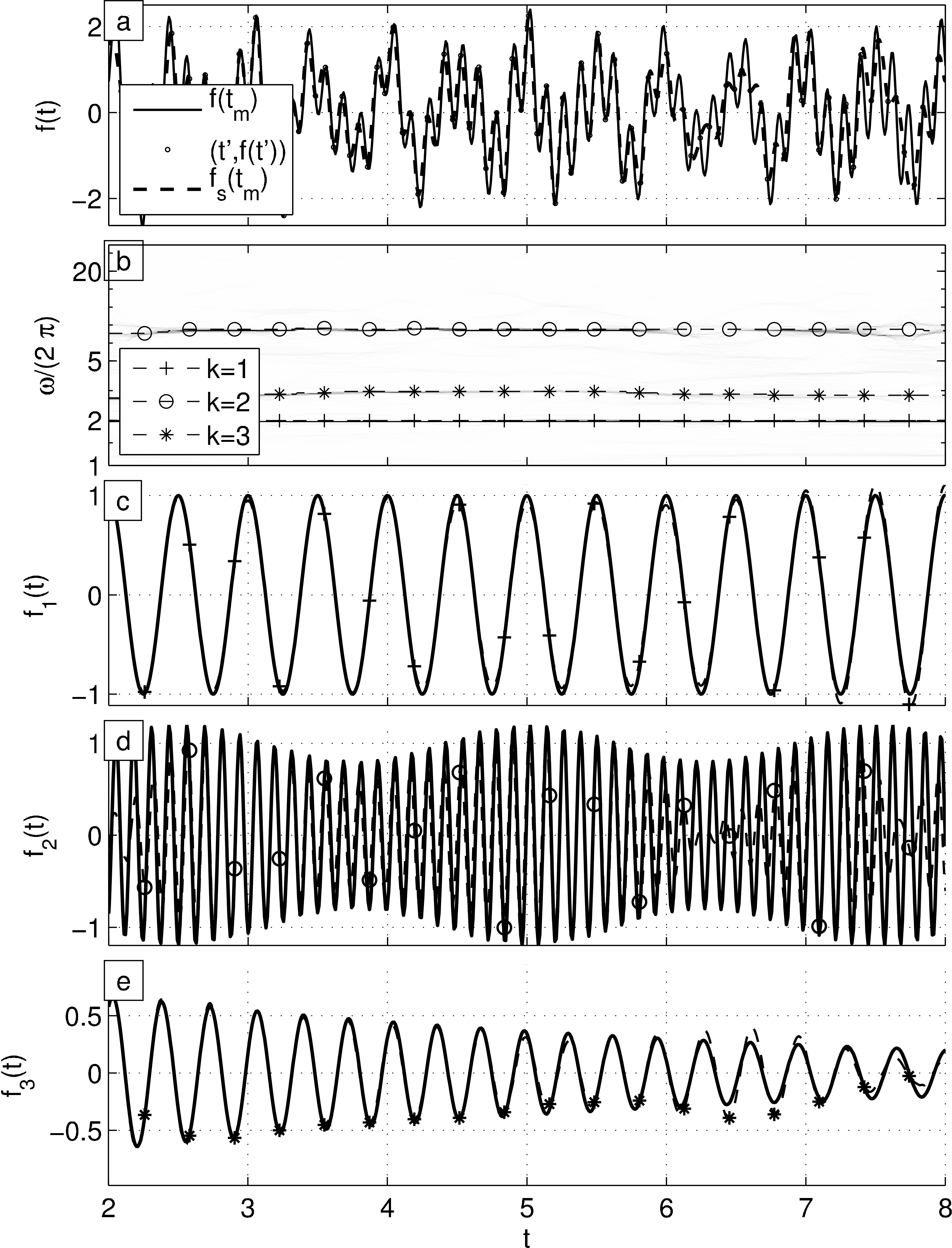}
   \caption[Nonuniform Sampling, Synchrosqueezing, and Component Extraction from $f$]
    {\label{fig:nonunif}
     (a) Nonuniform Sampling of $f$, with spline interpolation estimate $\tf_s$.
     (b) Synchrosqueezing $\wt{T}_{\tf_s}$ of $\tf_s$; components
     extracted via \eqref{eq:Cextract}.
     (c-e) Extracted components $\tf^*_k$ compared to originals
     $\tf_k$, $k=1,2,3$ (respectively).
   }
\end{figure}

Figs.~\ref{fig:nonunif}(b-e) show the results of
Synchrosqueezing and component extraction of $\tf_s$, for $t \in
[2,8]$.  All three components are well separated in
the TF domain.  The second component is the most difficult to
reconstruct, as it contains the highest frequency information.  Due to stability
(Thm. \ref{SSStableThm} and Cor. \ref{cor:splinestable}), extraction
of components with mid-range IFs is more stable to the error $e(t)$.
Fig.~\ref{fig:nonunif} shows that reconstruction errors are
time localized to the locations of errors in $\tf_s$.

\subsection{White Noise and Reconstruction}
We take the signal $f(t)$ of \eqref{eq:fsum}, now regularly sampled on
the fine grid with $\Delta t = 10/1024$ ($n=1024$ samples) as before,
and corrupt it with white Gaussian noise having a standard deviation
of $\sigma_N = 1.33$.  This signal, $\tf_N$ (see
Fig.~\ref{fig:noise}(a)) has an SNR of $-1$ dB.

\begin{figure}[ht]
  \centering
  \includegraphics[width=.9\columnwidth]{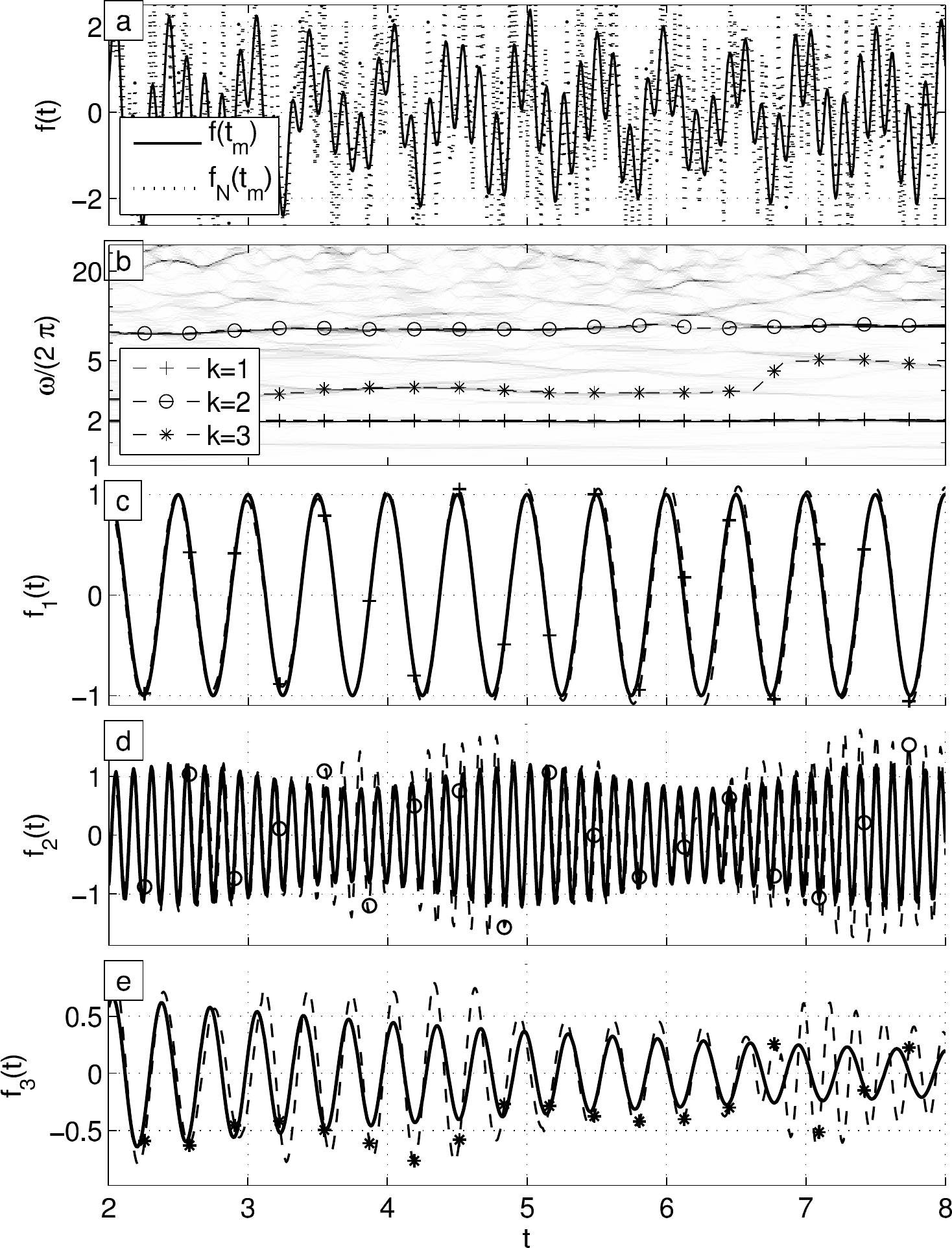}
   \caption[Synchrosqueezing, and Component Extraction from $\tf_N$]%
    {\label{fig:noise}
     (a) Uniform sampling of $f$, $\tf$, corrupted by normal white noise
     with standard deviation $\sigma_N = 1.33$ (SNR is $-1$ dB): $\tf_N$.
     (b) Synchrosqueezing $\wt{T}_{\tf_N}$ of $\tf_N$; components
     extracted via \eqref{eq:Cextract}.
     (c-e) Extracted components $\tf^*_k$ compared to originals
     $\tf_k$, $k=1,2,3$ (respectively).
   }
\end{figure}

Figs.~\ref{fig:noise}(b-e) show the results of
Synchrosqueezing and component extraction of $\tf_N$, for $t \in
[2,8]$.  As seen in
Fig.~\ref{fig:noise}(b), most of the additional energy, caused by
the white noise, appears in the higher frequencies.
Again, all three components are well separated in
the TF domain, though now the third, lower-amplitude, component
experiences a ``split'' at $t \approx 6.5$.
Reconstruction of signal components is less reliable in
locations of high frequencies and low magnitudes (note the axis in
Fig.~\ref{fig:noise}(e) is half that of the others).
This again numerically confirms Thm. \ref{SSStableThm}:
components with mid-range IFs and higher amplitudes are more stable to
the noise.

\section{\label{sec:wcmp}Invariance to the underlying transform}

As mentioned in \S\ref{sec:analysis} and in \cite{Daubechies2010},
Synchrosqueezing is invariant to the underlying choice of
transform.  The only differences one sees in practice are due to two factors: the
time compactness of the underlying analysis atom (e.g. mother
wavelet), and the frequency compactness of this atom.  That is,
$\abs{\psi(t)}$ should fall off quickly away from zero,
$\wh{\psi}(\xi)$ is ideally zero for $\xi<0$, and $\Delta$ (of
Thm. \ref{SSThm}) is small.

Fig.~\ref{fig:wcmp} shows the effect of Synchrosqueezing the
discretized spline signal $\tf_s$ of the synthetic
nonuniform sampling example in \S\ref{sec:ssnonunif}, using three
different complex CWT mother wavelets.  These wavelets are:
\begin{align*}
&\textbf{a. Morlet (shifted Gaussian)} \\
&\qquad \wh{\psi}_a(\xi) \propto \exp(-(\mu-\xi)^2/2),
  \quad \xi \in \bbR\\
&\textbf{b. Complex Mexican Hat} \\
&\qquad \wh{\psi}_b(\xi) \propto \xi^2 \exp(-\sigma^2 \xi^2/2),
  \quad \xi > 0\\
%&\textbf{c. Shifted Complex Hermitian Hat} \\
%&\qquad \wh{\psi}_d(\xi) \propto (\xi-\mu) ( \xi-\mu+1 ) \exp(-(\xi-\mu)^2/2),
% \quad \xi > \mu\\
&\textbf{c. Shifted Bump} \\
&\qquad \wh{\psi}_d(\xi) \propto 
  \exp\left(- (1- ((\xi-\mu)/\sigma )^2 )^{-1} \right), \\
&\qquad \xi \in [\sigma(\mu-1), \sigma(\mu+1)]
\end{align*}
where for $\psi_a$ we use $\mu=2\pi$, for $\psi_b$ we use $\sigma=1$,
% $\psi_c$ we use $\mu=5$,  %% Shifted Complex Hermitian Hat
for and for $\psi_c$ we use $\mu=5$ and $\sigma=1$.

\begin{figure}[h]
  \centering
  \includegraphics[width=.9\columnwidth]{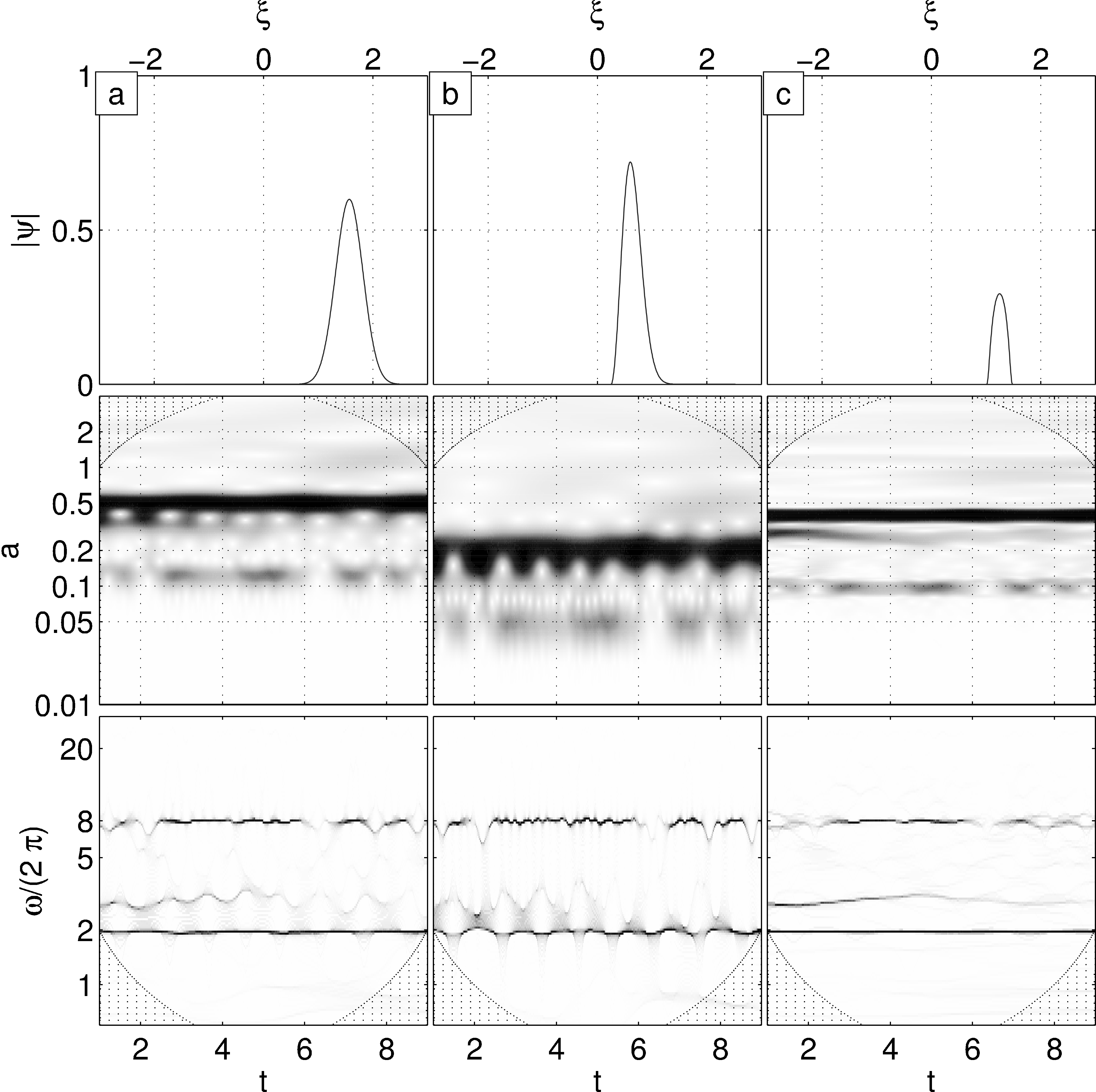}
  \caption[Wavelet and Synchrosqueezing transforms of $\tf_s$.]
   {\label{fig:wcmp} 
    Wavelet and Synchrosqueezing transforms of $\tf_s$.
    Columns (a-c) represent choice of mother wavelet $\psi_a \ldots \psi_c$.
    Top row: $|2 \wh{\psi}(4\xi)|$.  Center row: $|W_{f_s}|$.  Bottom row:
    $|T_{f_s}|$.
    }
\end{figure}

The Wavelet representations of $\tf_s$ differ due to
differing mother wavelets, but the Synchrosqueezing representation is
mostly invariant to these differences.  As expected from
Thm. \ref{SSThm}, more accurate representations are given
by wavelets having compact frequency support on $\xi$
away from $0$.

\section{\label{sec:ssfuture}Conclusions and Future Work}
Synchrosqueezing can be used to extract the instantaneous spectra of,
and filter, a wide variety of signals that include complex simulation data
(e.g. dynamical models), and physical signals
(e.g. climate proxies).  A careful implementation runs in
$O(n_v n \log^2 n)$ time, and is stable (in theory and in practice)
to errors in these types of signals.

%We have shown how it can be used
%to help address an important problem in climatology.
Areas in which Synchrosqueezing has shown itself to be an important
analysis tool include ECG analysis (respiration and T-end detection),
meteorology and oceanography (large-scale teleconnection and
ocean-atmosphere interaction), and climatology.  Some of these
examples are described in the next chapter.

Additional future work includes theoretical analysis of the
Synchrosqueezing transform, including the development of
Synchrosqueezing algorithms that directly support nonuniform sampling,
the analysis of Synchrosqueezing when the signal is perturbed by
Gaussian, as opposed to bounded, noise, and extensions to higher
dimensional data.

\chapter{\label{ch:ssapp}Synchrosqueezing: Applications%
\chattr{Section \ref{sec:resp} of this chapter are % and \ref{sec:Tend} 
  based on work in collaboration with Hau-Tieng~Wu and Gaurav~Thakur.
  Section \ref{sec:SSpaleo} is based on work in collaboration with
  Neven~S.~Fu\v{c}kar, International Pacific Research Center,
  University of Hawaii, as submitted in \cite{Brevdo2011b}.}}

\section{Introduction}

The theoretical results of Chapter \ref{ch:ss} provide important
guarantees and guidelines for the use of Synchrosqueezing in data
analysis techniques.  Here, we focus on two specific applications in
which Synchrosqueezing, in combination with preprocessing methods such
as spline interpolation, provides powerful new analysis tools.

This chapter is broken down into two sections.  First, we use
Synchrosqueezing and spline interpolation to estimate patients'
respiration from the R-peaks (beats) in their Electrocardiogram (ECG)
signals.  This extends earlier work on the ECG-Derived Respiration problem.

Second, we visit open problems in paleoclimate studies of the last
2.5\,Myr, where Synchrosqueezing provides improved insights.
We compare a calculated solar flux index with a deposited $\delta^{18}
O$ paleoclimate proxy over this period.  Synchrosqueezing cleanly
delineates the orbital cycles of the solar radiation, provides an
interpretable representation of the orbital signals of
$\delta^{18}O$, and improves our understanding of the effect that the
solar flux distribution has had on the global climate.
Compared to previous analyses of these data, the Synchrosqueezing
representation provides more robust and precise estimates in the
time-frequency plane.

\section{\label{sec:resp}ECG Analysis: Respiration Estimation}

We first demonstrate how Synchrosqueezing can be
combined with nonuniform subsampling of a single lead
ECG recording to estimate the instantaneous
frequency of, and in some cases extract, a patients's respiration signal.
We verify the accuracy of our estimates by comparing them with the
instantaneous frequency (IF) extracted from a simultaneously recorded
respiration signal.

The respiratory signal is usually recorded mechanically via, e.g.,
spirometry or plethysmography. There are two common disadvantages to
these techniques. First, they require the use of complicated devices
that might interfere with natural breathing. Second,
they are not appropriate in many situations, such as
ambulatory monitoring. However, having the respiratory signal is often
important, e.g. for the diagnosis of obstructive sleep apnea. Thus,
finding a convenient way to directly
record, or indirectly estimate, information about the respiration
signal is important from a clinical perspective. 

ECG is a cheap, non-invasive, and ubiquitous technique, in which
voltage differences are passively measured between electrodes (leads)
connected to a patient's body (usually the chest and arms).  The
change of the thoracic electrical impedance caused by inhalation and
exhalation, and thus physiological respiratory information, is
reflected in the ECG amplitude.  The
respiration-induced distortion of ECG was first studied in
\cite{einthoven} and \cite{flaherty}. A well-known ECG-Derived
Respiration (EDR) technique \cite{MMZM85}
experimentally showed that ``electrical rotation'' during the
respiratory cycle is the main contributor to the distortion of ECG
amplitude, and that the contribution of thoracic impedance variations
is relatively minor.  These prior work confirm that analyzing ECG may
enable us to estimate respiration. More details about EDR are available
in \cite{ecgbook}.

Relying on the coupling between physiological respiration and R-peak
amplitudes (the tall spikes in Fig.~\ref{fig:ecgFs}(a)), we
use the R-peaks as a proxy for the respiration signal.  More
specifically, we hypothesize that the R peaks, taken as samples of the
envelope of the ECG signal $f_E(t)$, have the same IF profile as the
true respiration signal $f_R(t)$.  By sampling $f_E(t)$ at the R
peaks and performing spline interpolation on the resulting samples, we
hope to see a time shifted, amplitude scaled, version of $f_R(t)$
near the respiratory frequency (0.25Hz).

\begin{figure}[h]
\centering
\includegraphics[width=.9\columnwidth]{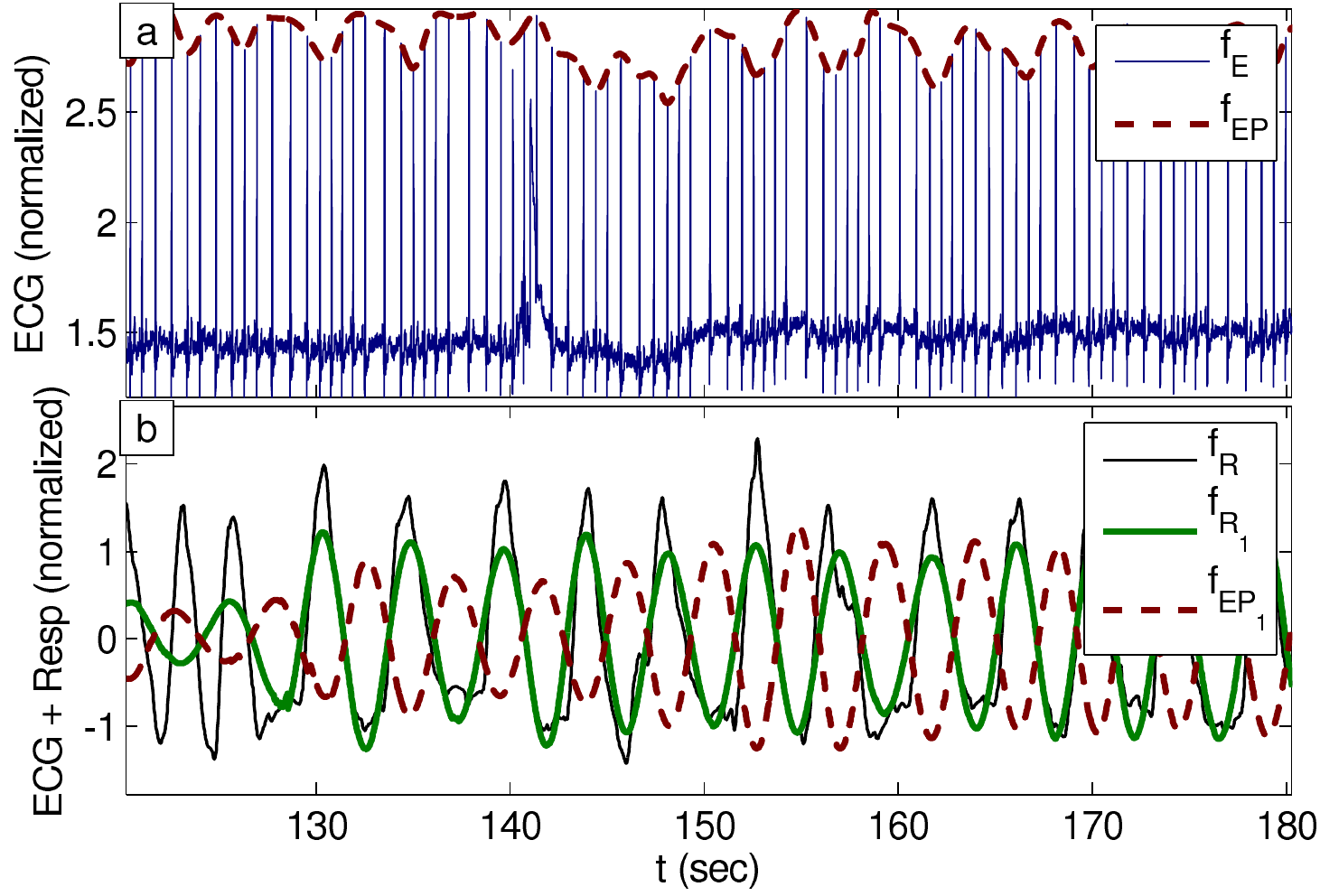}
\caption[Original and Filtered ECG signal and spline R-peak envelope]%
{\label{fig:ecgFs}
(a) ECG signal and spline R-peak envelope; $t \in [120,180]$.
(b) Respiration, full (black), filtered (green), and estimated from ECG R-peak envelope (red).}
\end{figure}

\begin{figure}[h]
  \centering
\includegraphics[width=.9\columnwidth]{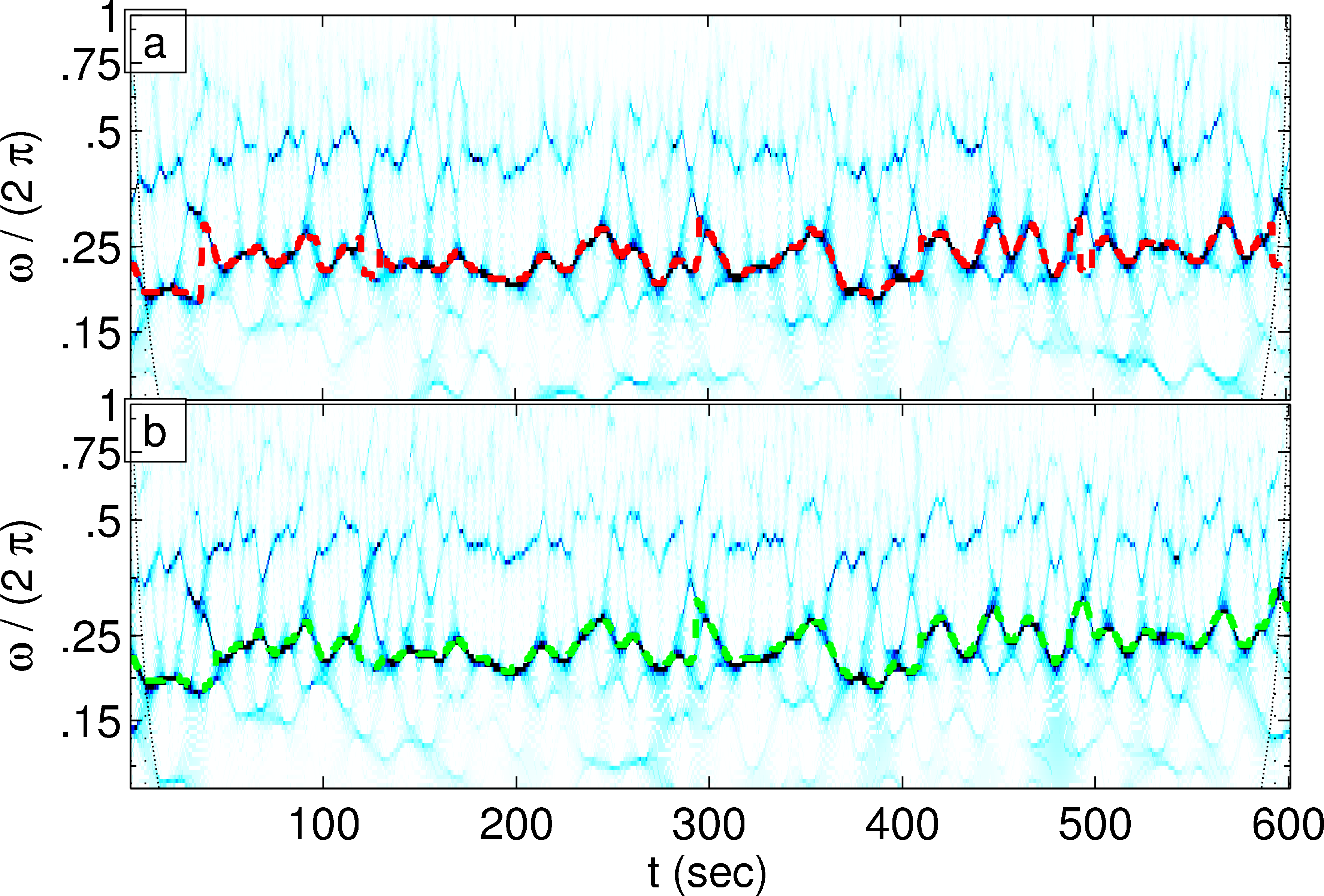}
\caption[Synchrosqueezing of ECG spline R-peak envelope $\tf_{EP}$]
{\label{fig:ecgTs} (a) Synchrosqueezing of ECG spline R-peak
  envelope $\tf_{EP}$ with main extracted curve representing
  $\tf_{EP_1}$; $t \in [0,600]$ and $\omega \in [.2\pi, 2\pi]$. 
  (b) Synchrosqueezing of Respiration $\tf_R$ with main extracted
  curve representing $\tf_{R_1}$.}
\end{figure}

In Fig.~\ref{fig:ecgFs}, we show the lead II ECG signal and the
true respiration signal (via respiration belt) of a healthy $30$ year
old male, recorded over a $10$ minute interval ($t \in [0,600]$ sec).
The sampling rates of the ECG and respiration signals are respectively
400Hz and 50Hz within this interval. There are $846$ R peaks
appearing at nonuniform times $t'_{m}\in[0,600]$,
$m = 1,\ldots,846$. We run cubic spline interpolation on the R-peaks
$\{(t'_m,f_E(t'_m))\}$ to get $f_{EP}(t)$, which we discretize at 50Hz
(with $n=30000$) to get $\tf_{EP}$. Fig.~\ref{fig:ecgTs} shows the
result of running Synchrosqueezing on $\tf_R$ and $\tf_{EP}$.
The computed IF, $\tf_{EP_1}$, turns out to be a good (shifted and
scaled) approximation to the IF of the true respiration, $\tf_{R_1}$.
It can be seen, from Figs. \ref{fig:ecgFs} and
\ref{fig:ecgTs}, that the spacing of respiration
cycles in $f_R(t)$ is reflected by the main IF of $\tf_{EP}$: closer
spacing corresponds to higher IF values, and wider spacing to
lower values.

These results were confirmed by tests on several subjects.
Thanks to the stability of Synchrosqueezing (Thm. \ref{SSStableThm} and
Cor. \ref{cor:splinestable}), this algorithm has the potential
for broader clinical usage.

\subsection{Notes on Data Collection and Analysis Parameters}

The ECG signal $\tf_E$ was collected at 400Hz via a MSI MyECG E3-80.
The respiration signal $\tf_R$ was collected at 50Hz via a
respiration belt and PASCO SW750.  The ECG signal was filtered to
remove the worst nonstationary noise by thresholding signal values
below the $0.01\%$ and above the $99.99\%$ quantiles, to these
quantile values. The ECG R-peaks were then extracted from $\tf_E$
by first running the physionet
\texttt{ecpguwave}~%
\footnote{{ecgpuwave may be found at: \url{http://www.physionet.org/physiotools/ecgpuwave/}}}
program, followed by
a ``maximum'' peak search within a 0.2 sec window of
each of the ecgpuwave-estimated R-peaks.  

For Synchrosqueezing, the
parameters $\gamma = 10^{-8}$ and $\lambda = 10^5$ were used for
thresholding $W_f$ and extracting contours from both the R-peak spline
and respiratory signals.

%
%\section{\label{sec:Tend}ECG Analysis: T-end Detection}
%

\section{\label{sec:SSpaleo}Paleoclimatology: Aspects of the mid-Pleistocene transition}
Next, we apply Synchrosqueezing to analyze the
characteristics of a calculated index of the incoming solar radiation
(insolation) and to
measurements of repeated transitions between glacial (cold) and
interglacial (warm) climates during the Pleistocene epoch:
$\approx$\,1.8\,Myr to 12\,kyr before the present.

The Earth's climate is a complex, multi-component nonlinear, system
with significant stochastic elements \cite{Pierrehumbert2010}. The key
external forcing field is the insolation at the top of the
atmosphere (TOA). Local insolation has predominately harmonic
characteristics in time (diurnal cycle, annual cycle and
Milankovi\'{c} orbital cycles). However, response of planetary
climate, which varies at all time scales \cite{Huybers2006}, also
depends on random perturbations (e.g., volcanism), solid boundary
conditions (e.g., plate tectonics and global ice distribution),
internal variability and feedbacks (e.g., global carbon
cycle). Various paleoclimate records or proxies provide us with
information about past climates beyond observational records. Proxies
are biogeochemical tracers, i.e., molecular or isotopic properties,
imprinted into various types of deposits (e.g., deep-sea sediment),
and they indirectly represent physical conditions (e.g., temperature)
at the time of deposition. We focus on climate variability during the
last 2.5\,Myr (that also includes the late Pliocene) as recorded by
${\delta}^{18}O$ in foraminiferal shells at the bottom of the ocean
(benthic forams). Benthic $\delta^{18}O$ is the deviation of the ratio
of $^{18}O$ to $^{16}O$ in sea water with respect to the
present-day standard, as imprinted in benthic forams during their
growth.  It increases with glaciation during cold climates because $^{16}O$
evaporates more readily and accumulates in ice sheets. Thus, benthic
${\delta}^{18}O$ can be interpreted as a proxy for either
high-latitude temperature or global ice volume.

\begin{figure}[h!]
  \centering
   \includegraphics[width=.9\columnwidth]{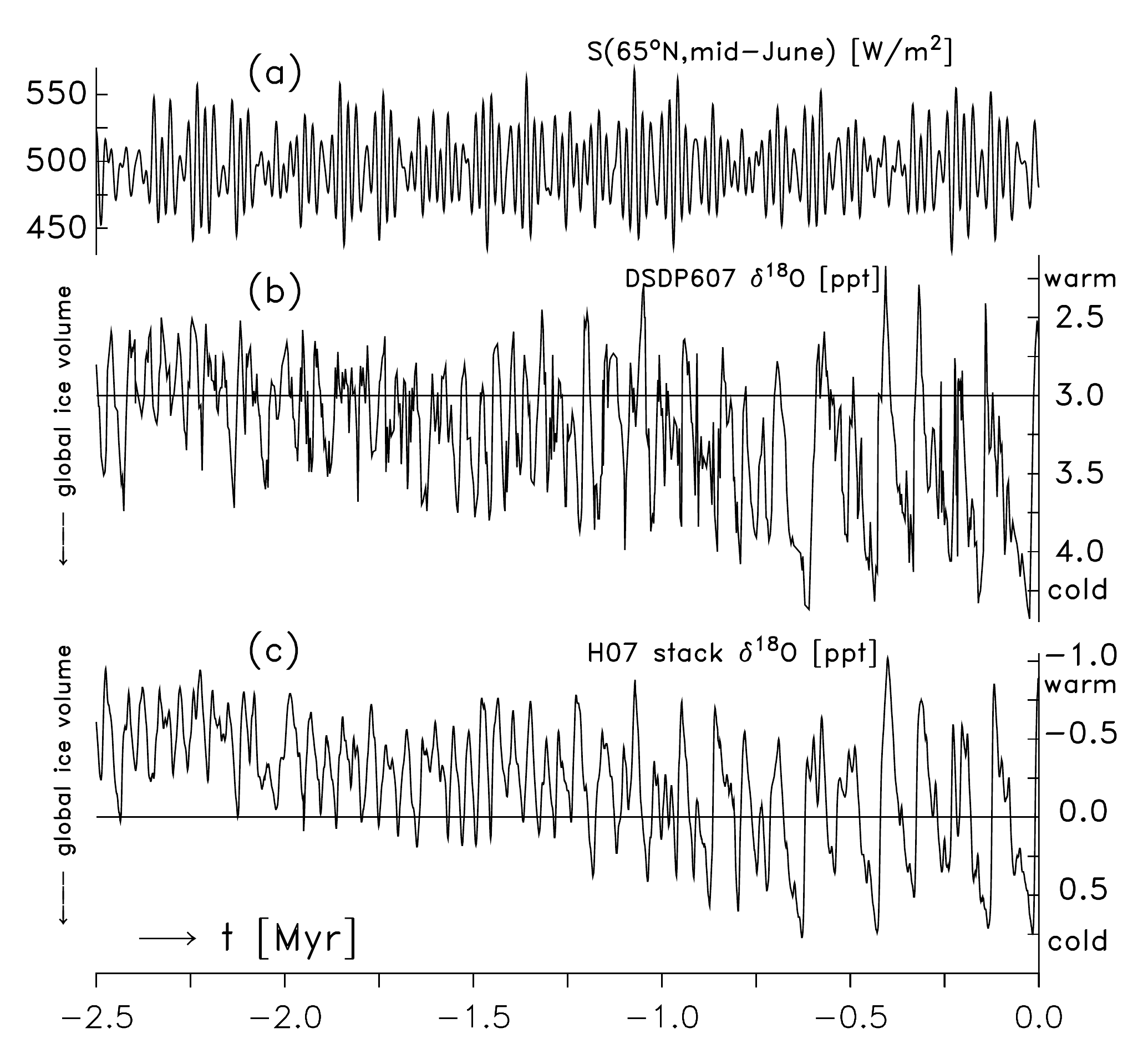}
   \caption[Calculated mid-June insolation flux and climate response as recorded by benthic ${\delta}^{18}O$]%
    {\label{fig:paleo_ts} (a) Calculated mid-June insolation
     flux at $65^{o}N$: $f_{SF}$.  Climate response as recorded by
     benthic ${\delta}^{18}O$ (b) in DSDP607 core: $f_{CR1}$, and (c) in
     H07 stack: $f_{CR2}$.}
\end{figure}

We first examine a calculated element of the TOA solar forcing
field. Fig.~\ref{fig:paleo_ts}(a) shows $f_{SF}$, the mid-June
insolation at $65^{o}N$ at 1\,kyr intervals \cite{Berger1992}. This
TOA forcing index does not encompass the full complexity of solar
radiation structure and variability, but is commonly used to gain
insight into the timing of advances and retreats of ice sheets in the
Northern Hemisphere in this period (e.g., \cite{Hays1976}). % ,Imbrie1992}).
The Wavelet and Synchrosqueezing
decompositions in Fig.~\ref{fig:paleo_wavelet}(a) and
Fig.~\ref{fig:paleo_synsq}(a), respectively, show the key harmonic
components of $f_{SF}$. The application of a shifted bump mother
wavelet (see \S\ref{sec:wcmp}) yields an upward
shift of the spectral features along the scale axis in each of the
the representations in Fig.~\ref{fig:paleo_wavelet}.  Therefore
the scale $a$ should not be used to directly infer periodicities.
In contrast, the Synchrosqueezing spectrums in Fig.~\ref{fig:paleo_synsq}
explicitly present time-frequency (or here specifically
time-periodicity) decompositions with a sharper structure, and
are not affected by the scale shift inherent in the choice of mother
wavelet.

Fig.~\ref{fig:paleo_synsq}(a) clearly shows the presence of
strong precession cycles (at periodicities $\tau$=19\,kyr and
23\,kyr), obliquity cycles (primary at 41\,kyr and secondary at
54\,kyr), and very weak eccentricity cycles (primary periodicities at
95\,kyr and 124\,kyr, and secondary at 400\,kyr). This is in contrast
with Fig.~\ref{fig:paleo_wavelet}(a), which contains blurred and shifted
spectral structures only qualitatively similar to
Fig.~\ref{fig:paleo_synsq}(a).

We next analyze the climate response during the last 2.5\,Myr as deposited
in benthic ${\delta}^{18}O$ in long sediments cores.
(in which deeper layers contain forams settled further back in time).
Fig.~\ref{fig:paleo_ts}(b) shows
$f_{CR1}$: benthic ${\delta}^{18}O$, sampled at irregular time intervals
from a single core, DSDP Site 607, in the North Atlantic
\cite{Ruddiman1989}.  This signal was spline interpolated to 1\,kyr
intervals prior to the spectral analyses.  Fig.~\ref{fig:paleo_ts}(c)
shows $f_{CR2}$: the benthic ${\delta}^{18}O$ stack (H07) calculated
at 1\,kyr intervals from fourteen cores (most of them from the Northern
Hemisphere, including DSDP607) using the extended depth-derived age
model \cite{Huybers2007}.
%The ${\delta}^{18}O$ of the cores included in the
%H07 stack vary over different ranges primarily due to different
%ambient temperatures at different depth of sea floor. Therefore, 
Prior to combining the cores in the H07 stack, the record mean % these
between 0.7\,Myr ago and the present was subtracted from each
${\delta}^{18}O$ record; this is the cause of the differing vertical
ranges in Figs. \ref{fig:paleo_ts}(b-c). Noise due
to local climate characteristics and measurement errors of
each core is reduced when we shift the spectral analysis from DSDP607
to the stack; and this is particularly visible in the finer
scales and higher frequencies.
\begin{figure}[h!]
  \centering
   \includegraphics[width=.9\columnwidth]{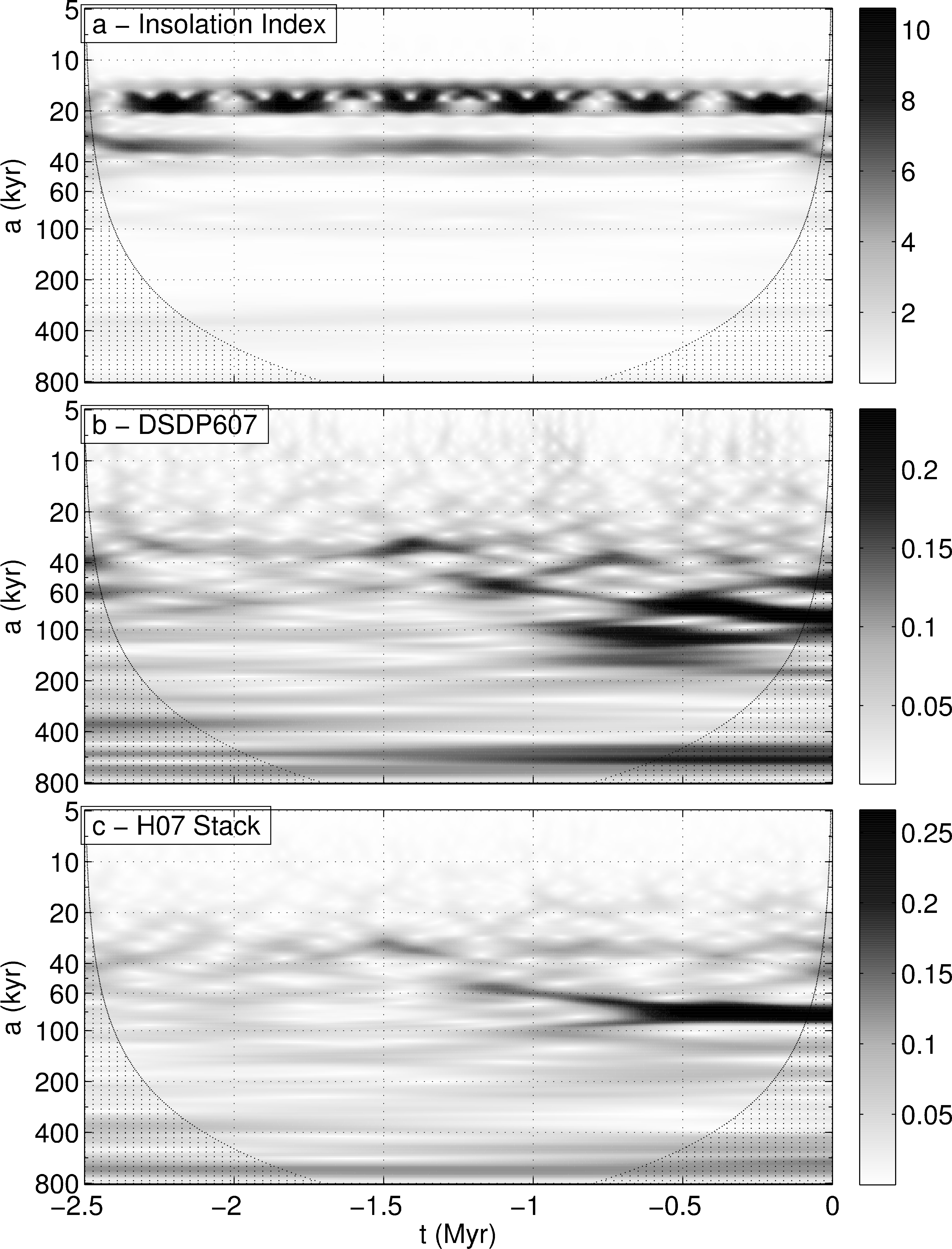}
   \caption[Wavelet magnitude time evolution of insolation index $f_{SF}$ and climate response]%
    {\label{fig:paleo_wavelet} Wavelet magnitude time
     evolution of (a) insolation index $f_{SF}$, and climate
     response in benthic ${\delta}^{18}O$ of (b) DSDP607 core,
     $f_{CR1}$, and (c) H07 stack, $f_{CR2}$.  }
\end{figure}
\begin{figure}[h!]
  \centering
   \includegraphics[width=.9\columnwidth]{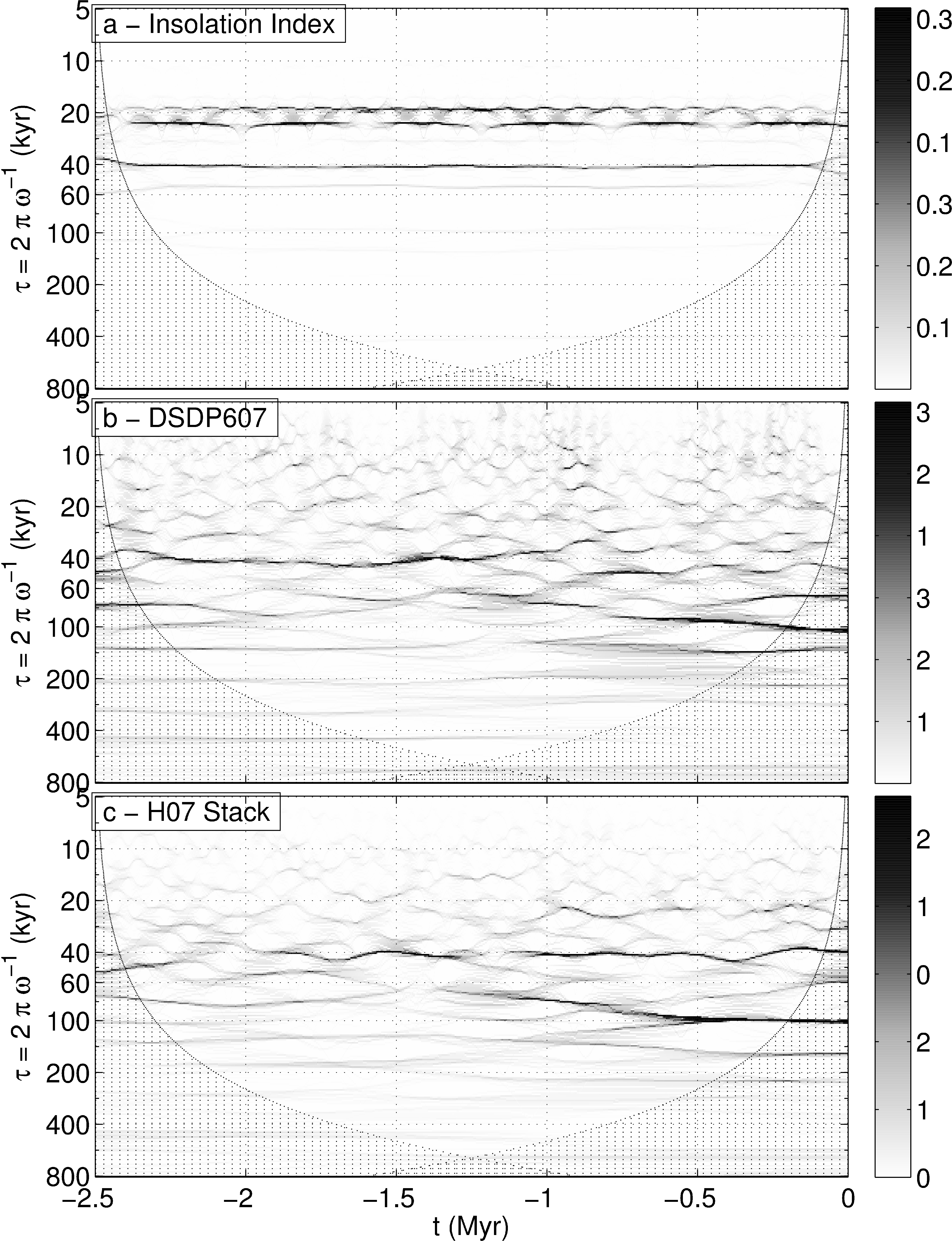}
   \caption[Synchrosqueezing spectral magnitude of insolation index $f_{SF}$ and climate response]%
    {\label{fig:paleo_synsq} Explicit time-periodicity
     decomposition of Synchrosqueezing spectral magnitude of (a) solar
     forcing index $f_{SF}$, and climate response in benthic
     ${\delta}^{18}O$ of (b) DSDP607 core, $f_{CR1}$, and (c) H07
     stack, $f_{CR2}$.  }
\end{figure}
The Synchrosqueezing decomposition in
Fig.~\ref{fig:paleo_synsq}(c) is a more precise time-frequency
representation of the stack than a careful STFT analysis
\cite[Fig.~4]{Huybers2007}.  In addition, it shows far less
stochasticity above the obliquity band as compared to
Fig.~\ref{fig:paleo_synsq}(b), enabling the 23\,kyr precession cycle
to become mostly coherent over the last 1\,Myr. Thanks to the
stability of Synchrosqueezing, the spectral differences below the
obliquity band are less pronounced between
Fig.~\ref{fig:paleo_synsq}(b) and Fig.~\ref{fig:paleo_synsq}(c).
Overall, the stack reveals sharper 
time-periodicity evolution of the climate system than
DSDP607 or any other single core possibly could.  The Wavelet
representations in Figs. \ref{fig:paleo_wavelet}(b-c)
also show this suppression of noise in
the stack (in more diffuse and scale shifted patterns).
Figs. \ref{fig:paleo_spectrum}(a) through \ref{fig:paleo_spectrum}(c)
show that the time average of Synchrosqueezing magnitudes (normalized
by $1/R_\psi$) is directly comparable with the Fourier spectrum, but
delineates the harmonic components much more clearly (not shown).

During the last 2.5\,Myr, the Earth experienced a gradual decrease in
global background temperature and $\text{CO}_2$ concentration, and an
increase in mean global ice volume accompanied with
glacial-interglacial oscillations that have intensified towards the present
(this is evident in Fig.~\ref{fig:paleo_ts}(b) and
\ref{fig:paleo_ts}(c)). The mid-Pleistocene transition, occurring
gradually or abruptly sometimes between 1.2\,Myr and 0.6\,Myr ago, was
the shift from 41\,kyr-dominated glacial cycles to 100\,kyr-dominated
glacial cycles recorded in deep-sea proxies (e.g., \cite{Ruddiman1986,
Clark2006, Raymo2008}). The origin of this strong 100\,kyr cycle
in the late-Pleistocene climate and the prior incoherency of the precession
band are still unresolved questions. Both types of spectral analyses of
selected ${\delta}^{18}O$ records indicate that the climate system
does not respond linearly to external periodic
forcing.

Synchrosqueezing enables the detailed time-frequency
decomposition of a noisy, nonstationary, climate time series due to
stability (Thm. \ref{SSStableThm}) and more precisely reveals key
modulated signals that rise above the stochastic background. The gain
(the ratio of the climate response amplitude to insolation forcing
amplitude) at a given frequency or period, is not constant.
%Except within the
%obliquity band, there is no significant correlation between the magnitudes
%of forcing and response in precession and eccentricity bands. 
The response to the 41\,kyr obliquity cycle is present almost
throughout the entire Pleistocene in
Fig.~\ref{fig:paleo_synsq}(c). The temporary incoherency of the
41\,kyr component starting about 1.25\,Myr ago 
roughly coincides with the initiation of a lower frequency signal
($\approx$\,70\,kyr) that evolves into a strong 100\,kyr component in
the late Pleistocene (about 0.6\,Myr ago). 
Inversion (e.g., spectral integration) of the Synchrosqueezing
decomposition of $f_{SF}$ and $f_{CR2}$ across the key
orbital frequency bands in Fig.~\ref{fig:paleo_filtered} again emphasize
the nonlinear relation between insolation and climate evolution.
Specifically, in Fig.~\ref{fig:paleo_filtered}(a) the amplitude
of the filtered precession signal of $f_{CR2}$ abruptly rises
1\,Myr ago, while in Fig.~\ref{fig:paleo_filtered}(c) the amplitude
of the eccentricity signal shows a gradual increase.

\begin{figure}[h!]
  \centering
   \includegraphics[width=.9\columnwidth]{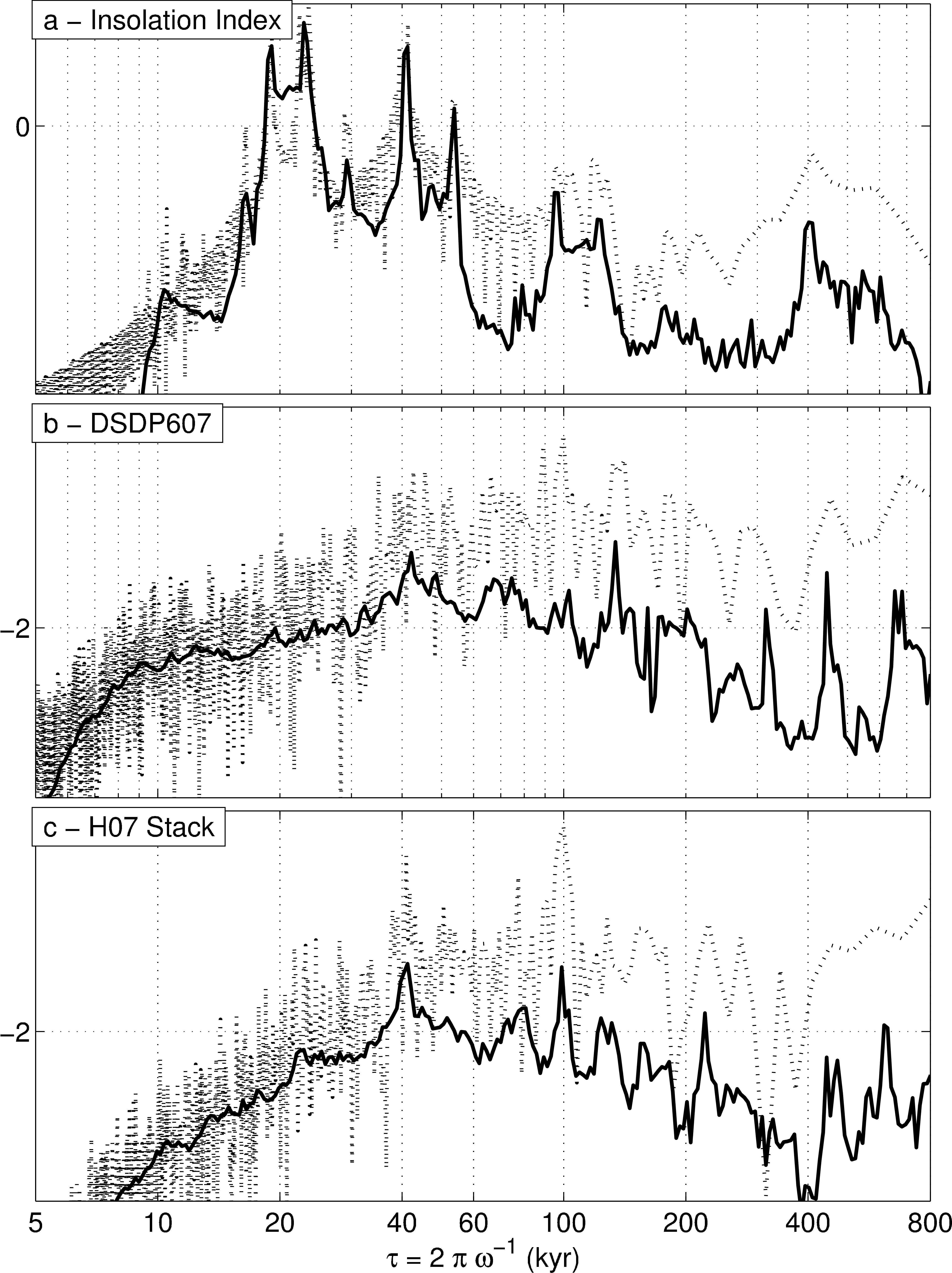}
   \caption[Comparison of Fourier magnitude and averaged
     Synchrosqueezing magnitude of insolation index $f_{SF}$ and
     climate response]%
   {\label{fig:paleo_spectrum} Comparison of Fourier 
     magnitude (gray) and Synchrosqueezing magnitude averaged
     over the entire period (black) of (a) insolation index
     $f_{SF}$, and climate response in benthic ${\delta}^{18}O$ of (b)
     DSDP607 core, $f_{CR1}$, and (c) H07 stack, $f_{CR2}$.  
     Synchrosqueezing averages are normalized by $1/R_\psi$ to
     correspond to the true signal magnitudes.}
\end{figure}

Synchrosqueezing analysis of the solar insolation index and
benthic $\delta^{18}O$ makes a significant contribution in three
important ways.  First, it produces spectrally sharp traces
of complex system evolution through the high-dimensional climate state
space (compare with, e.g., \cite[Fig.~2]{Clark2006}).  Second, it
delineates the effects of noise on specific
frequency ranges when comparing a single core to the stack. 
Low frequency components are mostly robust to noise induced by both
local climate variability and the measurement process.  Third,
thanks to its precision, Synchrosqueezing allows the filtered
reconstruction of signal components within frequency bands.

Questions about the key physical processes governing large scale
climate variability over the last 2.5\,Myr can be answered with
sufficient accuracy only by precise 
data analysis and the development of a hierarchy of
models at various levels of complexity that reproduce the key aspects
of Pleistocene history. The resulting dynamic stochastic
understanding of past climates may benefit our ability
to predict future climates.

\begin{figure}[h!]
  \centering
   \includegraphics[width=.9\columnwidth]{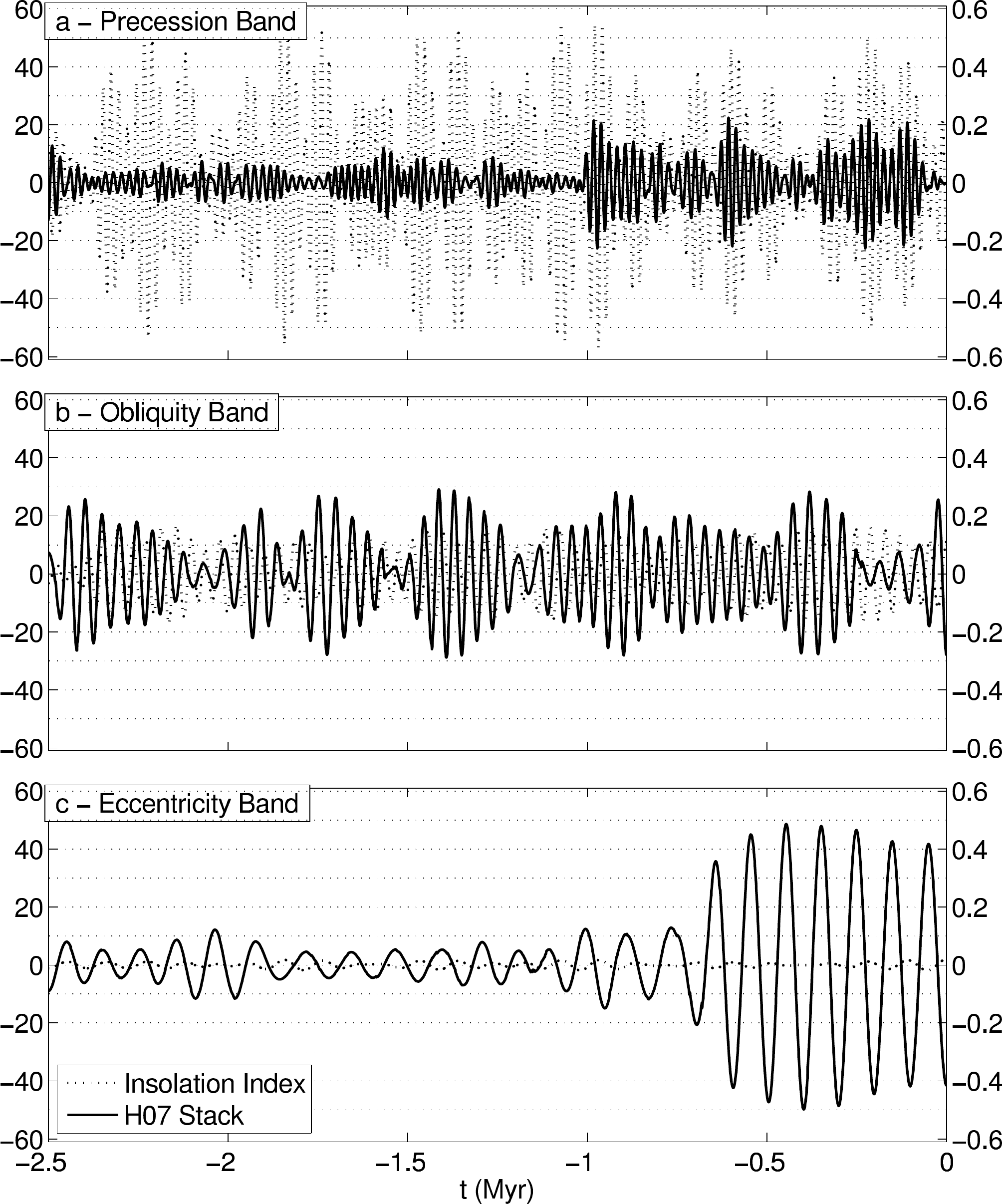}
   \caption[Milankovi\'{c} orbital components extracted by the inverse
     Synchrosqueezing transforms of insolation index $f_{SF}$ and
     climate response]%
    {\label{fig:paleo_filtered} Milankovi\'{c} orbital
     components extracted by the inverse Synchrosqueezing transforms
     of insolation index $f_{SF}$ (gray curve with vertical scale
     on the left) and climate response in benthic ${\delta}^{18}O$
     stack $f_{CR2}$ (black curve with vertical scale on the right)
     over (a) precession band (integrated from 17\,kyr to 25\,kyr),
     (b) obliquity band (40\,kyr - 55\,kyr), and (c) eccentricity band
     (90\,kyr - 130\,kyr).  }
\end{figure}

%\section{Modern Climate}

%\section{Conclusion and Future Work}

\chapter{\label{ch:sltr}Multiscale Dictionaries of Slepian Functions on the Sphere%
\chattr{This chapter is based on ongoing work in collaboration with
Frederik~J.~Simons, Department of Geosciences, Princeton University.}}

\section{Introduction}

The estimation and reconstruction of signals from their
samples on a (possibly irregular) grid is an old and important problem
in engineering and the natural sciences.  Over the last several
centuries, both approximation and sampling techniques have been
developed to address this problem.
Approximation theorems provide (possibly probabilistic) guarantees
that a function can be approximated to a specified precision with a
bounded number of coefficients in an alternate basis or frame.  In
general, such guarantees put constraints on the function (e.g.,
differentiability) and the domain (e.g., smoothness
and compactness). Sampling theorems guarantee that a function can be
reconstructed to a given precision from either point samples or some
other form of sampling technique (e.g., linear combinations of point
observations, as in the case of compressive sensing).  Again, sampling
theorems place requirements on both the sampling (e.g., grid
uniformity or a minimum sampling rate), on the original function (e.g.,
a bandlimit), and/or on the domain (e.g., smoothness, compactness).

Approximation and sampling techniques are closely linked due to
their similar goals.  For example, a signal can be estimated via
its representation in an alternate basis (e.g., via Riemannian
sums that numerically calculate projections of point samples onto the
basis functions).  The estimate then follows by expanding the function
in the given basis.  Regularization (the approximation) of the
estimate can be performed by excluding the basis elements assumed to
be zero.

\begin{figure}[h!]
\centering
\includegraphics[width=.75\linewidth]{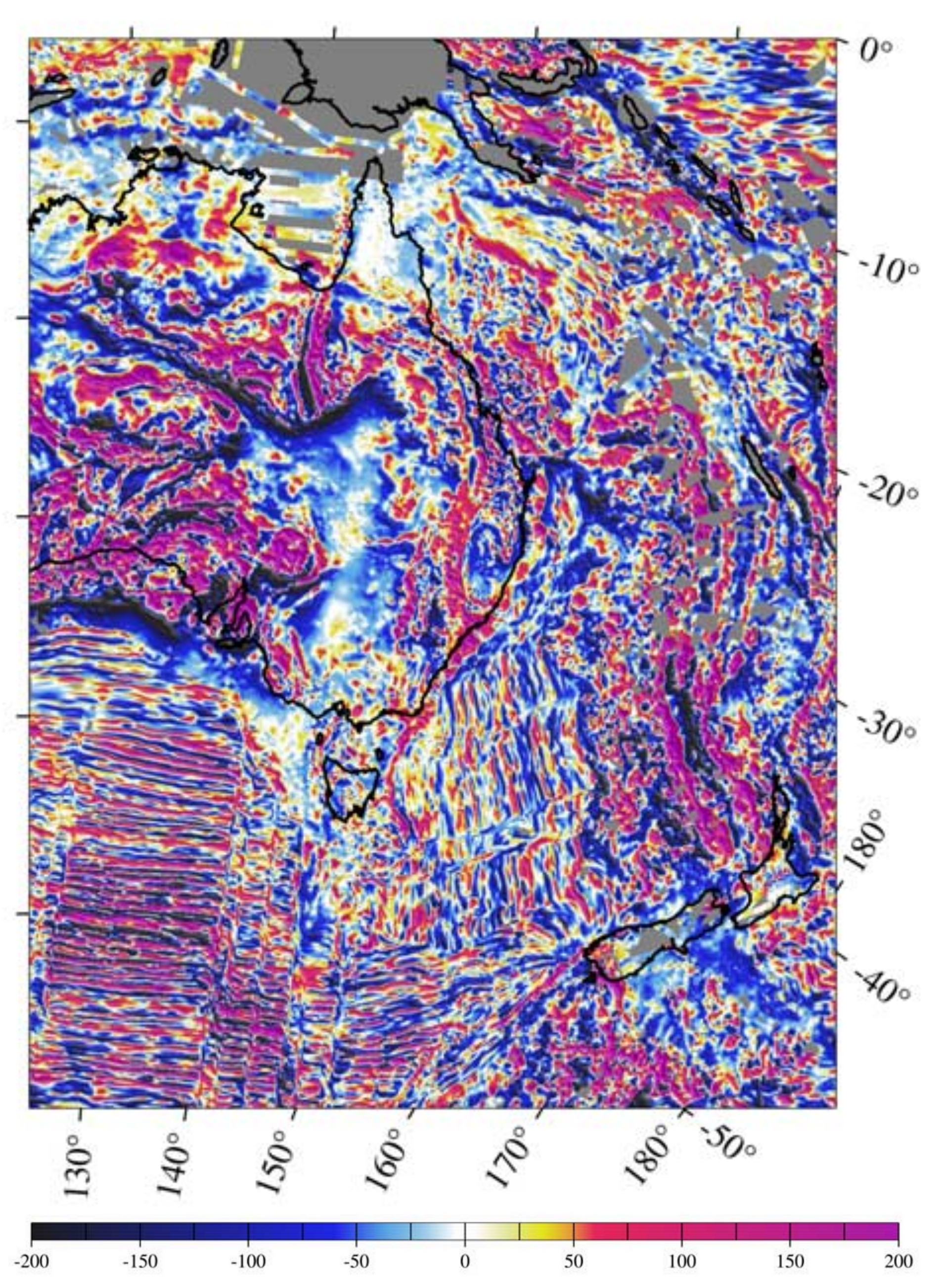}
\caption[EMAG2: Earth Magnetic Anomaly Grid (in nT), for the region
  covering Australia and New Zealand]{\label{fig:emag2}EMAG2: Earth
  Magnetic Anomaly Grid (nT), for the region covering Australia and
  New Zealand; reprinted from Fig.~6 of~\cite{Maus2009}.  Note the
  higher frequency, locally directionally consistent components in the
  ocean basins, and the lower frequency components on the continents.}
\end{figure}

In this chapter, we focus on the approximation and sampling problem
for subsets of the sphere $\cR \subset S^2$.  First, we are
interested in the representation of signals that are
bandlimited%
\footnote{In the sense that they have compact support
  in the spherical harmonic basis (see App.~\ref{app:fourier}).}
but whose contribution of higher frequencies arises from within a certain
region of interest (ROI), denoted $\cR$ (see, e.g.,
Fig.~\ref{fig:emag2}).  Second, we are interested in
reconstructing such functions from point samples within the ROI.  To
this end, we construct \emph{multiscale} dictionaries of functions that are
bandlimited on the sphere $S^2$ and space-concentrated in contiguous
regions $\cR$.

Our constructions are purely numerical, and are motivated by
subdivision schemes for wavelet constructions on the interval.  By
construction, the functions have low coherency (their pairwise inner
products are bounded in absolute value).  As a result, thanks to new
methods in sparse approximation (see, e.g.,
\cite{Candes2007,Gurevich2008}), they are good candidates for the
approximation and reconstruction of signals that are locally
bandlimited.

\section{Notation and Prior Work}
\label{sec:sltrprior}

We will focus on the important prior work in signal
representation, working our way up to Slepian functions on the
sphere --- the functions upon which our construction is based.

Before proceeding, we first introduce some notation.  For two
sequences (vectors) $\wh{f}, \wh{g} \in \ell^2$ and a subset of the
natural numbers $\Omega$, we refer to the product
${\ip{\wh{f}}{\wh{g}}_\Omega = \sum_{i \in \Omega} \wh{f}_i
\ol{\wh{g}_i}}$.  When $\Omega$ is omitted, we assume the sum is over
all indices.  The norm of $\wh{f}$ is denoted ${{\norm{\wh{f}}_\Omega^2
  = \ip{\wh{f}}{\wh{f}}}_\Omega}$.
For square-integrable functions $f$ and $h$ on a Riemannian manifold
$(\cM,g)$ (that is, $f,h \in L^2(\cM)$), and $\cR$ some subset of
$\cM$, we denote the inner product ${\ip{f}{h}_\cR = \int_\cR f(x)
\ol{h(x)} d\mu(x)}$, where $d\mu(x)$ is the volume element associated
with the metric $g$ (see Apps. \ref{app:diffgeom}--\ref{app:fourier}).
The norm is again defined as ${\norm{f}_\cR = \sqrt{\ip{f}{f}}_\cR}$.  The
type of inner product, manifold, and metric will be clear from the
context. Notation referring specifically to functions on the sphere
$S^2$ and the spherical harmonics may be found in
\S\ref{sec:fouriersphere}.

The prior literature in this area pertains to sampling,
interpolation, and basis functions on the real line, and we focus on
these next.

\subsection{Reconstruction of Bandlimited, Regularly Sampled Signals
  on $\bbR$}

A signal $f(t)$ on the real line is defined as bandlimited when it has
no frequencies higher than some bandlimit $W$.  That is, $\wh{f}(\omega) = 0$ for
$\abs{\omega} > W$.  The simplest version of the sampling theorem, as
given by Shannon \cite[\S II]{Shannon1949}, states that if $f$ is
sampled at regular intervals of with a frequency at or above the
sampling limit $t_s = 1/2W$, it can be exactly reconstructed via
convolution with the sinc low-pass filter:
\beq
\label{eq:sampthm}
f(t) = \sum_{n = -\infty}^\infty x_n \sinc(2 W t - n),
\eeq
where $x_n = f(t_s n)$ are the samples and $\sinc(z) = \sin(\pi z)/(\pi z)$.

Shannon also \emph{heuristically} describes \cite[\S III]{Shannon1949} that if a
bandlimited function $f$, with bandlimit~$W$, is also timelimited to
an interval $[-T,T]$ (that is, all of its samples, taken at
rate~$k_s$, are exactly 0 outside of this interval), then it requires
$N = \floor{2 T W}$ samples on this interval to reconstruct.  Thus a function
that is both time- and bandlimited as described above can be
described using only $N$ numbers---and the dimension of such functions
is $N$, which is called the Shannon number.

Shannon's heuristic definition was based on the (at that time)
well-known fact that a bandlimited, \emph{substantially} spacelimited%
\footnote{From Shannon's paper, we assume ``substantially'' implies
 space-concentrated.}
function can be represented well with~$N$ numbers.  In fact,
it is impossible to construct exactly space- and frequency- limited functions on
the line; see, e.g., the Paley-Wiener theorem%
\footnote{This theorem states that the Fourier transform of a compact
  function is entire.  The only entire function with an accumulation
  point of zeros (e.g., a compactly supported one) is the zero function.}
\cite[Thm.~7.22]{Rudin1973}.
Thus a different technique is needed for the estimation of bandlimited
functions on an interval, from samples only within that
interval.  As we will describe next the optimal representation for
this $N$-dimensional space is given by the ordered basis of
Slepian functions.  We will now focus on the construction of Slepian
functions, and will provide a more rigorous definition of the Shannon
number for the space of bandlimited, space-concentrated
functions.

\subsection{\label{sec:slepconstr1d}An Optimal Basis for Bandlimited Functions on the Interval}

The estimation of bandlimited signals on an interval requires the
construction of an optimal basis to represent such functions.  The
Prolate Spheroidal Wave Functions (PSWF)
\cite{Landau1961,Landau1962,Slepian1961}, also known as Slepian
functions, are one such basis.
The criterion of optimal concentration is with respect to the ratio of
$L^2$ norms.  Let $\cR = [-T,T]$ be the interval on the line, and
$\Omega=[-W,W]$ be the bandlimit in frequency.  The PSWF are
the orthogonal set of solutions to the variational problem
\begin{align}
\label{eq:sleptime1d}
\maximize_{g \in L^2(\bbR)} \quad &\lambda = \frac{\int_\cR g^2(t)
  dt}{\int_\bbR g^2(t) dt}, \\
\text{subject to} \quad & \hat{g}(\omega) = 0, \quad \omega \not\in \Omega,
\nonumber
\end{align}
where for $\alpha=1,2,\ldots$, the value $\lambda_\alpha$ is
achieved by $g_\alpha$, and we impose the orthonormality
constraint $\ip{g_\alpha}{g_{\alpha'}} = \delta_{\alpha \alpha'}$.
Here $\lambda_\alpha$ is the measure of concentration of $g_\alpha$
on~$\cR$; these eigenvalues are bounded: $0 < \lambda_\alpha < 1$.
We use the standard ordering of the Slepian functions, wherein
${\lambda_1 > \lambda_2 > \cdots}$ (the first Slepian function is
the most concentrated within $\cR$, the second is the second most
concentrated, and so on).

The orthonormal solutions to \eqref{eq:sleptime1d} satisfy the
integral eigenvalue problem \cite{Slepian1961,Simons2006b}:
\begin{align}
\label{eq:slep1dtimeint}
\int_\Omega &K_F(\omega,\omega') \wh{g}_\alpha(\omega') d\omega' = \lambda_\alpha
\wh{g}_\alpha(\omega),\ \omega \in \Omega, \quad \text{where}\\
&K_F(\omega,\omega') = \frac{\sin T(\omega-\omega)}{\pi(\omega-\omega')}.
\nonumber
\end{align}

Note that $K_F$ is a smooth, symmetric, positive-definite kernel
with eigenfunctions $\set{\wh{g}_\alpha}_\alpha$
(and associated eigenvalues $\set{\lambda_\alpha}_\alpha$).
Mercer's theorem therefore applies \cite[\S 97,98]{Riesz1956}, and we can write
\beq
\label{eq:DFmercer}
K_F(\omega,\omega') = \sum_{a \geq 1} \lambda_\alpha \wh{g}_\alpha(\omega) \ol{\wh{g}_\alpha(\omega')}
\qquad \text{and} \qquad
K_F(\omega,\omega) = \sum_{a \geq 1} \lambda_\alpha \abs{\wh{g}_\alpha(\omega)}^2.
\eeq

In practice, \eqref{eq:slep1dtimeint} can be solved exactly and
efficiently using a special ``trick'': this integral equation commutes
with a special second-order differential operator and thus its solution
can be found via the solution of a PDE of Sturm-Liouville type.  The
values of the $g_\alpha$'s, or of the $\wh{g}_\alpha$'s, can be evaluated
exactly at any set of points on their domains, via the factorization
of a special tridiagonal matrix~\cite{Simons2010,Slepian1978}.  This
construction is beyond the scope of this chapter.

It has been shown that the Slepian functions are all either very well
concentrated within the interval, or very well concentrated outside
of it.  That is, the eigenvalues of the Slepian functions are
all either nearly $1$ or nearly $0$~\cite[Table~1]{Slepian1961}.
Furthermore the values of any bandlimited signal with concentration
$1-\epsilon$ on the interval, can be estimated to within a squared error
(in $L^2$) bounded by $\epsilon$, for any arbitrarily
$\epsilon > 0$, by approximating this signal using a linear
combination of all the Slepian functions whose eigenvalues are near
unity~\cite[Thm.~3]{Landau1962}, and no fewer~\cite[Thm.~5]{Landau1962}.  

The number of Slepian functions required to approximate a bandlimited,
space-concentrated function can be therefore be calculated by summing
their energies:
\beq
\label{eq:shanline}
N_{T,W} = \sum_{\alpha \geq 1} \lambda_\alpha
= \sum_{\alpha \geq 1} \lambda_\alpha \int_\Omega \abs{\wh{g}_\alpha(\omega)}^2 d\omega
= \int_{\Omega} K_F(\omega,\omega) d\omega = \frac{2 T W}{\pi},
\eeq
where we used \eqref{eq:DFmercer} after swapping the sum and integral.

The difference between this value of $N$ and Shannon's version is due
to changes in normalization of Fourier transforms and the
constant $\pi$ inside the $\sin$ in \eqref{eq:sampthm}.

\subsection{An Optimal Basis for Bandlimited Functions on subregions
  of the Sphere $S^2$}
\label{sec:sltrsphere}

The construction of Slepian functions on the sphere proceeds similarly
to the interval case, with differences due to the compactness of
$S^2$.  Let $\cR$ be a closed and connected subset of $S^2$ and let
the frequency bandlimit
%\footnote{For a description of Fourier analysis on $S^2$, see
%App.~\ref{app:fourier}.}
be ${\Omega=\set{(l,m) : 0 \leq l \leq L, -l \leq m \leq l}}$.
The set of square-integrable functions on $S^2$ with bandlimit~$\Omega$, 
which we will call $L^2_\Omega(S^2)$,  is an $(L+1)^2$
dimensional space.  This follows because any $f \in L^2_\Omega(S^2)$
can be written as%
\footnote{The spherical harmonics $Y_{lm}$ and their properties are
  given in App.~\ref{app:fourier}.}
\beq
\label{eq:plm2xyz}
f(\theta,\phi) = \sum_{l=0}^L \sum_{m=-l}^l \wh{f}_{lm} Y_{lm}(\theta,\phi),
\eeq
and therefore
$$
\dim L^2_\Omega(S^2) = \sum_{l=0}^L (2l+1) = (L+1)^2.
$$
For the rest of the chapter, we will refer to the bandlimits ``$\Omega$''
and ``$L$'' interchangeably.

We now proceed as in \cite{Simons2006b}.
Slepian functions concentrated on $\cR$ with bandlimit~$L$ are
the orthogonal set of solutions to the variational problem
\begin{align}
\label{eq:slepsphere}
\maximize_{g \in L^2(S^2)} \quad &\lambda = \frac{\int_\cR g^2(x) d\mu(x)}{\int_{S^2} g^2(x) d\mu(x)}, \\
\text{subject to} \quad & \hat{g}_{lm} = 0, \quad (l,m) \not\in \Omega,
\nonumber
\end{align}
where the value $\lambda_\alpha$ is achieved by $g_\alpha$,
$\alpha = 1,2,\ldots$, and we impose the orthonormality
constraint
\beq
\label{eq:sleponorm}
\int_{S^2} g_\alpha(x) g_{\alpha'}(x) d\mu(x) = \delta_{\alpha \alpha'}.
\eeq
Finally, as before, we use the standard order for the
Slepian functions: ${\lambda_1 \geq \lambda_2 \geq \cdots}$ (i.e., in
decreasing concentration).  Note that in contrast to the 1D case, the
concentration inequalities are not strict due to possible geometric
degeneracy.  Due to orthonormality, the Slepian functions also fulfill
the orthogonality constraint
\beq
\label{eq:slepo}
\int_{\cR} g_\alpha(x) g_{\alpha'}(x) d\mu(x) = \lambda_\alpha
\delta_{\alpha \alpha'}.
\eeq

The problem \eqref{eq:slepsphere} admits an integral
formulation equivalent to \eqref{eq:slep1dtimeint}.  In addition, as the Fourier
basis on $S^2$ is countable, the construction can be reduced to a
matrix eigenvalue problem.

Writing $g$ in \eqref{eq:slepsphere} via its Fourier series, as in
\eqref{eq:plm2xyz}, reduces the problem~to~\cite[Eq.~33]{Simons2010}
\begin{align}
\label{eq:slepspheremat}
\sum_{(l',m') \in \Omega}& K_{lm,l'm'} \wh{g}_{l'm'} = \lambda \wh{g}_{lm},
\quad (l,m) \in \Omega, \\
\text{where } &K_{l'm',lm} = \ip{Y_{lm}}{Y_{l'm'}}_\cR.
\nonumber
\end{align}
From now on, we will denote by $K$ the $(L+1)^2 \x (L+1)^2$
matrix with coefficients $K_{l'm',lm}$.  The vector $K \wh{g}$ is thus
an $(L+1)^2$-element vector, indexed by coefficients $(l,m)$, with $(K
\wh{g})_{lm} = \sum_{(l',m') \in \Omega} K_{lm,l'm'} \wh{g}_{l'm'}$
for any $(l,m) \in \Omega$.  The matrix $K$ is called the spectral
localization kernel; it is real, symmetric and positive-definite.
We can now rewrite \eqref{eq:slepsphere} as a proper eigenvalue
problem: we solve
\beq
\label{eq:slepsphereeig}
K \wh{G} = \wh{G} \Lambda
\eeq
where $\wh{G} = \left(\wh{g}_1~\cdots~\wh{g}_{(L+1)^2}\right)$ is the
orthonormal matrix of eigenvectors and the
diagonal matrix $\Lambda$ is composed of eigenvalues in decreasing
order: ${\Lambda = \diag{\lambda_1~\lambda_2~\cdots~\lambda_{(L+1)^2}}}$.
As $K$ is positive-definite and symmetric, its eigenvectors,
${\wh{g}_\alpha,~\alpha=1,2,\cdots,(L+1)^2}$, form an orthogonal set 
that spans the space $L^2_\Omega(S^2)$.  The spatial functions can be
easily calculated via \eqref{eq:plm2xyz} and efficient recursion
formulas for the spherical harmonics.

The Shannon number, $N_{\abs{\cR'},L}$ is again defined as the sum of the
eigenvalues,
\beq
\label{eq:shansphsum}
N_{\abs{\cR'},L} = \sum_{\alpha = 1}^{(L+1)^2} \lambda_\alpha = \text{Tr}(\Lambda)
= \text{Tr}(\wh{G}).
\eeq
It can also be shown (see,~e.g.,~\cite[\S4.3]{Simons2006b}) that the
Shannon number is given by a formula similar~to~\eqref{eq:shanline}:
\beq
\label{eq:shansph}
N_{\abs{\cR'},L} = \frac{\abs{\cR'}}{4 \pi} (L+1)^2,
\eeq
where $\abs{\cR'} = \int_{x \in \cR'} d\mu(x)$ is the area of $\cR'$.

\subsection{\label{sec:calcslepD}Calculation of the Spectral Localization Kernel $K$}
In contrast to the 1D construction of
\S\ref{sec:slepconstr1d}, and with the exception of the cases when the
domain $\cR$ has azimuthal and/or equatorial symmetry (see, e.g.,
\cite{Simons2006b}), there is no known differential operator that
commutes with the spatial integral version of~%
\eqref{eq:slepspheremat}.  Nevertheless, thanks to the discrete nature
of the problem, we can find tractable solutions using simple
numerical analysis.

We now briefly discuss the calculation of the symmetric
positive-definite matrix $K$ of \eqref{eq:slepspheremat}; basing the
discussion on the work in \cite[\S4.2]{Simons2007}.  As the
calculation of the largest eigenvalues and associated eigenvectors
of $K$ can be performed using standard efficient iterative solvers,
the main computational complexity lies in constructing the matrix
itself.

The problem of calculating $K$ reduces to numerically
estimating the constrained spatial inner product between spherical
harmonics,
\beq
\label{eq:YlmYlpmp}
K_{lm,l'm'} = \int_\cR Y_{lm}(x) Y_{l'm'}(x) d\mu(x),
\eeq
when we are given the (splined) boundary $\partial\cR$ (a closed simple
curve in $S^2$).  This is performed via a semi-analytic integration
over a grid.  We first find the northernmost and southernmost
colatitudes, $\theta_n$ and $\theta_s$, of $\partial \cR$.  For a
given colatitude $\theta$, we can find the westernmost and easternmost
points, $\phi_{e}(\theta)$ and $\phi_{w}(\theta)$ of $\partial
\cR$.  If $\cR$ is nonconvex, there will be some $I(\theta)$ number of
such points, which we denote $\phi_{e,i}(\theta)$~and~%
${\phi_{w,i}(\theta), i=1,\ldots,I(\theta)}$.
The integral \eqref{eq:YlmYlpmp} thus becomes
\begin{align}
\label{eq:Ylmp1}
K_{lm,l'm'} &= \int_{\theta_s}^{\theta_n} X_{lm}(\theta)
  X_{l'm'}(\theta) \Phi_{mm'}(\theta) \sin \theta d\theta, \\
\label{eq:Ylmp2}
\text{where} \quad \Phi_{mm'}(\theta) &=
 \sum_{i=1}^{I(\theta)} \int_{\phi_{e,i}(\theta)}^{\phi_{w,i}(\theta)}
 \sfN_{m} \sfN_{m'} \sfS_{m}(\phi) \sfS_{m'}(\phi) d\phi, \\
\sfN_{m} &= \sqrt{2 - \delta_{0m}},
\nonumber \\
\text{and} \quad \sfS_{m}(\phi) &=
\begin{cases}
\cos m \phi & \text{if } m \leq 0, \\
\sin m \phi & \text{if } m > 0.
\end{cases}
\nonumber
\end{align}
Above, $X_{lm}$ is the colatitudinal portion of $Y_{lm}$;
for more details, see App.~\ref{app:fourier}.

Equation \eqref{eq:Ylmp1} is calculated via Gauss-Legendre
integration using the Nystr\"{o}m method, first by discretizing the
colatitudinal integral into $J$ points $\set{\theta_j}_{j=1}^J$, and
then evaluating the integral \eqref{eq:Ylmp2} analytically 
at each point $\theta_j$.  The discretization number $J$
for the numerical integration is chosen large enough that the
\emph{spatial-domain} eigenfunctions, as calculated via
the diagonalization of $K$ and application of \eqref{eq:plm2xyz},
satisfy the Slepian orthogonality relations \eqref{eq:sleponorm} and
\eqref{eq:slepo} to within machine precision.

A second way involves the expansion of $\ip{Y_{lm}}{Y_{l'm'}}$ into
spherical harmonics --- the expansion coefficients are the
quantum-mechanical Wigner~$3j$ functions, which can be calculated
recursively.  The remaining integral over a single spherical harmonic
can be performed recursively in the manner of \cite{Paul1978}, which
is exact.  See also \cite{Eshagh2009}.

\section{\label{sec:sltrtree}Multiscale Trees of Slepian Functions}

We now turn our focus to numerically constructing a dictionary $\cD$ of
functions that can be used to approximate mostly low bandwidth signals
on the sphere.  As we will see in the next section, this dictionary
allows for the reconstruction of a variety of signals from their point
samples.

To construct $\cD$, we first need some definitions.  Let $\cR \subset
S^2$ be a simply connected subset of the sphere.  Let $L$ be the
bandwidth: the dictionary $\cD$ will be composed of functions
bandlimited to harmonic degrees $0 \leq l \leq L$.  The construction
is based on a binary tree.  Choose a positive integer (the node
capacity) $n_b$; each node of the tree corresponds to the first $n_b$
Slepian functions with bandlimit $L$ and concentrated on a subset 
$\cR' \subset \cR$.  The top node of the tree corresponds to the entire
region $\cR$, and each node's children correspond to a division of
$\cR'$ into two roughly equally sized subregions (the subdivision scheme
will be described soon).  As the child nodes will be concentrated in
disjoint subsets of $\cR'$, all of their corresponding functions and
children are effectively incoherent.

\begin{figure}[h!]
\centering
\includegraphics[width=.85\linewidth]{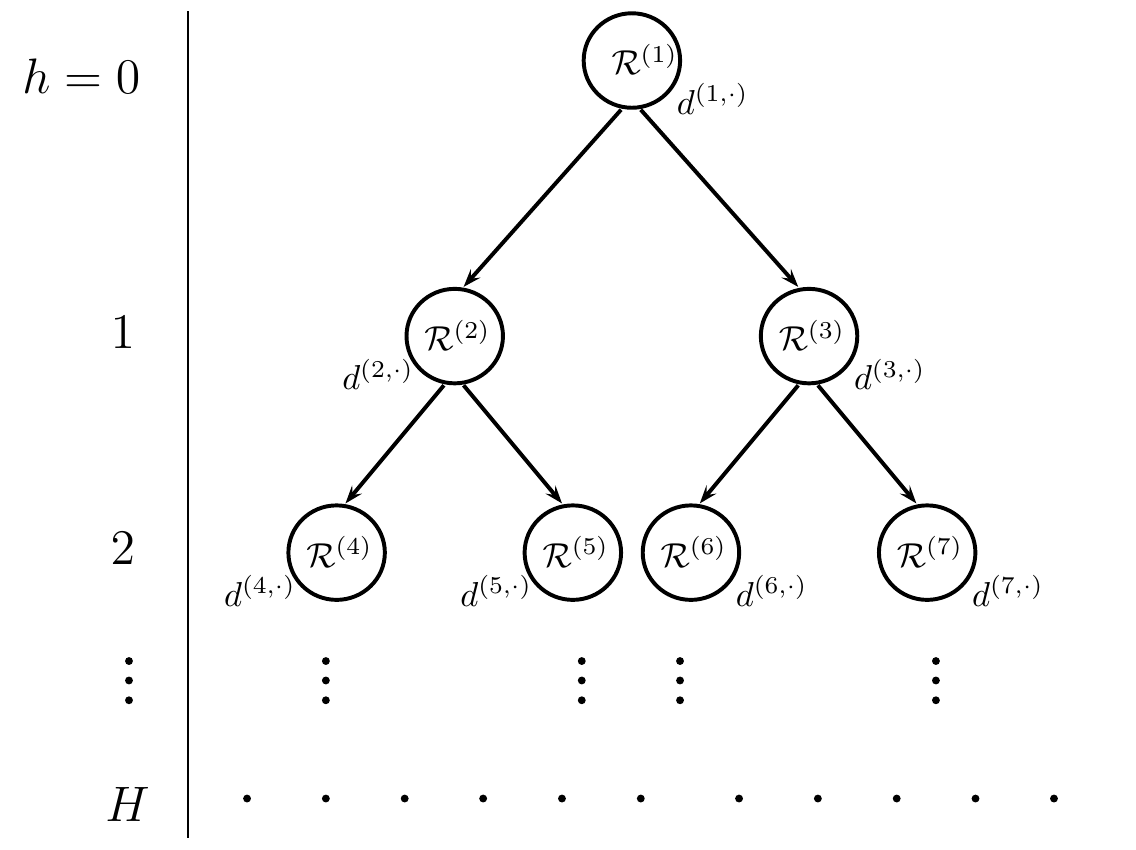}
\caption[The binary tree subdivision scheme and associated
  dictionary~$\cD_{\cR,L}$.]{\label{fig:sltrdiag}The binary tree
  subdivision scheme and associated dictionary~$\cD_{\cR,L}$.  We
  define the top-level region $\cR$ as $\cR^{(1)}$ and the generic
  subsets $\cR'$ as $R^{(j)}$.}
\end{figure}

We now fix a height $H$ of the tree: the number of times to subdivide
$\cR$.  The height is determined as the maximum number of binary
subdivisions of $\cR$ that can have $n_b$ well concentrated functions.
That is, we find the minimum integer $H$ such that 
$$
n_b \geq N_{2^{-H} \abs{\cR}, L}
$$
with the solution
$$
H = \ceil{\log_2\left(\frac{\abs{\cR}}{4 \pi} \frac{(L+1)^2}{n_b}\right)}.
$$
A complete binary tree with height $H$ has $2^{H+1}-1$ nodes, so from
now on we will denote the dictionary 
$$
\cD_{\cR,L,n_b} =
\set{d^{(1,1)},d^{(1,2)},\cdots,d^{(1,n_b)},\cdots,d^{\left(2^{H+1}-1,1\right)},
  \cdots,d^{\left(2^{H+1}-1,n_b\right)}}
$$
as the set of
$\abs{\cD_{\cR,L,n_b}} = n_b \, (2^{H+1}-1)$ functions thus constructed
on region $\cR$ with bandlimit $L$ and node capacity
$n_b$.
Fig.~\ref{fig:sltrdiag} shows the tree diagram of the
subdivision scheme.  We use the standard enumeration of nodes wherein
node $(j,\cdot)$ is subdivided into child nodes $(2j,\cdot)$ and
$(2j+1,\cdot)$, and at a level $0 \leq h \leq H$, the nodes are
indexed from $2^h \leq j \leq 2^{h+1}-1$.  More specifically, for
${j = 1,2,\ldots}$, we have ${\cR^{(j)} = \cR^{(2j)} \cup \cR^{(2j+1)}}$.
Furthermore, letting $g^{\cR'}_\alpha$ be the $\alpha$'th Slepian
function on $\cR'$ (the solution to \eqref{eq:slepsphere} with
concentration region $\cR'$), we have that
$$
d^{(j,\alpha)} = g^{\cR^{(j)}}_\alpha.
$$

Fig.~\ref{fig:sleptrafrica} shows an example of the construction when
$\cR$ is the African continent.  Note how, for example, $d^{(4,1)}$
and $d^{(5,1)}$ are the first Slepian functions associated with
the subdivided domains of $\cR^{(2)}$.
\begin{figure}[h!]
\centering
\begin{subfigure}[b]{.29\linewidth}
\includegraphics[width=\linewidth]{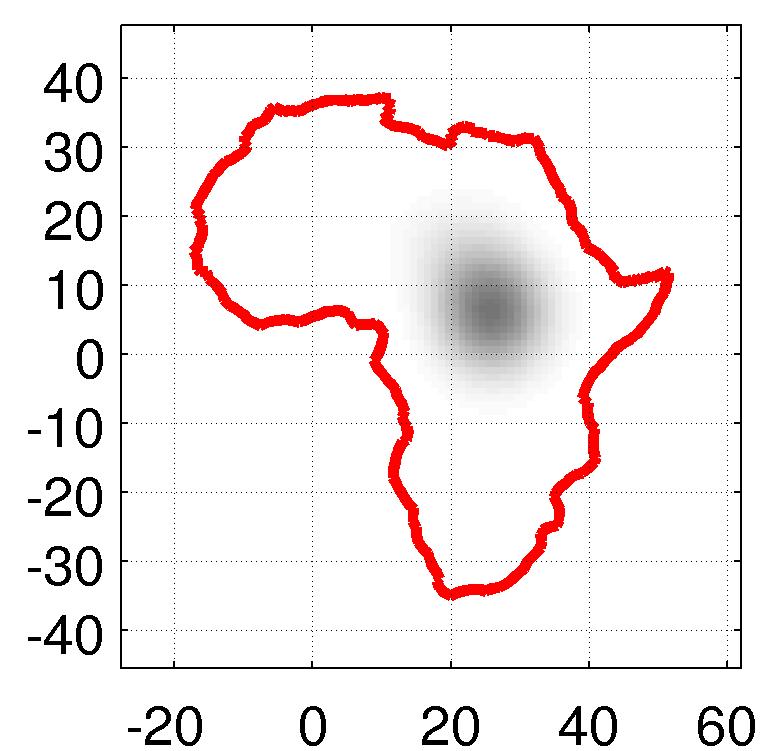}
\caption{$d^{(1,1)}$}
\end{subfigure}
\begin{subfigure}[b]{.29\linewidth}
\includegraphics[width=\linewidth]{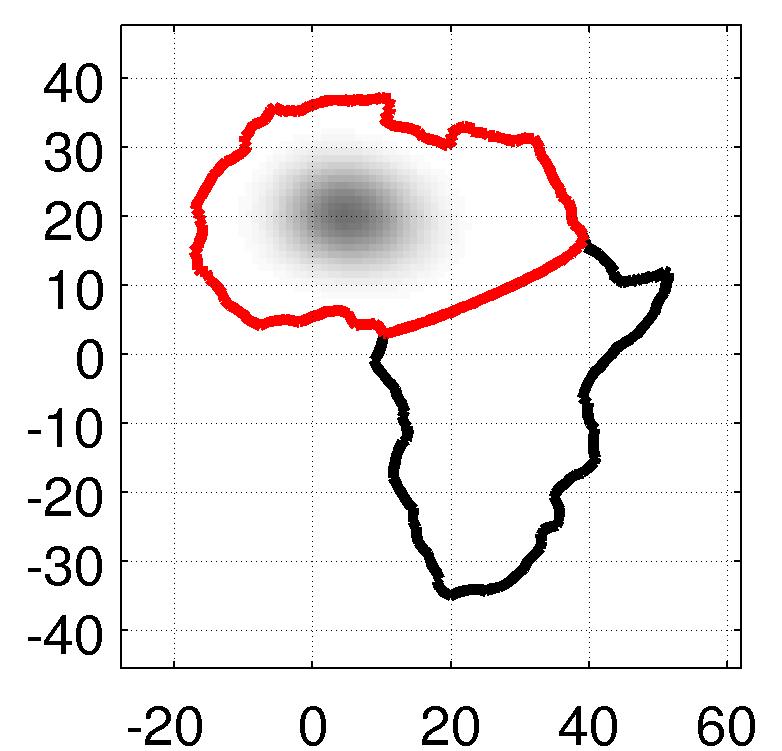}
\caption{$d^{(2,1)}$}
\end{subfigure}
\begin{subfigure}[b]{.29\linewidth}
\includegraphics[width=\linewidth]{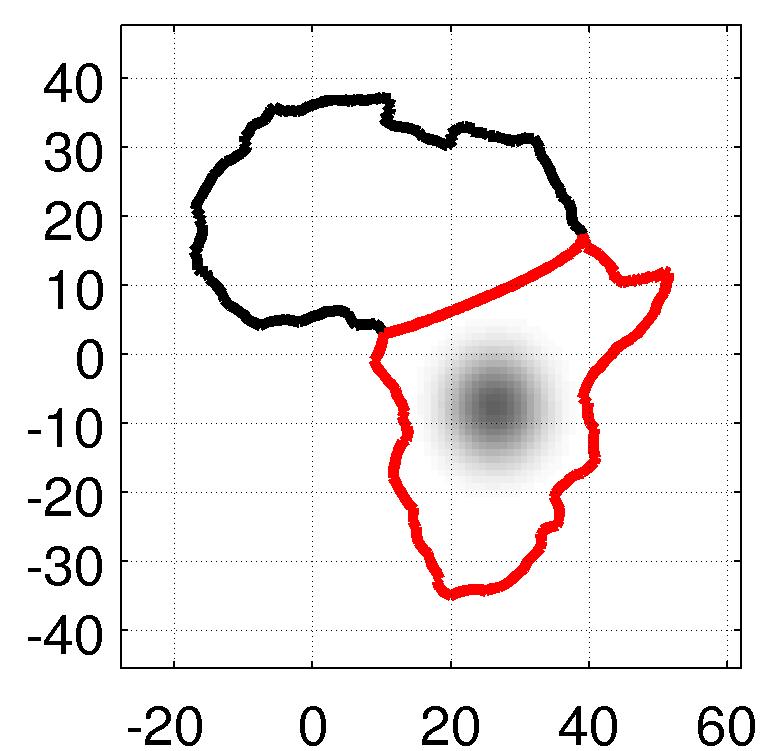}
\caption{$d^{(3,1)}$}
\end{subfigure}
\\
\begin{subfigure}[b]{.29\linewidth}
\includegraphics[width=\linewidth]{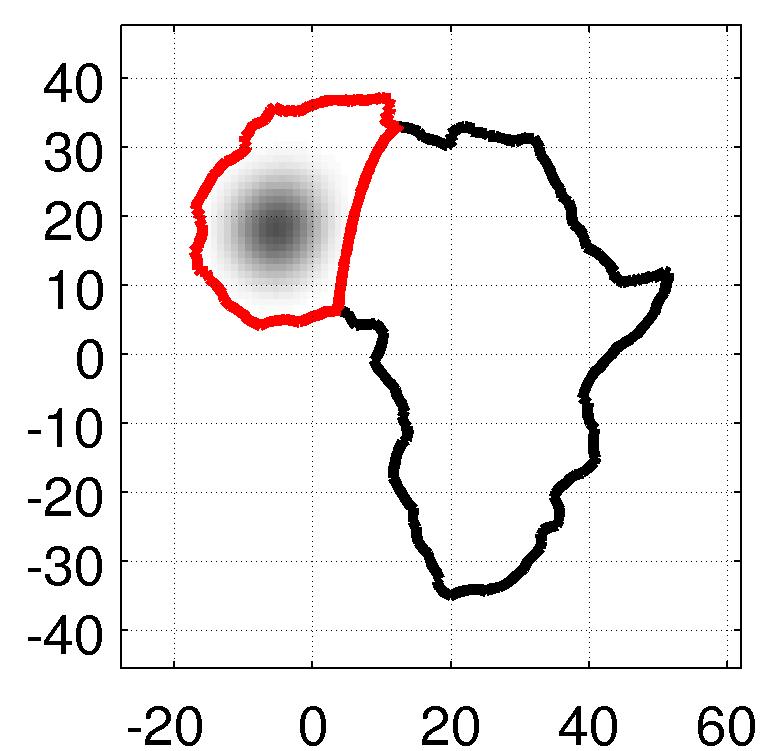}
\caption{$d^{(4,1)}$}
\end{subfigure}
\begin{subfigure}[b]{.29\linewidth}
\includegraphics[width=\linewidth]{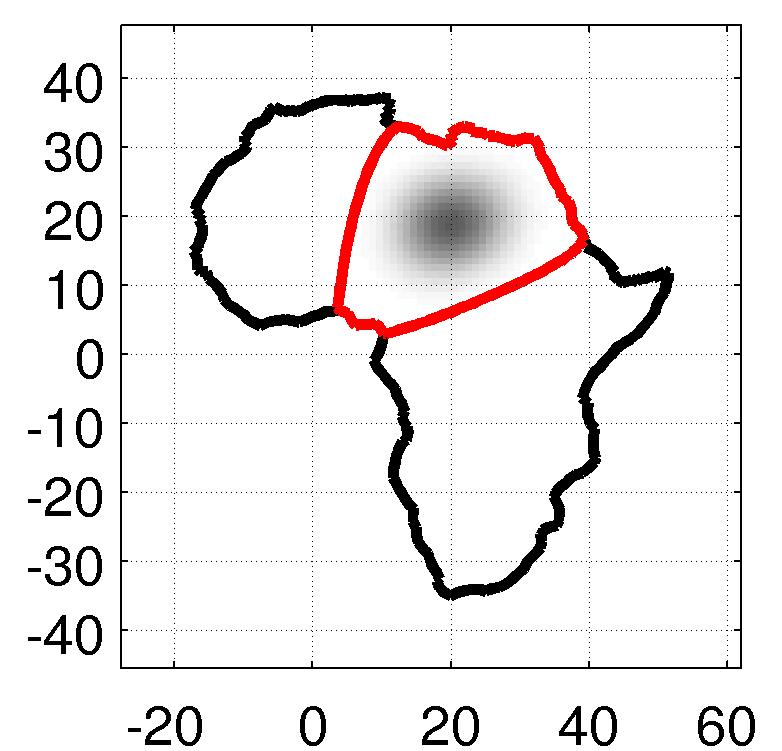}
\caption{$d^{(5,1)}$}
\end{subfigure}
\begin{subfigure}[b]{.29\linewidth}
\includegraphics[width=\linewidth]{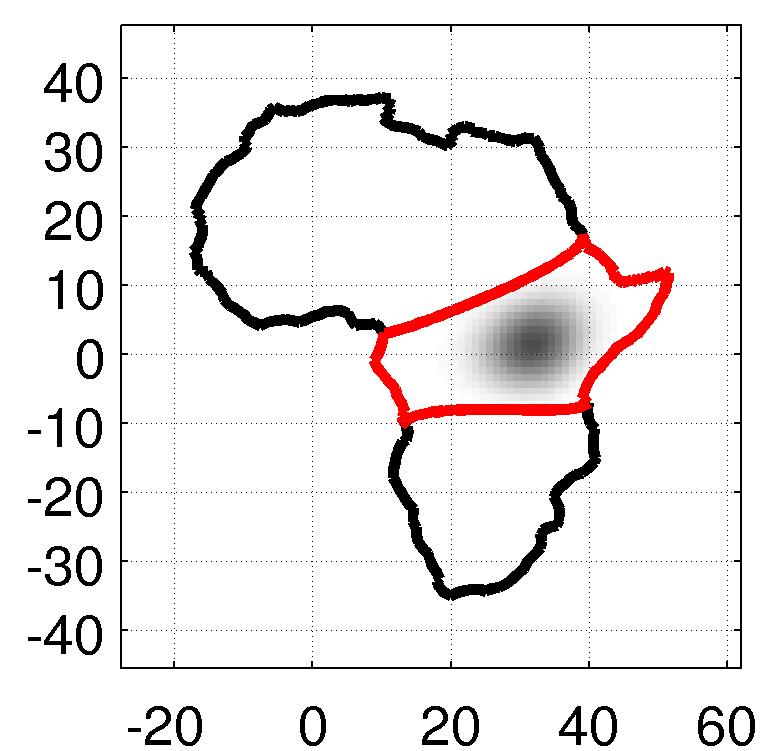}
\caption{$d^{(6,1)}$}
\end{subfigure}
\\
\begin{subfigure}[b]{.29\linewidth}
\includegraphics[width=\linewidth]{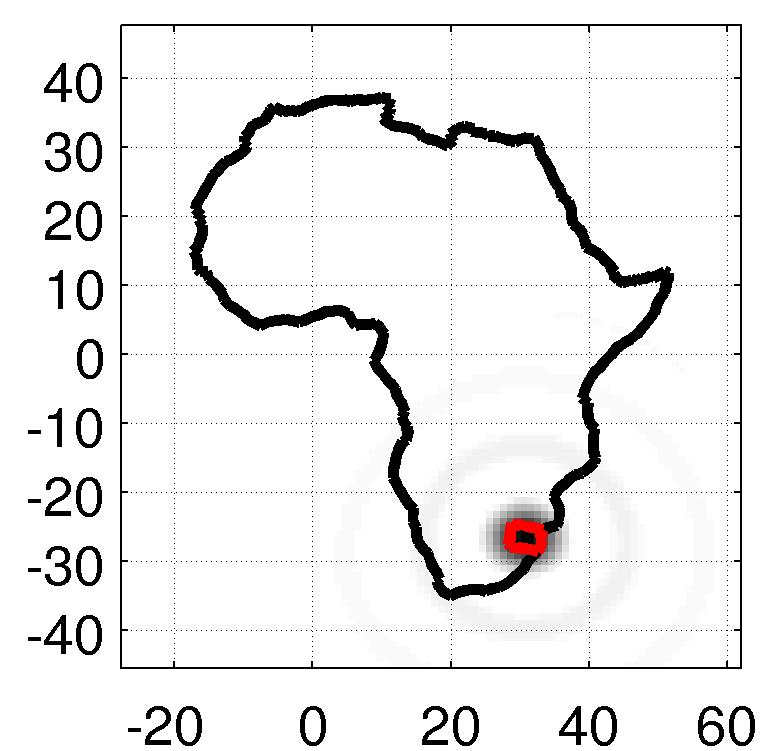}
\caption{$d^{(250,1)}$}
\end{subfigure}
\begin{subfigure}[b]{.29\linewidth}
\includegraphics[width=\linewidth]{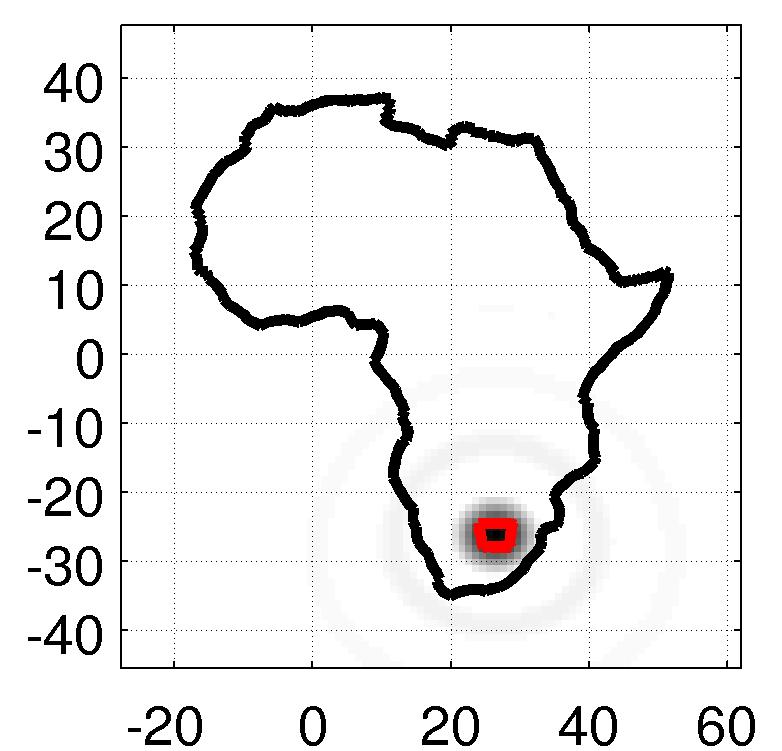}
\caption{$d^{(251,1)}$}
\end{subfigure}
\begin{subfigure}[b]{.29\linewidth}
\includegraphics[width=\linewidth]{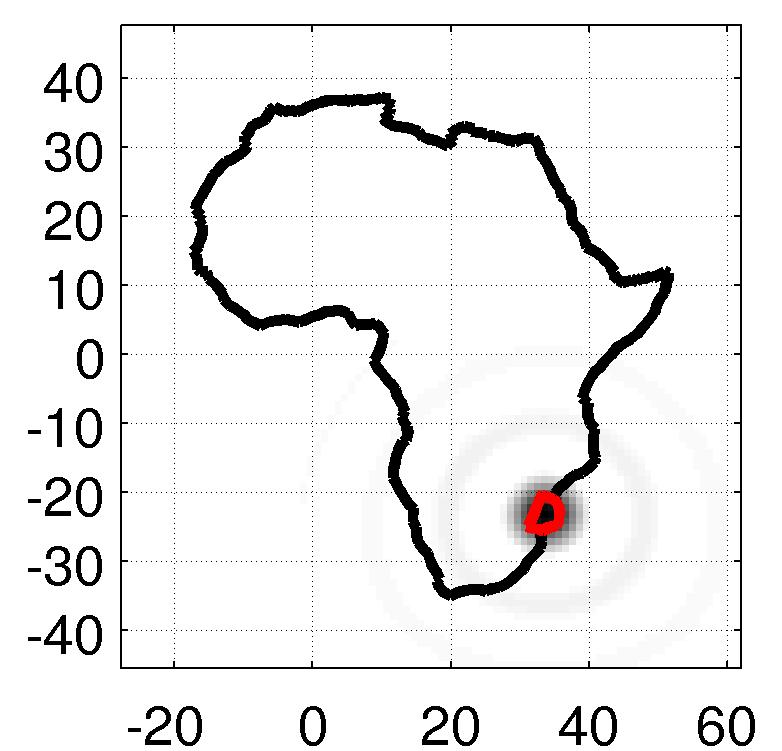}
\caption{$d^{(252,1)}$}
\end{subfigure}
\\
\begin{subfigure}[b]{.29\linewidth}
\includegraphics[width=\linewidth]{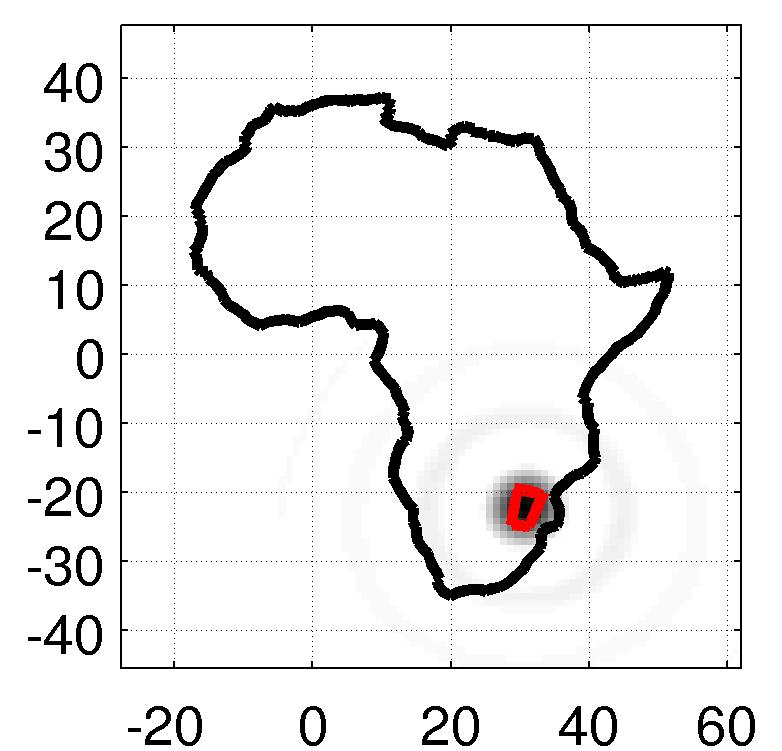}
\caption{$d^{(253,1)}$}
\end{subfigure}
\begin{subfigure}[b]{.29\linewidth}
\includegraphics[width=\linewidth]{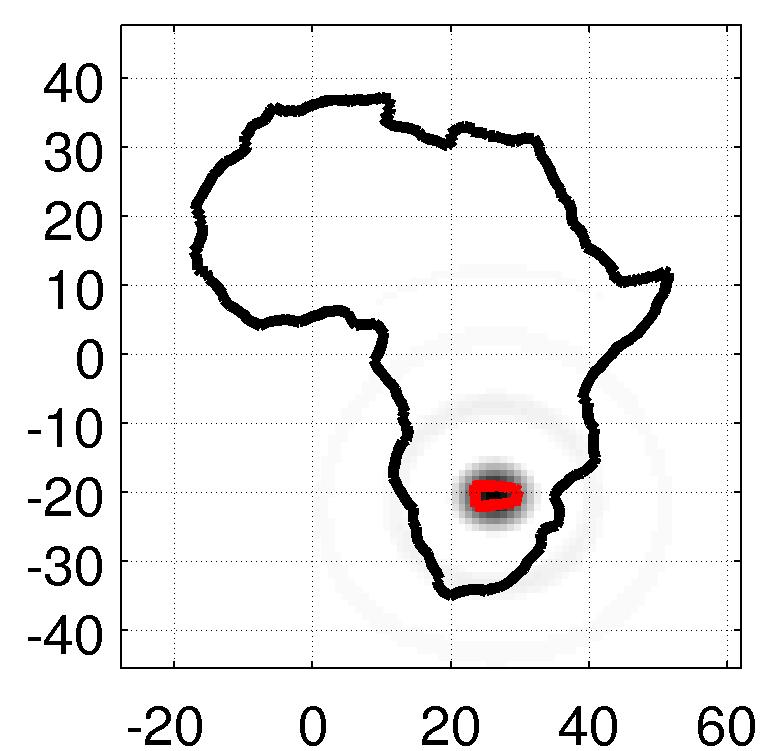}
\caption{$d^{(254,1)}$}
\end{subfigure}
\begin{subfigure}[b]{.29\linewidth}
\includegraphics[width=\linewidth]{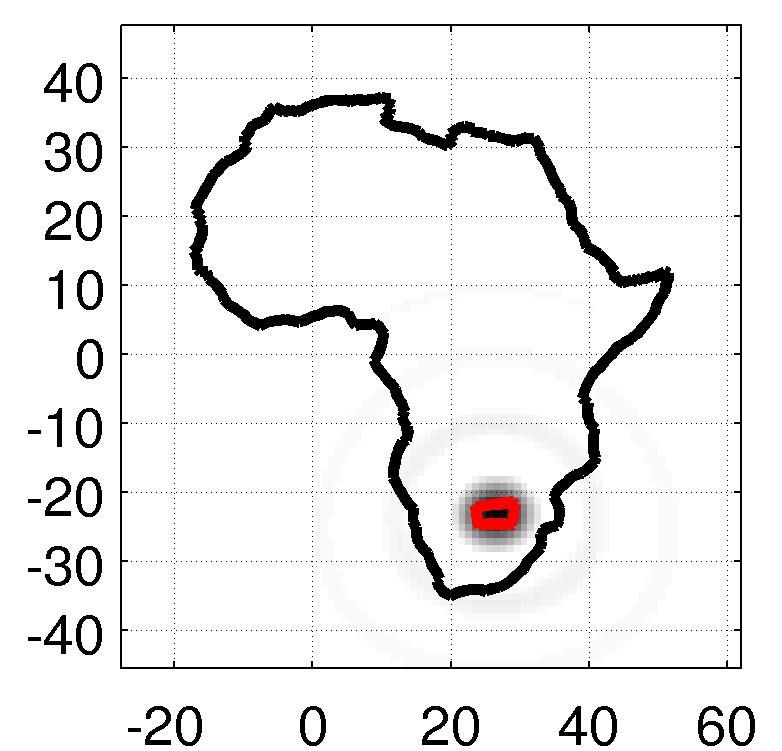}
\caption{$d^{(255,1)}$}
\end{subfigure}
\caption[Slepian Tree Dictionary $\cD_{\text{Africa},36,1}$]%
{\label{fig:sleptrafrica}Slepian Tree Dictionary
  $\cD_{\text{Africa},36,1}$ (having size $\abs{\cD}=255$);
  functions $d^{(1,1)}$ through $d^{(6,1)}$ and $d^{(250,1)}$ through
  $d^{(255,1)}$.
  The x-axis is longitude, the y-axis is colatitude.  Regions of
  concentration $\set{\cR^{(i)}}$ are outlined.}
\end{figure}

To complete the top-down construction, it remains to decide how to
subdivide a region $\cR'$ into equally sized subregions.
For roughly circular connected domains, the first Slepian function has
no sign changes, and the second Slepian function has a single
zero-level curve that subdivides the region into approximately equal
areas; when $\cR'$ is a spherical cap, the subdivision is
exact~\cite{Simons2006b}.  We thus subdivide a region $\cR'$ into the
two nodal domains associated with the second Slepian function on that
domain; see Fig.~\ref{fig:sleptrafricasecond} for a visualization of
the subdivision scheme as applied to the African continent.

\begin{figure}[h!]
\centering
\begin{subfigure}[b]{.29\linewidth}
\includegraphics[width=\linewidth]{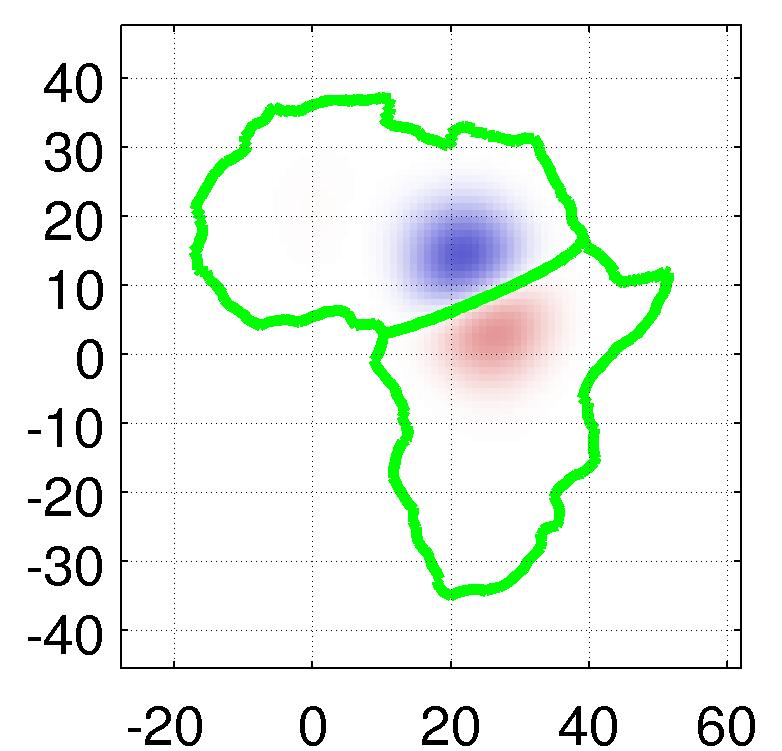}
\caption{$d^{(1,2)}$}
\end{subfigure}
\begin{subfigure}[b]{.29\linewidth}
\includegraphics[width=\linewidth]{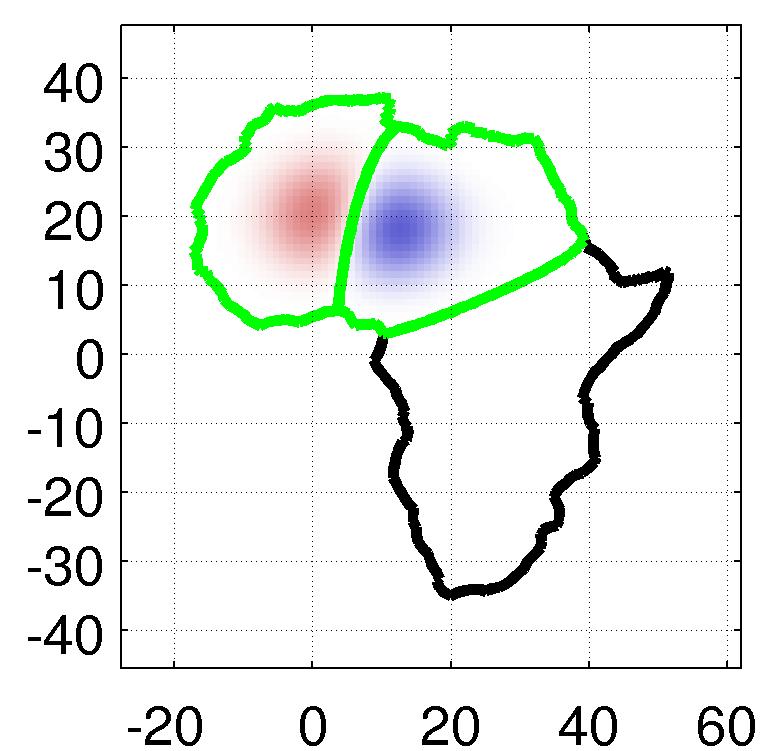}
\caption{$d^{(2,2)}$}
\end{subfigure}
\begin{subfigure}[b]{.29\linewidth}
\includegraphics[width=\linewidth]{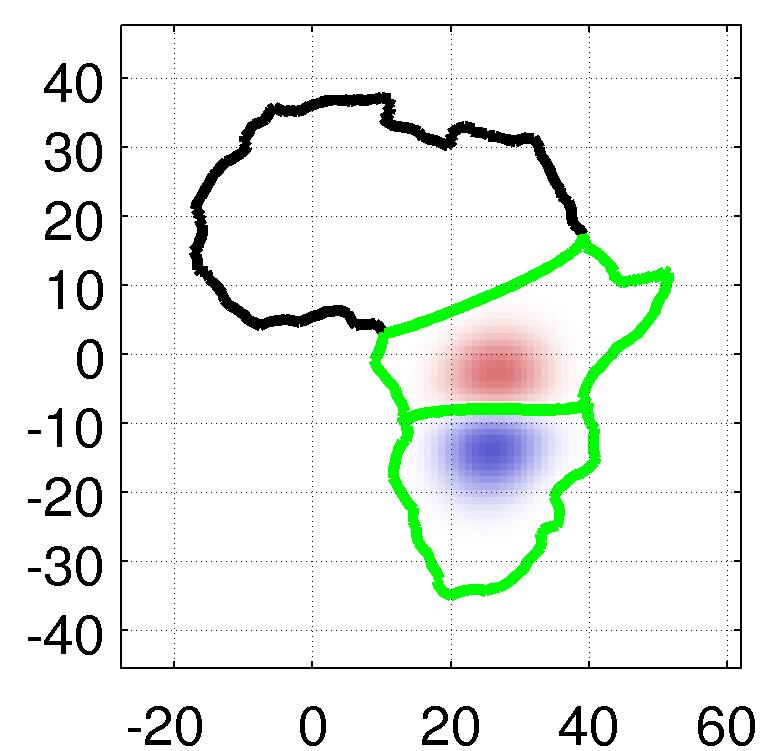}
\caption{$d^{(3,2)}$}
\end{subfigure}
\\
\begin{subfigure}[b]{.29\linewidth}
\includegraphics[width=\linewidth]{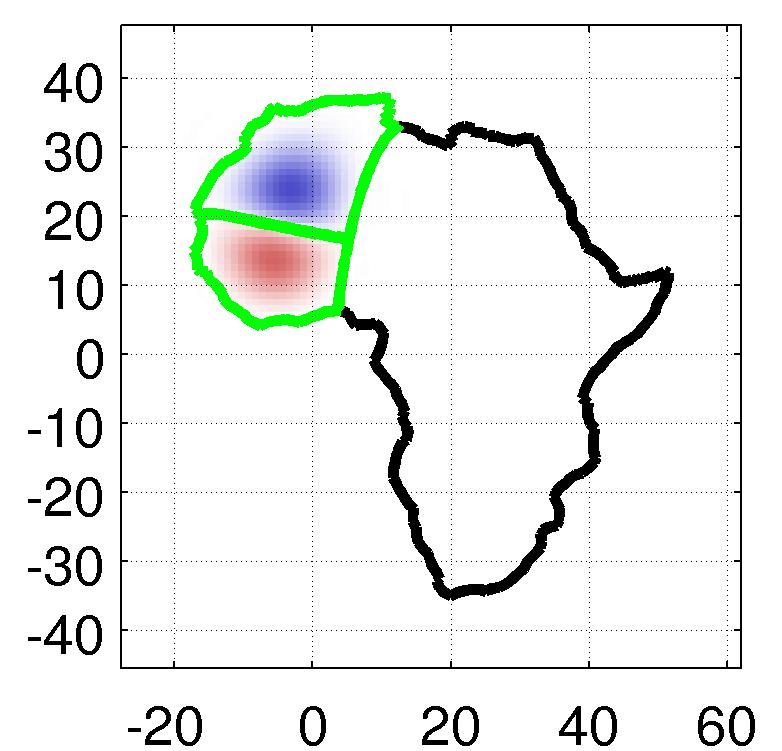}
\caption{$d^{(4,2)}$}
\end{subfigure}
\begin{subfigure}[b]{.29\linewidth}
\includegraphics[width=\linewidth]{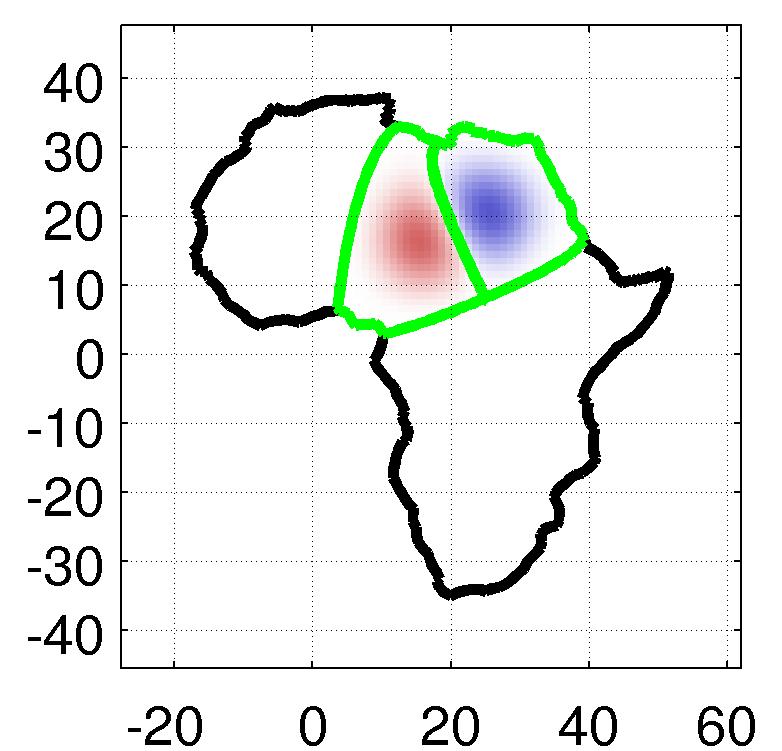}
\caption{$d^{(5,2)}$}
\end{subfigure}
\begin{subfigure}[b]{.29\linewidth}
\includegraphics[width=\linewidth]{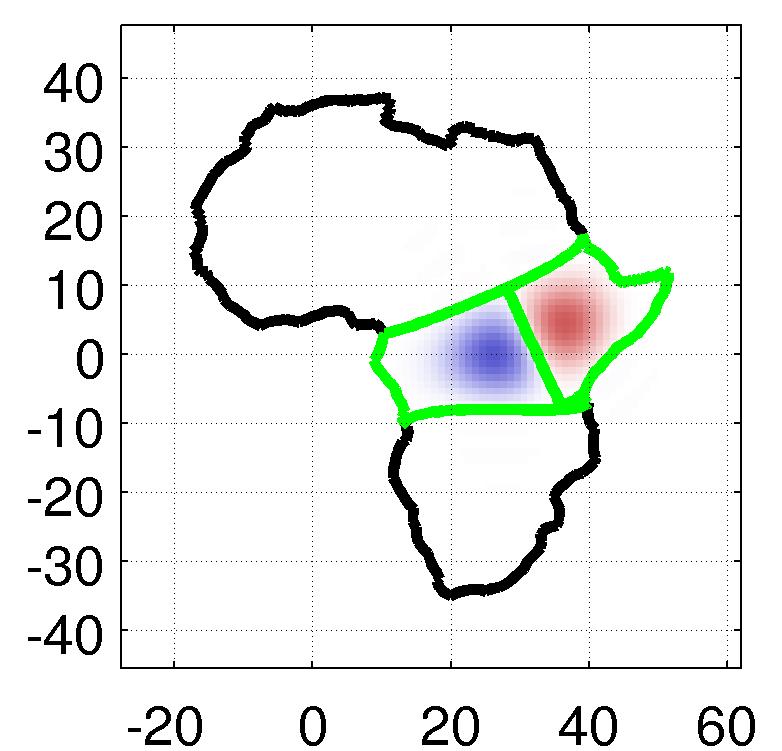}
\caption{$d^{(6,2)}$}
\end{subfigure}
\caption[Second Slepian functions associated with the regions in Figs.~\ref{fig:sleptrafrica}(a-f)]%
{\label{fig:sleptrafricasecond}Second Slepian functions associated
  with the regions in Figs.~\ref{fig:sleptrafrica}(a-f).
  The x-axis is longitude, the y-axis is colatitude.  Regions of
  concentration, and the central dividing contour (the zero-level set),
  are drawn in green.  Blue and red represent the sign of the
  Slepian function values.}
\end{figure}

\section{\label{sec:sltrprop}Concentration, Range, and Incoherence}
The utility of the Tree construction presented above depends on its
ability to represent bandlimited functions in a region $\cR$, and its
efficacy at reconstructing functions from point samples in $\cR$.
These properties, in turn, reduce to questions of concentration,
range, and incoherence:
\begin{itemize}
\item Dictionary $\cD$ is concentrated in $\cR$ if its
  functions are concentrated in $\cR$.
\item The range of dictionary $\cD$ is the subspace spanned by its
  elements.  Ideally, the basis formed by the first
  $N$ Slepian functions on $\cR$ is a subspace of the range~of~$\cD$.
\item When $\cD$ is incoherent, pairwise inner products of
  its elements have low amplitude: pairs of functions are
  approximately orthogonal.  This, in turn, is a useful property when
  using $\cD$ to estimate signals from point samples, as we will see
  in the next section.
\end{itemize}
In this section, we provide several techniques for analyzing these
properties for a given dictionary $\cD$, providing numerical
examples as we go along.

Unlike the eigenvalues of the Slepian functions on $\cR$, not all of
the eigenvalues of the elements of $\cD_{\cR}$ reflect their
concentration within this top-level (parent) region.  We thus define
the modified concentration value
\beq
\nu^{(j,\alpha)} = \int_{\cR} \left[d^{(j,\alpha)}(x)\right]^2 d\mu(x).
\eeq
Recalling that $\norm{d^{(j,\alpha)}}_2 = 1$, the value
$\nu$ is simply the percentage of energy of the
$(j,\alpha)^{\text{th}}$ element that is concentrated in $\cR$.  This value is
always larger than the element's eigenvalue,
which relates its fractional energy within the smaller subset $\cR^{(j)}$.
Figs.~\ref{fig:slepafricaeigvals},~\ref{fig:sleptrafricaeigvals},~%
and~\ref{fig:sleptrafricaconc} compare the eigenvalues of the Slepian
functions on the African continent with those of the Tree
construction, as well as with numerically calculated%
\footnote{Calculations performed using gridded Gauss-Legendre
integration similar to that in \S\ref{sec:calcslepD}.}
values of $\nu$.

\begin{figure}[h!]
\centering
\includegraphics[width=.49\textwidth]{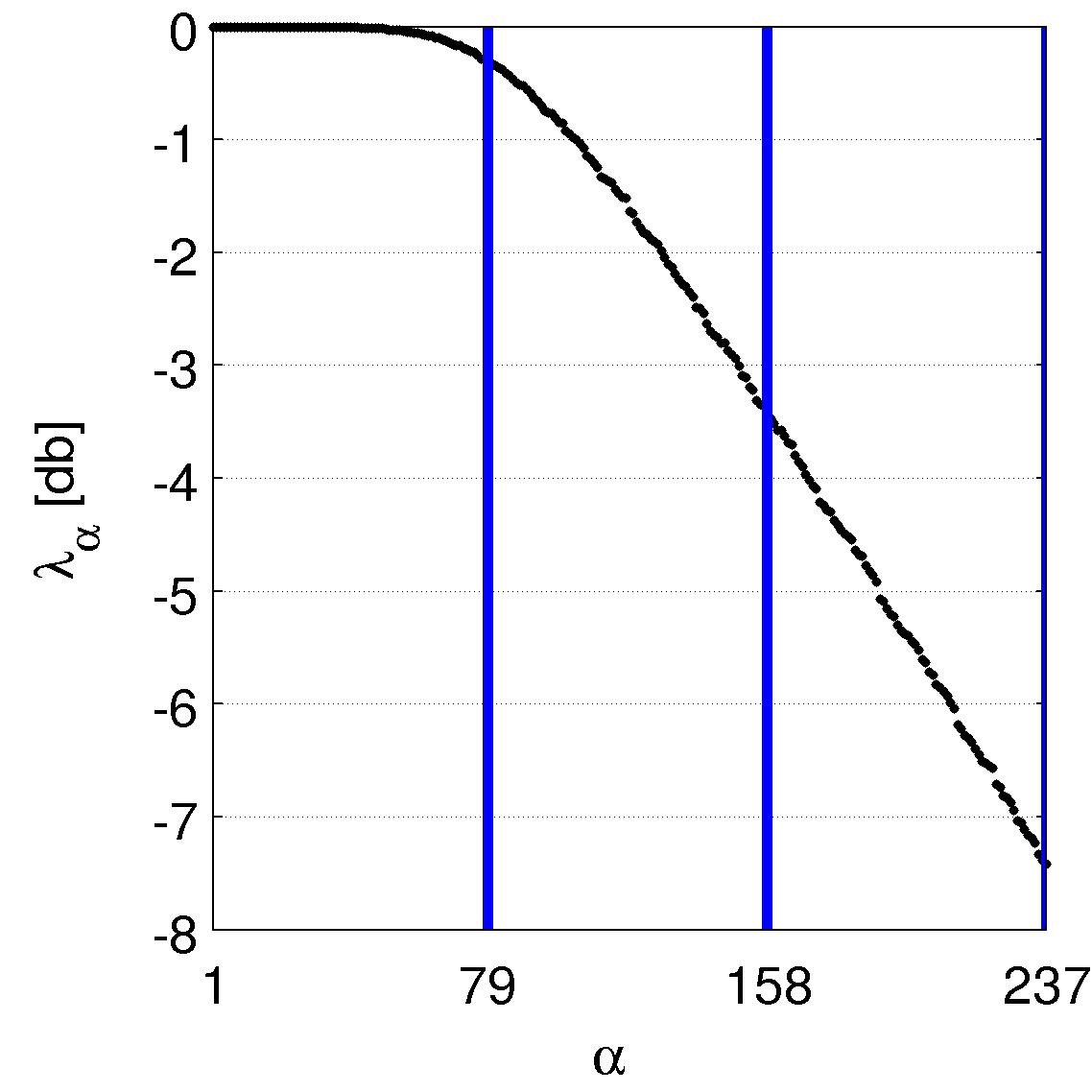}
\caption[Eigenvalues of the first $3 N_{\text{Africa},36}$ Slepian
  functions]{\label{fig:slepafricaeigvals}Eigenvalues of the first
  $N_{\text{Africa},36}$ Slepian functions for the African continent,
  normalized and on a base-10 log scale.  Blue lines
  correspond to integer multiples of the Shannon number~%
  ${N_{\text{Africa},36} \approx 79}$.}
\end{figure}

\begin{figure}[h!]
\centering
\includegraphics[width=.49\textwidth]{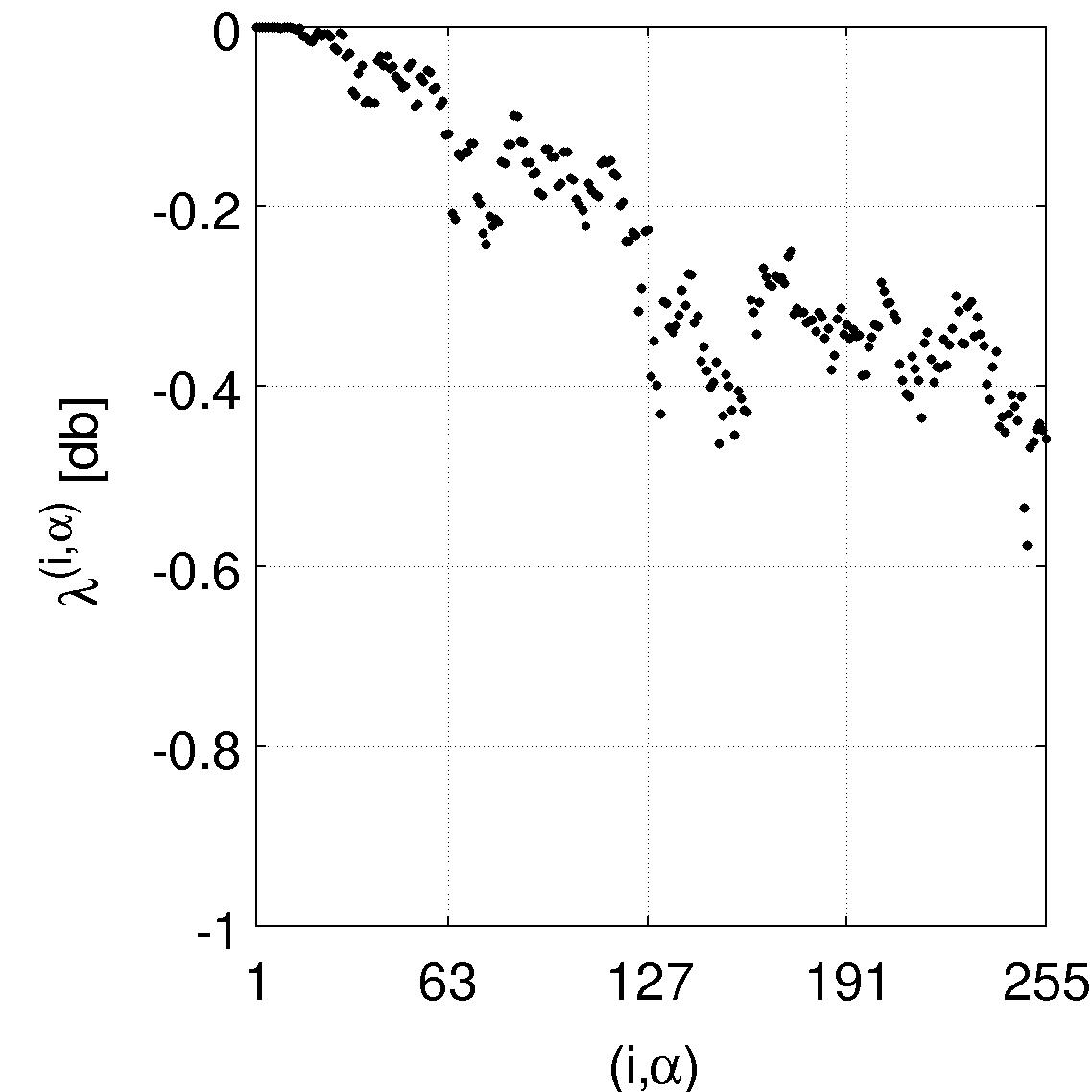}
\includegraphics[width=.49\textwidth]{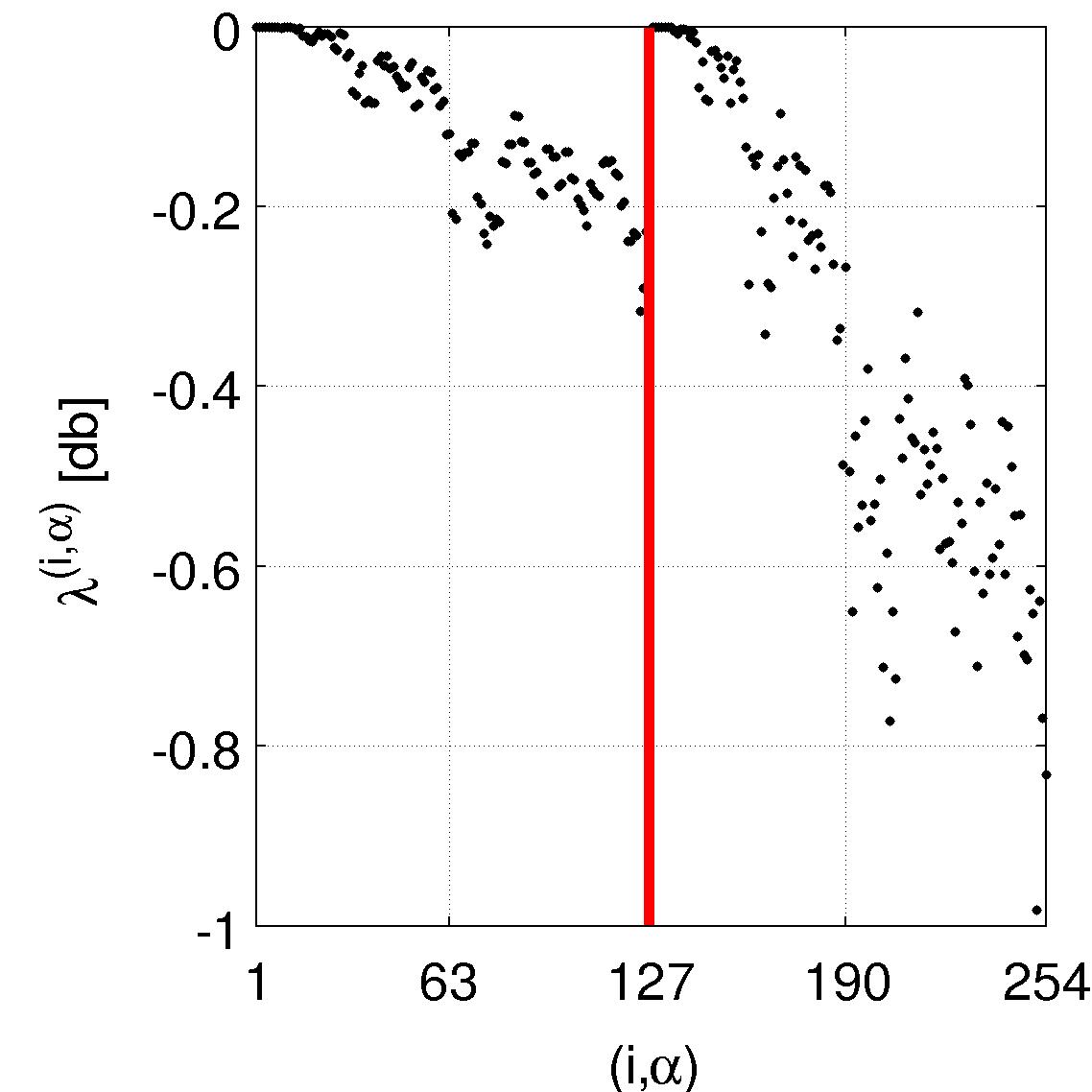}
\caption[Eigenvalues of the dictionaries ${\cD}_{\text{Africa},36,1}$ and
  ${\cD}_{\text{Africa},36,2}$]{\label{fig:sleptrafricaeigvals}Eigenvalues of
  the dictionary elements of ${\cD}_{\text{Africa},36,1}$ (left) and
  ${\cD}_{\text{Africa},36,2}$ (right), normalized and on a base-10
  log scale.  On the right pane, the thick line separates the 127 elements with
  $\alpha=1$ (left) and $\alpha=2$ (right).}
\end{figure}

\begin{figure}[h!]
\centering
\includegraphics[width=.49\textwidth]{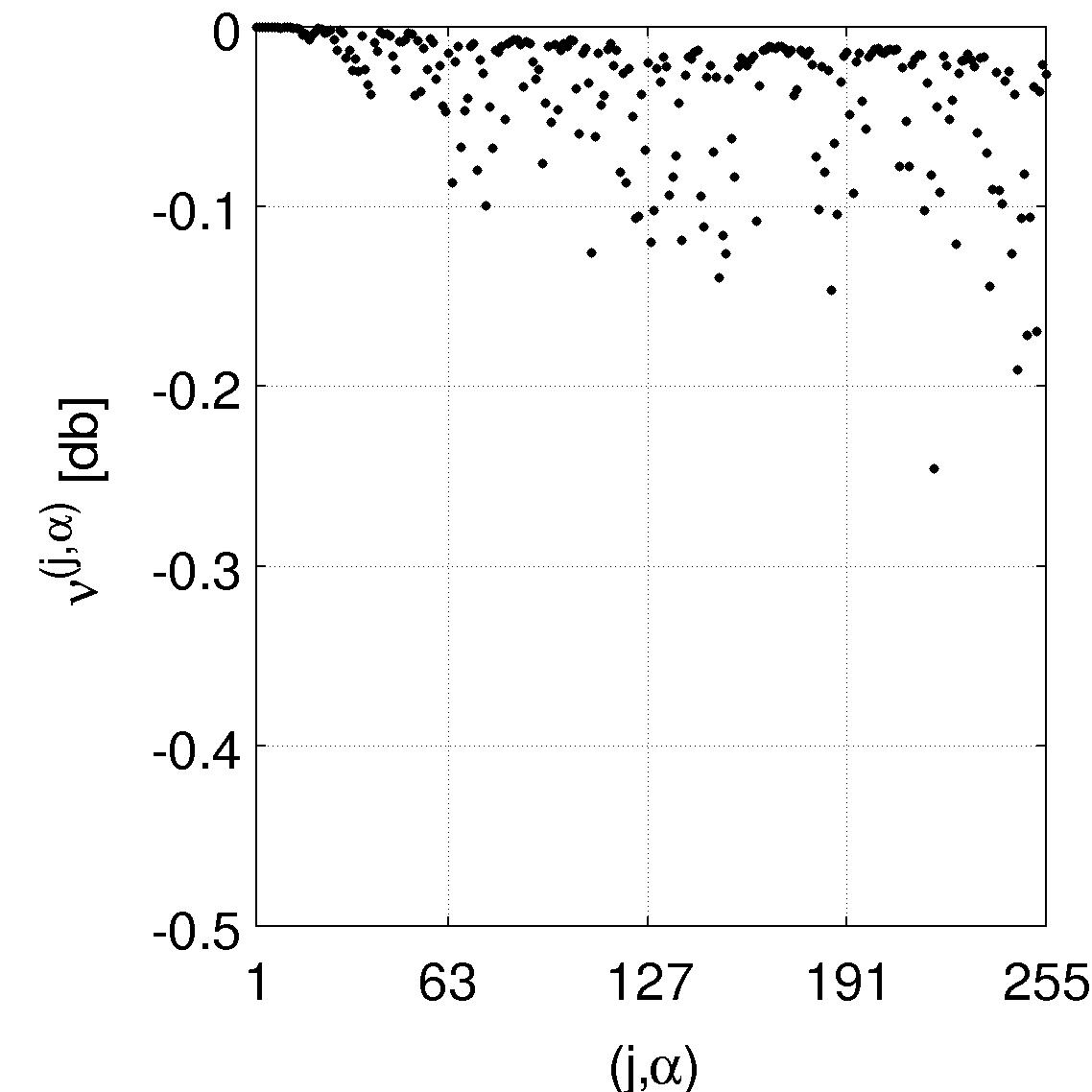}
\includegraphics[width=.49\textwidth]{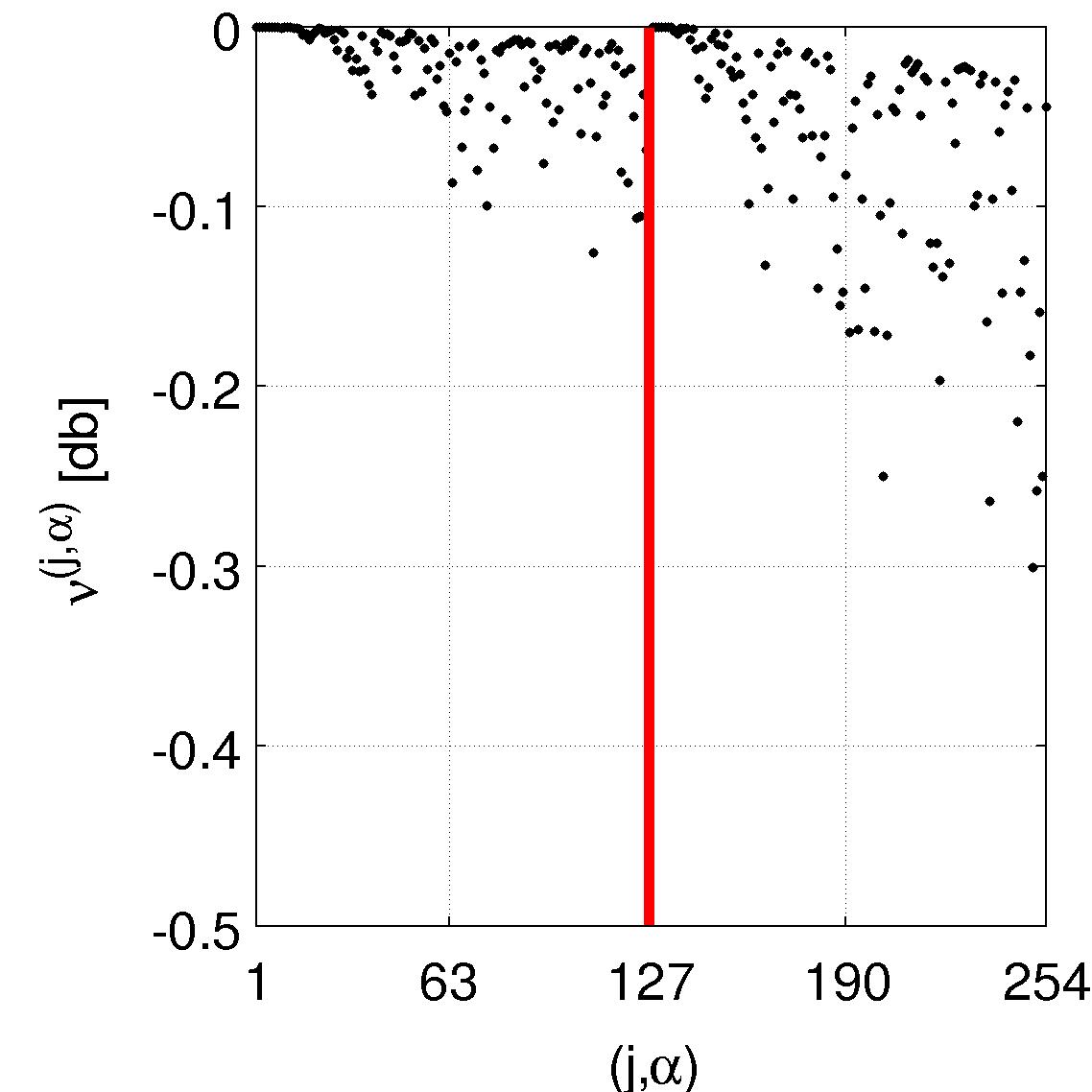}
\caption[Concentration values $\nu$ of the dictionaries ${\cD}_{\text{Africa},36,1}$ and
  ${\cD}_{\text{Africa},36,2}$]{\label{fig:sleptrafricaconc}Concentrations
  $\nu$ of the dictionary elements of
  ${\cD}_{\text{Africa},36,1}$ (left) and ${\cD}_{\text{Africa},36,2}$
  (right), normalized and on a base-10 log scale.
  On the right pane, the thick line separates the 127 elements with
  $\alpha=1$ (left) and $\alpha=2$ (right).}
\end{figure}

The size of dictionary $\cD_{\cR,L,n_b}$ is generally larger than
the Shannon number $N_{\abs{\cR},L}$ for any node capacity $n_b$, and
as a result it cannot form a proper basis (it has too many
functions).  Ideally, then, we require that elements of the range of
the dictionary spans the space of the first $N_{\abs{\cR},L}$ Slepian
functions.  We discuss two visual approaches for determining if this
is the case.

Though the spatial nature of the construction makes it
clear that dictionary elements tend to cover the entire domain $\cR$,
we also investigate the spectral energies of these elements; and
compare them with the energies of the $N_{\abs{\cR},L}$ Slepian
functions on $\cR$, which ``essentially'' form a basis for bandlimited
functions in $\cR$. The spectral energy density of a function $f$ is
given for each degree~${l=0,\ldots,L}$~by~\cite[Eq.~38]{Dahlen2008}:
$$
S^f_l = \frac{1}{2l+1} \sum_{m=-l}^l \abs{\wh{f}_{lm}}^2.
$$

Figs. \ref{fig:slepafricaspec} and \ref{fig:sleptrafricaspec} compare
the power spectra of the Slepian functions on the African continent
with the power spectra of two dictionaries given by the tree
construction.  While the Slepian functions are concentrated within
specific ranges of the harmonics, the tree construction leads to
spectra that depend on the degree.  The dictionary elements with
$\alpha=1$ tend to either contain mainly low-frequency harmonics or,
for elements concentrated on smaller regions, have a more flat
harmonic response within the bandlimit.  Dictionary elements
associated with higher order Slepian functions have a more pass-band
response when concentrated on larger regions and a more flat response
within the higher frequencies of the bandlimit when concentrated on
smaller regions.  So while it is clear that the Slepian functions span
the bandlimited frequencies, the spectral amplitude plots do not relay
this as clearly for the tree construction.

\begin{figure}[ht]
\centering
\includegraphics[width=.49\textwidth]{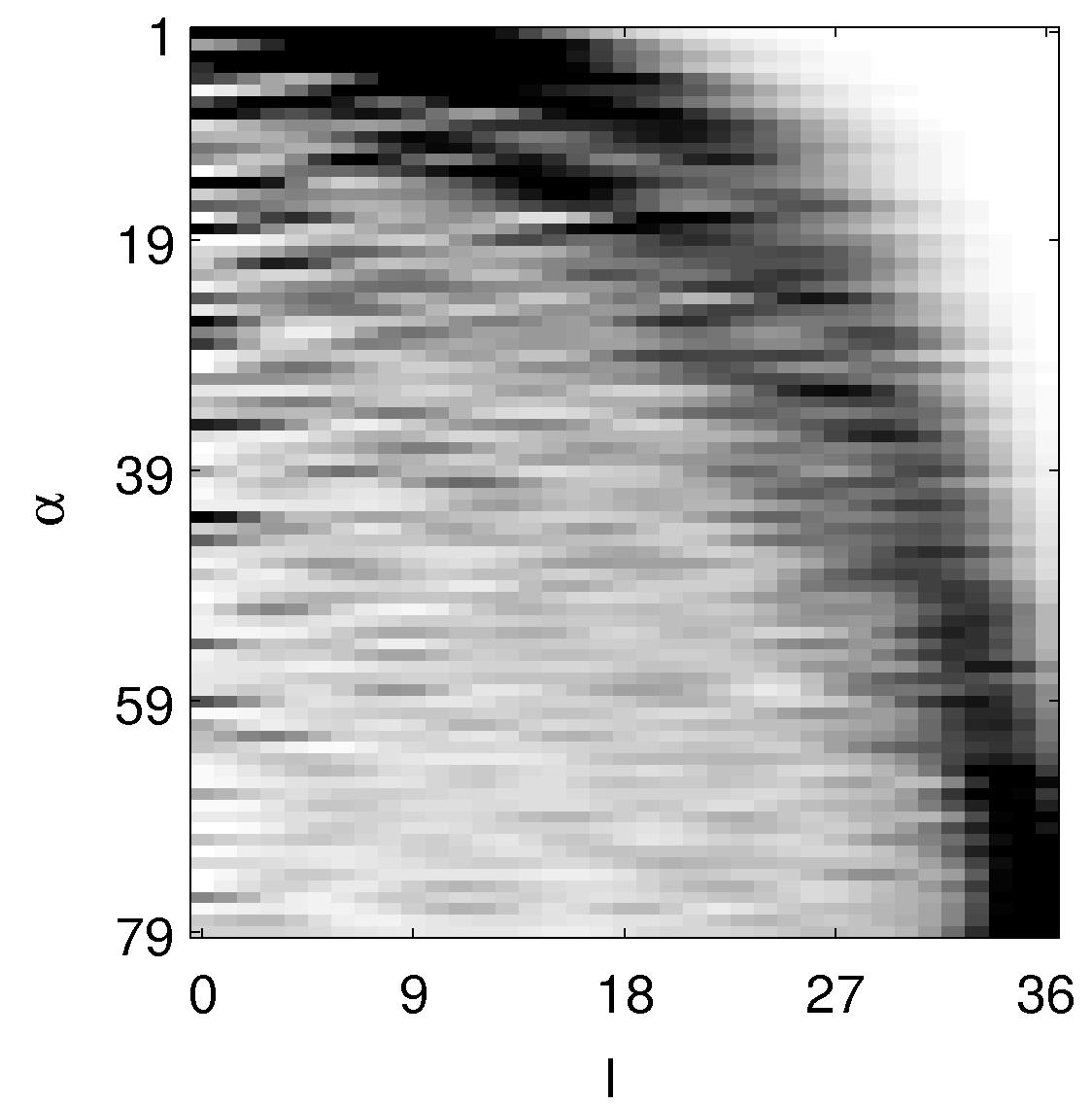}
\caption[Power Spectrum $S_l$ of the first $N_{\text{Africa},36}$ Slepian
  functions]{\label{fig:slepafricaspec}Power Spectrum $S_l$ of the first
  $N_{\text{Africa},36}$ Slepian functions for the African continent.
  The x-axis is degree $l$, the y-axis is Slepian function index
  $\alpha$.  Values are between 0 (white) and 1 (black).}
\end{figure}

\begin{figure}[ht]
\centering
\includegraphics[width=.49\textwidth]{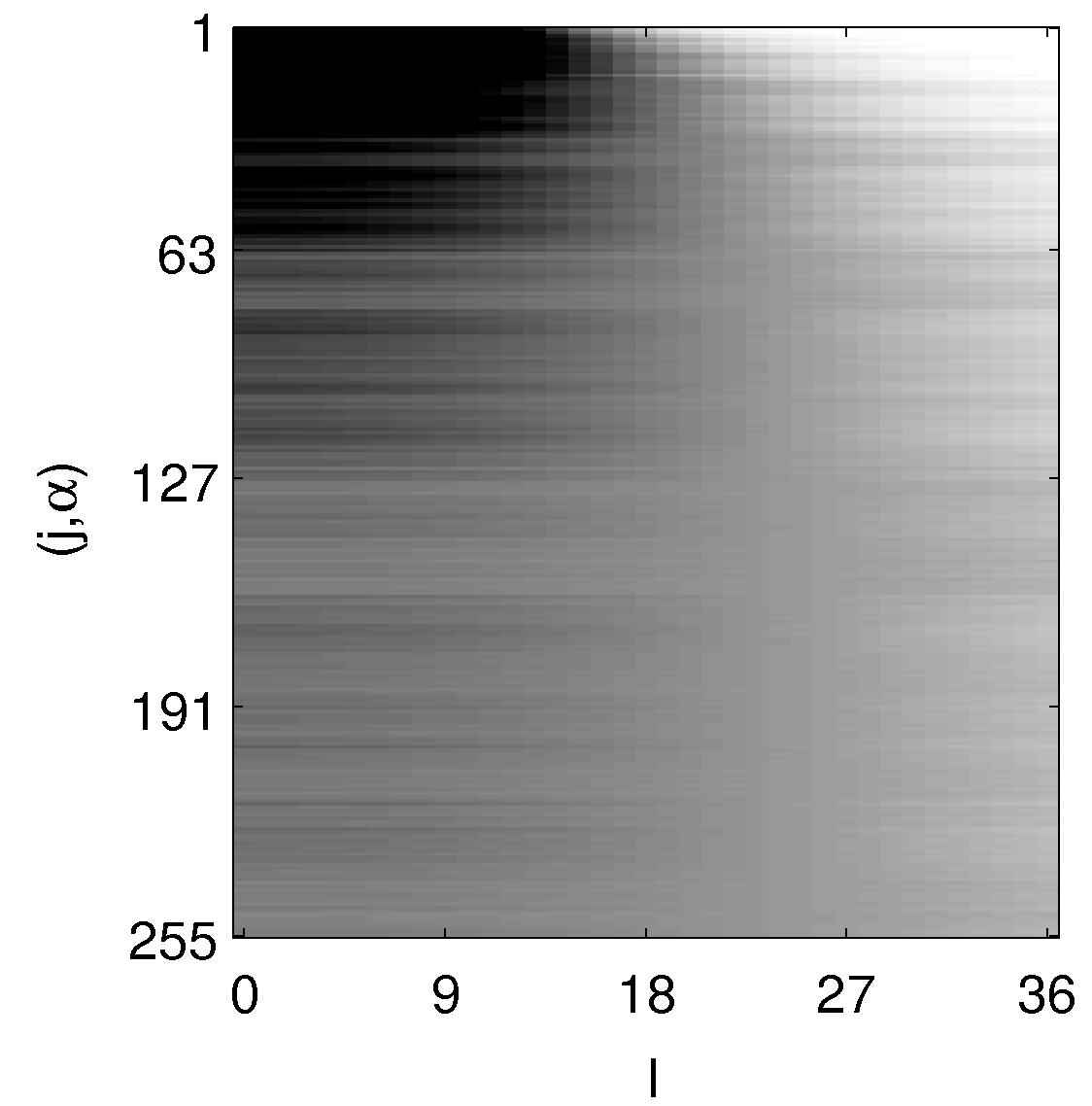}
\includegraphics[width=.49\textwidth]{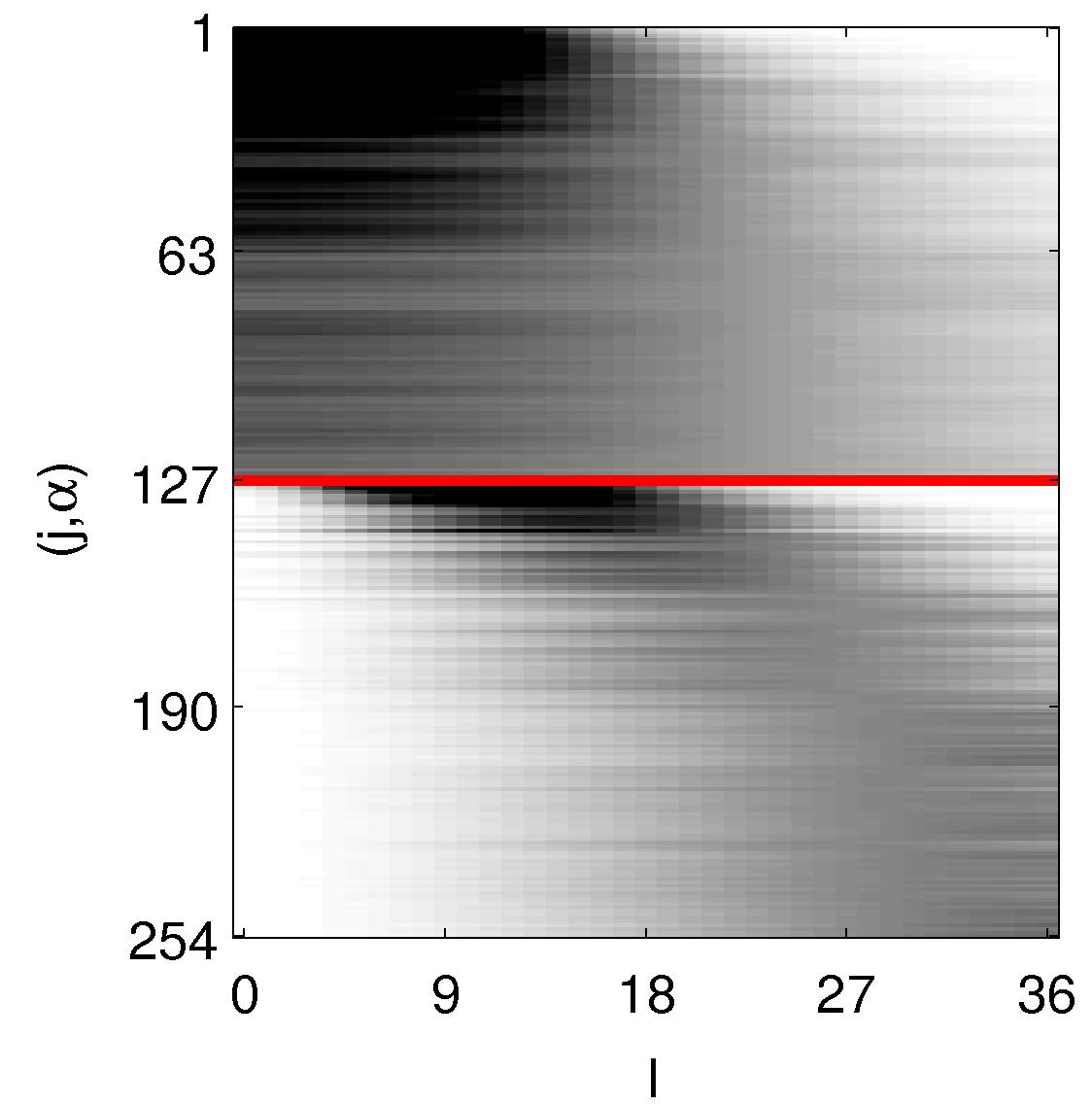}
\caption[Power Spectra $S_l$ of the
  elements of the dictionaries ${\cD}_{\text{Africa},36,1}$ and
  ${\cD}_{\text{Africa},36,1}$.]{\label{fig:sleptrafricaspec}Power
  Spectra $S_l$ of the elements of the dictionaries
  ${\cD}_{\text{Africa},36,1}$ (left) and ${\cD}_{\text{Africa},36,1}$
  (right).  The x-axis is degree $l$, the y-axis is dictionary index
  $(j,\alpha)$.   
  On the right pane, the red line separates the 127 elements with
  $\alpha=1$ (top) and $\alpha=2$ (bottom).
  Values are between 0 (white) and 1 (black).}
\end{figure}

A complementary answer to the question of the range of $\cD$ is given by
studying the angle between the subspaces spanned by elements of $\cD$
and the first $\alpha$ functions of the Slepian basis, for
${\alpha=1,2,\ldots}$~\cite{Wedin1983}.  The angle between two subspaces
$A$ and $B$ of $\bbC^n$ (having possibly different dimensions), is
given by the formula
\begin{align}
\angle(A,B) &= \min\left(\sup_{x \in A} \angle(x,B), \sup_{y \in B} \angle(y,A)\right), \quad \text{where} \\
\angle(x,B) &= \inf_{y \in B} \angle(x,y) = \cos^{-1} \frac{\norm{P_B x}}{\norm{x}}.
\end{align}

Here, $P_B$ is the orthogonal projection operator onto space $B$ and
all of the norms are with respect to the given subspace.
The angle $\angle(A,B)$ is symmetric, nonnegative, and zero iff
$A \subset B$ or $B \subset A$; furthermore it is invariant under
unitary transforms applied to both on $A$ and $B$ (such as Fourier
synthesis), and admits a triangle inequality.  It is thus a good
indicator of distance between two subspaces; furthermore, it can be
calculated accurately\footnote{See MATLAB function \texttt{subspace}.}
given two matrices whose columns span $A$ and $B$.
We can therefore identify the matrices $A$ and $B$ with the subspaces
spanned by their columns.

Let $\left(\wh{G}_{\cR,L}\right)_{1:\alpha}$ denote the matrix containing the
first $\alpha$ column vectors of $\wh{G}$ from \eqref{eq:slepsphereeig}.
Further, let $\wh{D}$ denote the $(L+1)^2 \x \abs{\cD_{\cR,L,n_b}}$ matrix
containing the spherical harmonic representations of the elements of
$\cD_{\cR,L,n_b}$.  Fig.~\ref{fig:slepvstrdist} shows
$\angle(\wh{G}_{1:\alpha}, \wh{D})$ for ${\cR=\text{Africa}}$ with ${L=36}$. 
The Shannon number is ${N_{\text{Africa},36} \approx 79}$ (see~%
Fig.~\ref{fig:slepafricaeigvals}~and~\eqref{eq:shansphsum}).
From this figure, it is clear that while the dictionaries
$\cD_{\text{Africa},36,1}$ and $\cD_{\text{Africa},36,1}$ do not
strictly span the space of functions bandlimited to $L=36$ and
optimally concentrated in Africa, they are a close approximation: the
column span of $\left(\wh{G}_{\text{Africa},36}\right)_{1:\alpha}$ is
nearly linearly dependent with the spans of
$\wh{D}_{\text{Africa},36,1}$ and $\wh{D}_{\text{Africa},36,2}$, for
$\alpha$ significantly larger than $N_{\text{Africa},36}$.

\begin{figure}[ht]
\centering
\includegraphics[width=.49\textwidth]{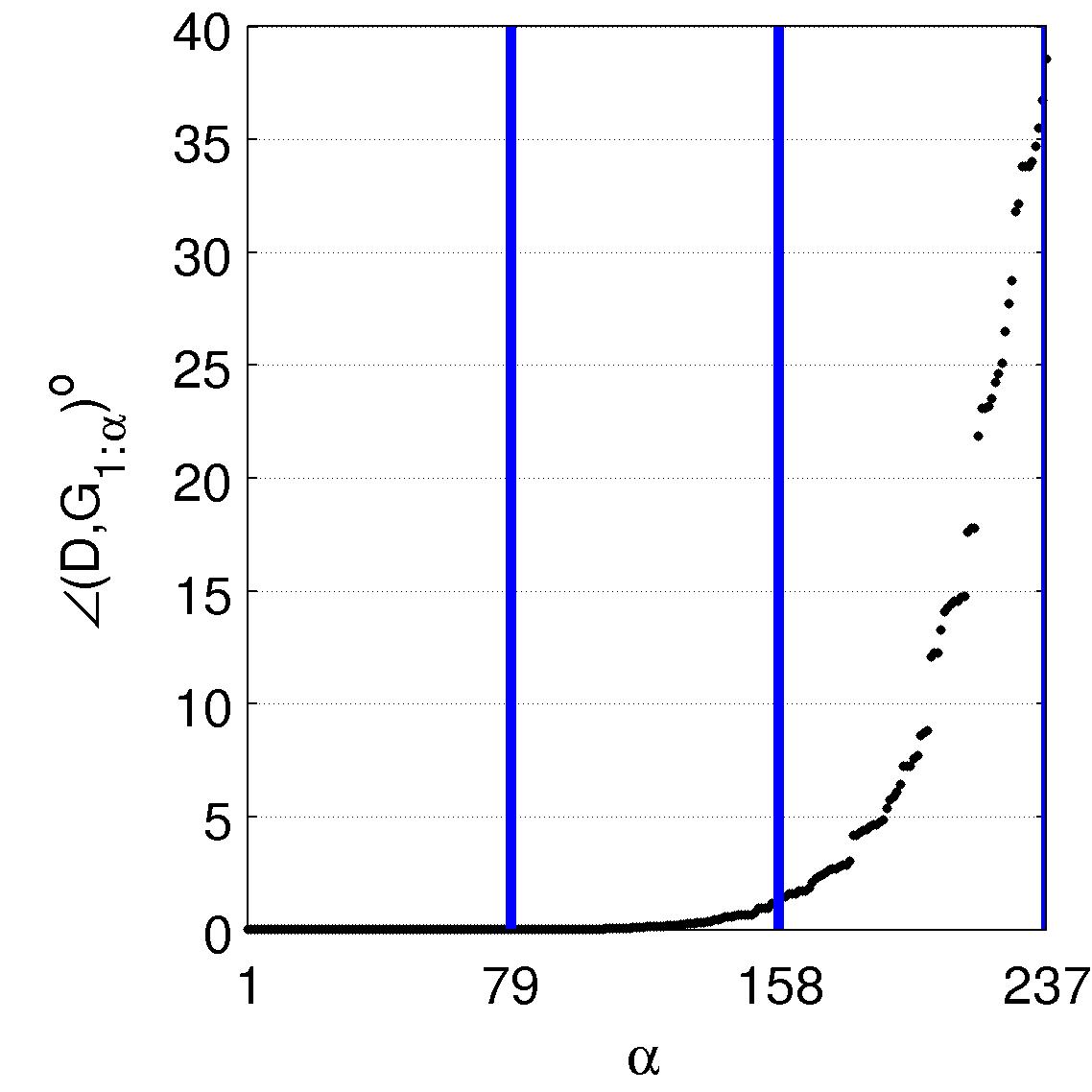}
\includegraphics[width=.49\textwidth]{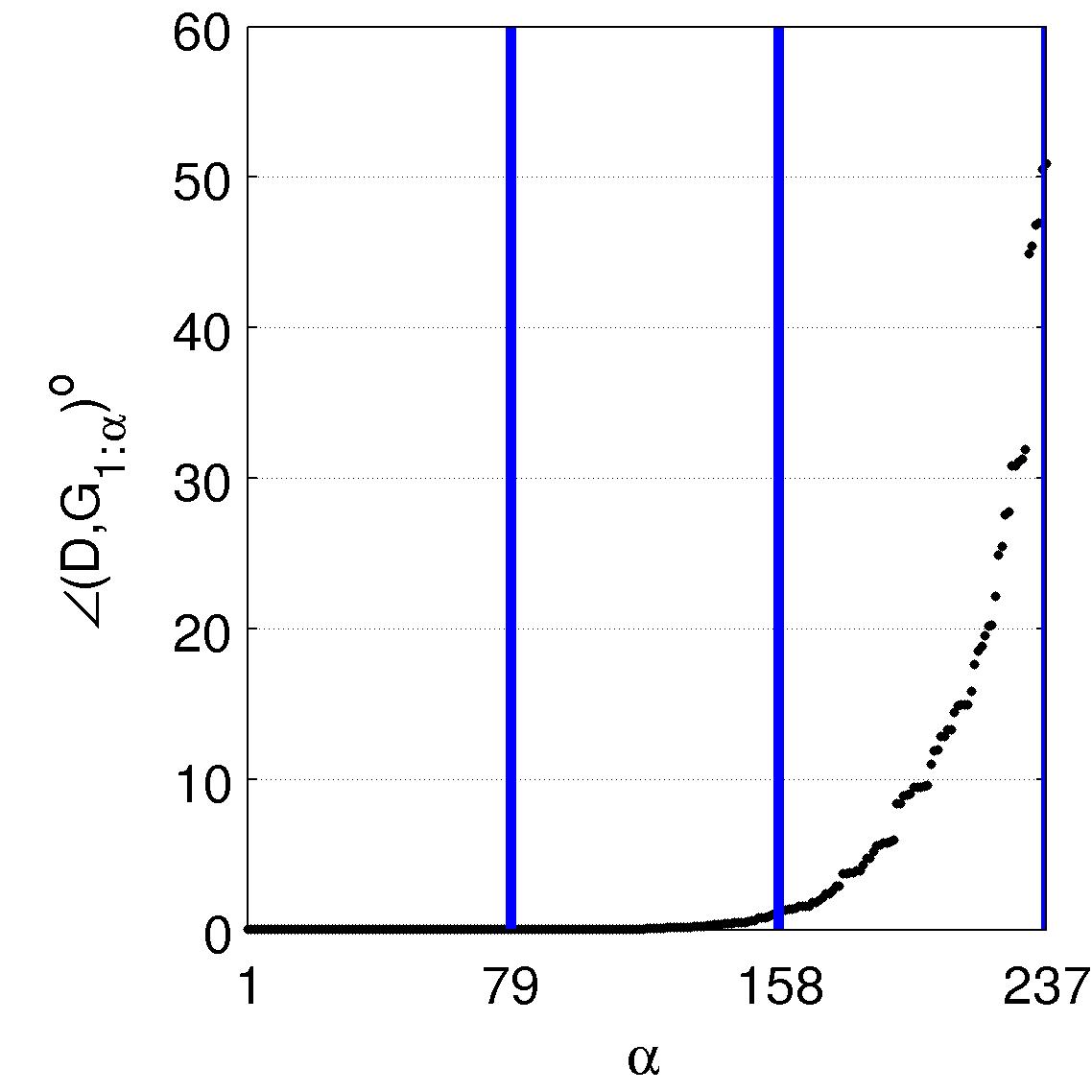}
\caption[Angles between the spaces spanned by $\left(\wh{G}_{\text{Africa},36}\right)_{1:\alpha}$ and the
  dictionary matrices $\wh{D}_{\text{Africa,36,1}}$ and
    $\wh{D}_{\text{Africa,36,2}}$, $\alpha=1,2,\ldots,3
  N_{\text{Africa},L}$.]{\label{fig:slepvstrdist}Angles (in degrees)
  between the spaces spanned by
  $\left(\wh{G}_{\abs{\cR},36}\right)_{1:\alpha}$ and the dictionary
  matrices $\wh{D}_{\cR,36,1}$ (left) and $\wh{D}_{\cR,36,2}$ (right),
  $\cR=\text{Africa}$, $\alpha=1,\ldots,3 N_{\abs{\cR},L}$.
  Thick lines correspond to integer multiples of the Shannon number~%
  ${N_{\text{Africa},36} \approx 79}$.}
\end{figure}

Thanks to novel approaches to signal approximation, which we will
discuss in the next section, the requirement that the dictionary
elements form an orthogonal basis is less important than the property
of mutual incoherence. Mutual  incoherence in a dictionary
means that the inner product (the angle, when elements are of unit norm)
between pairs of elements is almost always very low.
Figs. \ref{fig:sleptrip} and \ref{fig:sleptripecdf} numerically show
that the two tree constructions on continental Africa have good
incoherency properties: most dictionary element pairs are nearly
orthogonal.

\begin{figure}[ht]
\centering
\begin{subfigure}[b]{.49\linewidth}
\includegraphics[width=\textwidth]{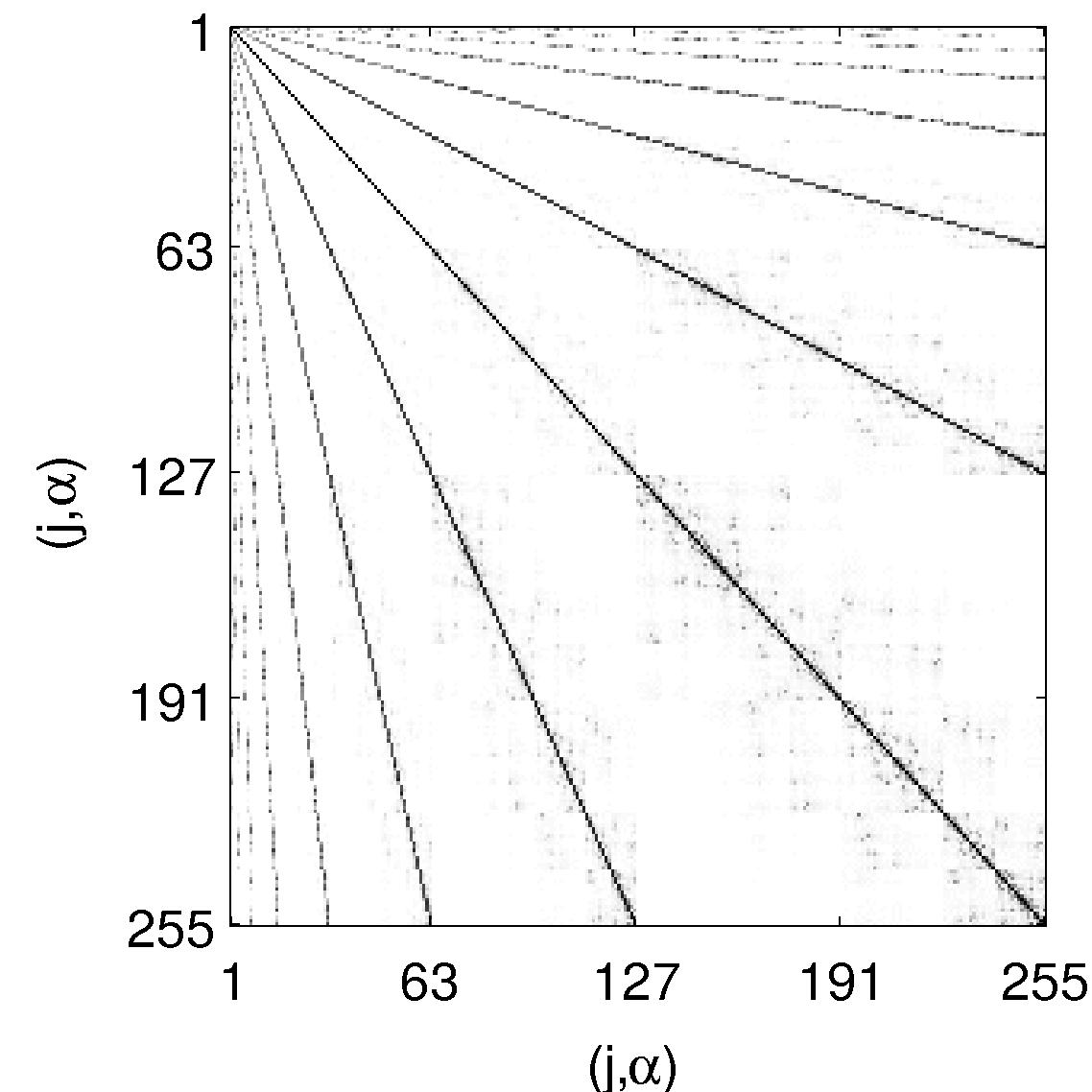}
\caption{$\cD_{\text{Africa},36,1}$}
\end{subfigure}
\begin{subfigure}[b]{.49\linewidth}
\includegraphics[width=\textwidth]{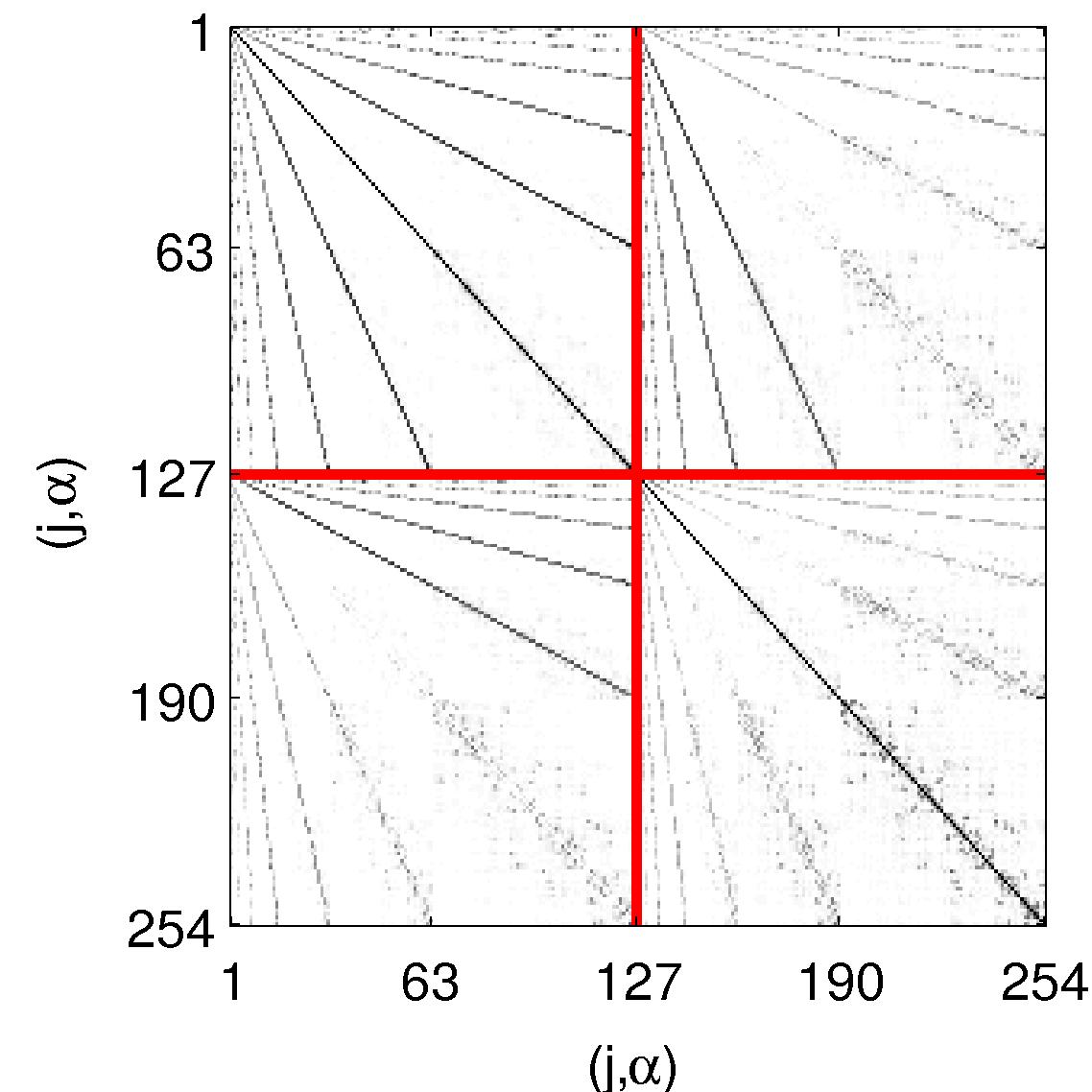}
\caption{$\cD_{\text{Africa},36,2}$}
\end{subfigure}
\caption[Magnitudes of pairwise inner products of dictionaries
  $\cD_{\text{Africa},36,1}$ and~$\cD_{\text{Africa},36,2}$]%
  {\label{fig:sleptrip}Magnitudes of pairwise inner products of
    dictionaries (a)~$\cD_{\text{Africa},36,1}$ and~(b)~$\cD_{\text{Africa},36,2}$.
    In (b), as before, thick lines separate inner products
    between the 127 elements with $\alpha=1$ (top left) and $\alpha=2$
    (bottom right), and their cross products.
    Values are between 0 (white) and 1 (black).
  }
\end{figure}

\begin{figure}[h]
\centering
\includegraphics[width=.49\textwidth]{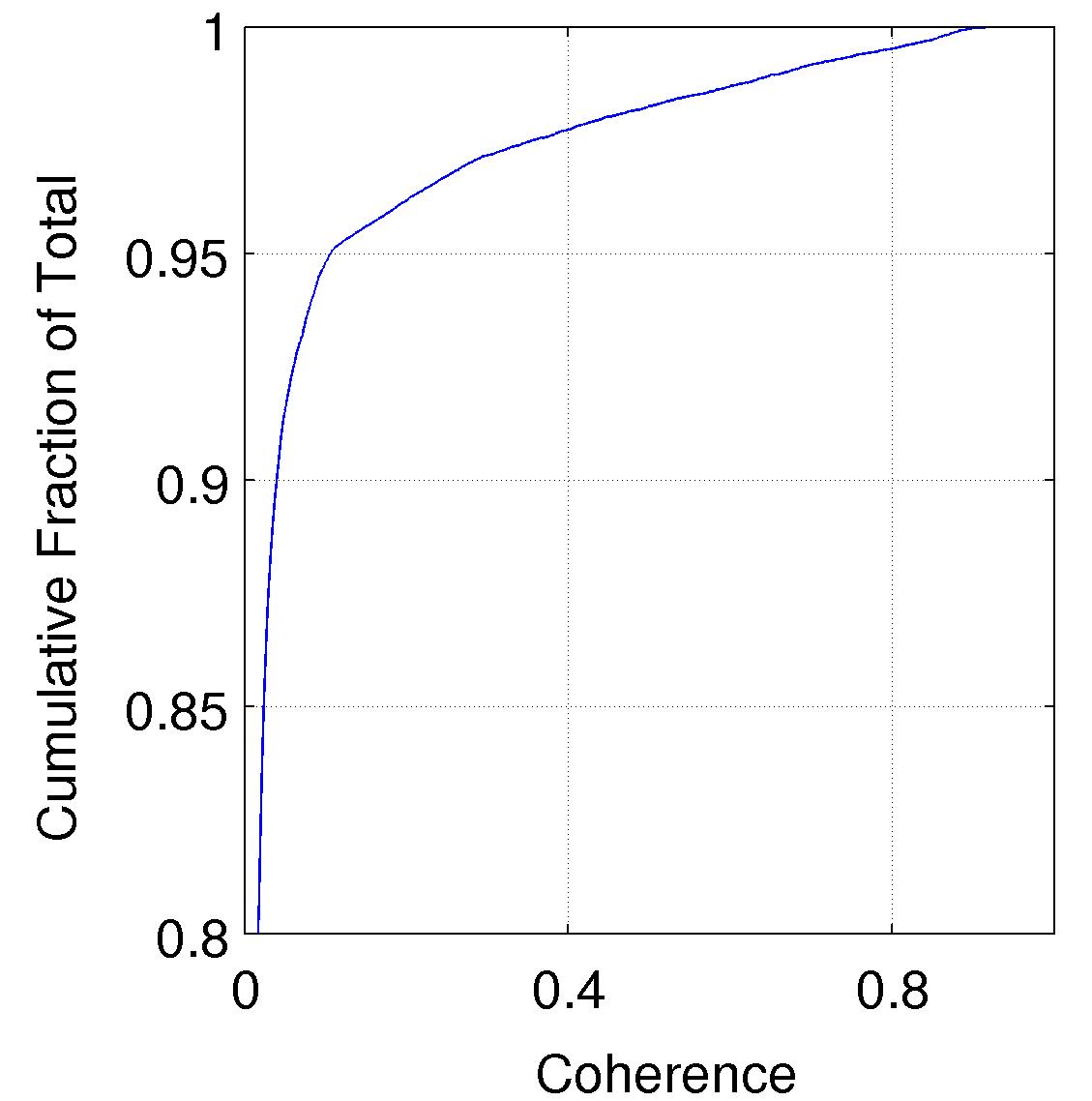}
\includegraphics[width=.49\textwidth]{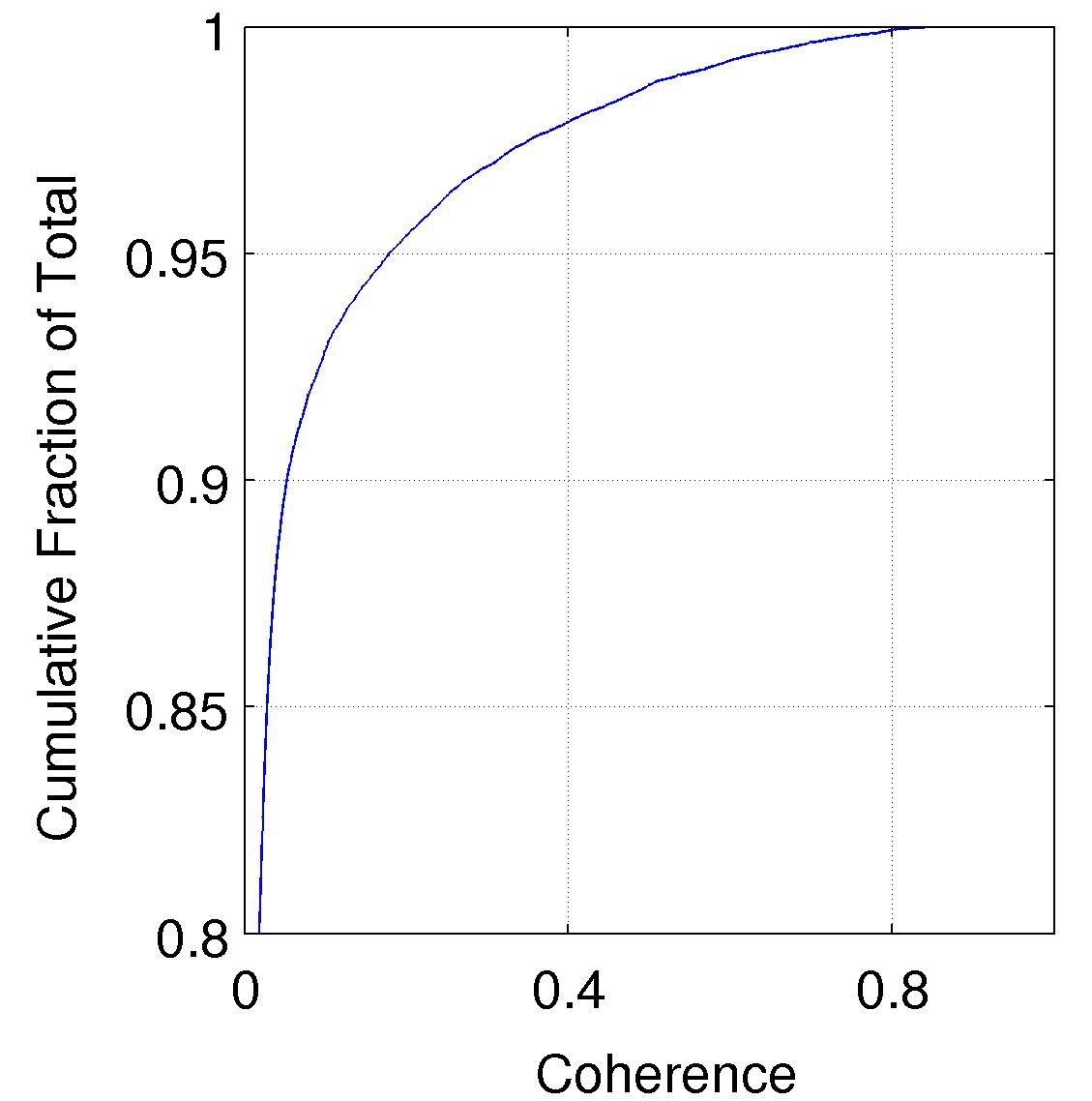}
\caption[Empirical CDFs of pairwise inner product magnitudes for
  dictionaries $\cD_{\text{Africa},36,1}$ and
  $\cD_{\text{Africa},36,1}$]{\label{fig:sleptripecdf}Empirical
  Cumulative Distribution Functions (CDFs) of pairwise inner product
  magnitudes for dictionaries $\cD_{\text{Africa},36,1}$ (left) and
  $\cD_{\text{Africa},36,1}$ (right).  Note, in both cases,
  approximately $95\%$ of pairwise inner products have a value within $\pm 0.1$.}
\end{figure}

In Fig. \ref{fig:sleptrip}(a), most pairwise inner products are
nearly zero, with the exception of nodes and their ancestors, which
share their parents' regions.  More specifically, as expected,
dictionary elements ${(j,1)}$, ${(2j,1)}$, ${(2j+1,1)}$, ${(2(2j),1)}$,
${(2(2j)+1,1)}$, ${(2(2j+1), 1)}$, ${(2(2j+1)+1,1), \ldots}$, tend to
have large inner products, while those elements with non-overlapping
borders do not.  This exact property is also visible in the two
diagonal submatrices of Fig. \ref{fig:sleptrip}(b).  In the
off-diagonals, due to the orthonormality of the construction, elements
of the form $(j,1)$ and $(j,2)$ are orthogonal.  In contrast, due to
the nature of the tree subdivision scheme, elements of the
form $(2j,1)$ or $(2j+1,1)$ and $(j,2)$ have a large magnitude inner
product.  However, the number of connections between nodes and their
ancestors is $O(n_b\, (2^H H))$, while the total number of pairwise
inner products is $O((n_b 2^{H})^2)$; and for reasonably sized values
of $L$ the ratio of ancestral connections to pairwise inner products
grows small (see Fig. \ref{fig:sleptripecdf}).

%% An important followup question to that of mutual coherence is whether
%% the dictionary forms a frame on $\cR$.
%% A frame on $\cR$ is a set of (possibly linearly dependent) functions
%% $\set{e_j}$ for which there exist values $0 < A \leq B < \infty$
%% such that for any $f \in L^2_\Omega(S^2)$,
%% $$
%% A \norm{f}^2_\cR \leq \sum_{j} \ip{f}{e_j}_\cR \leq B \norm{f}^2_\cR,
%% $$
%% where the norms and inner products are all restricted to the domain
%% $\cR$.  Frames allow the fast and stable
%% reconstruction~\cite[Ch.~5]{Mallat2009} of the original function $f$
%% from its frame coefficients $\set{\ip{f}{e_j}_\cR}$.

%% Due to the purely numerical construction of $\cD_\cR$ it is still
%% unclear if this is the~case.

\section{\label{sec:sltrlin}Solution Approaches for Linear~Systems}
As we will show in \S\ref{sec:sltrapprox}, the signal approximation
problem on the sphere reduces to a linear problem of the form
\beq
\label{eq:linrecon}
y = A~x,
\eeq
where $y \in \bbR^m$ are samples of a function on a region ${\cR
\subset S^2}$, $x \in \bbR^n$ are estimate coefficients in
the given dictionary, and ${A \in \bbR^{m \x n}}$ represents a spatial
discretization of the dictionary elements.

While the number of samples $m$ is usually larger than the number of
dictionary elements, in practice, due to the nature of sampling and
discretization (e.g.,~\cite[Fig.~1b]{Slobbe2011}), the rank~$r$~of~$A$
is significantly lower than~$m$.  There are several examples
in the literature concerning spherical harmonics, for which
\eqref{eq:linrecon} is invertible (that is, with $r = n \approx m$).
For example, it can be shown via a Shannon-type theorem that when $x$
represents the spherical harmonic coefficients of a function with
bandlimit $L$, $y$ represents its samples on a special semi-regular
grid, and $A$ represents the discretization of the harmonics to the
grid, the signal $x$ can be reconstructed
exactly~\cite[Thm.~3]{Driscoll1994}.  In this 
case, the semi-regular grid consists of $m=4 L^2$ points on the entire
sphere $S^2$, and the number of harmonic coefficients (also the rank
of $A$) is~${r=n=(L+1)^2}$.  In the limit of large bandlimit $L$,
the ratio of the samples to unknowns is $4$.  As such, the sampling
scheme presented in that paper is in a sense optimal in order of
magnitude.  Nevertheless, even there the rank is significantly lower
than the sample number.

Under the assumption that $r < m$, we can write the
compact~Singular~Value~Decomposition~(SVD)~\cite[Chapter~7]{Horn1985}~of~$A$:
\beq
\label{eq:svdA}
A = U^{}_+ \Sigma^{}_+ V_+^* , \quad U_+^* U^{}_+ = I_{r \x r},
\text{ and } V_+^* V^{}_+ = I_{r \x r}.
\eeq
Here $U_+ \in \bbR^{m \x r}$, $\Sigma_+ \in \bbR^{r \x r}$ is a
diagonal matrix of positive singular values, and 
$V_+ \in \bbR^{n \x r}$.  We can thus rewrite \eqref{eq:linrecon} as
$$
\wh{y} = \wh{A} x, \quad \wh{y} =
U_+^* y, \text{ and } \wh{A} = \Sigma^{}_+ V_+^*,
$$
where $\wh{A} \in \bbR^{r \x n}$ and $r$ may be smaller than $n$.

We must therefore focus on \eqref{eq:linrecon} with all three
possible cases: the overdetermined case $m > n$ (with rank $r=n$), the
case $m = n$ (with rank $r=m=n$), and especially the underdetermined
case $m > n$ (with rank $r=n$).  We first quickly review the
overdetermined case.

\subsection{\label{sec:sltrover}Overdetermined and Square Cases: $m \geq n$}
When the system \eqref{eq:linrecon} is overdetermined, there many not
be one exact solution $x$.  However, in this case it is possible
to minimize the sum of squared errors
\beq
\label{eq:minlsqr}
x_M = \argmin_{x \in \bbR^n} \norm{y - A x}_2^2.
\eeq
Taking gradients with respect to vector $x$ and setting the resulting
system equal to zero, we get the least squares solution
\beq
\label{eq:linlsqr}
x_M = (A^* A)^{-1} A^* y.
\eeq

Note that \eqref{eq:minlsqr} arises from the Maximum Likelihood
formulation of a statistical model in which the samples $y$ are
observed as~\cite[Chapter~3]{Bishop2006}
$$
y = Ax + n,
$$
where $n \sim \cN(0,\sigma^2 I_{m \x m})$ is i.i.d Gaussian noise.

When $m=n=r$, \eqref{eq:linrecon} has exactly one solution.  The
matrix $A$ and its conjugate are invertible and,
using the identity $(AB)^{-1} = B^{-1}A^{-1}$,
\eqref{eq:linlsqr} reduces to ${x_M = A^{-1} x}$, as expected.

\subsection{\label{sec:sltrunder}Underdetermined Case: $m < n$}

When the linear system $A$ has more unknowns than equations
(or~rank~$r$), additional modeling or regularization is required.  We
discuss two possible statistical models on $x$, their limiting
cases, and the resulting computational considerations.

\subsubsection{Prior: $x$ distributed according to a Gaussian distribution}
One possible way to model the underdetermined case is via a
statistical model in which the coefficients of $x$ are generated
i.i.d. with a zero-mean, fixed variance Gaussian distribution:
\begin{align*}
y|x &= Ax + n, \\
n &\sim \cN(0,\sigma I_{m \x m}), \text{ and} \\
x &\sim \cN(0,\sigma_s I_{n \x n}).
\end{align*}

The maximum a posteriori (MAP) estimate follows from Bayes' rule:
\begin{align*}
\Pr(x|y) &\propto \Pr(y|x)\Pr(x) \\
 &= \frac{1}{(2 \pi)^{(m+n)/2} \sigma^{m/2} \sigma_s^{n/2}}
  \exp\set{-\left(\norm{y-Ax}^2_2/2 \sigma^2 + \norm{x}^2_2/2 \sigma_s^2\right)}.
\end{align*}
When $\sigma$ and $\sigma_s$ are fixed, maximizing $\Pr(x|y)$ is
equivalent to minimizing its negative logarithm.  The MAP problem in
this case becomes
\beq
\label{eq:minlsqrl2}
x_{N} = \argmin_{x \in \bbR^n} \norm{y-Ax}_2^2 + \gamma^2 \norm{x}_2^2,
\eeq
where ${\gamma = \sigma / \sigma_s}$.

Simple calculus again provides the solution, also known as the
Tikhonov regularized solution to \eqref{eq:minlsqrl2}:
\beq
x_{N} = (A^* A + \gamma^2 I)^{-1} A^* y.
\eeq

When the model noise power $\sigma$ grows small with respect to the
signal power $\sigma_s$, the regularization term $\gamma$ goes to
zero.  In this limiting case, we can write the limiting
solution~as~\cite[pp~421-422]{Horn1985}
\beq
\label{eq:minlsqrMP}
x_{MP} = \lim_{\gamma^2 \to 0} (A^* A + \gamma^2 I)^{-1} A^* y = A^\dagger y,
\eeq
where $A^\dagger$ is the Moore-Penrose generalized inverse of $A$.
For overdetermined and exactly determined systems, $A^*A$ has a well
defined inverse and $A^\dagger$ coincides with the matrix in
\eqref{eq:linlsqr}.  For underdetermined systems, the matrix
$A^\dagger$ is still well defined, and is given by
$$
A^\dagger = U^{}_+ \Sigma_+^{-1} V_+^*,
$$
where $U,\Sigma,V$ are given by the SVD as in \eqref{eq:svdA}.  Note
that $x_{MP}$ also solves the convex, quadratic optimization
problem
$$
x_{MP} = \argmin_{x \in \bbR^m} \norm{x}_2^2 \text{ subject to }
y = Ax.
$$

The Tikhonov and Moore-Penrose solutions \eqref{eq:minlsqrl2} and
\eqref{eq:minlsqrMP} are a common approach to solving underdetermined
inverse problems in the Geosciences literature
(see~e.g.,~\cite{Xu1998}).  However, depending on the dictionary used
for the representation of $x$ the Gaussian prior may not be an
ideal one; as it encourages \emph{all} of the coefficients to be
nonzero.

One of the major underlying foundations of this work includes
recent results in representation theory, which have shown that
overcomplete (redundant) multiscale frames and dictionaries with
certain incoherency properties can provide stable and noise-robust
estimates to ill-posed inversion problems.  The basic requirement in
the estimation stage is that the solution is as ``simple'' as possible:
most of the coefficients in of the solution $x$ are zero; $x$ is
\emph{sparse}.  We discuss this next.

\subsubsection{Prior: $x$ distributed according to a Laplace distribution}
It is well known in the statistics community (and, most
recently, in the Compressive Sensing literature), that applying a
super Gaussian prior $\Pr(x)$ induces MAP solutions that have many
zero components and a few large magnitude ones.  The zero-mean Laplace
distribution is one such particularly convenient distribution.  We
model each component of $x$, $x_i$ as i.i.d. Laplace distributed:
$$
\Pr(x) = \prod_{i =1}^n \frac{1}{2b} \exp(-\abs{x_i}/b) =
\frac{1}{(2b)^n} \exp(-\norm{x}_1/b).
$$

In a manner identical to that of the previous section, for fixed noise
power $\sigma^2$ and signal scale $b>0$, the MAP problem can be
reduced to
\beq
\label{eq:minlsqrl1}
x_{L} = \argmin_{x \in \bbR^n} \norm{y-Ax}_2^2 + \eta \norm{x}_1,
\eeq
where $\eta = 2 \sigma / b$.  For any $\eta>0$, there is a $t > 0$
such that the following problem is identical:
\beq
\label{eq:lasso}
\argmin_{x \in \bbR^n} \norm{y-Ax}_2^2 \text{ subject to } \norm{x}_1
\leq t,
\eeq
where for a given $\eta$, there exists a $t(\eta) > 0$ such that
\eqref{eq:lasso} gives a solution identical to \eqref{eq:minlsqrl1},
and $t(\eta)$ decreases monotonically with $\eta$.  Furthermore, for any
given $t$ to \eqref{eq:lasso}, an $\eta$ can be found for
\eqref{eq:minlsqrl1} that provides the identical
solution~\cite[\S12.4.2]{Mallat2009}; this result essentially follows
from the method of Lagrange multipliers.  In other words, the two
convex problems are completely equivalent.  They are also equivalent
to the popularly studied Compressive Sensing problem
$$
\argmin_{x \in \bbR^n} \norm{x}_1 \text{ subject to } \norm{y-Ax}_2^2
\leq \epsilon,
$$
where $\epsilon$ grows monotonically in $\eta$ and/or $t^{-1}$.

Fast and robust solvers for the convex quadratic optimization problems
\eqref{eq:minlsqrl1} and \eqref{eq:lasso} have been the subject of
study for many years.  For our calculations
we use the LASSO solver\footnote{The LASSO \texttt{glmnet}
  package:~\url{http://www-stat.stanford.edu/~tibs/lasso.html}.}
(see,~e.g.,~\cite{Tibshirani1996}).
LASSO is an iterative solution method that provides the full solution
paths to \eqref{eq:lasso}.  It is computationally efficient for small-
to medium- scale linear systems.  For systems with more than tens
of thousands of unknowns, there are a variety of other techniques for
the solution of \eqref{eq:lasso} that are
much more computationally tractable, though possibly less accurate
(see, e.g.,~\cite{Daubechies2010b}~and~\cite[\S12.4,~\S12.5]{Mallat2009}).

\subsubsection{Debiasing the $\ell_1$ solution}

An alternative approach to finding a sparse solution of
\eqref{eq:linrecon} is to attempt to minimize the problem
\beq
\label{eq:min0norm}
\min_{x \in \bbR^n} \norm{y-Ax} \text{ subject to } \norm{x}_0 \leq \kappa,
\eeq
where $\norm{x}_0 = \abs{\supp{x}}$ and 
${\supp{x} = \set{i : x_i \neq 0}}$.  That is, find the best matching
data to the model where the number of nonzero coefficients of the
model is bounded by $\kappa$.
Unfortunately, this combinatorial problem is nonconvex and therefore
usually intractable.  In some cases, it can be shown that the solution
is equivalent to the $\ell_1$ problem \eqref{eq:lasso}, when the
matrix $A$ fulfills one of a number of special properties, e.g., the
Restricted Isometry Property (RIP) or Null Space Property (NSP).  This
result is, in fact, a celebrated equivalence result in Compressive
Sensing~\cite{Candes2006,Donoho2006}.  Unfortunately, for physical
discretization matrices $A$, the RIP and its equivalents are
difficult to check~\cite{dAspremont2008,Juditsky2008}.

In practice, problem \eqref{eq:min0norm} can be reduced into two parts:
\begin{enumerate}
\item Estimate the support of $x$, $\cS = \supp{x}$, such that $\abs{\cS} < r$.
\item Solve the overdetermined system \eqref{eq:minlsqr} via
  \eqref{eq:linlsqr} on the reduced set $\supp{x}$ by
  keeping only the columns of $A$ associated with $\supp{x}$,
  $A_{\supp{x}}$, when solving the least squares problem.
\end{enumerate}
A tractable solution to the first part is to use the output
of the $\ell_1$~(LASSO) estimator:
$$
\cS = \supp{x} = \supp{x_L}.
$$
The final estimate is the vector $x_D$ where
\begin{align}
\label{eq:minsqrl1deb}
x'_{D} &= \argmin_{x' \in \bbR^{\abs{\cS}}} \norm{y-A_{\cS} x'}_2^2,
\text{ and }
\\
\nonumber x_{D,i} &= x'_{D,i} \delta_{i \in \cS},
\quad i=1,\ldots,n.
\end{align}
This alternative solution is also sometimes called the \emph{debiased}
$\ell_1$ solution, because after the $\ell_1$ minimization step, the
bias of the $\ell_1$ penalty is removed via least squares on the
estimated support.

As we will show in the next section, this final combination of
support estimate based on a sparsity-inducing prior, followed by
the solution to an overcomplete least squares problem on this support
set, allows for improvements in signal approximation over currently
standard techniques in geophysics.

\section{\label{sec:sltrapprox}Signal Approximation Models for Subsets of the Sphere~$S^2$}

We now turn to the problem of estimating a signal from noisy and/or
incomplete observations on a subset $\cR$ of the sphere.  Following
the notation of~\cite{Simons2010}, suppose we observe data (samples of
some function $f$) on a set of points within the region $\cR$,
consisting of signal $s$ plus noise $n$.  We are interested in
estimating the signal within~$\cR$ from these samples.

While in practice, most signals of interest in geophysics are not
bandlimited, this assumption allows us to perform estimates, and can
be thought of as a regularization of the signal, similar in nature to
assumptions of a maximum frequencies in audio analysis.  Furthermore,
as in 1D signal processing, constraints on physical sampling and high
frequency noise always reduce the maximum determinable frequency.
See, for example, a noise analysis for satellite observations in the GRACE
mission~\cite[Fig.~1]{Wahr1998}, and the effects of noise on power
spectral estimation for the CMB dataset~\cite[Fig.~12]{Dahlen2008}.

Let ${\cX = \set{x_i}_{i=1}^m, x_i \in \cR}$ be a set of points on
which data $f$ is observed, and let the corresponding observations be~%
$\set{f_i = f(x_i)}_{i=1}^m$, which we denote with the vector~$\tf$.
Then via the harmonic expansion \eqref{eq:plm2xyz}, we can write
\beq
\label{eq:fmodelsph}
f_i = \sum_{l=0}^L \sum_{m=-l}^l \wh{s}_{lm} Y_{lm}(\theta_i,\phi_i) + \nu_i,
\eeq
where the $\wh{s}_{lm}$ are the harmonic expansion coefficients of the
signal $s$ and $\nu_i = \nu(x_i)$ is a realization of the noise.  We will also
denote by $\tnu$ the vector of samples of the noise process.  Let
$Y$ be the $\abs{\cX} \x (L+1)^2$ harmonic sensing matrix, with 
$$
{Y_{i,lm} = Y_{lm}(\theta_i,\phi_i)} \text{ for } i =
1,\ldots,\abs{\cX}, \text{ and } (l,m) \in \Omega.
$$
Then we can rewrite \eqref{eq:fmodelsph} as
\beq
\label{eq:fmodelsphmat}
\tf = Y \wh{s} + \tnu.
\eeq
By restricting the bandlimit to $L$, we restrict the function $s$ to
lie in $L^2_\Omega(S^2)$.  Moreover, we are only interested in
estimating $s$ on $\cR \subset S^2$.

The Slepian functions are another basis for $L^2_\Omega(S^2)$, in which
the functions are ordered in terms of their concentration
on $\cR$.  As such, we may rewrite the samples via their Slepian
expansion:
\beq
\label{eq:fmodelslep}
f_i = \sum_{\alpha=1}^n \wt{s}_\alpha g_\alpha(\theta_i, \phi_i) + \nu_i
\eeq
where now the $\wt{s}_\alpha$ are the Slepian expansion coefficients,
and we are free to constrain $n$ from $1$ through $(L+1)^2$.  By
setting $n$ to $N_{\abs{\cR},L}$, we concentrate the estimate to
$\cR$, while choosing $n=(L+1)^2$ leads to a representation equivalent
to \eqref{eq:fmodelsph}.

Using the harmonic expansion of the Slepian functions, we rewrite
\eqref{eq:fmodelslep} via the spherical harmonics:
\beq
\label{eq:fmodelslepsph}
f_i = \sum_{\alpha=1}^n \wt{s}_\alpha \sum_{l=0}^L \sum_{m=-l}^l
\wh{g}_{\alpha,lm} Y_{lm}(\theta_i,\phi_i) + \nu_i,
\eeq
and, using the terminology $\wh{G}$ of \S\ref{sec:sltrprop} to
denote the harmonic expansion matrix of the Slepian functions,
and the ``colon'' notation to denote restrictions of matrices and
vectors to specific index subsets,  we can
rewrite~\eqref{eq:fmodelslep} in matrix notation:
\beq
\label{eq:fmodelslepmat}
\tf = Y \wh{G}_{1:n} \wt{s}_{S,{1:n}} + \tnu.
\eeq
As just described, by setting $n < (L+1)^2$ this model assumes that
all but the first $n$ of the Slepian coefficients $\wt{s}_S$ are zero.

Finally, using the Tree dictionary construction
of~\S\ref{sec:sltrtree} and the notation of~\S\ref{sec:sltrprop}, for
a given dictionary $\cD_{\cR,L,n_b}$ we can model the signal with the
linear model
\beq
\label{eq:fmodelsltr}
f_i = \sum_{\alpha=1}^{n_b} \sum_{j=1}^{2^{H+1}-1} \wt{s}_{T,(j,\alpha)}
d^{(j,\alpha)}_{\cR,L,n_b}(\theta_i,\phi_i) + \nu_i.
\eeq
Writing the dictionary elements via their harmonic expansions, the
matrix formulation of \eqref{eq:fmodelsltr} becomes
\beq
\label{eq:fmodelsltrmat}
\tf = Y \wh{D}_{\cR,L,n_b} \wt{s}_T + \tnu.
\eeq
As we will see next, this alternative way of describing bandlimited
functions on $\cR$ has a number of advantages.

\subsection{\label{sec:sltrinv}Regularized Inversion and Numerical Experiments}

In practice, the sensing matrix $Y$ is highly rank
deficient: depending on the sensing grid points $\cX$, its rank $r$
tends to be significantly smaller than the maximum possible value $n$,
$\dim L^2_\Omega(S^2) = (L+1)^2$.  As such, the estimation of
$s$ via direct inversion of \eqref{eq:fmodelsphmat} is ill
conditioned: it must be regularized.

As discussed in \S\ref{sec:sltrlin}, the most common form of
regularization is via the Moore-Penrose pseudoinverse:
\eqref{eq:linlsqr} for overdetermined systems or \eqref{eq:minlsqrMP}
for underdetermined ones.  Following the discussion of
\S\ref{sec:sltrsphere} and the statistical analyses
in~\cite[\S3]{Simons2010} and~\cite[\S7]{Simons2006}, the
``classically'' optimal way to estimate $\wt{s}$ is by restricting the 
reconstruction to be concentrated within $\cR$: that is, first by
choosing a small $n$ in \eqref{eq:fmodelslepmat}, such that $n < r$,
and then applying \eqref{eq:linlsqr} to estimate $\wt{s}_S$.  In
practice, we can consider $n$ ranging from $1$ to $N_{\abs{R},L}$
because the Shannon number is less than rank $r$.  We will call this
first estimation method Slepian~Truncated~Least~Squares~(STLS).

We now propose, first, a simple alternative approach: assume sparsity of
$\wt{s}$ (i.e., with respect to the Slepian basis).  As the Slepian
basis was initially constructed to promote sparsity in the
representation of bandlimited functions concentrated on
$\cR$~[\S3.1.2]\cite{Simons2010}, we expect that this assumption
should lead to estimates that are equivalent to STLS, if not
better.  The basic idea is to let $n=(L+1)^2$ in
\eqref{eq:fmodelslepmat}, and use the solution method
\eqref{eq:minsqrl1deb} with sparsity penalties $\eta_1 \geq \eta_2
\geq \cdots$ sufficiently large that only a few nonzero coefficients
are found in the support.  We consider a range of values $\eta$ from a
maximum $\ol{\eta}$ that induces only one nonzero coefficient, to a
minimum $\ul{\eta}$ that induces $N_{\abs{\cR},L}$.  We call this
estimation method Slepian~$\ell_1$~+~Debias~(SL1D).  Though in this
case we do not explicitly require the estimate to be well concentrated
in $\cR$ via choice of basis functions, by minimizing the squared
error between sample values on $\cX \subset \cR$ we expect that most
of estimated support will be within the first Slepian functions.

With the Tree construction of \S\ref{sec:sltrtree}, we have
a new dictionary of elements that are both concentrated in $\cR$,
bandlimited, and multiscale.  As such, these dictionaries are
excellent candidates for estimation via the $\ell_1$+~Debias technique
\eqref{eq:minsqrl1deb}.  This method is similar to the previous one:
apply \eqref{eq:minsqrl1deb} to the model \eqref{eq:fmodelsltrmat},
choosing a range of $\ell_1$ penalties $\eta$ that lead to between $1$
and $N_{\abs{\cR},L}$ dictionary elements in the support of
$\wt{x}_T$.  We call this the Slepian~Tree~$\ell_1$~Debias~(STL1D) method.

The experiments below numerically show that the two new estimation
(inversion) methods SL1D and STL1D provide improved performance over
the classic STLS, in terms of average reconstruction error over the
domain of interest $\cR$ using a small number of coefficients, for
several important types of bandlimited signals.  Furthermore, as
expected the multiscale and spatially concentrated dictionary elements
of the Tree construction provide improved estimation performance when
the signal is ``red'', i.e., when it contains more energy in the lower
harmonic components.

\subsubsection{Bandpass Filtered POMME Model}
Fig. \ref{fig:pomme4dlocal} shows a bandpass filtered version of the
radial component of Earth's crustal magnetic field, which we will call
$p(\theta,\phi)$: a preprocessed version of the 
output of the POMME model~\cite{Maus2006}.  The signal $p$
has been:
\begin{enumerate}
\item Bandpassed between $l_{\text{min}}=9$ and $l_{\text{max}}=36$.
\item Spatially tapered (multiplied) by the
first Slepian function bandlimited to ${L_t=18}$ and concentrated within
Africa.
\item Low-pass filtered to have maximum frequency~$L=36$~via direct
projection onto the first $(L+1)^2$ spherical harmonics using standard
Riemannian sum-integral approximations~\cite[Eq.~80]{Simons2010},
i.e., direct inversion.
\end{enumerate}
It can be shown~\cite[\S2]{Wieczorek2005} that the harmonics
of the tapered signal, at degree $l$, receive contributions from the
original coefficients in the range from $\abs{l-L_t}$ to $l+L_t$.  As
a result, only the first $l_{\text{max}}-L_t$ degree coefficients are
reliable estimates of the original signal's harmonics.

Samples of $p$ are given via the forward model \eqref{eq:fmodelsph}
(with $\tnu=0$), from the low-pass filtered ``ground truth'' signal
$\wt{p}$, on the intersection of the African continent $\cR$ with the
grid
$$\cX^* = \set{(k_\theta
  \Delta,k_\phi 
  \Delta), \Delta = 0.25^o, k_\theta = \pm 0, \pm 1, \ldots, k_\phi =
  \pm 0, \pm 1, \ldots}
$$
We denote this reduced set $\cX$; it contains $\abs{\cX} = 40250$
points.  For $L=36$, as before, the Shannon number of Africa is
$N_{\text{Africa},36} \approx 79$, the dimension of the bandlimited
space is $\dim L^2_\Omega(S^2) = (L+1)^2 = 1639$, and the rank of
the discretization matrix in \eqref{eq:fmodelsph}, $Y_{1:(L+1)^2}$, is
${r=528} \ll 1639$.

\begin{figure}[ht]
\centering
\includegraphics[width=.65\linewidth]{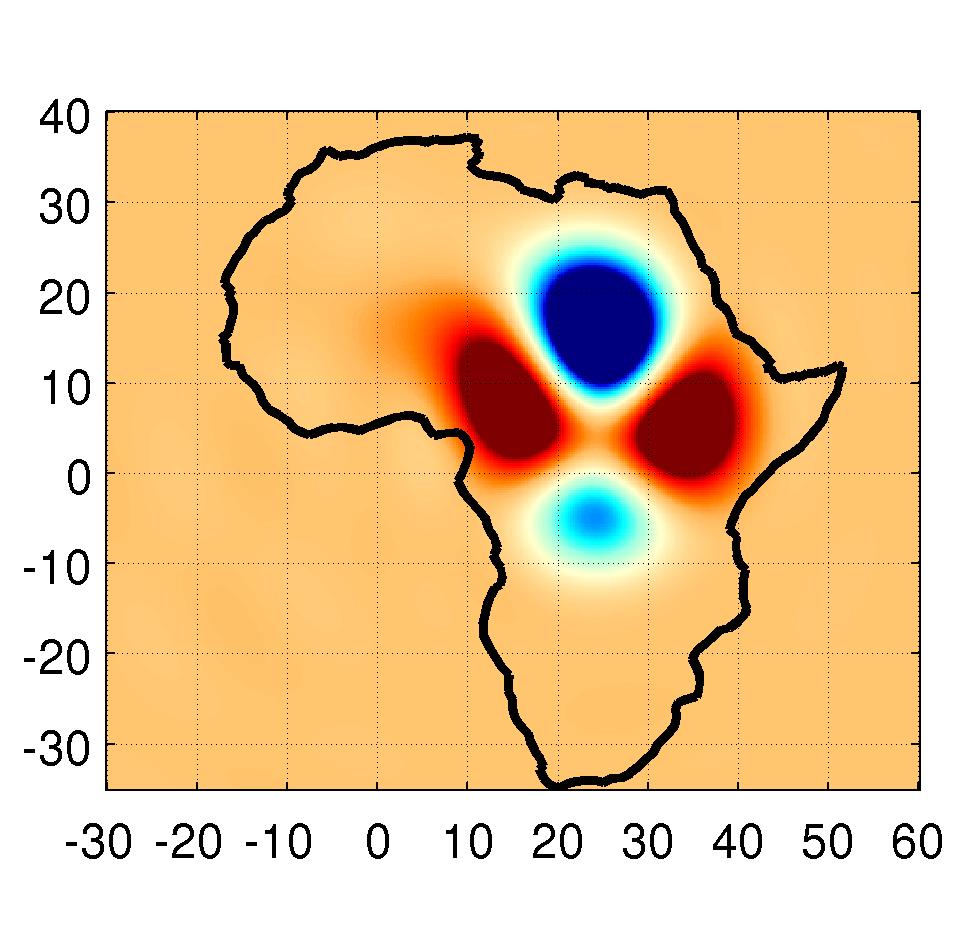}
\caption[The tapered and preprocessed ``ground
  truth'' POMME-4 signal $p$, sampled on the regular grid $\cX^*$
  around Africa.]{\label{fig:pomme4dlocal}The tapered and preprocessed ``ground
  truth'' POMME-4 signal $p$, sampled on the regular grid $\cX^*$
  around Africa.  Red colors are high magnitude positive values, blue
  are high magnitude negative values.}
\end{figure}

Figs. \ref{fig:pomme4slepr}, \ref{fig:pomme4slepl1r}, and
\ref{fig:pomme4sleptrr} show intermediate results in the estimation of
$p$ via the three methods: STLS, SL1D, and STL1D, respectively.
Specifically, they show the absolute error between the original
sampled signal $\tp$, and the expansion via \eqref{eq:plm2xyz} of the
three estimates.  In Fig. \ref{fig:pomme4slepr}, the number of nonzero
coefficients in the estimate is determined by the Slepian truncation
number $n$, while in Figs. \ref{fig:pomme4slepl1r} and
\ref{fig:pomme4sleptrr}, the number of nonzero coefficients are
indirectly determined by the parameter $\eta$ after the support
estimation stage, as per our earlier discussion.  As such, the number
of nonzero components do not always match that of
Fig. \ref{fig:pomme4slepr}; instead nearby values are used when
found.

%% Partial reconstructions for each of 3 types; for dictionary use nb=1
%% Also show tree support regions. (4th plot)
\begin{figure}[h!]
\centering
\begin{subfigure}[b]{.29\linewidth}
\includegraphics[width=\linewidth]{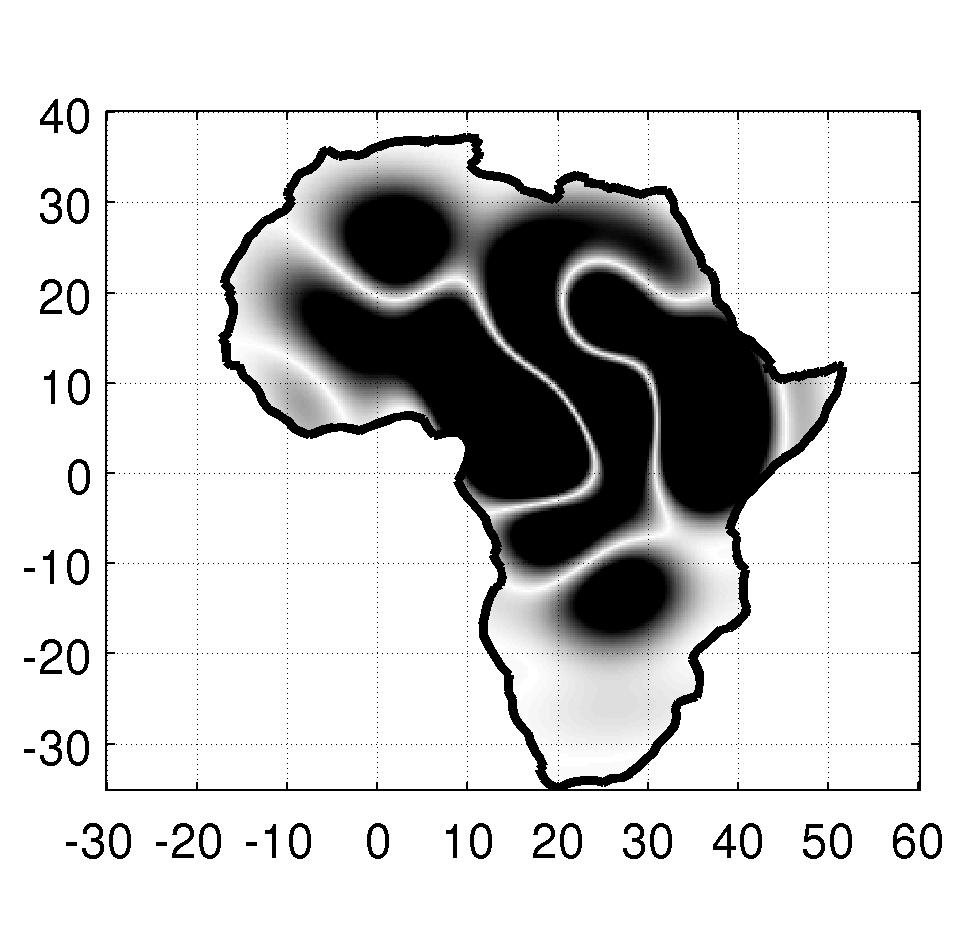}
\caption{$e_{10}$}
\end{subfigure}
\begin{subfigure}[b]{.29\linewidth}
\includegraphics[width=\linewidth]{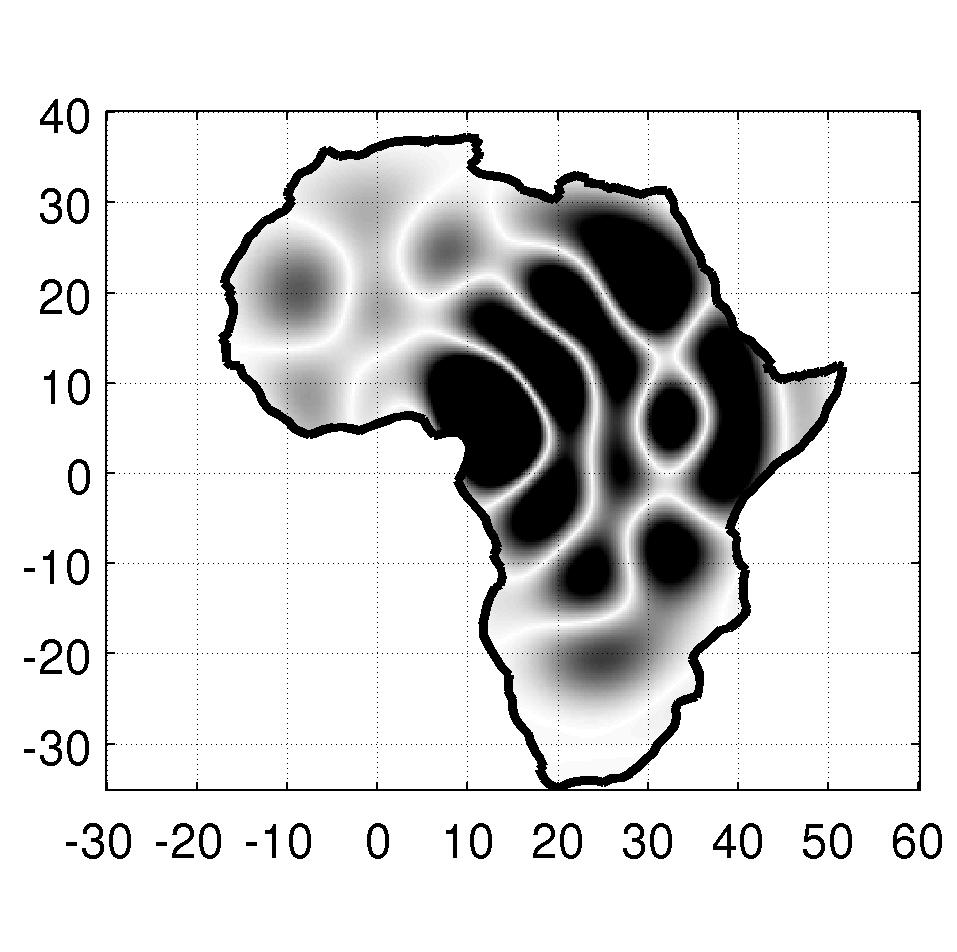}
\caption{$e_{20}$}
\end{subfigure}
\begin{subfigure}[b]{.29\linewidth}
\includegraphics[width=\linewidth]{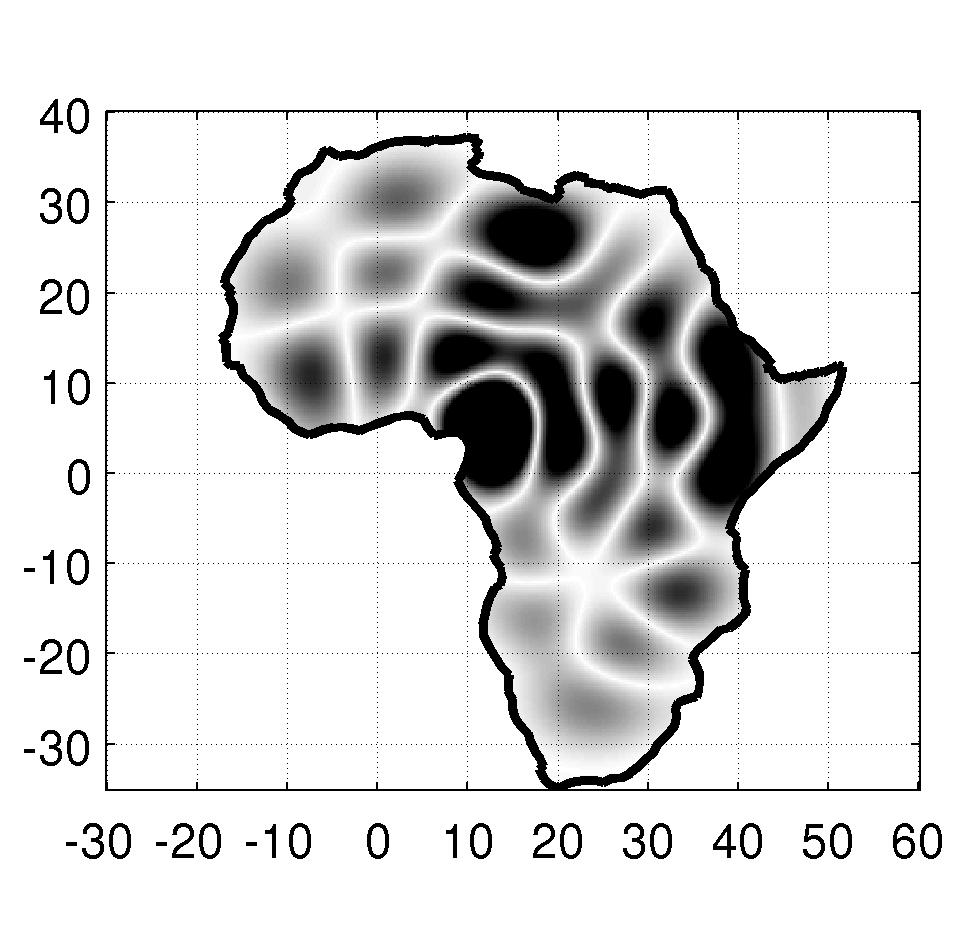}
\caption{$e_{30}$}
\end{subfigure}
\\
\begin{subfigure}[b]{.29\linewidth}
\includegraphics[width=\linewidth]{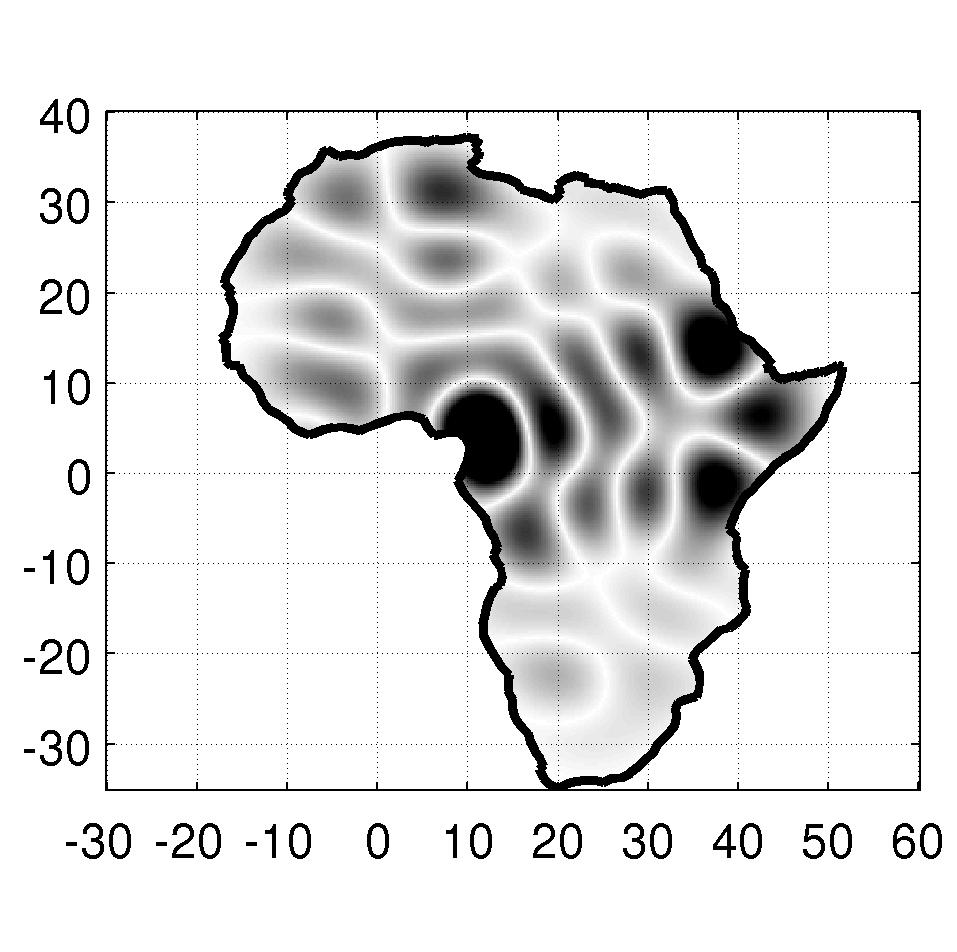}
\caption{$e_{40}$}
\end{subfigure}
\begin{subfigure}[b]{.29\linewidth}
\includegraphics[width=\linewidth]{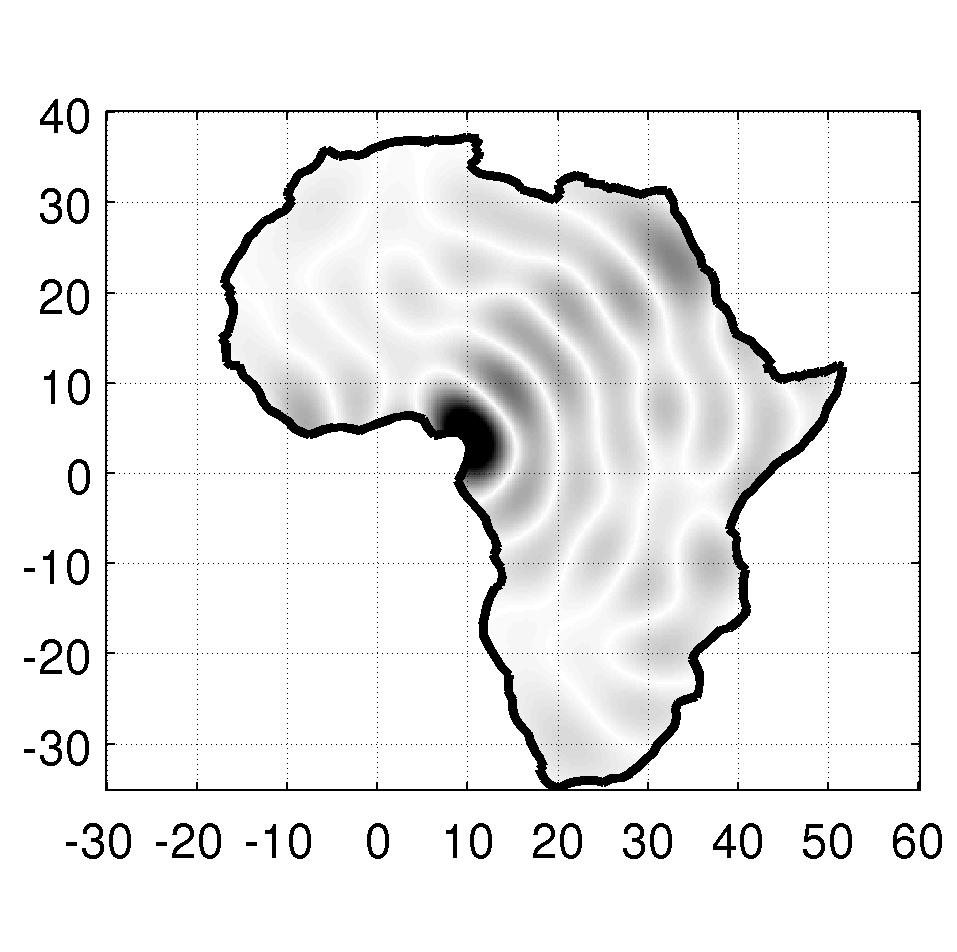}
\caption{$e_{50}$}
\end{subfigure}
\begin{subfigure}[b]{.29\linewidth}
\includegraphics[width=\linewidth]{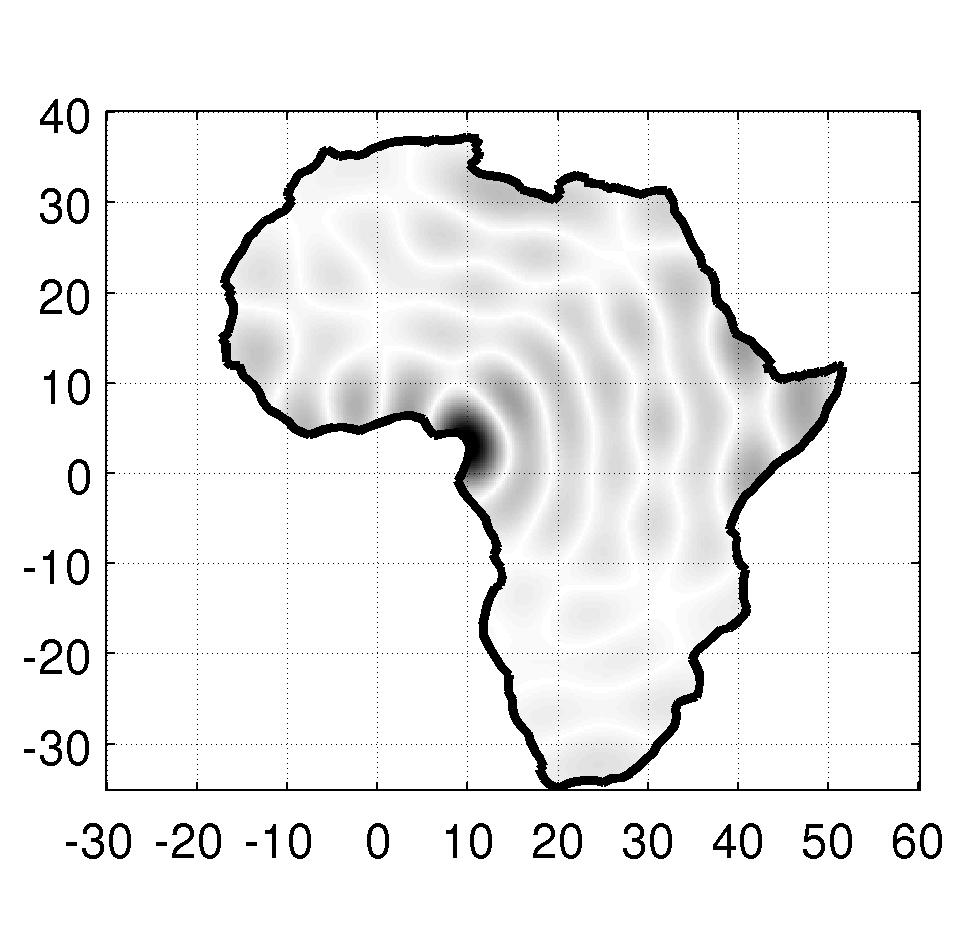}
\caption{$e_{60}$}
\end{subfigure}
\caption[Slepian Truncated Least Squares Reconstruction of POMME-4]%
{\label{fig:pomme4slepr}Residual errors of the Slepian Truncated Least
  Squares Reconstruction of POMME-4 data, using the $L=36$ basis concentrated on Africa.
  Labels above describe the number of nonzero entries in the
  reconstructed estimate.  Absolute error values range between 0 
  (white) to 50 (black) and above (thresholded black).}
\end{figure}

\begin{figure}[h!]
\centering
\begin{subfigure}[b]{.29\linewidth}
\includegraphics[width=\linewidth]{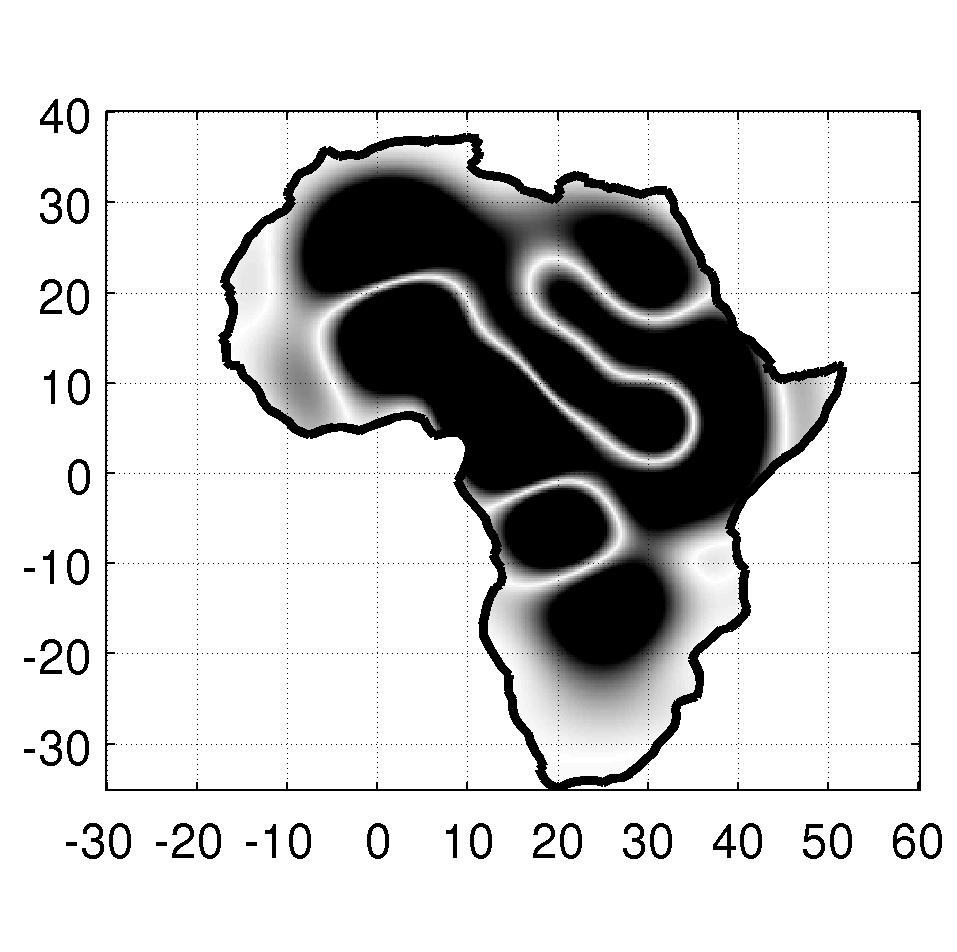}
\caption{$e_{10}$}
\end{subfigure}
\begin{subfigure}[b]{.29\linewidth}
\includegraphics[width=\linewidth]{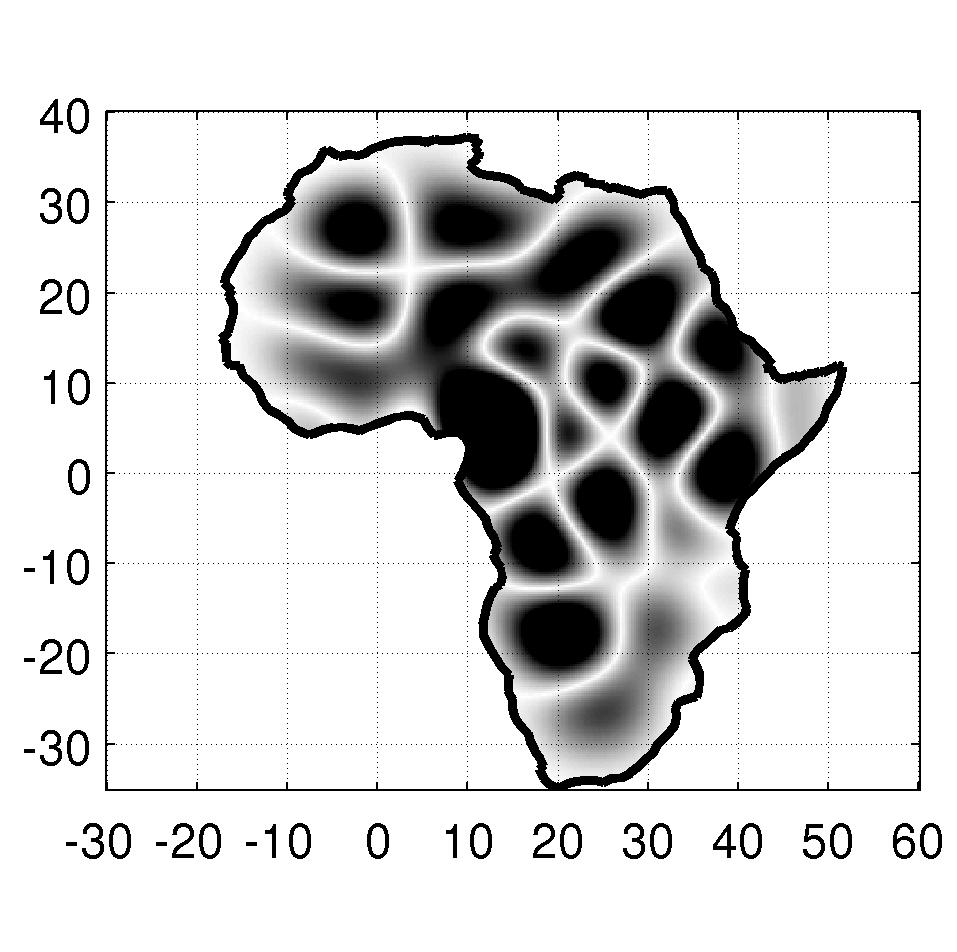}
\caption{$e_{20}$}
\end{subfigure}
\begin{subfigure}[b]{.29\linewidth}
\includegraphics[width=\linewidth]{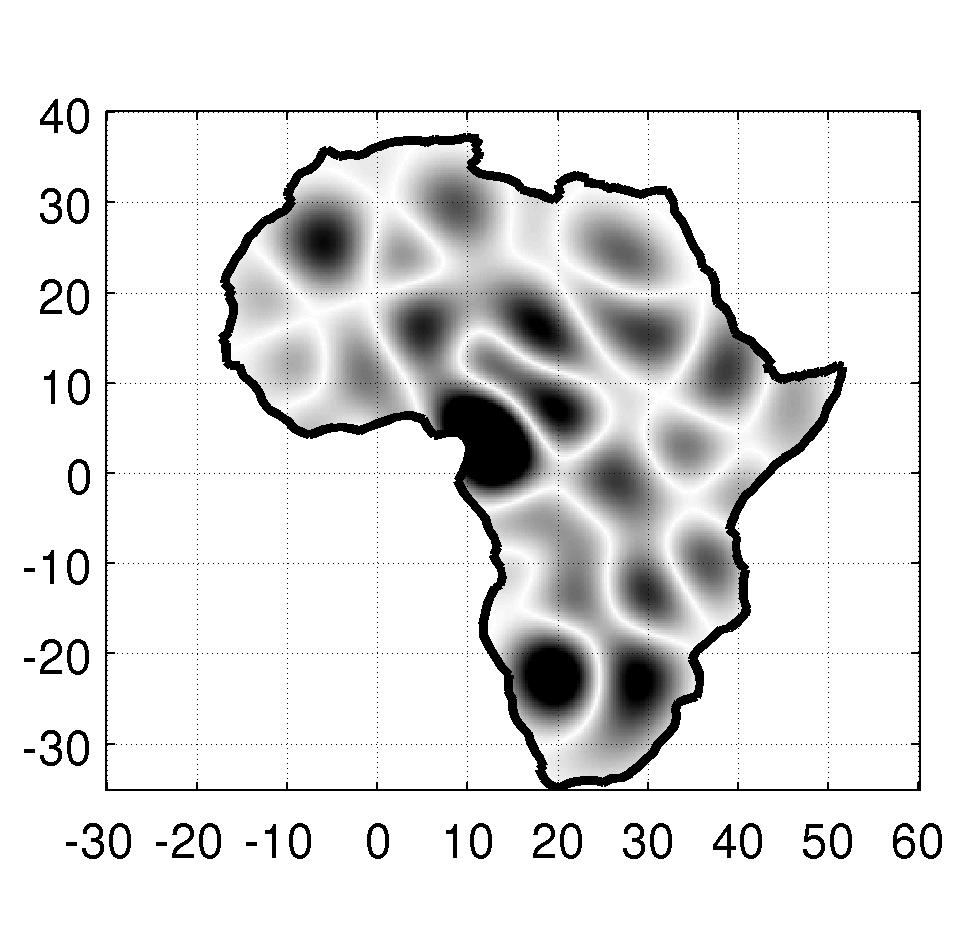}
\caption{$e_{30}$}
\end{subfigure}
\\
\begin{subfigure}[b]{.29\linewidth}
\includegraphics[width=\linewidth]{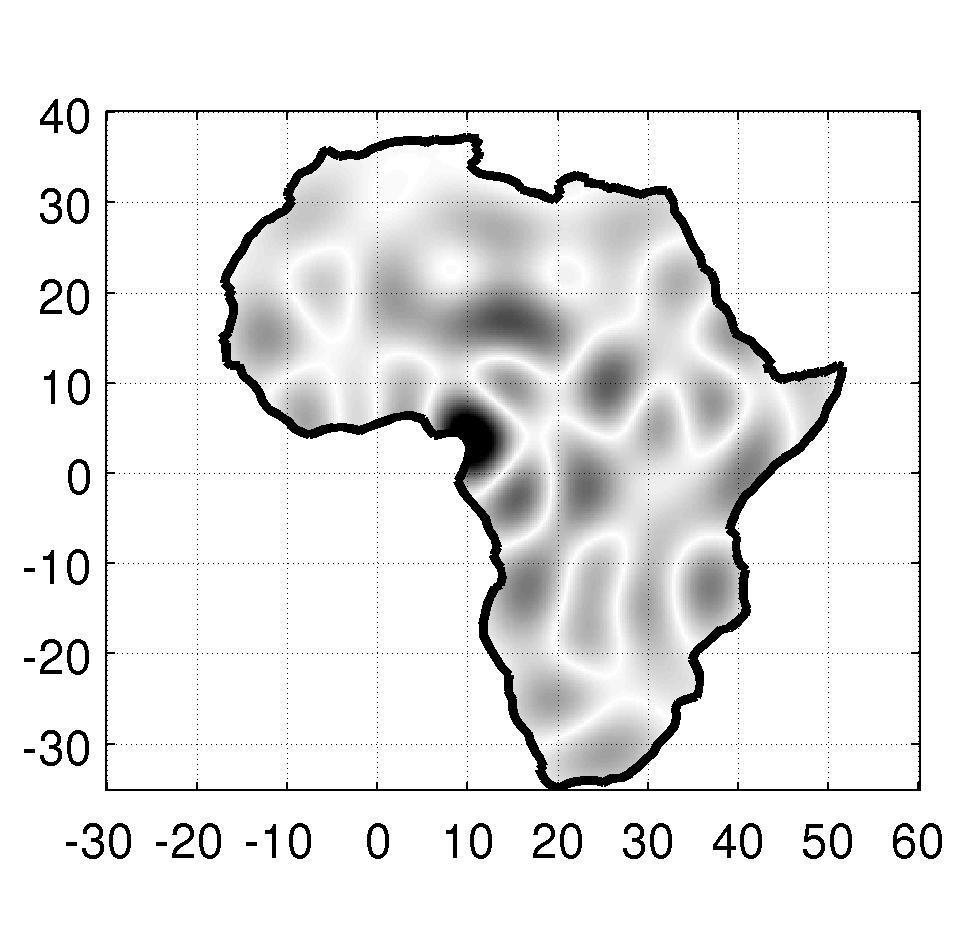}
\caption{$e_{42}$}
\end{subfigure}
\begin{subfigure}[b]{.29\linewidth}
\includegraphics[width=\linewidth]{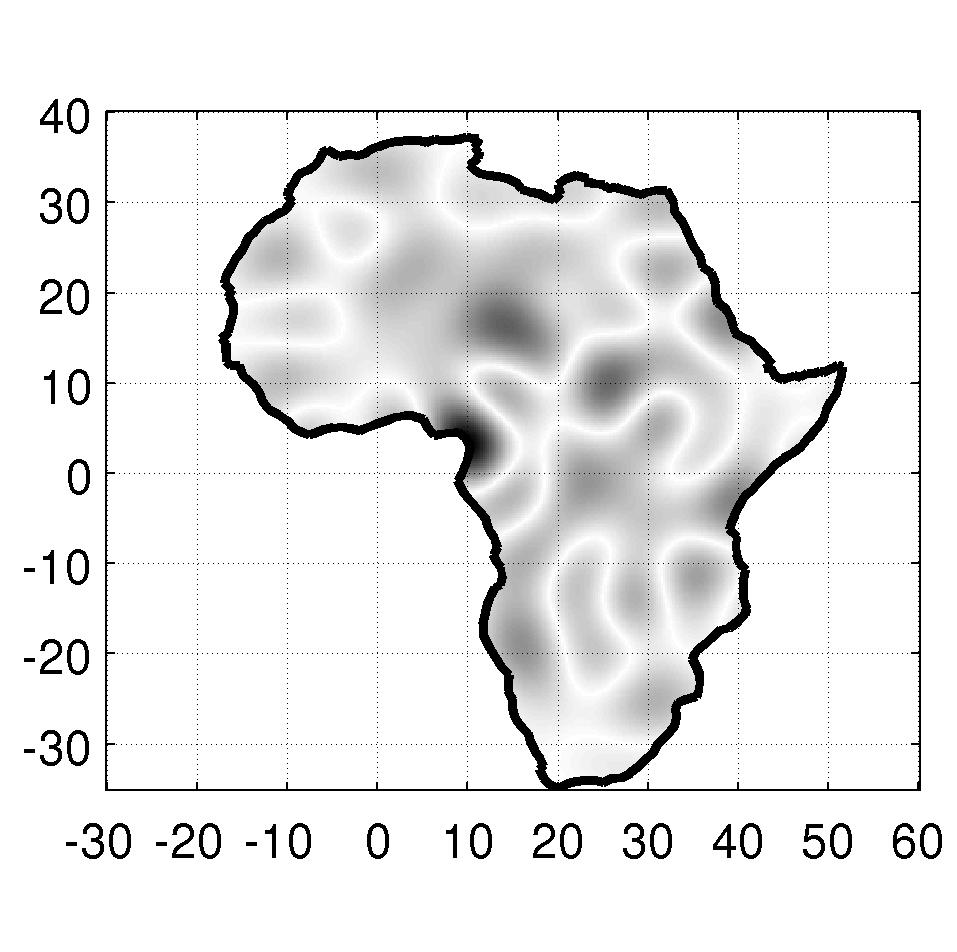}
\caption{$e_{50}$}
\end{subfigure}
\begin{subfigure}[b]{.29\linewidth}
\includegraphics[width=\linewidth]{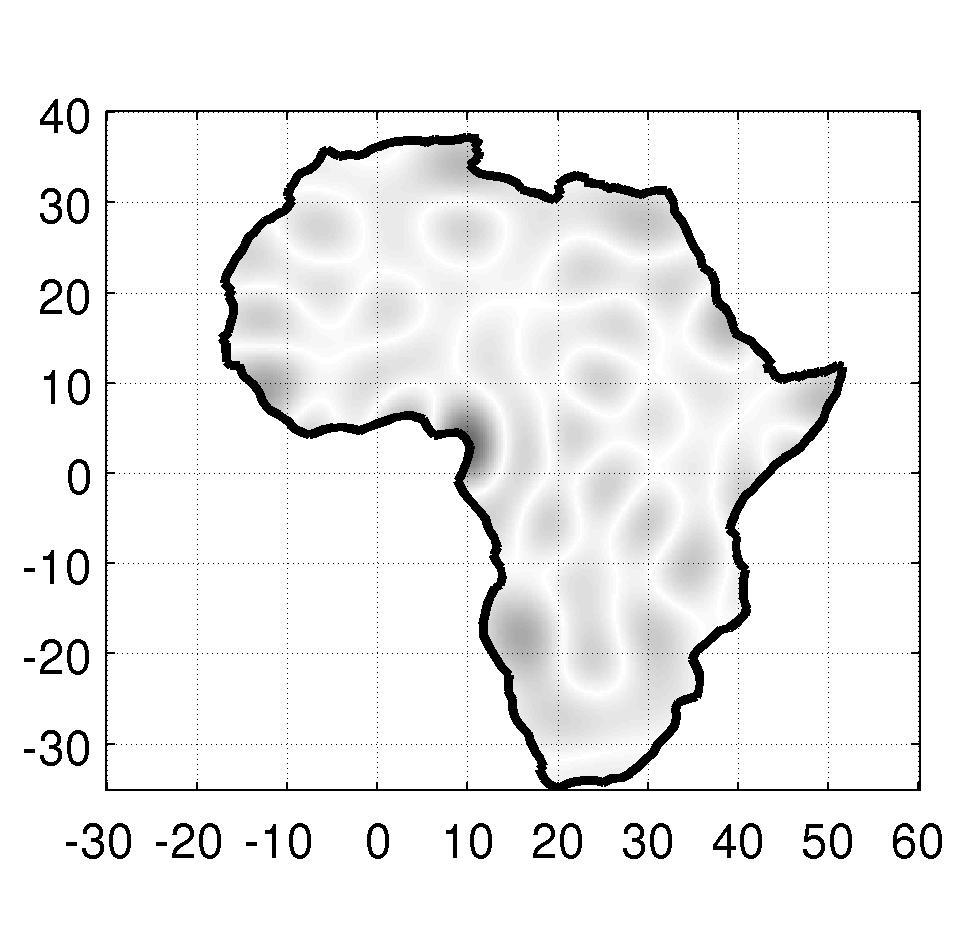}
\caption{$e_{61}$}
\end{subfigure}
\caption[Slepian~$\ell_1$~+~Debias Reconstruction of POMME-4]%
{\label{fig:pomme4slepl1r}Residual errors of the
  Slepian~$\ell_1$~+~Debias Reconstruction of POMME-4 data, using the
  $L=36$ basis concentrated on Africa.  Labels
  above describe the number of nonzero entries
  in the reconstructed estimate.  Absolute error values range between
  0 (white) to 50 (black) and above (thresholded black).}
\end{figure}

\begin{figure}[h!]
\centering
\begin{subfigure}[b]{.29\linewidth}
\includegraphics[width=\linewidth]{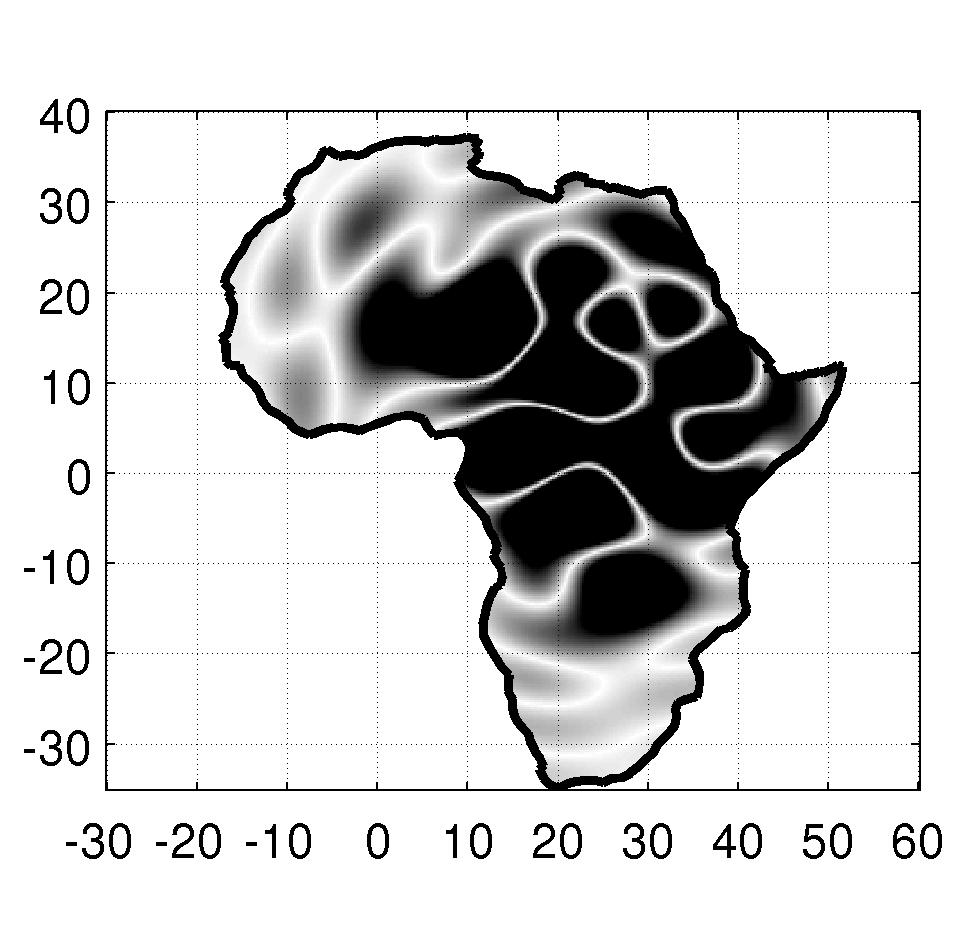}
\caption{$e_{9}$}
\end{subfigure}
\begin{subfigure}[b]{.29\linewidth}
\includegraphics[width=\linewidth]{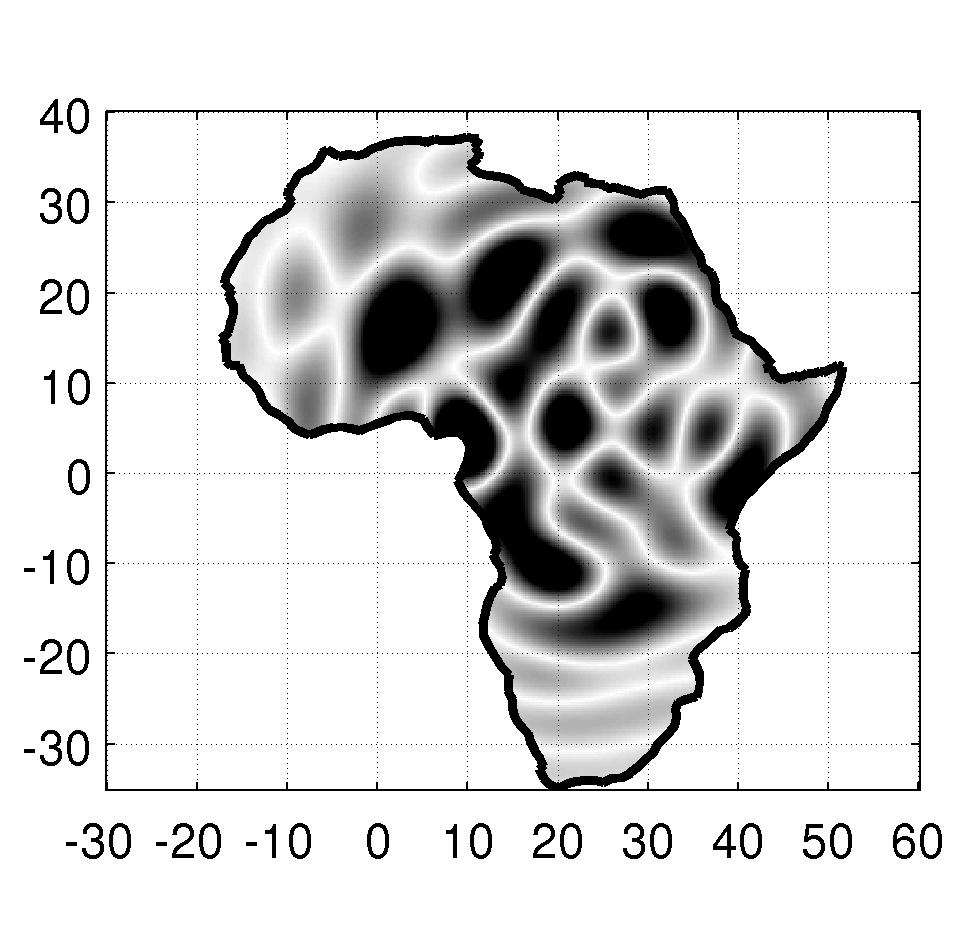}
\caption{$e_{20}$}
\end{subfigure}
\begin{subfigure}[b]{.29\linewidth}
\includegraphics[width=\linewidth]{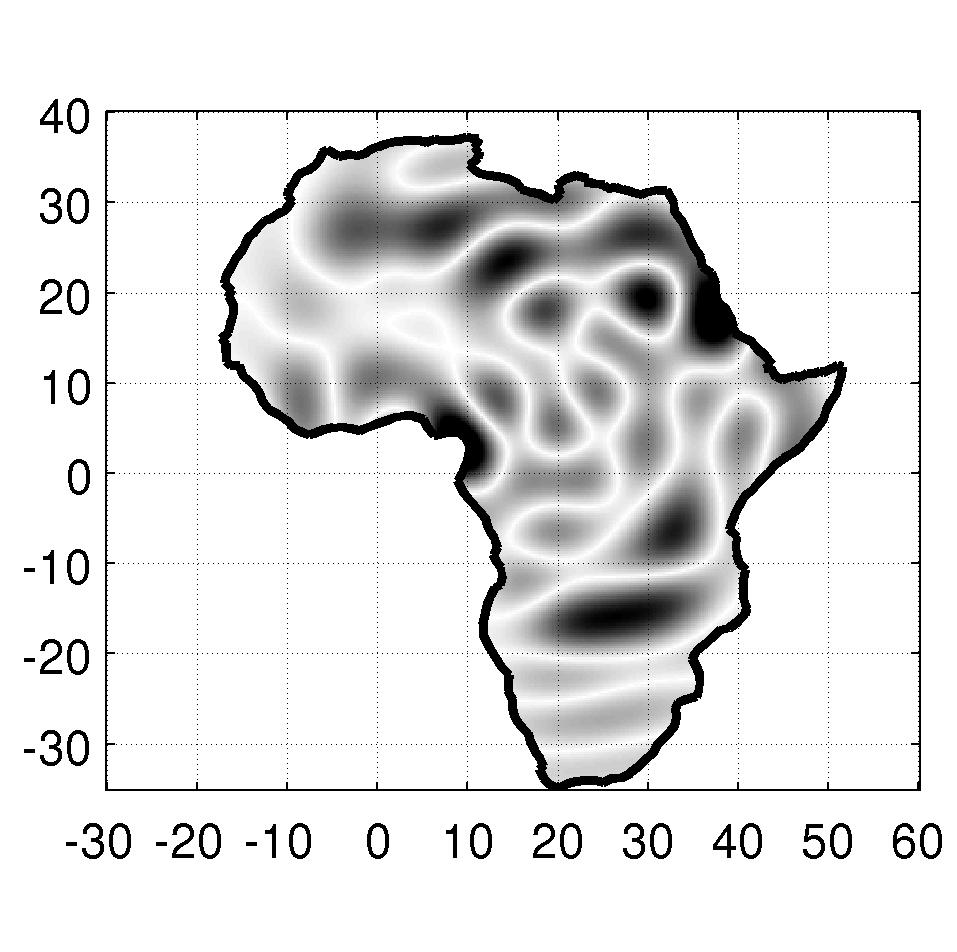}
\caption{$e_{31}$}
\end{subfigure}
\\
\begin{subfigure}[b]{.29\linewidth}
\includegraphics[width=\linewidth]{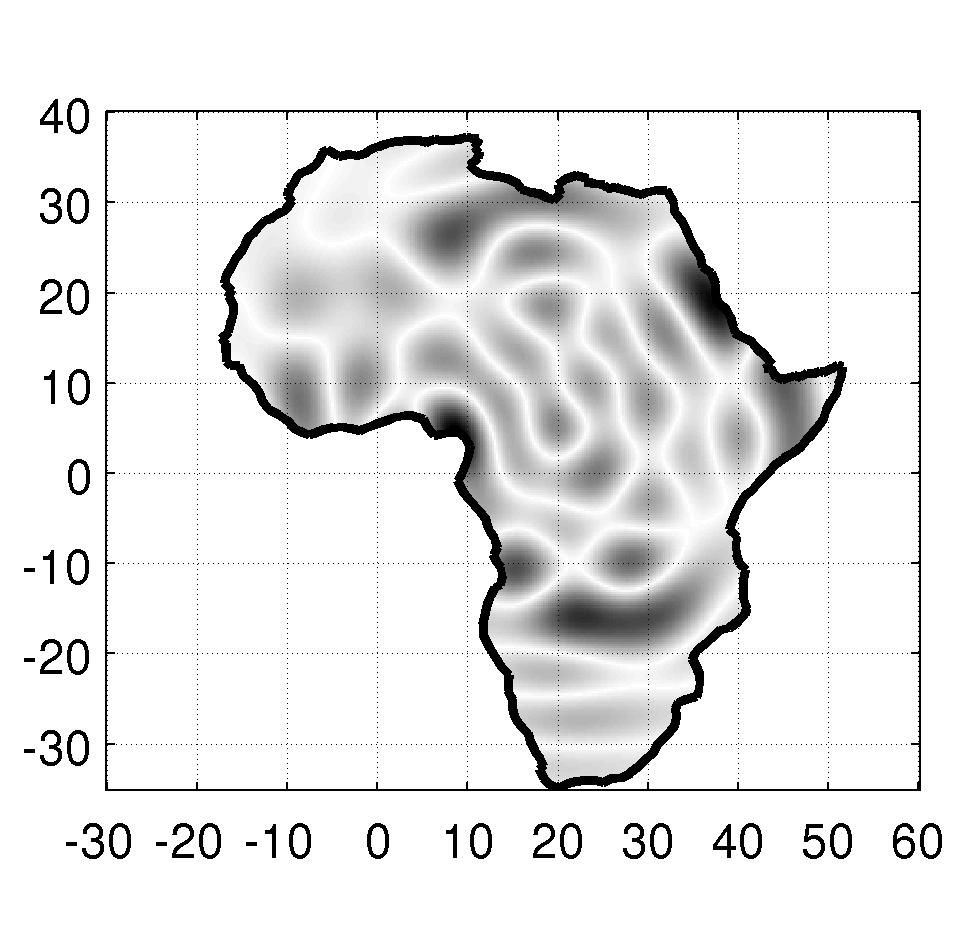}
\caption{$e_{40}$}
\end{subfigure}
\begin{subfigure}[b]{.29\linewidth}
\includegraphics[width=\linewidth]{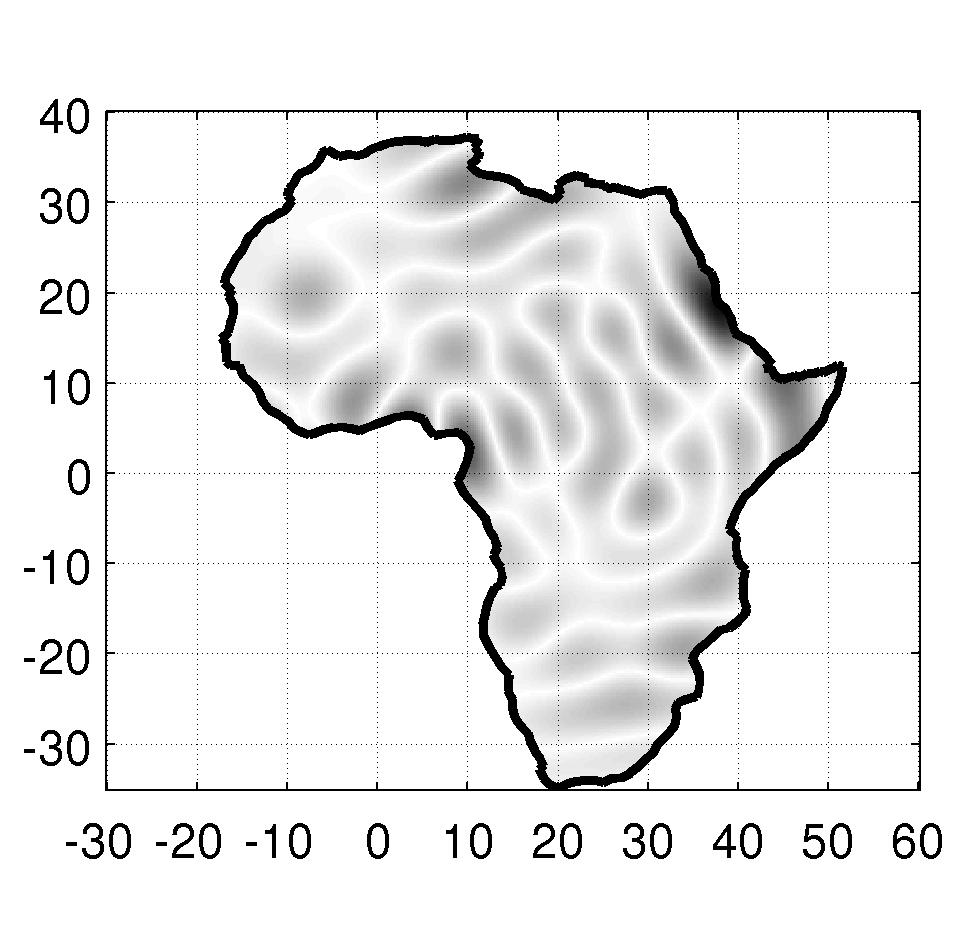}
\caption{$e_{50}$}
\end{subfigure}
\begin{subfigure}[b]{.29\linewidth}
\includegraphics[width=\linewidth]{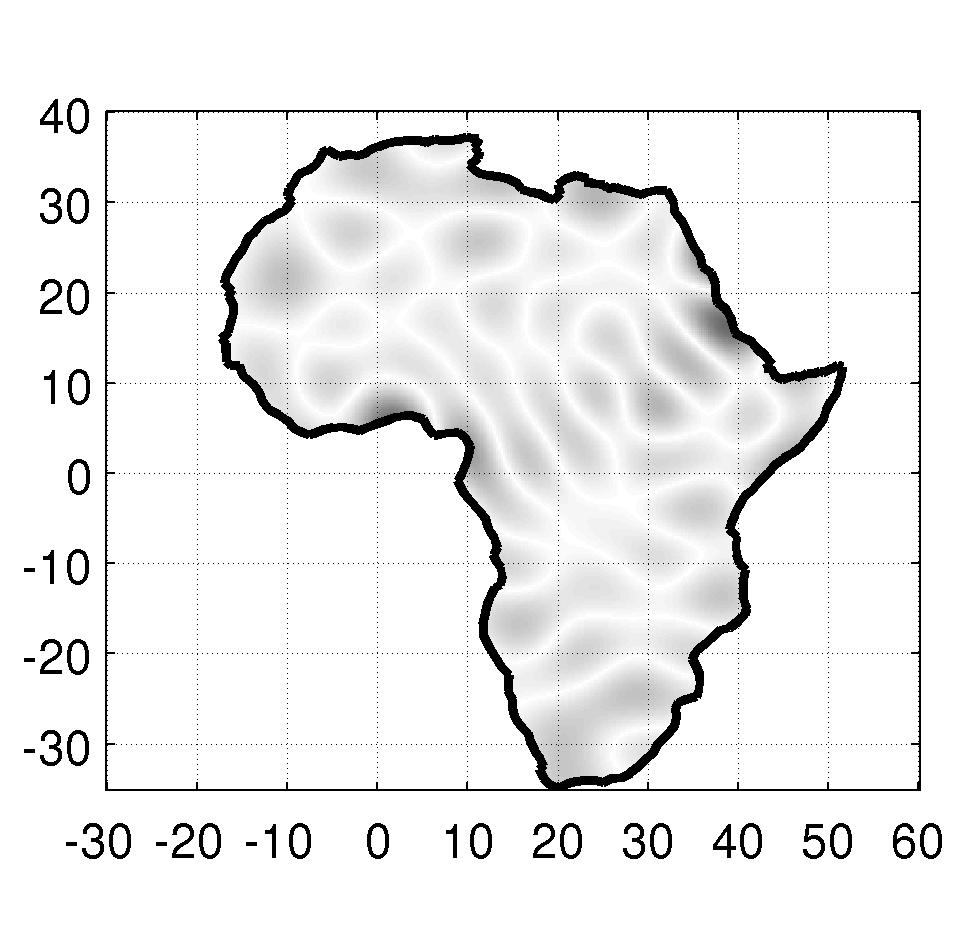}
\caption{$e_{62}$}
\end{subfigure}
\caption[Slepian~Tree~$\ell_1$~Debias Reconstruction of POMME-4]%
{\label{fig:pomme4sleptrr}Residual errors of the
  Slepian~Tree~$\ell_1$~Debias Reconstruction of POMME-4 data, using
  dictionary $\cD_{\text{Africa},36,1}$.  Labels above describe the
  number of nonzero entries in the reconstructed estimate.  Absolute
  error values range between 0 (white) to 50
  (black) and above (thresholded black).}
\end{figure}

As the dictionary elements given by the Tree construction are
localized in both scale and location, we can graphically show which
elements are ``turned on'' through the solution path, as more and more
elements in the support are chosen to be nonzero.
Fig.~\ref{fig:pomme4sleptrsuppr} shows the supporting regions
$\set{\cR^{(j,\alpha)}}$ of dictionary $\cD_{\text{Africa},36,1}$
associated the solutions given in the corresponding panels of
Fig.~\ref{fig:pomme4sleptrr}.  Clearly, larger scale dictionary
elements are chosen first; these reduce the residual error the most.
As more and more dictionary elements are added during the $\ell_1$
based inversion process, finer and finer details are included in the
reconstruction.

\begin{figure}[h!]
\centering
\begin{subfigure}[b]{.29\linewidth}
\includegraphics[width=\linewidth]{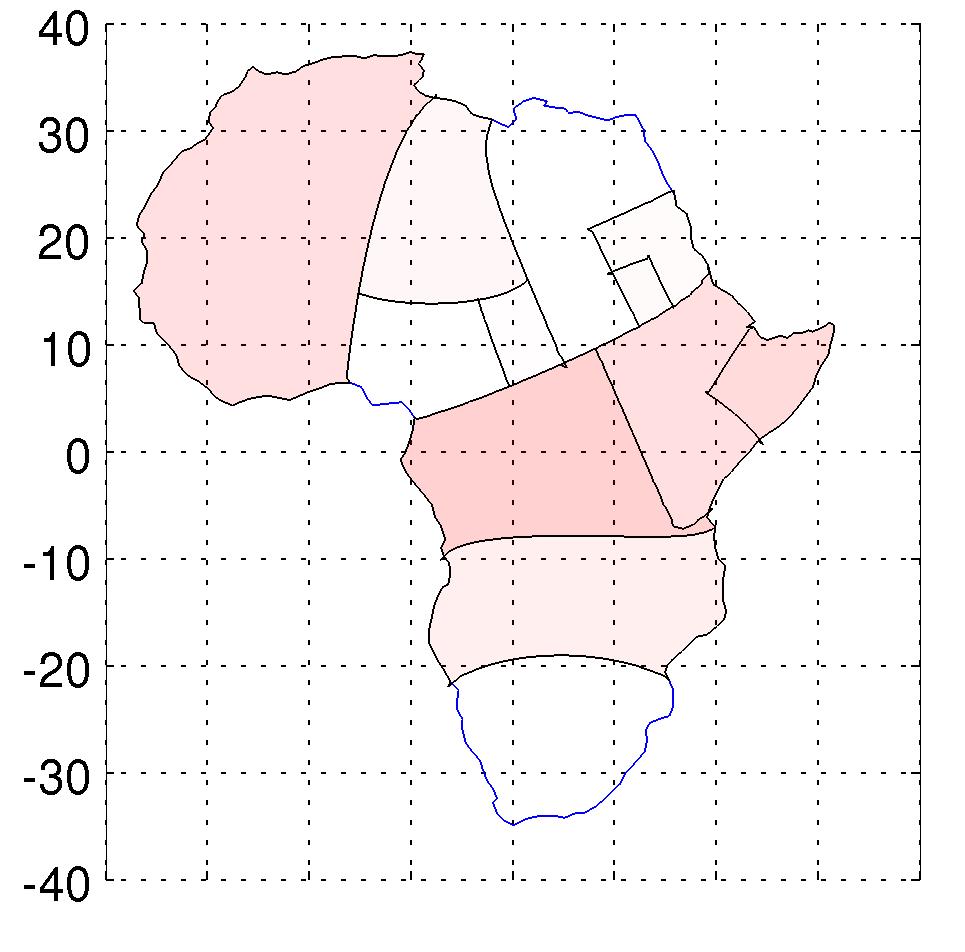}
\caption{$\cS_{9}$}
\end{subfigure}
\begin{subfigure}[b]{.29\linewidth}
\includegraphics[width=\linewidth]{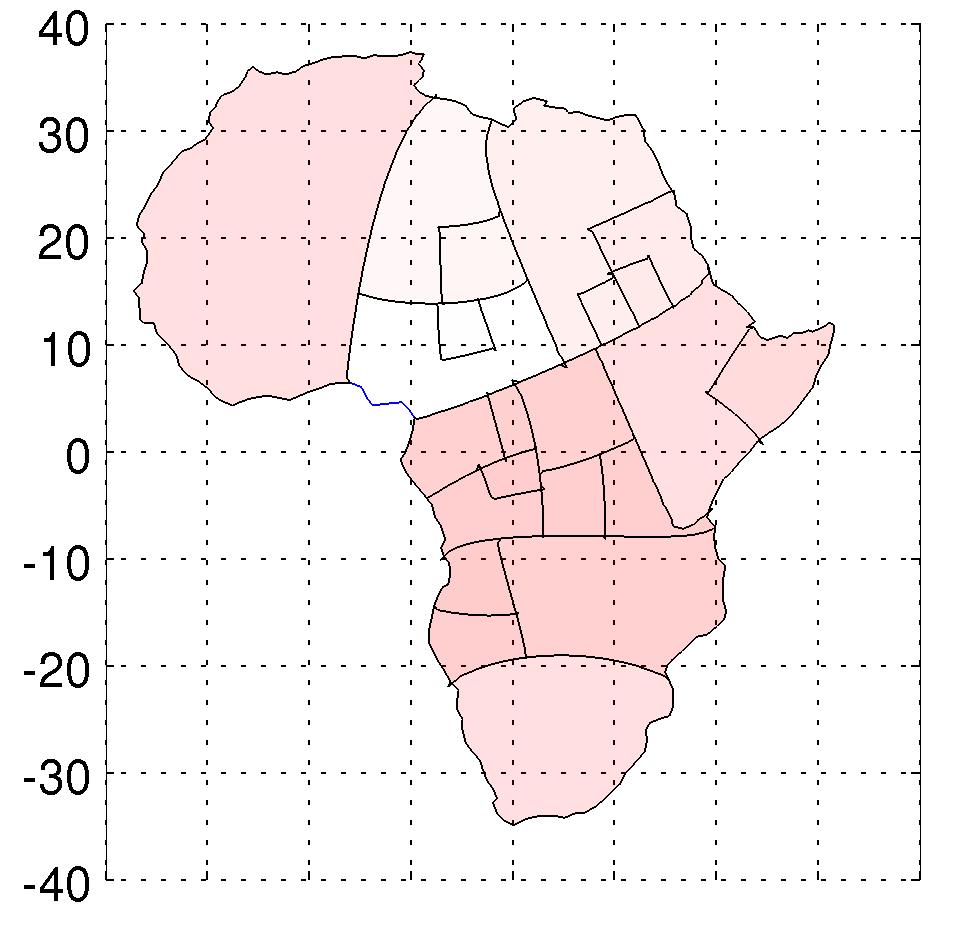}
\caption{$\cS_{20}$}
\end{subfigure}
\begin{subfigure}[b]{.29\linewidth}
\includegraphics[width=\linewidth]{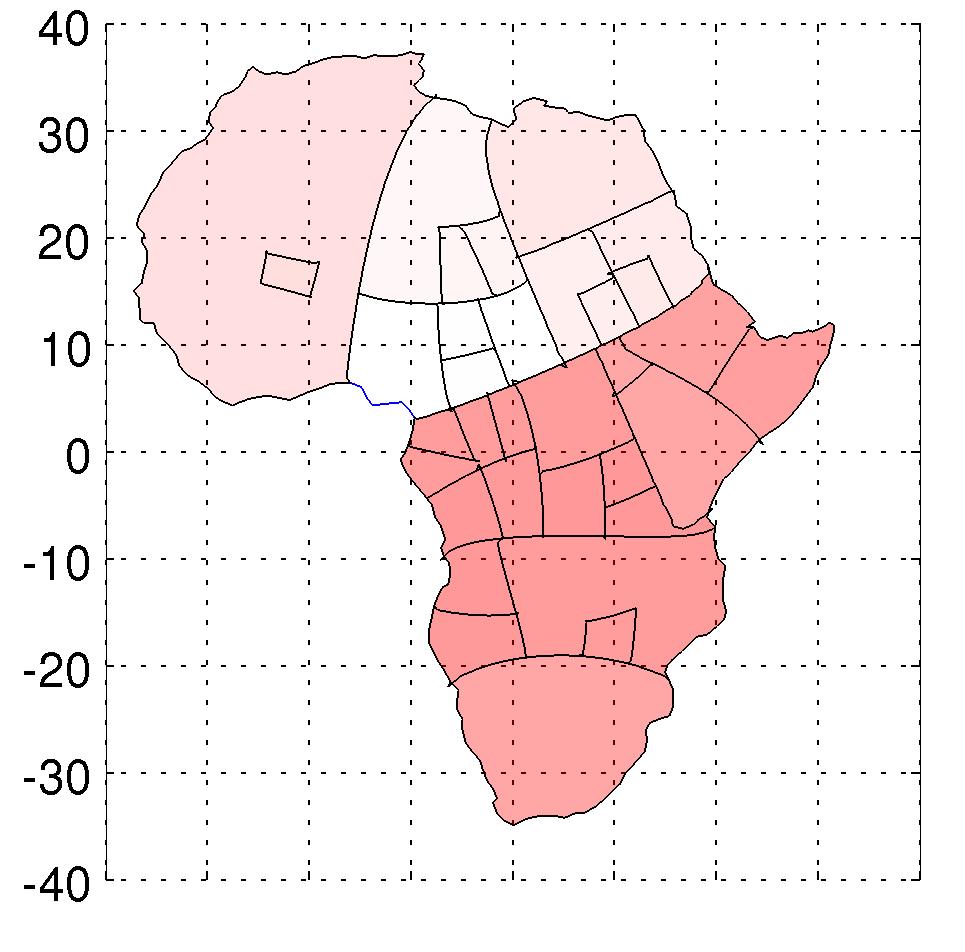}
\caption{$\cS_{31}$}
\end{subfigure}
\\
\begin{subfigure}[b]{.29\linewidth}
\includegraphics[width=\linewidth]{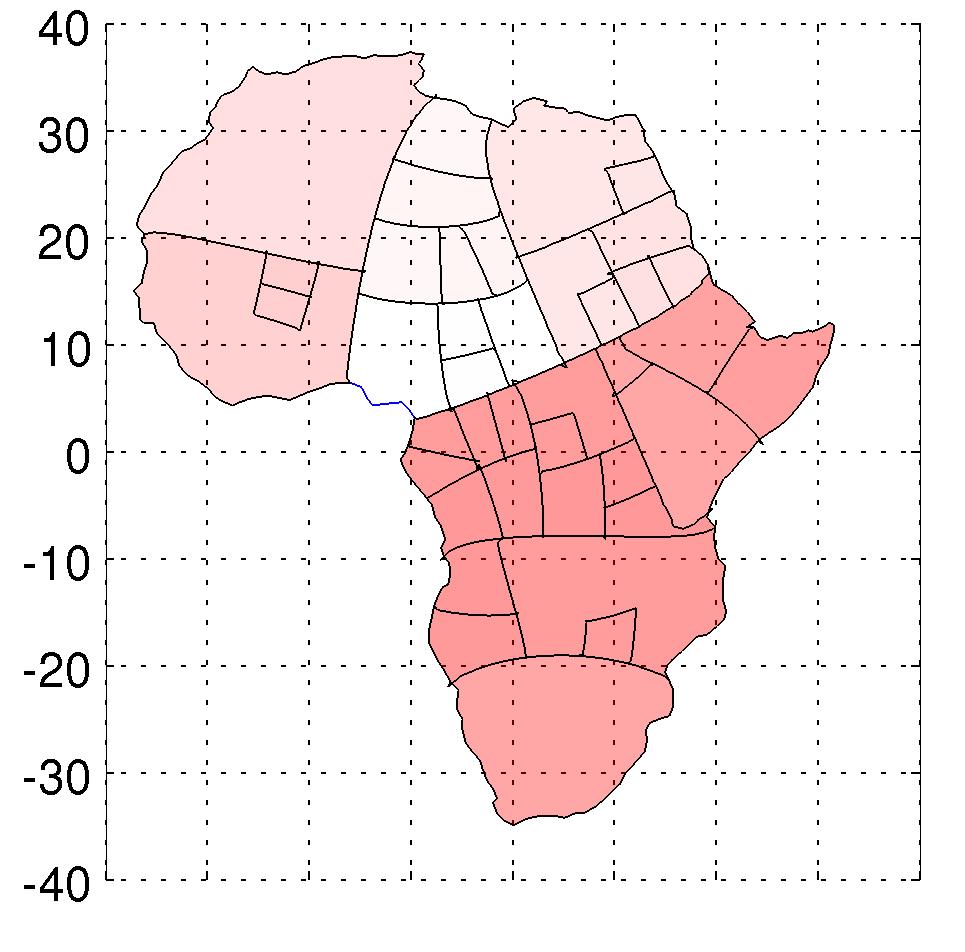}
\caption{$\cS_{40}$}
\end{subfigure}
\begin{subfigure}[b]{.29\linewidth}
\includegraphics[width=\linewidth]{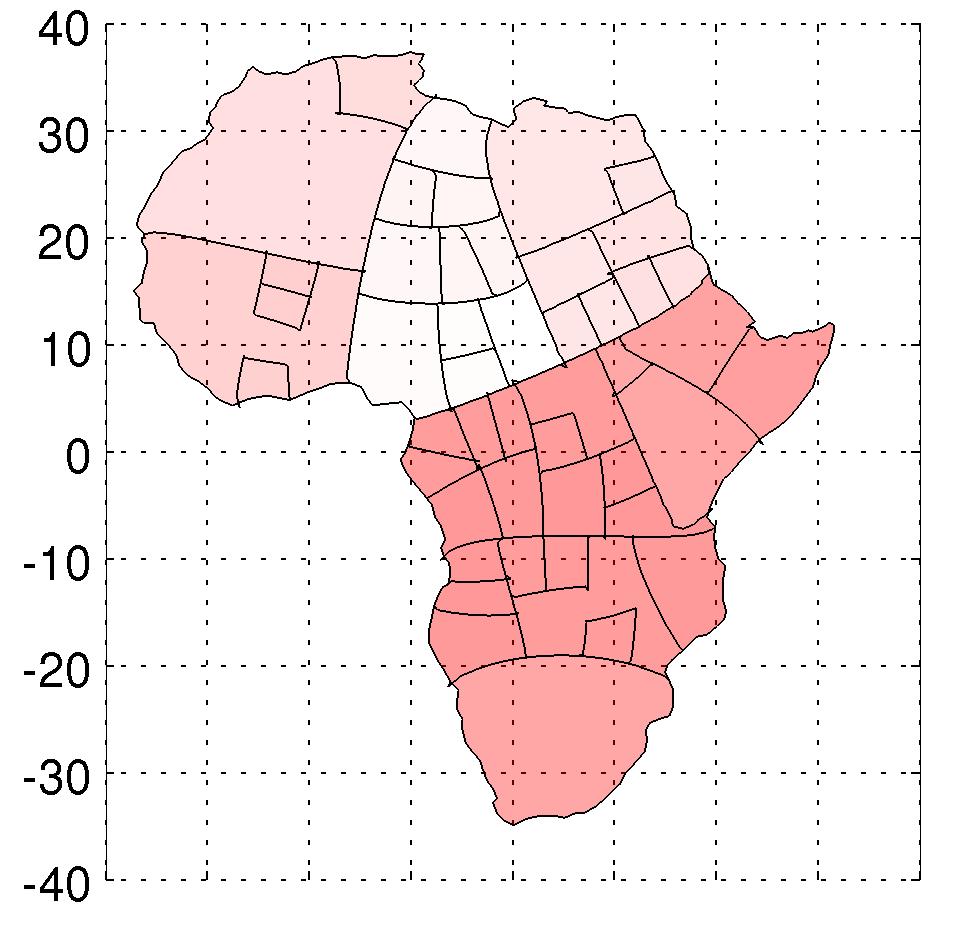}
\caption{$\cS_{50}$}
\end{subfigure}
\begin{subfigure}[b]{.29\linewidth}
\includegraphics[width=\linewidth]{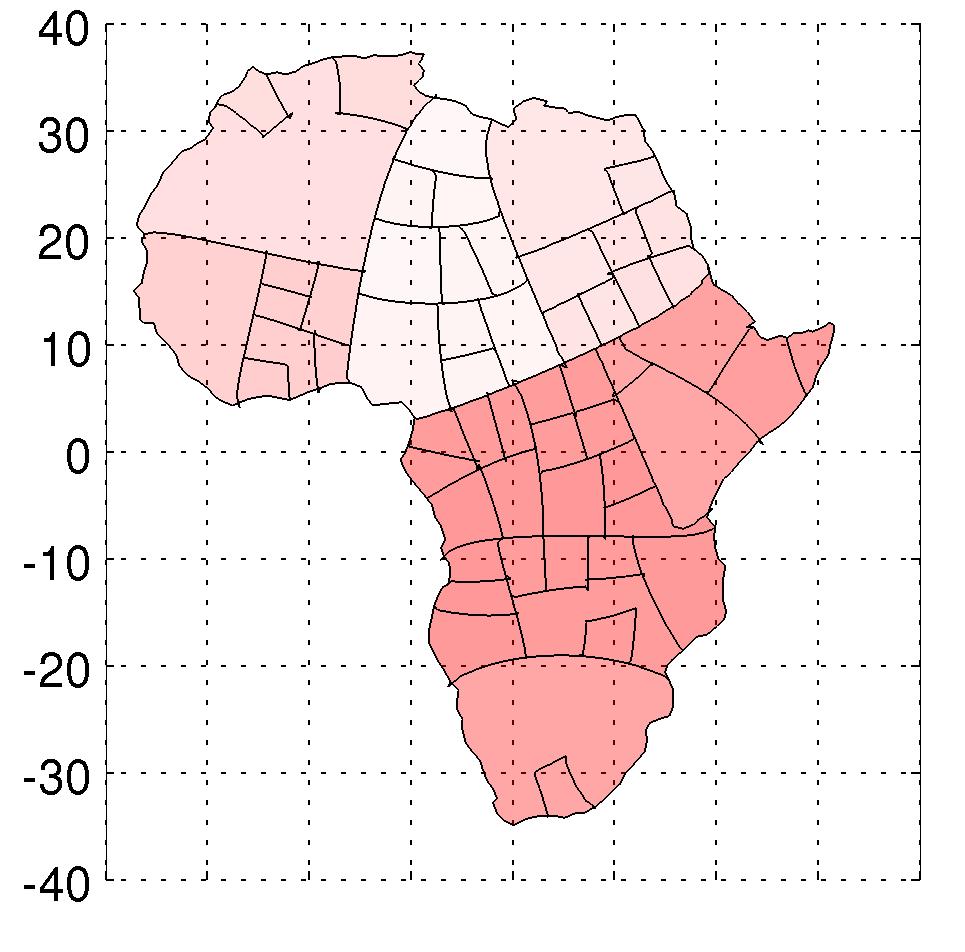}
\caption{$\cS_{62}$}
\end{subfigure}
\caption[Support sets estimated during Slepian~Tree~$\ell_1$~Debias Reconstruction of POMME-4]%
{\label{fig:pomme4sleptrsuppr}Support sets as estimated during the
  Slepian~Tree~$\ell_1$~Debias Reconstruction of POMME-4.  The
  dictionary used was $\cD_{\text{Africa},36,1}$.  These panes
  are associated with their corresponding panes in Fig.~\ref{fig:pomme4sleptrr}.}
\end{figure}

Fig. \ref{fig:pomme4error} compares, on a logarithmic scale, the
spatial residual errors (sum of squared differences) between the
three estimates, as a function of the number of nonzero components
allowed.  Clearly, when a small number of nonzero components
is allowed, the sparsity-based estimators outperform the standard
Slepian truncation-based inversion.  

%% POMME-4 Error
\begin{figure}[h!]
\centering
\includegraphics[width=.65\linewidth]{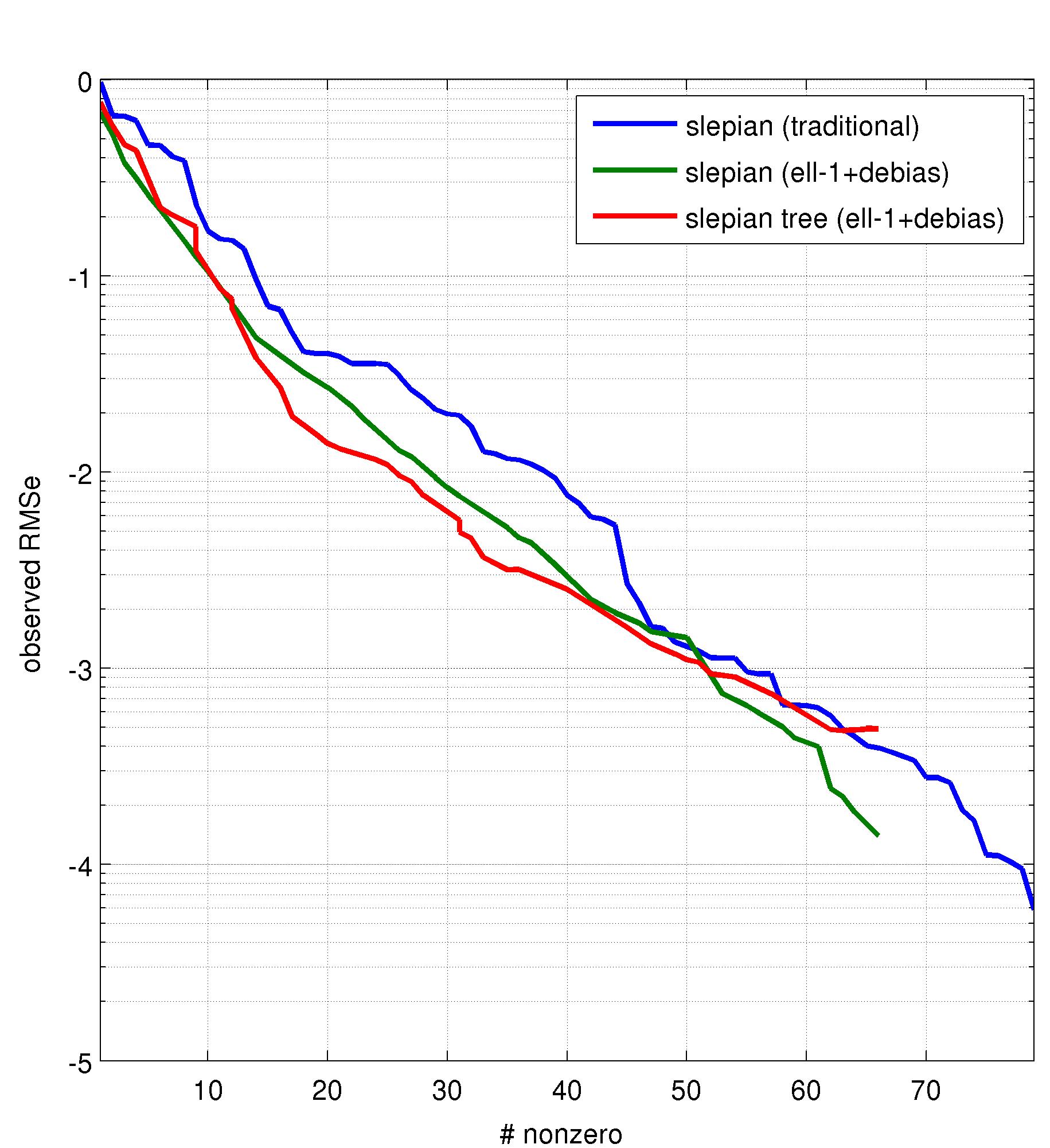} \\
\caption[Normalized residual error during reconstruction of POMME-4
  data]{\label{fig:pomme4error}Normalized residual error during
  reconstruction of POMME-4 data vector $\tp$.  Values are
  proportional to the sum of squared differences between $\tp$ and
  the spatial expansions of its estimates via STLS, SL1D, and STL1D, on
  the grid $\cX$.  The x-axis denotes the number of
  nonzero components allowed.  The y-axis is on a base-10 logarithmic
  (dB) scale.}
\end{figure}

%% Consistency
\begin{figure}[h!]
\centering
\includegraphics[width=.45\linewidth]{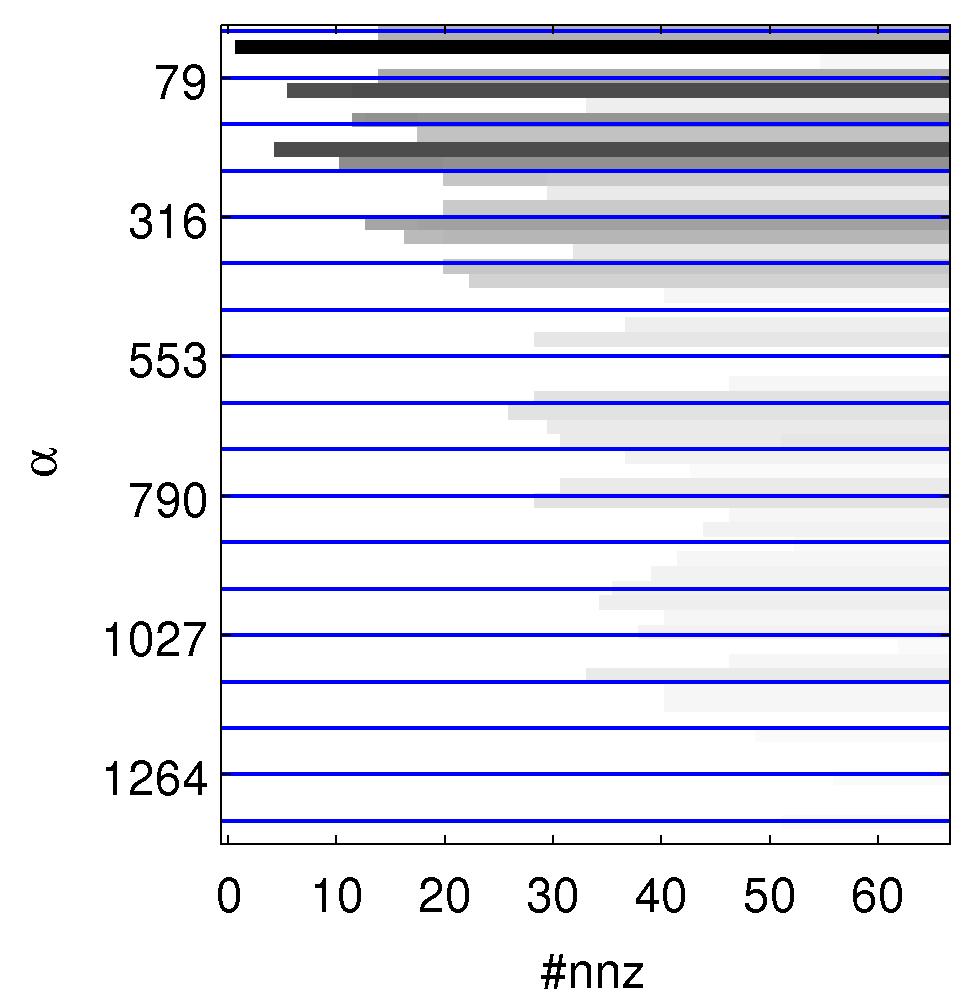}
\includegraphics[width=.45\linewidth]{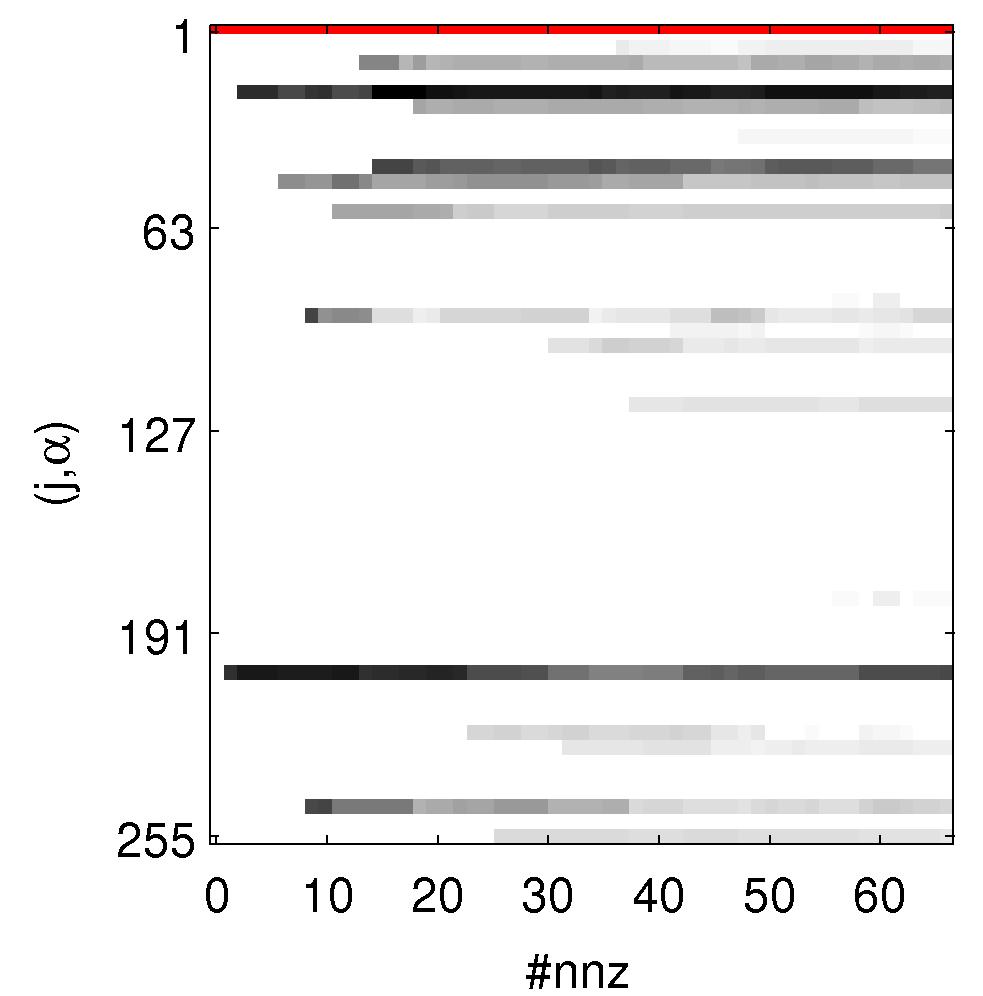}
\caption[Solution paths for the SL1D and STL1D estimators on the
  POMME-4 data]{\label{fig:pomme4solpath}Solution paths for the SL1D
  and STL1D estimators on the POMME-4 data $\tp$.  On the left, the
  solution path (coefficient magnitudes vs. support size) for the SL1D
  estimator.  Blue horizontal lines represent multiples of the Shannon
  number $N_{\text{Africa},36} \approx 79$.  On the right, the solution
  path for the STL1D estimator with dictionary~$\cD_{\text{Africa},36,1}$.}
\end{figure}

To study the consistency of the $\ell_1$-based SL1D
and STL1D estimators, and as a measure of how closely the SL1D
method matches the classical truncation strategy, we plot the solution
paths of these two estimators.  Fig. \ref{fig:pomme4solpath} shows
that while lower order (better-concentrated in $\cR$) Slepian functions were
chosen early on, when $n$ was small.  However, as more and more
nonzero indices were allowed, less well concentrated Slepian
functions, with lower magnitudes, were included in the solution,
probably as small ``tweaks'' to the estimate near the edges of the
region.  Once a Slepian function was included into the solution, its
magnitude did not change much throughout the solution path (as other
elements were added).

In contrast to the behavior of SL1D, the Slepian Tree
solution chose particular elements localized to the main features of
the signal, not simply elements that are well concentrated in all of
Africa on a large scale.  In addition, as the size of the support was
allowed to increase, the magnitudes of some coefficients were
decreased as new elements were added.  This supports the general
statement that multiscale dictionaries, when combined with
sparsity-inducing reconstruction techniques, ``fit'' the support to
the nature of the data.  Figs. \ref{fig:pomme4error} and
\ref{fig:pomme4solpath} thus help to clarify the behavior of the STL1D
estimator.

%% Averages for white and red noise

\subsubsection{White and Pink Noise}

As a second experiment, we generated multiple observations of either
pink or white noise fields, with a bandlimit of $L=36$, on the sphere
$S^2$, via randomization in the spherical harmonic representation of
$L^2_\Omega(S^2)$.  We then sampled these on $\cX$ to and attempted to
reconstruct them only in Africa, as in the previous example.

A random field $r$ with spectral slope $\beta$, up to degree $L$, is
defined as having the harmonic coefficients
\begin{align}
\wh{r}_{lm} &= l^{\beta/2} N_l^{-1} n_{lm} \quad (l,m) \in \Omega, \\
\text{where } n_{lm} &\operatorname*{\sim}^{\text{i.i.d.}} \cN(0,1)
\quad (l,m) \in \Omega, \nonumber \\
\text{and } N_l &= (2l+1)^{-1} \textstyle{\sqrt{\sum_{m=-l}^l n_{lm}^2}}. \nonumber
\end{align}

For a white noise process, with equal signal power across its
spectrum, $\beta = 0$.  Most spatial processes in geophysics, however,
have some $\beta < 0$; their power drops off with degree $l$.  Fields
with $\beta$ near $-2$ are considered to be ``pink'', while fields
with $\beta$ near $-4$ are ``red''.  The more red a noise process, the
higher its spatial correlation, the less ``random'' it looks.  The
earth's geopotential field, for example, is modeled as having ${\beta =
-4.036}$.  In our experiments, we use $\beta=0$ to generate white
noise processes and $\beta=-2$ for pink.

For each of $T=200$ iterations, we generated both pink and white noise
fields $r$ (bandlimited to $L=36$) and sampled them on $\cX$ (the grid
over the African continent) to get the vectors $\tr$.  As before, the
discretization matrix $Y$ is of rank $r=528$ so direct inversion is
impossible.  We again performed reconstruction via the three methods
STLS, SL1D, and STL1D. As in Fig. \ref{fig:pomme4error}, for each of
the iterations we calculated the normalized residual: the sum of
squared differences between the estimates, as expanded on $\cX$, and
the original samples, normalized by the sum of squares
$\norm{\tr}_2^2$.

Figs.~\ref{fig:slepwhitenoiseavg} and~\ref{fig:sleprednoiseavg} show
the mean normalized error over the $T$ iterations for white and pink
noise, respectively.  For white noise, the ability to use any of the
Slepian functions clearly provides an advantage for the SL1D
algorithm.  In contrast, for pink noise the spatial localization of
the dictionary elements and the smaller size of the dictionary give an
edge to the Tree-based estimator.  Clearly, however, both estimators
lead to lower normalized observation error on average, as compared
with the classically optimal Slepian truncation method STLS.

\begin{figure}[hb]
\centering
\includegraphics[width=.55\linewidth]{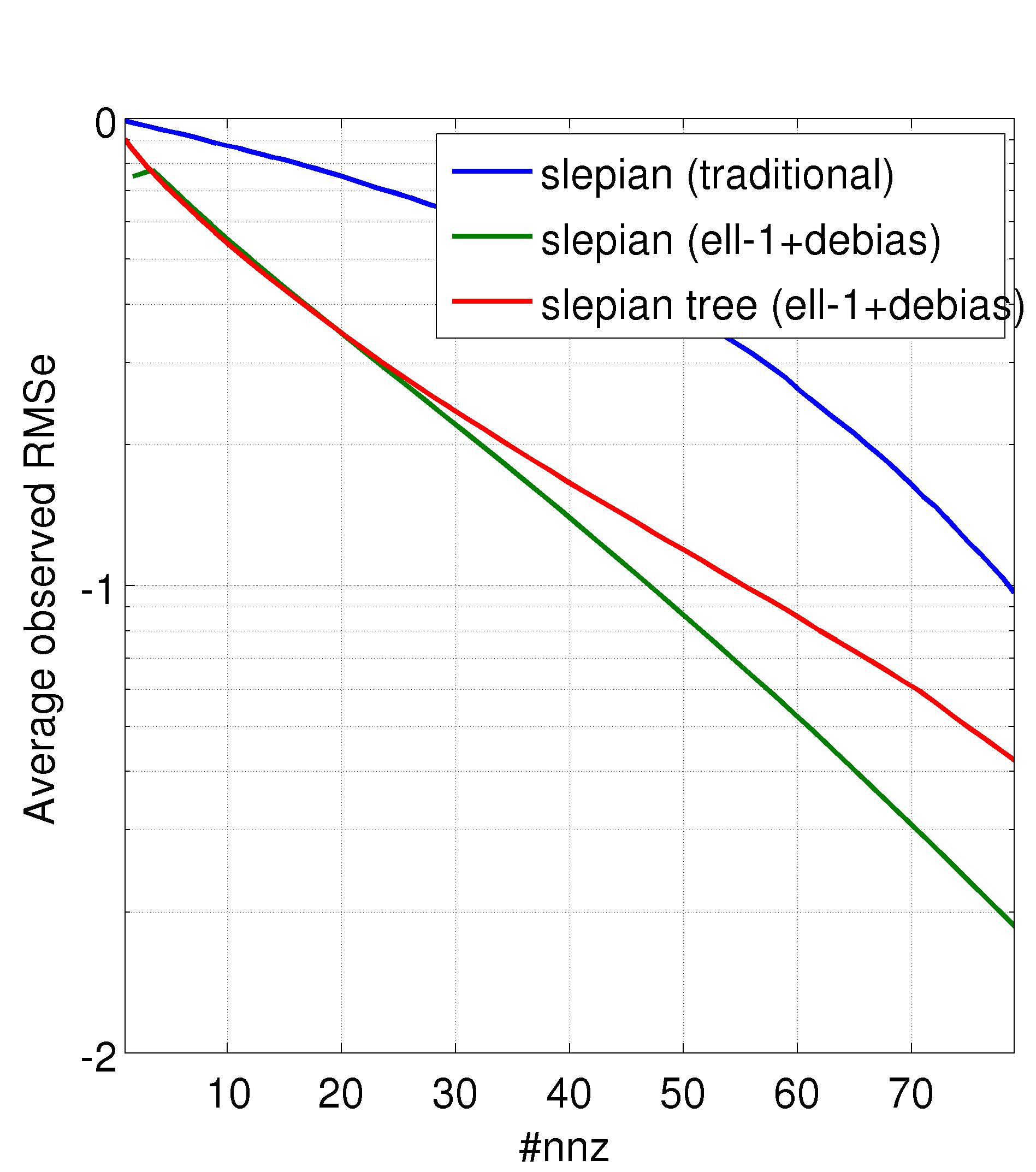}
\caption[White noise ($\beta=0$) average reconstruction
  error]{\label{fig:slepwhitenoiseavg}White noise ($\beta=0$) average
  reconstruction error, normalized by signal power within each
  iteration, on a base-10 log scale.  Average taken over $T=200$
  iterations.  X-axis represents the number of nonzero elements
  allowed in the support.  The Tree dictionary is~$\cD_{\text{Africa},36,1}$.}
\end{figure}

\begin{figure}[ht]
\centering
\includegraphics[width=.55\linewidth]{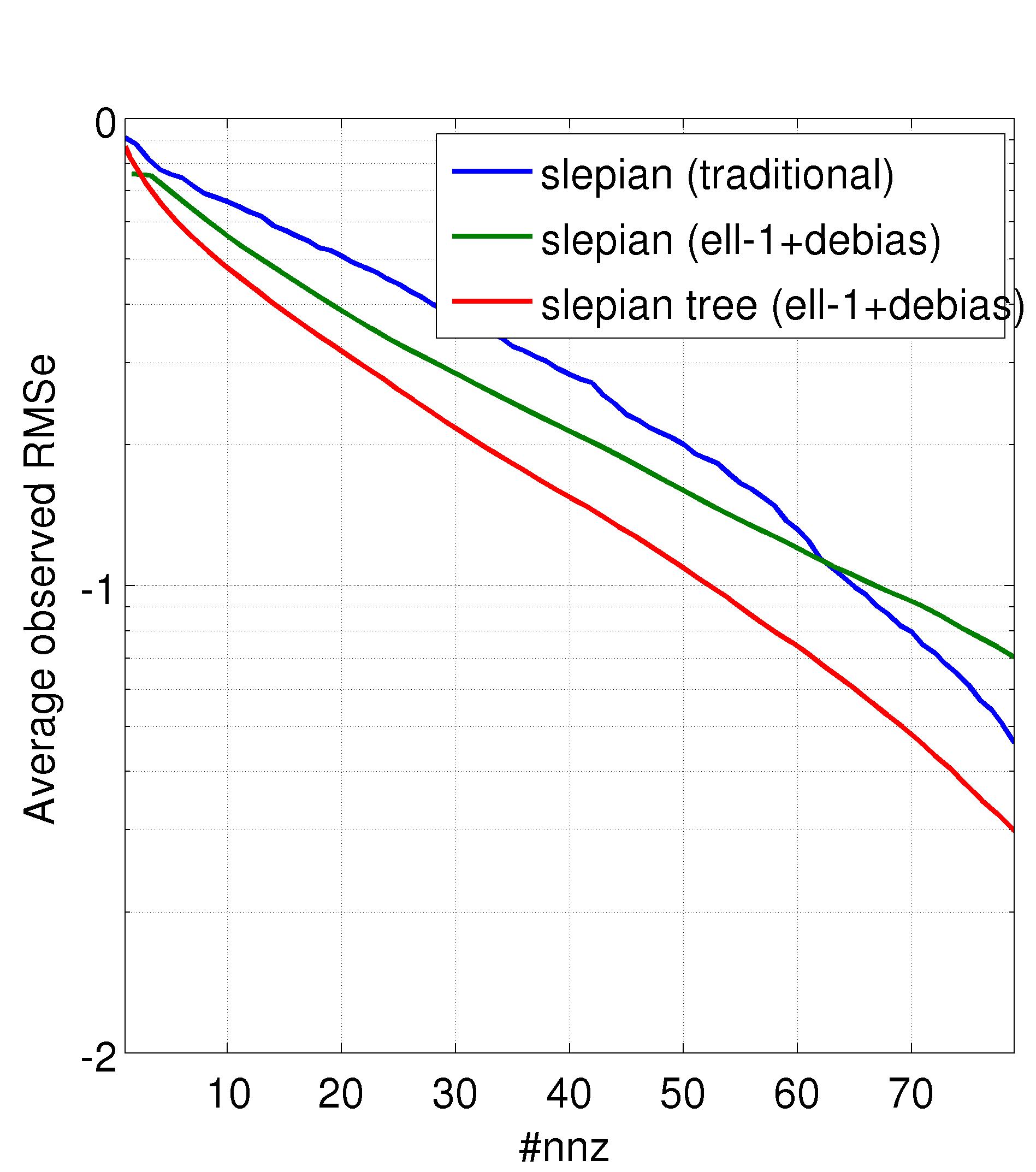}
\caption[Pink noise ($\beta=-2$) average reconstruction
  error]{\label{fig:sleprednoiseavg}Pink noise ($\beta=-2$) average
  reconstruction error, normalized by signal power within each
  iteration, on a base-10 log scale.  Average taken over
  $T=200$ iterations.  X-axis represents the number of nonzero elements
  allowed in the support.  The Tree dictionary is~$\cD_{\text{Africa},36,1}$.}
\end{figure}

\section{Conclusion and Future Work}

We have motivated and described a construction for
dictionaries of multiscale, bandlimited functions on the sphere.  When
paired with the modern inversion techniques of \S\ref{sec:sltrunder},
these dictionaries provide a powerful tool for the approximation
(inversion) of bandlimited signals concentrated on subsets of the
sphere.
The numerical examples in \S\ref{sec:sltrinv} provide
good evidence for the efficacy of the estimators SL1D and STL1D.  More
simulations are required to confirm and explore their numerical
accuracy.

In addition, more theoretical analysis of the existing dictionary
constructions (e.g., their concentration properties)
is also required.  Especially when working in concert with the
$\ell_1$-based estimators, questions of coherence are especially
important~\cite{Gurevich2008,Candes2006}.

The theoretical underpinnings of the SL1D estimator have not been
studied, to our knowledge.  In contrast, the identically equivalent
question of estimating the support of the Fourier transform of a
signal, given its (possibly nonuniform) samples, is one that has been
studied extensively in the Compressive Sensing
community (starting with, e.g.,~\cite{Candes2006,Candes2006b}).

The top-down subdivision based scheme described in this chapter is not
the only way to construct multiscale dictionaries.  Followup work may
include one or more of the following ideas:
\begin{itemize}
\item Instead of estimating an optimal height $H$ during construction,
  simply prune a tree element $d^{(j,\alpha)}$ if its spectral
  concentration $\lambda^{(j,\alpha)}$ or concentration in $\cR$,
  $\nu^{(j,\alpha)}$, is below a minimum threshold.  This allows for
  more adaptive and better concentrated dictionary elements near
  high-curvature borders.
\item While the dictionaries described here describe ``summary''
  functions (for $\alpha=1$), it is possible to use Gram-Schmidt
  orthogonalization to construct an alternate ``difference''
  dictionary by orthogonalizing each node with its parent and
  sibling.  Such dictionaries would be better tuned to find ``edges'',
  and would provide sparser representations for mostly smooth data.
  In practice, this leads to better performance of $\ell_1$-based
  estimators like STL1D.
\item Other subdivision construction schemes should be considered.  For
  example, when the subregion $\cR$ is highly nonconvex (e.g., when
  $\cR$ is the interior of the Earth's oceans), even the second
  Slepian function contains more than one mode.  In this case, it is
  unclear how to subdivide the domain from the top down.  Instead, a
  bottom-up approach would work, wherein a fine grid is constructed on
  the region $\cR$, and grid elements are ``merged'' until their area
  is large enough that reasonably well concentrated Slepian functions
  with bandwidth $L$ will fit in them.
\end{itemize}

The ultimate goal of the constructions in this chapter is an
overcomplete multiscale \emph{frame} of bandlimited functions that are
well concentrated on $\cR$, can be constructed quickly, and admit fast
forward and inverse transforms.  That is, we seek a methodology
similar to the Wavelet transforms but allowing for bandlimits.  The work 
here should be considered a stepping stone in that direction as it
shares many of the properties of third generation Wavelets treated
elsewhere, especially the one that may be most important: numerical
accuracy in the solution of ill-posed inverse problems.

%\include{ch-conclusion/chapter-conclusion}

% Appendix?
\appendix % all chapters following will be labeled as appendices
\chapter{Differential Geometry: Definitions \label{app:diffgeom}}

The manifold $(\cM,g)$ with metric tensor $g$ admits an atlas of charts
$\set{(\phi_\alpha,U_\alpha)}_\alpha$, where $\phi_\alpha : U_\alpha
\to \bbR^d$ is a diffeomorphism from the open subset $U_\alpha \subset
\cM$ \cite[Chs. 0,4]{doCarmo1992}.
The choice of a standard orthonormal basis in $\bbR^d$ defines a
corresponding basis for the tangent plane to $\cM$ at each $x \in
U_\alpha$, as well as a local coordinate system $\set{v^j}$ near $x$.
The components of the tensor $g$ are defined as
the inner products between the partial derivatives of $\phi_\alpha$ in
this coordinate system:
$g_{ij}(x) = \ip{\partial \phi_\alpha(x) / \partial v^i}{\partial
\phi_\alpha(x)/ \partial v^j}$
where $v = \phi_\alpha(x)$. The inverse tensor,
denoted by $g^{ik}(x) = (g^{-1}(x))_{ik}$, is smooth, everywhere
symmetric and positive definite because $g$ is.  We use the notation
$\abs{g} = \det{(g_{ij})}$.  Note that we will often drop the position
term $x$; for example, $g^{-1} = g^{-1}(x)$.

We assign to $\cM$ the standard
gradient $\grad$, inner product $\cdot$, divergence $\grad \cdot$, and
Laplacian $\lap = \grad \cdot \grad$ at a point $x$
\cite[Ch. 1]{Rosenberg1997}.
For $f$ a twice-differentiable function on $\cM$ (i.e. $f \in C^2(\cM)$)
and $\bm{f}$ and $\bm{g}$ differentiable vector fields on $\cM$,
\begin{align*}
\partial_i f(x) = \frac{\partial f(x)}{\partial v^i}
\qquad (\grad f)^j = \sum_i g^{ij}(x) \partial_i f(x)
\qquad \bm{f} \cdot \bm{g} = \sum_{i,j} g_{ij}(x) \bm{f}^i(x) \bm{f}^j(x) \\
\grad \cdot \bm{f} = \sum_i \frac{1}{\sqrt{\abs{g(x)}}} \partial_i
    \left(\sqrt{\abs{g(x)}} \bm{f}^i(x) \right)
\qquad   \lap f = \sum_{i,j} \frac{1}{\sqrt{\abs{g(x)}}} \partial_i
    \left(\sqrt{\abs{g(x)}} g^{ij}(x) \partial_j f(x) \right),
\end{align*}
where again $v = \phi_\alpha(x)$ and $1 \leq i, j, k \leq d$.
We use the definitions above throughout, as well as the norm
definition $\norm{\bm{f}}^2 = \bm{f} \cdot \bm{f}$; thus
$\norm{\grad f}^2 = \sum_{i,j} g^{ij} \partial_i f \partial_j f$.

\chapter{Spectral Theory \label{app:fourier}}

In this appendix we state some basic facts about the existence
of the Fourier transform for functions in $L^2(\bbR^n)$.  We also
discuss the existence and properties of Fourier series representations
for functions in $L^2(\cM)$, where $(\cM,g)$ is a
Riemannian manifold with metric $g$ (see App. \ref{app:diffgeom}).  A
comprehensive review of Fourier transforms and convolutions on general
(possibly non-compact) spaces is available in \cite[Ch. 7]{Rudin1973},
and on Riemannian Manifolds in particular in
\cite[Ch. 1]{Rosenberg1997}.

\section{Fourier Transform on $\bbR^n$}

The Fourier transform of a function $f \in L^2(\bbR^n)$ is formally defined,
for all $\omega \in \bbR^n$, as
$$
\wh{f}(\omega) = \int_{x \in \bbR^n} f(x) \ol{\psi_\omega(x)} dx,
$$
where $\psi_\omega(x) = e^{i x \cdot \omega}$ solves the eigenvalue problem
$$
\Delta \psi_\omega(x) = -\lambda_\omega^2 \psi_\omega(x), \qquad x \in \bbR^n
$$
with $\Delta f(x) = \sum_{j=1}^n \frac{\partial^2}{\partial x_j^2} f(x)$ and $\lambda_\omega =
\norm{\omega}_2$.

The Fourier inversion theorem and its extensions \cite[Thms. 7.7 and
7.15]{Rudin1973} state that $\wh{f}$ is well defined, and that $f$
can be represented via $\wh{f}$.  For all $x \in \bbR^n$,
$$
f(x) = (2 \pi)^{-n} \int_{\omega \in \bbR^n} \wh{f}(\omega) \psi_\omega(x) dx.
$$

\section{Fourier Analysis on Riemannian Manifolds}

Suppose $\cH$ is a separable Hilbert space, e.g., $\cH = L^2(M)$ where
$M$ is a compact set.  Further, suppose $T : \cH \to \cH$ is a
bounded, compact, self-adjoint operator.  Then the Hilbert-Schmidt theorem
\cite[Thm. VI.16]{Reed1980} states that there is a complete
orthonormal basis $\set{\psi_k}$ for $\cH$ such that $T \psi_k(x) =
\lambda_k \psi_k(x)$, $k=0,1,\ldots$, and $\lambda_k \to 0$ as $k \to
\infty$.  The Laplacian operator $T = \Delta$ defined on compact
subsets of $M \subset \bbR^n$ therefore admits a complete orthonormal
basis (as its inverse is compact), and from now on we refer to
$\set{\psi_k}$ as the eigenfunctions of the Laplacian.  This set is
called the Fourier basis.

We can therefore write any function $f \in L^2(M)$ as
$$
f(x) = \sum_{k=0}^\infty \wh{f}_k \psi_k(x) \qquad \text{where}
\qquad
\wh{f}_k = \int_M f(x) \ol{\psi_k(x)} dx,
$$
where $\set{\wh{f}_k}$ are called the Fourier coefficients.  The
calculation of Fourier coefficients is called analysis, and the
reconstruction of $f$ by expansion in the Fourier basis is called
synthesis.

On a Riemannian manifold $(\cM,g)$ a similar analysis applies
\cite[Thm. 1.29]{Rosenberg1997}:

\begin{prop}[Hodge Theorem for Functions]
\label{prop:hodge}Let $(\cM,g)$ be a compact connected oriented
Riemannian manifold.  There exists an orthonormal basis of
$L^2(\cM)$ consisting of eigenfunctions of the Laplacian.  All the
eigenvalues are positive, except that zero is an eigenvalue with
multiplicity one.  Each eigenvalue has finite multiplicity, and the
eigenvalues accumulate only at infinity.
\end{prop}

The Laplacian given above is the negative of the Laplace-Beltrami
operator defined in \S\ref{app:diffgeom}, and the eigenfunctions have
Neumann boundary conditions.

Fourier analysis and synthesis can be written as
$$
f(x) = \sum_{k=0}^\infty \wh{f}_k \psi_k(x), x \in \cM \qquad
\text{where} \qquad
\wh{f}_k = \int_\cM f(x) \ol{\psi_k(x)} d\mu(x),
$$
where the analysis integral above is with respect to the volume
metric $g$.  Fourier analysis is a unitary operation (this is
known as Parseval's theorem).  Note that practical definitions of the
forward and inverse Fourier operators often differ by the choice and
placement of normalization constants, in order to simplify notation
(as is the case throughout this thesis).

\subsection{Fourier Analysis on the Circle $S^1$}

On the circle, $S^1$, parametrized by $\set{\theta : \theta \in
  [0,2\pi)}$ with $d\mu(\theta) = d\theta$, we have $\Delta e^{i
 \theta k} = - k^2 e^{i k \theta}$ for 
$k = 0, \pm 1, \pm 2, \ldots$.  We can therefore decompose $L^2(S^1)$ via
projections onto $\set{e^{i \theta k}}_k$.  That is, we can write
$$
f(\theta) = \sum_{k = -\infty}^\infty \wh{f}_k e^{ i \theta k} \qquad \text{where} \qquad
\wh{f}_k = \frac{1}{2 \pi} \int_{S^1} f(\theta) e^{- i \theta k} d\theta
$$
for any $f \in L^2(S^1)$.

The analysis, as defined, is an isometry.  By Parseval's theorem, for
$f, g \in L^2(S^1)$,
$$
\int_{S^1} f(\theta) \ol{g(\theta)} d\theta = \sum_{k=-\infty}^\infty \wh{f}_k \ol{\wh{g}_k}
$$

This analysis can also be applied to functions on $L^2[0,2\pi]$ with
periodic boundary conditions by identifying the interval with $S^1$.

\subsection{Fourier Analysis on the Sphere $S^2$\label{sec:fouriersphere}}

The unit sphere, $S^2$, can be parametrized by $\set{(\theta,\phi) : \theta
  \in [0,\pi], \phi \in [0,2\pi)}$ where $\theta$ is the colatitude and
$\phi$ is the longitude.  In this case, the Laplacian is
$$
\Delta f(\theta,\phi) = \frac{1}{\sin \theta} \partial_\theta
  \left(\sin \theta \partial_\theta f(\theta,\phi) \right)
  + \frac{1}{\sin^2 \theta} \partial_{\phi \phi} f(\theta,\phi)
$$
and the volume element is $d\mu(\theta,\phi) = \sin \theta d\phi d\theta$.

For our purposes, we are interested in real-valued functions on the
sphere.  As such, we study the \emph{real-valued} eigenfunctions of
the Laplacian on $S^2$.  The real surface spherical harmonics,
$\set{Y_{lm}(\theta,\phi)}$, are parametrized by
the degree $l$ and order $m$, where $l = 0, 1, \ldots$ and $m = -l,
\ldots, l$.  These can be given by \cite{Simons2006}:
\begin{align}
Y_{lm}(\theta,\phi) &=
\begin{cases}
\sqrt{2} X_{l\abs{m}}(\theta) \cos(m \theta) & -l \leq m < 0, \\
X_{l0}(\theta) & m = 0,\\
\sqrt{2} X_{lm}(\theta) \sin(m \theta) & 0 < m \leq l,
\end{cases} \qquad \text{where}\\
X_{lm}(\theta) &= (-1)^m \left( \frac{(2l+1)}{4 \pi} \frac{(l-m)!}{(l+m)!} \right)^{1/2}
P_{lm} (\cos \theta),
\end{align}
and $P_{lm}(t)$ are the associated Legendre functions of degree $l$
and degree $m$ \cite[\S8.1.1]{Abramowitz1964}.  Each spherical
harmonic $Y_{lm}$ fulfills the eigenvalue relationship $\Delta Y_{lm}
= -l(l+1) Y_{lm}$.

We can therefore decompose $L^2(S^2,\bbR)$ via
projections onto $\set{Y_{lm}(\theta,\phi)}_{lm}$, where now
$\ol{Y_{lm}} = Y_{lm}$.  That is, we can write
$$
f(\theta,\phi) = \sum_{lm} \wh{f}_{lm} Y_{lm}(\theta,\phi) \qquad \text{where} \qquad
\wh{f}_{lm} = \int_{S^2} f(\theta,\phi) Y^{m}_l(\theta,\phi) d\mu(\theta,\phi),
$$
for any $f \in L^2(S^2,\bbR)$.

The analysis, as defined, is an isometry.  By a
Parseval's theorem, for real-valued functions
$f, g \in L^2(S^2)$,
$$
\int_{S^2} f(\theta,\phi) g(\theta,\phi) d\mu(\theta,\phi) 
 = \sum_{l \geq 0} \sum_{m = -l}^l \wh{f}_{lm} \wh{g}_{lm}.
$$

%\include{ch-appendicies/implementation}
%\include{ch-appendicies/printing}

% Make the bibliography single spaced
\singlespacing
\bibliographystyle{siam}

% add the Bibliography to the Table of Contents
\cleardoublepage
\ifdefined\phantomsection
  \phantomsection  % makes hyperref recognize this section properly for pdf link
\else
\fi
\addcontentsline{toc}{chapter}{Bibliography}

% include your .bib file
\bibliography{thesis,paleo,respiration}

\end{document}